\documentclass[12pt,reqno]{amsart}

\usepackage{fullpage}
\usepackage{amsmath, amsthm, amsopn, amssymb, enumerate}

\usepackage{latexsym }
\usepackage{graphics}
\usepackage{graphicx}
\usepackage{verbatim} 
\usepackage{hyperref} 
\usepackage[mathscr]{eucal}

\newtheorem{thm}{Theorem}[section]
\newtheorem{lemma}[thm]{Lemma}

\newtheorem{conj}[thm]{Conjecture}
\newtheorem{prop}[thm]{Proposition}
\newtheorem{obs}[thm]{Observation}

\newtheorem*{claim*}{Claim}

\numberwithin{figure}{section}

\theoremstyle{definition}

\newtheorem{defn}[thm]{Definition}

\newtheorem{rmk}[thm]{Remark}

\newcommand{\ds}{\displaystyle}

\newcommand{\ol}{\overline}

\newcommand{\one}{\mathbf{1}}

\def\A{\mathcal{A}}
\def\B{\mathcal{B}}
\def\C{\mathcal{C}}
\def\D{\mathcal{D}}
\def\E{\mathcal{E}}
\def\F{\mathcal{F}}
\def\HH{\mathcal{H}}
\def\I{\mathcal{I}}

\def\K{\mathcal{K}}
\def\LL{\mathcal{L}}
\def\M{\mathcal{M}}
\def\n{\mathcal{N}}
\def\O{\mathcal{O}}

\def\Q{\mathcal{Q}}
\def\R{\mathcal{R}}

\def\T{\mathcal{T}}
\def\U{\mathcal{U}}
\def\V{\mathcal{V}}
\def\W{\mathcal{W}}
\def\X{\mathcal{X}}
\def\Y{\mathcal{Y}}
\def\Yh{\hat{\mathcal{Y}}}
\def\Z{\mathcal{Z}}

\def\Xb{\ol{X}}
\def\Yb{\ol{Y}}

\def\Xs{\ol{X}^*}
\def\Ys{\ol{Y}^*}
\def\Qs{Q^*}
\def\Xt{\tilde{X}}
\def\Yt{\tilde{Y}}
\def\Qt{\tilde{Q}}

\def\As{A^*}
\def\At{\tilde{A}}
\def\Nt{\tilde{N}}

\def\Ex{\mathbb{E}}

\def\N{\mathbb{N}}
\def\Pr{\mathbb{P}}

\def\RR{\mathbb{R}}
\def\XX{\mathbb{X}}
\def\YY{\mathbb{Y}}
\def\ZZ{\mathbb{Z}}

\def\a{\mathbf{a}}

\def\m{\mathbf{m}}

\def\le{\leqslant}
\def\ge{\geqslant}

\def\eps{\varepsilon}
\def\<{\langle}
\def\>{\rangle}
\def\Var{\textup{Var}}
\def\Cov{\textup{Cov}}
\def\Im{\textup{Im}}

\title{The triangle-free process and the Ramsey number $R(3,k)$}

\author{Gonzalo Fiz Pontiveros}

\author{Simon Griffiths}

\author{Robert Morris} 

 \address{
   Gonzalo Fiz Pontiveros, Simon Griffiths, Robert Morris  \hfill\break 
    IMPA, Estrada Dona Castorina 110, Jardim Bot\^anico, Rio de Janeiro, RJ, Brasil
 }
 \email{gf232|sgriff|rob@impa.br}

\thanks{Research supported in part by CNPq bolsas PDJ (GFP and SG) and by CNPq Proc.~479032/2012-2 and Proc.~303275/2013-8, and FAPERJ Proc.~201.598/2014 (RM)} 

\begin{document}

\begin{abstract}
The areas of Ramsey theory and random graphs have been closely linked ever since Erd\H{o}s' famous proof in 1947 that the `diagonal' Ramsey numbers $R(k)$ grow exponentially in $k$. In the early 1990s, the triangle-free process was introduced as a model which might potentially provide good lower bounds for the `off-diagonal' Ramsey numbers $R(3,k)$. In this model, edges of $K_n$ are introduced one-by-one at random and added to the graph if they do not create a triangle; the resulting final (random) graph is denoted $G_{n,\triangle}$. In 2009, Bohman succeeded in following this process for a positive fraction of its duration, and thus obtained a second proof of Kim's celebrated result that $R(3,k) = \Theta \big( k^2 / \log k \big)$. 

In this paper we improve the results of both Bohman and Kim, and follow the triangle-free process all the way to its asymptotic end. In particular, we shall prove that
$$e\big( G_{n,\triangle} \big) \,=\, \left( \frac{1}{2\sqrt{2}} + o(1) \right) n^{3/2} \sqrt{\log n },$$
with high probability as $n \to \infty$. We also obtain several pseudorandom properties of $G_{n,\triangle}$, and use them to bound its independence number, which gives as an immediate corollary
$$R(3,k) \, \ge \, \left( \frac{1}{4} - o(1) \right) \frac{k^2}{\log k}.$$
This significantly improves Kim's lower bound, and is within a factor of $4 + o(1)$ of the best known upper bound, proved by Shearer over 25 years ago.  
\end{abstract}

\maketitle

\section{Introduction}

For more than eighty years, since the seminal papers of Ramsey~\cite{Ramsey} and Erd\H{o}s and Szekeres~\cite{ESz}, the area now known as Ramsey theory has been of central importance in combinatorics. The subject may be summarised by the following mantra: ``Complete chaos is impossible!" or, more precisely (if less poetically), ``Every large system contains a well-ordered sub-system." The theory consists of a large number of deep and beautiful results, as well as some of the most important and intriguing open questions in combinatorics. These open problems have, over the decades, been a key catalyst in the development of several powerful techniques, most notably the Probabilistic Method, see~\cite{AS}. 

The archetypal Ramsey-type problem is that of bounding Ramsey numbers. The basic question is as follows: for which $n \in \N$ does it hold that every red-blue colouring of the edges of the complete graph $K_n$ contains either a red $K_k$  or a blue $K_\ell$? The Ramsey number, denoted $R(k,\ell)$, is defined to be the smallest such integer $n$. Shortly after Ramsey~\cite{Ramsey} proved that $R(k,\ell)$ is finite for every $k$ and $\ell$, Erd\H{o}s and Szekeres~\cite{ESz} proved in 1935 the explicit upper bound $R(k,\ell) \le {k + \ell - 1 \choose \ell - 1}$, which implies in particular that $R(k) := R(k,k) \le 4^k$, and that $R(3,k) = O(k^2)$. Despite an enormous amount of effort, the former bound has only recently been improved by a super-polynomial factor, by Conlon~\cite{Conlon}, who refined the earlier method of Thomason~\cite{T88}. A~constructive super-polynomial lower bound was not obtained until 1981, by Frankl and Wilson~\cite{FW}, whose beautiful proof used techniques from linear algebra, but an exponential lower bound was given already in 1947 by Erd\H{o}s~\cite{E47}, whose seminal idea (``Colour randomly!") initiated the study of probabilistic combinatorics. Over the past 65 years this bound has only been improved by a factor of 2, by Spencer~\cite{Sp75} in 1975, using the Lovasz Local Lemma~\cite{LLL}. In summary, the current state of knowledge in the `diagonal' case is as follows:
$$\bigg( \frac{\sqrt{2}}{e} - o(1) \bigg) k \cdot 2^{k/2} \, \le \, R(k) \, \le \, \exp\bigg( - \Omega(1) \frac{(\log k)^2}{\log \log k} \bigg) \cdot 4^k.$$

After the diagonal case, the next most extensively studied setting is the so-called `off-diagonal' case, $\ell = 3$, where much more is known. As noted above, it follows from the Erd\H{o}s--Szekeres bound mentioned above that $R(3,k) = O(k^2)$ and, in a visionary paper from 1961, Erd\H{o}s~\cite{E61} proved a  lower bound of order $k^2 / (\log k)^2$ by applying a deterministic algorithm to the random graph $G(n,p)$. An important breakthrough was obtained by Ajtai, Koml\'os and Szemer\'edi~\cite{AKSz80,AKSz81} in 1980, who proved that $R(3,k) = O(k^2 / \log k)$, and a little later by Shearer~\cite{Sh83}, who refined the method of~\cite{AKSz80} and obtained a much better constant. However, it was not until 1995 that a complimentary lower bound was obtained, in a famous paper of Kim~\cite{Kim}. We remark that the papers~\cite{AKSz81} and~\cite{Kim} were particular important, since they (respectively) introduced and greatly developed the so-called \emph{semi-random method}. 

More recently, Bohman~\cite{Boh} gave a new proof of Kim's result, using the triangle-free process (see below). In this paper, we shall significantly improve both of these results, following the triangle-free process to its (asymptotic) end, and proving a lower bound on $R(3,k)$ that is within a factor of $4 + o(1)$ of Shearer's bound. We remark that very similar results have recently been obtained independently by Bohman and Keevash~\cite{BK2} using related methods.


\subsection{Random graph processes}

The modern theory of random graph processes was initiated by Erd\H{o}s and R\'enyi~\cite{ER59,ER60} in 1959, who studied the evolution of the graph with edge set $\{e_1,\ldots,e_m\}$, where $(e_1,\ldots,e_N)$ is a (uniformly chosen) random permutation of $E(K_n)$. This model, now known as the Erd\H{o}s-R\'enyi random graph, is one of the most extensively studied objects in combinatorics, see~\cite{RG} or~\cite{JLR}. Two especially well-studied problems are the emergence of the `giant component' around $m = n/2$, see~\cite{Bela84,ER60,JKLP} (or, more recently, e.g.,~\cite{BJR,DKLP}), and the concentration (and location) of the chromatic number~\cite{Bela87,Bela88,Matula,SS} (or, more recently,~\cite{AN,Heckel}), where martingale techniques, which will play a important role in this work, were first used to study random graph processes.

In general, a \emph{random graph process} consists of a sequence of graphs $(G_0,G_1,\ldots)$, where the graph $G_m$ is chosen randomly according to some probability distribution which depends on the sequence $(G_0,\ldots,G_{m-1})$ (and often just on the graph $G_{m-1}$). The study of these objects has exploded in recent years, as the ubiquity of `random-like' graphs in nature has come to the attention of the scientific community. Particularly well-studied processes include the `preferential attachment' models of Barab\'asi and Albert~\cite{BA}, the so-called `Achlioptas processes' introduced by Achlioptas in 2000 (see, e.g.,~\cite{ASS,BohSci,BF,BKr,SW}), and studied most notably by Riordan and Warnke~\cite{RW11,RW12}. Random processes also played a key role in the breakthrough results of Johansson, Kahn and Vu~\cite{JKV} on the threshold for $H$-factors in random graphs, and of Keevash~\cite{Kee} on the existence of designs. 

A technique that has proved extremely useful in the study of random graph processes is the so-called `differential equations method'. In this method, whose application to random graphs was pioneered in the 1990s by Ruci\'nski and Wormald~\cite{RW1,RW2,Worm95,Worm99}, the idea is to `track' a collection of graph parameters, by showing that (with high probability) they closely follow the solution of a corresponding family of differential equations. This method has been used with great success in recent years (see for example~\cite{BBFP,Boh,BFL,BK}); we note in particular the recent result of Bohman, Frieze and Lubetzky~\cite{BFL} that the so-called `triangle removal process' ends with $n^{3/2 + o(1)}$ edges.

\subsection{The triangle-free process}

Consider the following random graph process $(G_m)_{m \in \N}$ on vertex set $[n] = \{1,\ldots,n\}$. Let $G_0$ be the empty graph and, for each $m \in \N$, let $G_m$ be obtained from $G_{m-1}$ by adding a single edge, chosen uniformly from those non-edges of $G_{m-1}$ which do not create a triangle. The process ends when we reach a maximal triangle-free graph; we denote by $G_{n,\triangle}$ this (random) final graph. 

The triangle-free process was first suggested by Bollob\'as and Erd\H{o}s at the ``Quo Vadis, Graph Theory?" conference in 1990. The main reason for the introduction of the triangle-free process (and the more general $H$-free process) was the hope that $G_{n,\triangle}$ may give a good lower bound for $R(3,k)$. As a first step, Bollob\'as and Erd\H{o}s asked for the size of the final graph, and made some simple observations about the triangle-free and $C_4$-free processes~\cite{Bela}. The first non-trivial results about these processes were obtained by Erd\H{o}s, Suen and Winkler~\cite{ESW}, who showed that, with high probability, $e(G_{n,\triangle}) \ge c n^{3/2}$ edges for some constant $c > 0$. (Throughout the paper, we write `with high probability' to mean with probability tending to 1 as $n \to \infty$, where $n = |V(G_{n,\triangle})|$ is the size of the vertex set in the triangle-free process.\footnote{In practice, the probabilities of our bad events will all be at most $n^{-\log n}$, see for example Theorem~\ref{finalthm}.}) 

Determining the order of magnitude of $e(G_{n,\triangle})$ remained an open problem for nearly 20 years until the breakthrough paper of Bohman~\cite{Boh}, who followed the triangle-free process for a constant proportion of its lifespan, and hence proved that 
\begin{equation}\label{eq:Bohthm}
e\big( G_{n,\triangle} \big) \, = \, \Theta\big( n^{3/2} \sqrt{\log n} \big).
\end{equation}
Shortly afterwards, Bohman and Keevash~\cite{BK} extended and generalized the method of~\cite{Boh} to the setting of the $H$-free process, where the triangle is replaced by an arbitrary `forbidden' graph $H$. Improving on earlier results of Bollob\'as and Riordan~\cite{BR00} and Osthus and Taraz~\cite{OT}, they proved that for any strictly balanced\footnote{$m_2(H) := \ds\max_{F \subseteq H}$ $\frac{e(F) - 1}{v(F) - 2}$, and a graph is said to be strictly balanced if $\frac{e(F) - 1}{v(F) - 2} < \frac{e(H) - 1}{v(H) - 2}$ for every $F \subsetneq H$.} graph $H$, the number of edges in the final graph $G_{n,H}$ satisfies
$$e\big( G_{n,H} \big) \, \ge \, c \cdot n^{2 - 1 / m_2(H)} (\log n)^{1 / (e(H) - 1)}$$
for some constant $c > 0$ which depends on $H$, which is conjectured to be within a constant factor of the truth. (See~\cite{Pic,Lutz2,Lutz3} for upper bounds in the cases $H = K_4$ and $H = C_\ell$.) Although in this paper we shall focus on the case $H = K_3$, we believe that our methods can also be applied in the general setting, and we plan to return to this topic in a future work.

We shall follow the triangle-free process until $o\big( n^{3/2} \sqrt{\log n} \big)$ steps from the end, and hence obtain a sharp version of Bohman's theorem. Our first main result is as follows.

\begin{thm}\label{triangle}
$$e\big( G_{n,\triangle} \big) \,=\, \left( \frac{1}{2\sqrt{2}} + o(1) \right) n^{3/2} \sqrt{\log n},$$
with high probability as $n \to \infty$.
\end{thm}

We shall moreover control various parameters associated with the graph process, showing that they take the values one would expect in a random graph of the same density. Thus, one may morally consider the main result of this paper to be the following imprecise statement: ``For all $m \le (1 + o(1)) e( G_{n,\triangle} )$, the graph $G_m$ closely resembles the Erd\H{o}s-R\'enyi random graph $G(n,m)$, except for the fact that it has no triangles". One significant consequence of this pseudo-theorem (or rather, of the precise theorems stated below), is that we can bound (with high probability) the independence number of $G_{n,\triangle}$, and hence obtain the following improvement of Kim's lower bound on $R(3,k)$. 

\begin{thm}\label{R3k}
$$\left( \frac{1}{4} - o(1) \right) \frac{k^2}{\log k} \, \le \, R(3,k) \, \le \, \big( 1 \pm o(1) \big) \frac{k^2}{\log k}$$
as $k \to \infty$.
\end{thm}

We repeat, for emphasis, that the upper bound in Theorem~\ref{R3k} was proved by Shearer~\cite{Sh83} over 25 years ago. We remark that the factor of $4 + o(1)$ difference between the two bounds appears to have two separate sources, each of which contributes a factor of two. To see this, recall that Shearer proved that
\begin{equation}\label{eq:Shearer}
\alpha(G) \, \ge \, \big( 1 - o(1) \big) \frac{n \log d}{d}
\end{equation}
for any triangle-free $n$-vertex graph with average degree $d$, which implies the bound stated above, since the independence number of such a graph is clearly also at least $d$. Our results (see Theorem~\ref{thm:indepsets}, below) imply that~\eqref{eq:Shearer} is within a factor of two of being best possible (in the critical range), and we suspect that $G_{n,\triangle}$ is asymptotically extremal. Moreover, the independence number of $G_{n,\triangle}$ is (perhaps surprisingly) roughly twice its maximum degree, rather than asymptotically equal to it. 
We conjecture that our bound is in fact sharp.

\begin{conj}
$$R(3,k) \, = \, \left( \frac{1}{4} + o(1) \right) \frac{k^2}{\log k}$$
as $k \to \infty$.
\end{conj}

We remark that, as in the diagonal case, the best known constructive lower bound for $R(3,k)$ is far from the truth. To be precise, Alon~\cite{Alon} constructed a triangle-free graph with $\Theta\big( k^{3/2} \big)$ vertices and no independent set of size $k$. 

We finish this section by giving a rough outline of the proof of Theorems~\ref{triangle} and~\ref{R3k}; a more extensive sketch (together with several of our other main results) is given in Section~\ref{sketchSec}. Our plan (speaking very generally) is to use the differential equations method, exploiting the `self-correcting' nature of the triangle-free process in order to prove bounds on the various parameters we will track that become \emph{tighter} as the process progresses. However, we can assure the reader who is intimidated by this technique that no actual differential equations will be needed in the proof. Moreover, apart from the use of a martingale concentration inequality due to Freedman~\cite{F75}, our proof will be completely self-contained. We remark that Telcs, Wormald and Zhou~\cite{TWZ} were the first to make use of self-correction while applying the differential equations method in combinatorics; other early applications were obtained by Bohman and Picollelli~\cite{BPi}, and by Bohman, Frieze and Lubetsky~\cite{BFL10,BFL}. 

The basic idea is to `track' a large collection of graph parameters, such that the (expected) rate of change of each depends only on some (generally small) subset of the others. We shall show that, for each of these parameters, the probability that it is the \emph{first} parameter to go astray (that is, to have normalized error larger than~$1$) is extremely small; the lower bound in Theorem~\ref{triangle} then follows by applying the union bound to our family of events. To prove Theorem~\ref{R3k} and the upper bound in Theorem~\ref{triangle}, we assume that these parameters were all tracking up to step $m^*$ (the point at which we lose control of the process), and show (by bounding the probability of some carefully chosen events) that 
$$\alpha\big( G_{n,\triangle} \big) \, \le \, \alpha\big( G_{m^*} \big) \, \le \, \big( \sqrt{2} + o(1) \big) \sqrt{n \log n}$$ 
with high probability, and that the maximum degree of $G_{m}$ is unlikely to increase by more than $o\big(  \sqrt{n \log n} \big)$ between step $m^*$ and the end of the process.

The most basic parameter we shall need to track is the number of \emph{open edges}, where $e \in E(K_n) \setminus E(G_m)$ is said to be \emph{open} if its endpoints have no common neighbour. To be precise, we will write 
\begin{equation}\label{def:O(Gm)}
O(G_m) \, = \, \big\{ e \in E(K_n) \setminus E(G_m) \,:\, e \not\subseteq N_{G_m}(v) \textup{ for every } v \in V(G_m) \big\}
\end{equation}
for the set of open edges, and $Q(m) = |O(G_m)|$ for their number. Observe that the open edges of $G_m$ are exactly those that can be added to the graph at step $m+1$, and that the process therefore ends exactly when $Q(m) = 0$. In order to control the evolution of $Q(m)$, we will need to track the family of random variables $\big\{ Y_e(m) : e \in O(G_m) \big\}$, where $Y_e(m)$ denotes the number of edges that are closed if $e$ is added to the graph at step $m+1$ (see Definition~\ref{Ydef}, below), and in order to control these variables we will also need to control the family $\big\{ X_e(m) : e \in O(G_m) \big\}$, where $X_e(m)$ denotes the number of open edges $f$ such that $e$ and $f$ form two sides of a triangle in $O(G_m)$ (see Definition~\ref{Xdef}, below). 
It turns out (see Section~\ref{sketchSec}) that the collection
$$Q(m) \cup \big\{ X_e(m) : e \in O(G_m) \big\} \cup \big\{ Y_e(m) : e \in O(G_m) \big\}$$
forms a `closed' system, in the sense that the expected change of each variable in step $m+1$ depends only on the collection at step $m$, and for this reason Bohman~\cite{Boh} was able to track each of these variables up to a small (but rapidly growing) absolute error.

In order to prove Theorem~\ref{triangle}, we shall need to control these parameters up to a much smaller absolute error; in fact the error term we need \emph{decreases super-exponentially quickly} in $t = m \cdot n^{-3/2}$. In order to obtain such a tiny error, we shall exploit the self-correcting nature of the triangle-free process; doing so requires three separate steps, which are all new and quite different from one another, and each of which relies crucially on the other two.

First, we show that $Q(m)$ evolves (randomly) with $\Xb(m)$ and $\Yb(m)$ (the \emph{averages} over all open edges of $G_m$ of the variables $X_e(m)$ and $Y_e(m)$, respectively) according to a `whirlpool-like' structure (see Section~\ref{XYQsec}). Using a suitably chosen Lyapunov function, we are able to show that this three-dimensional system is self-correcting, even though $Q(m)$ itself is not. 

Second, for each integer $1 \le k \le 3/\eps$ and open edge $e \in O(G_m)$, we track a variable $V_e^{(k)}(m)$, which is (roughly speaking) the $k^{th}$ derivative of $Y_e(m)$. To define this variable, consider for each $m \in \N$ the graph (the `$Y$-graph' of $G_m$) with vertex set $O(G_m)$, and an edge between each pair $\{f,f'\}$ such that $f' \in Y_f(m)$; then $V_e^{(k)}(m)$ is the weighted average\footnote{The weight of an edge $f$ in $V_e^{(k)}(m)$ is equal to the number of walks of length $k$ from $e$ to $f$.} of $Y_f(m)$ over the edges $f \in O(G_m)$ at walk-distance $k$ from $e$ in the $Y$-graph. 
Crucially, our error bounds on these variables decrease exponentially quickly in $k$, and using this fact we shall be able to prove self-correction. A vital ingredient in this calculation amounts to showing that a random walk on the $Y$-graph mixes in constant time (see Sections~\ref{MixSec1} and~\ref{MixSec2}), and the proof of this property of the $Y$-graph uses the fact that we can track certain `ladder-like' graph structures in $G_m$. 

Finally, in order to control the number of `ladder-like' structures, we shall in fact track the number of copies of \emph{every} graph structure $F$ which occurs in $G_m$ (at a given `root'), up to the point at which it is likely to disappear, and after this time we shall bound the number of copies up to a polylog-factor. Such a general result is not only interesting in its own right; it is necessary for our proof to work, because (for our martingale bounds) we need to track the maximum possible number of copies of each structure that are created or destroyed in a single step of the triangle-free process, which depends on (the number of copies of) several other structures, some of which may be tracking, and others not. This part of the proof is extremely technical, and making it work requires several non-trivial new ideas; we mention in particular the `building sequences' introduced in Section~\ref{BSsec}. 

The rest of the paper is organised as follows. In Section~\ref{sketchSec} we give an overview of the proofs of Theorems~\ref{triangle} and~\ref{R3k}, and state various more detailed results about the structures that occur in the graph $G_m$. In Section~\ref{MartSec} we introduce our martingale method, and use it to track the variables $X_e$, which form a particularly simple special case. In Sections~\ref{EEsec}--\ref{XYQsec} we prove the lower bound in Theorem~\ref{triangle}; more precisely, in Section~\ref{EEsec} we study general graph structures, in Section~\ref{Ysec} we control the variables $Y_e(m)$ and $V_e^{(k)}(m)$, and in Section~\ref{XYQsec} we track $\Xb(m)$, $\Yb(m)$ and $Q(m)$. Finally, in Section~\ref{indepSec}, we study the independent sets in $G_{n,\triangle}$, and deduce the upper bound in Theorem~\ref{triangle} and the lower bound in Theorem~\ref{R3k}. To avoid distracting the reader from the key ideas with too much clutter, we postpone a few of the more tedious calculations to an Appendix~\cite{App}.

\section{An overview of the proof}\label{sketchSec}

In this section we shall lay the foundations necessary for the proofs of Theorems~\ref{triangle} and~\ref{R3k}. In particular, we shall formally introduce the various families of variables which we will need to track, and state some of our key results about these variables. We shall  attempt to give the reader a bird's-eye view of how the various components of the proofs of our main theorems fit together, whilst simultaneously introducing various notations and conventions which we shall use throughout the paper, often without further comment. We strongly encourage the reader, before plunging into the details of Sections~\ref{EEsec}--\ref{indepSec}, to read carefully both this section, and Section~\ref{SecX}. In what follows, the main heuristic to keep in mind at all times is that $G_m$ should closely resemble the Erd\H{o}s-R\'enyi random graph $G(n,m)$, except in the fact that it contains no triangles. 

We begin by choosing some parameters \emph{which will be fixed throughout the paper}. Let $\eps > 0$ be an arbitrary, sufficiently small constant. Given $\eps$, we choose a sufficiently large constant $C = C(\eps) > 0$, and a function $\omega = \omega(n)$ which goes to infinity sufficiently slowly\footnote{In particular, we could set $\omega(n) = \log\log\log n$.} as $n \to \infty$. We shall assume throughout the paper that $n$ (and hence also $\omega(n)$) is sufficiently large.


For each $m \in \N$, let the \emph{time} $t$ after $m$ steps of the triangle-free process be defined by $m = t \cdot n^{3/2}$. (We shall use this convention throughout the paper without further comment.) Thus, setting
\begin{equation}\label{def:m^*t^*}
t^* := \bigg( \frac{1}{2\sqrt{2}} - \eps \bigg) \sqrt{\log n} \qquad \text{and} \qquad m^* := t^* \cdot n^{3/2},
\end{equation}
our aim is to follow the triangle-free process up to time $t^*$. 

Recall that we write $Q(m) = |O(G_m)|$ for the number of open edges of $G_m$, see~\eqref{def:O(Gm)}. The triangle-free process ends when $Q(m) = 0$, since only open edges of $G_m$ may be added in step $m + 1$. The lower bound in Theorem~\ref{triangle} is therefore an immediate corollary of the following theorem.\footnote{Here, and throughout the paper we write $a \in b \pm c$ to mean $b - c \le a \le b + c$.}

\begin{thm}\label{Qthm}
With high probability,   
\begin{equation}\label{Qtrack}
Q(m) \, \in \, e^{-4t^2} {n \choose 2} \pm e^{-2t^2} n^{7/4} (\log n)^3
\end{equation}
for every $m \le \big( \frac{1}{2\sqrt{2}} - \eps \big) n^{3/2} \sqrt{\log n}$.
\end{thm}

Note that if $t \le t^*$ then $e^{2t^2} \le n^{1/4 - \eps}$, and hence it follows from~\eqref{Qtrack} that $Q(m) > 0$ for every $m \le m^*$. The lower bound in Theorem~\ref{triangle} follows, since $\eps > 0$ was arbitrary.

To see why $Q(m)$ should decay at rate $e^{-4t^2}$, simply recall that $G_m$ resembles a random graph of density $p \approx 2m/n^2$; it follows that the proportion of open edges in $G_m$ should be roughly $(1 - p^2)^n \approx e^{-p^2 n} \approx e^{-4t^2}$. Note that the absolute error in our bound~\eqref{Qtrack} actually decreases over time; this is possible because we exploit the self-correcting nature of the process. We shall denote the (maximum allowed) relative error by
$$g_q(t) \, = \, e^{2t^2} n^{-1/4} (\log n)^3,$$
and say that $Q$ is \emph{tracking} up to step $m'$ if $Q(m) \in \big( 1 \pm g_q(t) \big) e^{-4t^2} {n \choose 2}$ holds for all $m \le m'$. Note that, by~\eqref{def:m^*t^*}, we have $g_q(t) \le n^{-\eps}$ for every $t \le t^*$. We remark that we expect (but do not prove) that this error term is best possible up to the polylog-factor. 

For each of our graph parameters, the first step will always be to determine the expected rate of change of that parameter, conditioned on the past. To simplify the notation a little, let us define, for any graph parameter $A$ and any graph $G$, a function $\Delta A(G) \colon E(K_n) \to \RR$ by
$$\Delta A(G)(e) \mapsto A(G \cup \{e\}) - A(G).$$
Thus, setting $A(m) := A(G_m)$, it follows that 
$$\Ex\big[ \Delta A(m) \big] \,=\, \Ex\big[ A(m+1) - A(m) \,\big|\, G_m \big],$$
where the expectation in both cases is over the uniformly random open edge of $G_m$ chosen in step $m+1$ of the triangle-free process. We emphasize that our process is Markovian, and thus all of the relevant information about the past (i.e., all the information needed to determine the distribution of the sequence $G_{m+1}, G_{m+2}, \ldots$) is encoded in the graph $G_m$.

The value of $\Ex\big[ \Delta Q(m) \big]$ is controlled by the following family of parameters, which were already introduced informally above. The variables $\big\{ Y_e(m) : e \in O(G_m) \big\}$ determine (the distribution of) the number of edges which are closed at each step.

\begin{defn}\label{Ydef}
We say that two open edges $e,f \in O(G_m)$ are \emph{$Y$-neighbours} in $G_m$ if $e$ and $f$ form two sides of a triangle, the third of which is in $E(G_m)$. For each edge $e \in E(K_n)$ and each $m \in \N$, define 
$$Y_e(m) \, := \, \Big| \Big\{ f \in O(G_m) \,:\, f \textup{ is a $Y$-neighbour of $e$ in $G_m$} \Big\} \Big|$$
if $e \in O(G_m)$, and set $Y_e(m) = Y_e(m-1)$ otherwise.
\end{defn}

We remark that, in order to ease our notation, we shall sometimes write $Q(m)$ and $Y_e(m)$ for (respectively) the \emph{collection} of open edges and the \emph{collection} of $Y$-neighbours of $e$ in $G_m$. (In fact, we shall do the same for all of the various variables we define below.) We trust that this slight abuse of notation will not cause any confusion. 

\begin{figure}[h]
\includegraphics[scale=1] {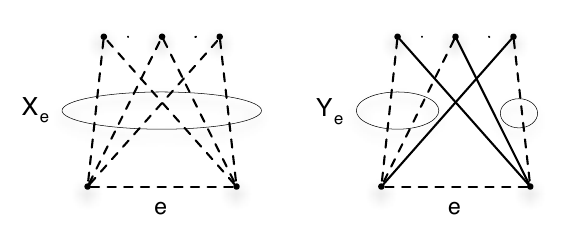}
\caption{The sets $X_e$ and $Y_e$.}
\end{figure}

Assuming once again that $G_m$ is well-approximated by $G(n,m)$, it follows that we would expect $Y_e(m)$ to be asymptotically equal to $2pn \cdot e^{-4t^2} \approx 4t e^{-4t^2} \sqrt{n}$ for every open edge $e \in O(G_m)$, and every $m \le m^*$. Showing that this is true turns out to be the most difficult part of our proof, and will be the key intermediate result in the proof of Theorem~\ref{triangle}. To begin, note that $f \in Y_e(m)$ if and only if $e \in Y_f(m)$, and define 
$$\Yb(m) \, = \, \frac{1}{Q(m)} \sum_{e \in Q(m)} Y_e(m),$$
the average number of $Y$-neighbours of an open edge in $G_m$. The following simple equation governs the expected change in $Q(m)$ at each step:
$$\Ex \big[ \Delta Q(m) \big] \, = \, - \, \Yb(m) - 1.$$
Thus, in order to control $Q(m)$, we shall need to track $\Yb(m)$, which is controlled (see Section~\ref{XYQsec}) by the following equation\footnote{More precisely, we shall show that these bounds hold for each $\omega \cdot n^{3/2} < m \le m^*$, as long as the other variables are still tracking.}:
\begin{equation}\label{eq:Yb}
\Ex \big[ \Delta \Yb(m) \big] \, \in \,  \frac{- \Yb(m)^2 + \Xb(m) - 2 \cdot \Var\big(Y_e(m)\big) \pm  O\big( \Yb(m) \big)}{Q(m)}.
\end{equation}
Here the variance is over the choice of a (uniformly) random open edge $e \in O(G_m)$. The term $\Xb(m)$ which appears in~\eqref{eq:Yb} is defined, similarly as for the $Y_e$, to be the average of the following collection of random variables:

\begin{defn}\label{Xdef}
We say that two open edges $e,f \in O(G_m)$ are \emph{$X$-neighbours} in $G_m$ if they form two sides of a triangle, the third of which is  also in $O(G_m)$. For each edge $e \in E(K_n)$ and each $m \in \N$, define 
$$X_e(m) \, := \, \Big| \Big\{ f \in E(K_n) \,:\, f \textup{ is an $X$-neighbour of $e$ in $G_m$} \Big\} \Big|$$
if $e \in O(G_m)$, and set $X_e(m) = X_e(m-1)$ otherwise.
\end{defn}

Appealing again to our random graph intuition, we would expect (and shall prove) that $X_e(m) \approx 2 e^{-8t^2} n$ for every open edge $e \in O(G_m)$. We define
$$\Xb(m) \, = \, \frac{1}{Q(m)} \sum_{e \in Q(m)} X_e(m),$$
the average number of $X$-neighbours of an open edge in $G_m$. The random variable $\Xb(m)$ is controlled (as long as the other variables are still tracking) by the following equation (see Section~\ref{XYQsec}):
$$\Ex \big[ \Delta \Xb(m) \big]  \, \in \, \frac{- \,2 \cdot \Xb(m) \Yb(m) - 3 \cdot \Cov\big( X_e(m), Y_e(m) \big)}{Q(m)}  \, \pm \, O\left( \frac{\Xb(m) + \Yb(m)^2}{Q(m)} \right).$$
Here, as in~\eqref{eq:Yb}, the covariance is over the choice of a (uniformly) random open edge $e \in O(G_m)$. Combining these equations, we will be able to show that the normalized errors of $\Yb$ and $Q$ form a stable two-dimensional system, and that $\Xb$ is self-correcting (see Section~\ref{WhirlSec}). Hence, by applying our martingale method (see Section~\ref{MartSec}) to a suitably-chosen Lyapunov function, we shall be able to prove both Theorem~\ref{Qthm} and the following bounds. 

\begin{thm}\label{XbYbthm}
With high probability,   
$$\Xb(m) \, \in \, \bigg( 1 \pm \frac{e^{2t^2} (\log n)^3}{n^{1/4}}  \bigg) \cdot  2 e^{-8t^2} n \quad \textup{and} \quad \Yb(m) \, \in \,  \bigg( 1 \pm \frac{e^{2t^2} (\log n)^3}{n^{1/4}}  \bigg) \cdot 4t e^{-4t^2} \sqrt{n}$$
for every $\omega \cdot n^{3/2} < m \le \big( \frac{1}{2\sqrt{2}} - \eps \big) n^{3/2} \sqrt{\log n}$.
\end{thm}

Note also that once again, although the relative error is increasing with $t$, the \emph{absolute} error is decreasing super-exponentially quickly. We shall write
$$\Qt(m) = e^{-4t^2} {n \choose 2}, \quad \Xt(m) = 2e^{-8t^2} n \quad \text{and} \quad \Yt(m) = 4t e^{-4t^2} \sqrt{n}$$
to denote the paths that we expect $Q(m)$, $\Xb(m)$ and $\Yb(m)$ to follow.

The alert reader will have noticed that in order to prove Theorems~\ref{Qthm} and~\ref{XbYbthm}, we are going to need some bounds on $\Var(Y)$ and $\Cov(X,Y)$. In fact, proving such bounds turns out to be the main obstacle in the proofs of Theorems~\ref{triangle} and~\ref{R3k}; more precisely, our main problem will be controlling the variables $Y_e(m)$, which are the key to the process. Our key intermediate result will therefore be the following bounds on $Y_e(m)$. 

\begin{thm}\label{Ythm}
With high probability,  
$$Y_e(m) \, \in \, 4t e^{-4t^2} \sqrt{n} \,\pm\, n^{1/4} (\log n)^3$$
for every open edge $e \in O(G_m)$ and every $m \le \omega \cdot n^{3/2}$, and
\begin{equation}\label{eq:Yevent}
Y_e(m) \, \in \, \Big( 1 \pm e^{2t^2} n^{-1/4} (\log n)^4 \Big) \cdot 4t e^{-4t^2} \sqrt{n}
\end{equation}
for every open edge $e \in O(G_m)$ and every $\omega \cdot n^{3/2} < m \le \big( \frac{1}{2\sqrt{2}} - \eps \big) n^{3/2} \sqrt{\log n}$.
\end{thm}

Let us denote the relative error in~\eqref{eq:Yevent} by
$$g_y(t) \, = \, e^{2t^2} n^{-1/4} (\log n)^4,$$
and observe that $g_y(t) \gg g_q(t)$ and that $g_y(t)^2 \Yt(m) \approx 1$ (up to a polylog factor). It seems plausible that the distribution of the $Y_e(m)$ might be `well-approximated' (in some sense) by a collection of independent Gaussians, each centred at $\Yt(m)$; however, our current techniques seem quite far from being able to prove such a strong statement. 

The proof of Theorem~\ref{Ythm} is somewhat intricate; we shall sketch here just the basic ideas. Observe first that the variable $Y_e(m)$ is governed by the equation\footnote{In fact this is not quite correct: each variable $Y_f(m)$ should be replaced by $Y_f(m) - 1$, since if the edge~$e$ is chosen in step $m+1$ of the triangle-free process then $\Delta Y_e(m) = 0$, see Definition~\ref{Ydef}. In the interest of simplifying the presentation of the overview, we shall for the time being ignore such rounding errors, as is often done in the case of ceiling and floor symbols.}
\begin{equation}\label{eq:Yeq}
\Ex \big[ \Delta Y_e(m) \big] \, = \, \frac{1}{Q(m)} \bigg( - \sum_{f \in Y_e(m)} Y_f(m) + X_e(m) \bigg).
\end{equation}
We remark that the main term (when $t$ is large) is the first one, since $\Yt(m)^2 \gg \Xt(m)$. Note that the rate of change of $Y_e(m)$ depends on itself, but also on the $Y$-values of its $Y$-neighbours. Thus, although it appears that $Y_e(m)$ should be self-correcting (since larger values of $Y_e(m)$ produce a more negative first derivative), this effect can be outweighed for a specific edge $e$ if the $Y$-values of the edges in $Y_e(m)$ are unusually large or small. 

For this reason, it is necessary to introduce a variable (for each open edge $e \in O(G_m)$) which counts the number of walks of length two in the $Y$-graph\footnote{Recall that this is the graph whose vertices are the open edges in $G_m$, and whose edges are pairs of $Y$-neighbours.}. However, in order to control this variable we shall need to track the number of walks of length three, and so on. This leads naturally to the following definition.

\begin{defn}
For each $k \in \N$, and each open edge $e \in O(G_m)$, let $U_e^{(k)}(m)$ denote the number of walks in the $Y$-graph of length $k$, starting from $e$, and set
$$V_e^{(k)}(m) \, = \, \frac{1}{U_e^{(k)}(m)} \sum_{f_1 \in Y_e(m)} \sum_{f_2 \in Y_{f_1}(m)} \dots \sum_{f_k \in Y_{f_{k-1}}(m)} Y_{f_k}(m),$$
the average of the $Y$-values reached via such walks. Moreover, set $V_e^{(0)}(m) = Y_e(m)$.
\end{defn}

The second statement in Theorem~\ref{Ythm} is the case $k = 0$ of the following theorem. 

\begin{thm}\label{Vthm}
For each $0 \le k \le \lceil 3 / \eps \rceil$, with high probability we have 
$$V_e^{(k)}(m) \, \in \, \Big( 1 \pm \eps^k \cdot e^{2t^2} n^{-1/4} (\log n)^4 \Big) \cdot 4t e^{-4t^2} \sqrt{n}$$
for every open edge $e \in O(G_m)$ and every $\omega \cdot n^{3/2} < m \le \big( \frac{1}{2\sqrt{2}} - \eps \big) n^{3/2} \sqrt{\log n}$. 
\end{thm}

In fact, in order to prove Theorem~\ref{Vthm}, we shall need to define and control (for each $k$) a more refined collection of variables, $\big\{ V_e^\sigma : \sigma \in \{L,R\}^k \big\}$, which takes into account whether each step of a walk in the $Y$-graph was taken with the `left' or `right' foot; that is, which endpoint  changed in the step from $e \in O(G_m)$ to $f \in Y_e(m)$, and which endpoint was common to~$e$ and~$f$. This is important because a key step in our proof will be to show that, for $k$ sufficiently large, after $k^2$ (random) steps in the $Y$-graph we are (in a certain sense) `well-mixed'. More precisely, we shall show that if we have `changed feet' at least $k$ times on this walk then the value of $Y_f(m)$ we reach may be well-approximated by $\Yb(m)$. On the other hand, if at some point we took $k$ consecutive steps with the same foot, with the other foot fixed at vertex $u$, then after these $k$ steps we will have (approximately) reached a uniformly-random open neighbour of~$u$. We refer the reader to Section~\ref{Ysec} for the details.

In order to prove our mixing results on the $Y$-graph, we shall need good bounds on the number of copies of certain structures in $G_m$, which correspond to paths in the $Y$-graph with either many or no changes of foot. In order to obtain such bounds, we shall in fact need to bound the number of copies of \emph{every} structure which can occur in $G_m$. We make the following definition.

\begin{defn}\label{def:structure}
A \emph{graph structure} $F$ consists of a set of (labelled) vertices $V(F)$, edges $E(F) \subseteq {V(F) \choose 2}$ and open edges $O(F)  \subseteq {V(F) \choose 2}$, where $E(F)$ and $O(F)$ are disjoint. 

Such a structure is said to be \emph{permissible} if every triangle in $E(F) \cup O(F)$ contains at least two edges of $O(F)$. 
\end{defn}

Note that structures that are not permissible cannot occur in $G_m$, since triangles with at most one open edge do not exist in $E(G_m) \cup O(G_m)$. We shall be interested in the number of copies of a graph structure $F$ `rooted' at a certain set of vertices. The following definition allows us to restrict our attention to those embeddings for which there (potentially) exists at least one copy of $F$. 

\begin{defn}\label{def:faithful}
Given a graph structure $F$ and an \emph{independent}\footnote{We will later need to extend this definition, and those below, to the case where $A$ is not an independent set in $F$ by removing the edges and open edges from $F[A]$. Hence, writing $\hat{F}^A$ for the structure thus obtained, we have $\Nt_A(F) := \Nt_A(\hat{F}^A)$, $t_A(F) := t_A(\hat{F}^A)$, and so on, see Section~\ref{BSsec}.} set $A \subseteq V(F)$, we say that an injective map $\phi \colon A \to V(G_m)$ is \emph{faithful} at time $t$ if the graph structure on $V(F)$ with open edge set $O(F)$ and edge set 
$$E(F) \cup \phi^{-1}\big( E\big( G_m[\phi(A)] \big) \big)$$
is permissible.
\end{defn}

Let us refer to a pair $(F,A)$ such that $F$ is a permissible graph structure and $A \subseteq V(F)$ is an independent set, as a \emph{graph structure pair}. If we are also given an injective map $\phi \colon A \to V(G_m)$, then we will refer to $(F,A,\phi)$ as a \emph{graph structure triple}. Given such a triple, if $\phi$ is faithful then we define 
$$N_\phi(F)(m) \, := \, \Big| \Big\{ \psi \colon V(F) \to V(G_m) \,:\, \psi \textup{ is an injective homomorphism and } \psi|_A = \phi \Big\} \Big|,$$
and set $N_\phi(F)(m) = N_\phi(F)(m-1)$ otherwise. Thus, as long as $\phi$ is faithful, $N_\phi(F)(m)$ counts the (labelled) copies\footnote{In other words, $\psi(u)\psi(v) \in E(G_m)$ whenever $uv \in E(F)$, and $\psi(u)\psi(v) \in O(G_m)$ whenever $uv \in O(F)$.} of $F$ in $G_m$ that agree with $\phi$ on $A$; note that if $\phi$ is not faithful at time~$t$, then $N_\phi(F)(m') = 0$ for every $m' \ge m$. We remark that, for the sake of brevity of notation, we shall often suppress the dependence of $N_\phi(F)$ on $m$. 

\begin{figure}[h]
\includegraphics[scale=1] {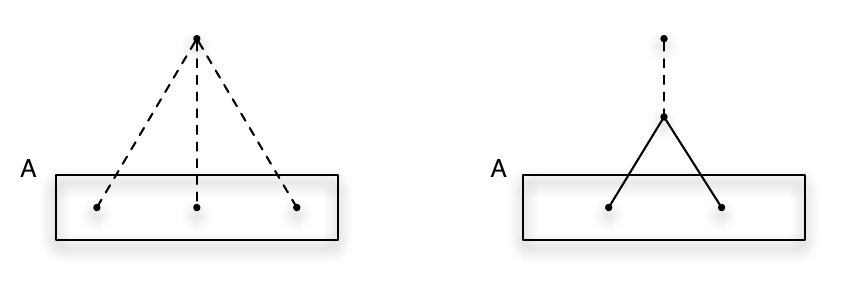}
\caption{Two graph structure pairs.}
\end{figure}

Let us first estimate the expected size of $N_\phi(F)$, i.e., the expected number of copies of $F$ rooted at $\phi(A)$, if $\phi$ is faithful. Recalling our guiding principle that $G_m$ resembles the Erd\H{o}s-R\'enyi random graph $G(n,m)$ (except in relation to containment of triangles), and that we expect that the densities of edges and open edges in $G_m$ will be roughly $2t / \sqrt{n}$ and $e^{-4t^2}$ respectively, it follows that we might expect $G_m$ to contain roughly
$$\Nt_A(F)(m) \, = \, \big( e^{-4t^2} \big)^{o(F)} \left( \frac{2t}{\sqrt{n}} \right)^{e(F)} n^{v_A(F)}$$
copies of $F$ rooted at $\phi(A)$, where $v_A(F) = |V(F)| - |A|$, $e(F) = |E(F)|$ and $o(F) = |O(F)|$. 

Our next theorem shows that $N_\phi(F)$ closely tracks the function $\Nt_A(F)$ for every graph structure triple ($F,A,\phi)$, as long as $\phi$ is faithful, up to a certain time $t_A(F) \in [0,t^*]$. Before we define $t_A(F)$, observe that if $F$ has an induced\footnote{Since we will only be interested in induced sub-structures, we shall identify each sub-structure of $F$ with its vertex set. Thus we write $A \subsetneq H \subseteq F$ to indicate that $H$ is a sub-structure of $F$, and $A \subsetneq V(H)$.} sub-structure $A \subsetneq H \subseteq F$ with $\Nt_A(H) < 1$, then we are unlikely to be able to track the number of copies of $F$, since it might be zero, but if $N_\phi(H)$ happens to be non-zero then $N_\phi(F)$ could be very large. This simple observation motivates the following definition of the \emph{tracking time} of the pair~$(F,A)$. 

\begin{defn}\label{def:tAF}
For each graph structure pair $(F,A)$, define
\begin{equation}\label{def:t*}
t_A^*(F) \, = \, \inf\Big\{ t > 0 \,:\, \Nt_A(F)(m) \le (2t)^{e(F)} \Big\} \, \in \, [0,\infty]
\end{equation}
and
$$t_A(F) \, = \, \min\Big\{ \min \big\{ t_A^*(H) : A \subsetneq H \subseteq F \big\}, \, t^* \Big\}.$$
We call $t_A(F)$ the \emph{tracking time} of the pair $(F,A)$. 
\end{defn} 

Note that we have $t_A^*(F) = 0$ if and only if $e(F) \ge 2v_A(F)$, and that if $A = V(F)$, then $t_A(F) = t^*$. Finally, define 
\begin{equation}\label{def:cFA}
c \, = \, c(F,A) \, := \, \max\bigg\{ \max_{A \subsetneq H \subseteq F} \bigg\{ \frac{2o(H)}{2v_A(H) - e(H)} \bigg\}, \, 2 \bigg\},
\end{equation}
for each graph structure pair $(F,A)$ with $t_A(F) > 0$.\footnote{For example, the first graph structure pair in Figure~2.2 has tracking time $t_A(F) = \frac{1}{2\sqrt{3}} \sqrt{\log n}$ and $c(F,A) = 3$, and the second has tracking time $t_A(F) = 0$.} We claim that $e^{ct^2} \le n^{1/4}$ when $t = t_A(F)$. Indeed, since $e^{2t^2} \le n^{1/4}$ for every $t \le t^*$, if $e^{ct^2} > n^{1/4}$ when $t = t_A(F)$ then there must exist $A \subsetneq H \subseteq F$ such that $e^{2o(H)t^2} > n^{(2v_A(H) - e(H))/4}$. But this implies that $t_A(F) \le t_A^*(H) < t$, which contradicts our assumption that $t = t_A(F)$.\footnote{Moreover, if $t_A(F) < t^*$ then in fact $e^{ct^2} = n^{1/4}$ when $t = t_A(F)$. Indeed, let $A \subsetneq H \subseteq F$ be such that $t_A(F) = t_A^*(H)$, and observe that $e^{4t^2 o(H)} = n^{v_A(H) - e(H)/2}$ when $t = t_A(F)$, and so $e^{ct^2} \ge n^{1/4}$, as required.}


We can now state the `permissible graph structure theorem'.

\begin{thm}\label{NFthm}
For every permissible graph structure $F$, and every independent set $A \subseteq V(F)$, there exists a constant $\gamma(F,A) > 0$ such that, with high probability, 
$$N_\phi(F)(m) \, \in \, \Big( 1 \pm e^{ct^2} n^{-1/4} (\log n)^{\gamma(F,A)} \Big) \big( e^{-4t^2} \big)^{o(F)} \left( \frac{2t}{\sqrt{n}} \right)^{e(F)} n^{v_A(F)}$$
for every $\omega < t \le t_A(F)$, and every faithful $\phi \colon A \to V(G_m)$. 
\end{thm}

We emphasize that when $t > t_A(F)$ then $N_\phi(F)$ is no longer likely to be tracking (in fact, it is quite likely to be zero).  Nevertheless, in this case we shall still give an upper bound on $N_\phi(F)$ which is (probably) best possible up to a polylog factor, see Theorem~\ref{EEthm}. As noted earlier, this latter bound will be a crucial tool in the proof of Theorem~\ref{NFthm}; more precisely, we will use it to bound the maximum step-size in our super- and sub-martingales. We note also that we shall need a separate argument, which is based on the proof in~\cite{Boh}, to control the variables $N_\phi(F)$ in the range $0 < t \le \omega$. 

Finally, in order to give an upper bound on the number of edges in $G_{n,\triangle}$, and to show that it is a good Ramsey graph, we shall need to prove the following theorem. 

\begin{thm}\label{thm:indepsets}
$$\Delta\big( G_{n,\triangle} \big) \, = \, \bigg( \frac{1}{\sqrt{2}} + o(1) \bigg) \sqrt{ n \log n } \qquad \text{and} \qquad \alpha\big( G_{n,\triangle} \big) \, \le \, \big( \sqrt{2} + o(1) \big) \sqrt{ n \log n }$$
with high probability as $n \to \infty$.
\end{thm}

The basic idea of the proof of Theorem~\ref{thm:indepsets} is to sum over sets $S$ (of the appropriate size) the probability that $S$ is the neighbourhood of a vertex in $G_{n,\triangle}$ / an independent set in $G_{m^*}$. In order to bound this probability we shall need to track the number of open edges inside $S$, which will be possible (via our usual martingale technique, see Section~\ref{MartSec}) only if there are not too many vertices that send more than $n^\delta$ (for some $0 < \delta \ll \eps$) edges of $G_{m^*}$ into $S$. The main challenge of the proof turns out to be dealing with the other case ($S$ has a large intersection with many neighbourhoods), see Section~\ref{indepSec}. We remark that, by comparison with $G(n,m)$, it seems very likely that our upper bound on $\alpha\big( G_{n,\triangle} \big)$ is asymptotically tight; however, we leave a proof of this statement as an open problem.

\section{Martingale bounds: the Line of Peril and the Line of Death}\label{MartSec}

In this section we introduce the method we shall use to bound the probability that we lose control of a `self-correcting' random variable. As in~\cite{Boh,BK}, we shall use martingales, but the technique here differs in some crucial respects. We remark that a similar method was introduced recently by Bohman, Frieze and Lubetzky~\cite{BFL10,BFL}. 

Recall that a \emph{super-martingale} with respect to a filtration $(\F_m)_{m = r}^{r+s}$ of $\sigma$-algebras, is a sequence of random variables $(M(m))_{m = r}^{r+s}$ such that the following hold for each $m \in [r,r+s]$: 
$$M(m) \textup{ is $\F_m$-measurable,} \quad \Ex\big[ |M(m)| \big] < \infty, \quad \textup{and} \quad \Ex\big[ M(m+1) \mid \F_m \big] \le M(m).$$ 
Our main tool is the following martingale bound, which follows from a classical inequality of Freedman~\cite{F75} (alternatively, see~\cite[Theorem~3.15]{Colin}, or the Appendix~\cite{App}). We remark that, for all of the super-martingales we shall use below, $M(m)$ will depend only on the sequence $(G_0,\ldots,G_m)$, i.e., only on the process up to step $m$.   

\pagebreak

\begin{lemma}\label{mart}
Let $(M(m))_{m = r}^{r+s}$ be a super-martingale with respect to a filtration $(\F_m)_{m = r}^{r+s}$, and suppose that 
\begin{equation}\label{eq:martlemma}
|M(m + 1) - M(m)| \le \alpha \qquad  \text{and} \qquad \Ex\big[ |M(m + 1) - M(m)| \;\big|\; \F_m \big] \le \beta
\end{equation}
for every $r \le m < r+s$. Then, for every $0 \le x \le \beta s$,
$$\Pr\big( M(r+s) > M(r) + x \big) \, \le \, \exp\left( - \frac{x^2}{4 \alpha \beta s} \right).$$
\end{lemma}

For some of the more straightforward martingale calculations (for example, when $t \le \omega$), it will be more convenient to use instead Bohman's method. We therefore state here, for convenience and comparison, the martingale lemma used in~\cite{Boh}.

\begin{lemma}\label{Bohmart}
Let $(M(m))_{m = r}^{r+s}$ be a super-martingale with respect to a filtration $(\F_m)_{m = r}^{r+s}$, and suppose that 
$$- \beta \le M(m+1) - M(m) \le \alpha$$ 
for every $r \le m < r+s$. Then, for every $0 \le x \le \min\{ \alpha, \beta \} \cdot s$,
$$\Pr\big( M(r+s) > M(r) + x \big) \, \le \, \exp\left( - \frac{x^2}{4 \alpha \beta s} \right).$$
\end{lemma}

We remark that in both cases, a corresponding bound also holds for sub-martingales.

\subsection{The Line of Peril and the Line of Death}

Let $A(m)$ be a `self-correcting' (see Definition~\ref{def:selfcorrecting}, below) random variable which we wish to track, i.e., show that 
$$A(m) \in \big( 1 \pm g(t) \big) \At(m)$$ 
for some functions $\At(m)$ and $g(t)$, and for all $m$ in a given\footnote{A typical case will be $a = \omega \cdot n^{3/2}$ and $b = t_A(F) \cdot n^{3/2}$. We will need to assume that $a \ge n^{3/2}$.} interval $[a,b]$. Let us define the normalized error to be 
$$\As(m) \, = \, \frac{A(m) - \At(m)}{g(t) \At(m)}.$$
In order to motivate the rather technical statement below (see Lemma~\ref{lem:self:mart}), we shall begin by outlining the basic idea behind the method, which is in fact rather simple. 

\medskip

\noindent \textbf{The basic idea:} Suppose that $A^*$ is a random walk on $\RR$, and we wish to bound the probability that the event 
$$\big\{ |A^*(m)| > 1 \big\} \cap \K(m-1)$$ 
holds for some $m \in [a,b]$, where $\K(1) \supseteq \ldots \supseteq \K(m^*)$ is some sequence of `good' events\footnote{For us, the event $\K(m)$ will essentially say that all of our variables are still tracking at step $m$.}. When $|A^*(m)| < 1/2$ we need only that the maximum step-size is~$o(1)$. When $|A^*(m)| \ge 1/2$, on the other hand, suppose that $A^*$ has a `self-correcting drift' (in expectation) back towards the line $A^*(m) = 0$. The idea is as follows: do nothing as long as $|A^*(m)| < 1/2$, and step in only when we reach the danger zone. More precisely, we control the probability that $A^*$ `crosses' from $|A^*(m)| = 1/2$ (the `Line of Peril') to $|A^*(m)| = 1$ (the `Line of Death') in a given interval, $[r,r+s]$. There are two cases to consider: either $A^*$ crosses this interval quickly, which is unlikely even without the drift, or $A^*$ crosses slowly, which is unlikely since it has to constantly swim against the current. In both cases we apply Lemma~\ref{mart} to a suitable super-martingale, which gives bounds of the form $n^{-\log n}$ for well-chosen $g(t)$. 

\medskip

Before giving a more detailed description of our martingale method, we remark that, in order to prove the various theorems stated in Section~\ref{sketchSec}, we shall consider a polynomial number of `bad' events of the following form: \begin{align*}
&\textup{``Variable $\As$ is the first of all the variables that we are tracking to cross its Line of}\\
& \textup{ Death, it does so after $r+s$ steps, and it last crossed its Line of Peril after $r$ steps."}
\end{align*}
Clearly, if some variable crosses its Line of Death, then an event of this type must occur. Since we prove bounds of the form $n^{-\log n}$ on the probability of each event, we will be able to deduce, by the union bound, that with high probability none of the bad events occurs. 

\begin{figure}[h]
\includegraphics[scale=0.8] {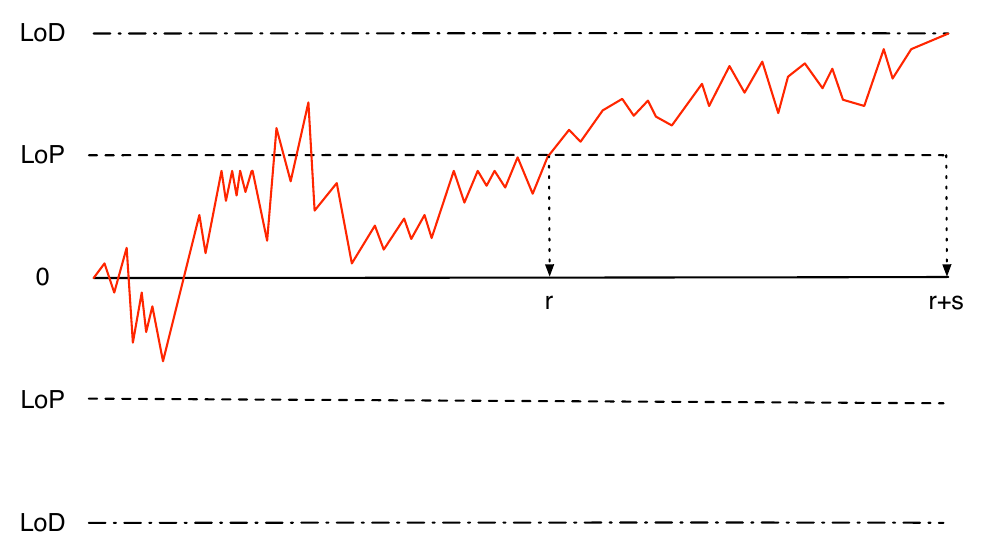}
\caption{The Lines of Peril and Death.}
\end{figure}

In order to make precise the general strategy outlined above for a variable $A^*$, we shall (roughly speaking) do the following:
\begin{itemize}
\item[$1.$] Choose a sequence of `good' events $\K(1) \supseteq \ldots \supseteq \K(m^*)$. \smallskip 
\item[$2.$] Prove that if $\K(m)$ holds then 
\begin{equation}\label{eq:Aselfcorrection}
\Ex\big[ \Delta \As(m) \big] \in \ds\frac{ct}{n^{3/2}} \Big( - \As(m) \pm \eps \Big),
\end{equation}
for some constant $c > 0$, and moreover
$$|\Delta \As(m)| \le \alpha(t) \qquad \text{and} \qquad \Ex\big[ |\Delta \As(m)| \big] \le \beta(t).$$ 
\item[$3.$] Call the line $|\As(m)| = \frac{1}{2}$ the \emph{Line of Peril} and $|\As(m)| = 1$ the \emph{Line of Death}. \smallskip
\item[$4.$] For each pair $(r,s) \in \N^2$, consider the event $\LL^A(r,s)$ that $|\As(m)|$ crosses the Line of Death for the first time in step $r+s$, and that it crossed the Line of Peril for the last time (before step $r+s$) in step $r+1$. \smallskip
\item[$5.$] Set $s_0 = \min\{ s, n^3 / r \}$, let $\alpha$ and $\beta$ denote the suprema of $\alpha(t)$ and $\beta(t)$, respectively, over the interval $[r,r+s_0]$, and check that $\beta s_0 \ge 1$, and that  
$$\alpha \cdot \beta \cdot n^{3/2} \,\le\, \frac{1}{(\log n)^3}.$$
\item[$6.$] Set $M(r) = |\As(r)|$ and, while $\K(m)$ holds and $|\As(m)| \ge 1/3$, set 
$$M(m) \, = \, |\As(m)| \,+\, \frac{cr(m-r)}{4n^3}.$$
We will use~\eqref{eq:Aselfcorrection} to show that $(M(m))_{m = r}^{r + s_0}$ is a super-martingale with respect to the filtration $(\F_m)_{m = r}^{r + s_0}$, where $\F_m$ encodes all of the information obtained by observing the process up to step $m$. We consider separately the cases $s_0 = s$ and $s_0 = n^3 / r$, noting that $s_0 \le n^{3/2}$ if $m \ge n^{3/2}$. \smallskip
\item[$7.$] In the former case, we ignore the self-correction and, using Lemma~\ref{mart}, bound the probability of $\LL^A(r,s) \cap \K(r+s) \cap \{ s_0 = s \}$ by
$$\Pr\left( M(s) \ge M(0) + \frac{1}{2} \right) \, \le \, \exp\left( - \frac{\delta}{\alpha \beta n^{3/2}} \right) \, \le \, n^{-\log n},$$
for some $\delta > 0$. The final inequality holds by our bound on $\alpha \beta n^{3/2}$.  \smallskip
\item[$8.$] In the latter case, we use the self-correction of $\As(m)$ to bound the probability that the martingale hasn't re-crossed the Line of Peril after $s_0$ steps. Indeed, it follows from $|A^*(r + s_0)| \ge 1/2 \ge |A^*(r)|$ that $M(s_0) \ge M(0) + c/4$, due to the `drift' over the interval $[r,r+s_0]$. Thus, by Lemma~\ref{mart} and as above, we can bound the probability of the event $\LL^A(r,s) \cap \K(r+s_0) \cap \{ s_0 \neq s \}$ by
$$\Pr\left( M(s_0) \ge M(0) + \frac{c}{4} \right) \, \le \, \exp\left( - \frac{\delta}{\alpha \beta n^{3/2}} \right) \, \le \, n^{-\log n},$$
for some $\delta = \delta(c) > 0$, as required. 
\end{itemize}
We remark that, crucially, our functions $\alpha(t)$ and $\beta(t)$ in each application will vary by only a relatively small factor over the range $m \in [r,r+s_0]$. Hence we shall not lose much in applying Lemma~\ref{mart} with $\alpha$ and $\beta$ fixed, and equal to the maximum of $\alpha(t)$ and $\beta(t)$ over this interval.

\subsection{A general lemma}

In this subsection we shall prove the main technical lemma of this section, which we will use several times in the proofs that follow. In order to state it, let us first make precise what it means for a random variable to be self-correcting.

\begin{defn}[Self-correcting]\label{def:selfcorrecting}
Let $A(m)$ be a random variable, let $g(t)$ and $h(t)$ be functions defined on $(0,t^*)$, let $\K(1) \supseteq \ldots \supseteq \K(m^*)$ be a nested sequence of events, and let $n^{3/2} \le a \le b \le m^*$. We say that the variable $A$ is \emph{$(g,h;\K)$-self-correcting} on the interval $[a,b]$ relative to the function $\At(m)$ if the following holds for every $a \le m \le b$: 

If $\K(m)$ holds, then the normalized error
$$A^*(m) \, = \, \frac{A(m) - \At(m)}{g(t) \At(m)}$$
satisfies the following inequalities:
$$A^*(m) \ge 1/3 \; \Rightarrow \;  \Ex\big[ \Delta A^*(m) \big] \le -h(t) \quad \text{and} \quad A^*(m) \le - 1/3 \; \Rightarrow \;  \Ex\big[ \Delta A^*(m) \big] \ge h(t).$$
\end{defn}

We remark that, in all of the applications below, $A^*(m)$ will satisfy an inequality of the form\footnote{We write $x \in (a \pm b)(c \pm d)$ to denote the fact that $x = yz$ for some $y \in a \pm b$ and $z \in c \pm d$.}
\begin{equation}\label{self:usual}
\Ex\big[ \Delta \As(m) \big] \, \in \, \big( c \pm d \big) \cdot \ds\frac{t}{n^{3/2}} \Big( - \As(m) \pm \eps \Big)
\end{equation}
when $\K(m)$ holds, for some constants $c > d \ge 0$ with $c - d \ge 2$. (In fact, in most of our applications $d = 0$.) It follows easily from~\eqref{self:usual} that $A$ is $(g,h;\K)$-self-correcting relative to $\At(m)$ for some function $h(t) = \Theta\big( t \cdot n^{-3/2} \big)$. 

Let us say that a function $\alpha(t)$ is $\lambda$-\emph{slow} on $[a,b]$ if it varies by at most a factor of $\lambda$ across any interval of the form $[x, x + 1/x] \subset [a,b]$. Note in particular that any function of the form $t^k e^{\ell t^2}$ is $\lambda $-slow on $[a,b]$ for some $\lambda = \lambda(k,\ell) > 0$ and any $1 \le a \le b < \infty$.

We are almost ready to state our main martingale lemma; the following definition will simplify the statement somewhat.

\begin{defn}\label{def:reasonable}
We say that a collection $(\lambda,\delta;g,h;\alpha,\beta;\K,I)$ is \emph{reasonable} if 
$$\begin{array}{lll}
& \bullet \; \textup{$\lambda \ge 1$ and $\delta \in (0,1/4)$ are constants,} \quad & \bullet \; \textup{$g$,~$h$,~$\alpha$~and~$\beta$ are functions defined on $(0,t^*)$,}\\[+0.7ex]
& \bullet \; \textup{$I = [a,b] \subseteq [n^{3/2},m^*]$ is an interval,} & \bullet \; \K(1) \supseteq \ldots \supseteq \K(m^*) \textup{ is a sequence of events,}
\end{array}$$
and moreover they satisfy the following conditions:
\begin{itemize}
\item[$(a)$] $\alpha$ and $\beta$ are $\lambda$-slow on $I$.
\item[$(b)$] $\min\big\{ \alpha(t), \, \beta(t), \, h(t) \big\} \ge \ds\frac{\delta t}{n^{3/2}}$ and $\alpha(t) \le \delta$ for every $m \in [a,b]$. 
\end{itemize}
\end{defn}

We can now state the main technical lemma we will use to bound the probability that a self-correcting random variable is the first to go off track.

\begin{lemma}\label{lem:self:mart}
Let $(\lambda,\delta;g,h;\alpha,\beta;\K,I)$ be a reasonable collection, and let $A$ be a random variable that is $(g,h;\K)$-self-correcting on $I = [a,b]$ relative to the function $\At$. Suppose that
$$|\Delta A^*(m)| \le \alpha(t) \qquad  \text{and} \qquad \Ex\big[ |\Delta A^*(m)| \big] \le \beta(t)$$
for every $m \in [a,b]$ for which $\K(m)$ holds, and that $| A^*(a) | < 1/2$. Then
$$\Pr\Big( \big\{ |A^*(m)| > 1 \big\} \cap \K(m-1) \text{ for some $m \in [a,b]$} \Big) \, \le \, n^4 \exp\left( - \frac{\delta'}{n^{3/2}} \min_{m \in [a,b]} \bigg\{ \frac{1}{\alpha(t) \beta(t)} \bigg\} \right),$$
where $\delta' = \delta^3 / (4\lambda)^2 > 0$. 
\end{lemma}

\begin{proof}
For each $(r,s) \in \N^2$, we define an event 
$$\LL_+^A(r,s) \, = \, \big\{ A^*(r) < 1/2 \big\} \cap \bigcap_{m = r+1}^{r+s-1} \big\{ 1/2 \le A^*(m) \le 1 \big\} \cap \big\{ A^*(r+s) > 1 \big\},$$
and a similar event $\LL_-^A(r,s)$ corresponding to crossing the interval $[-1,-1/2]$. If we have $| A^*(a) | < 1/2$ and $|A^*(m)| > 1$ for some $m \in [a,b]$, then the event $\LL_+^A(r,s) \cup \LL_-^A(r,s)$ must hold for some $a \le r < m$ and $1 \le s \le m - r$. By symmetry, it will therefore suffice to prove the following claim. 

\medskip
\noindent \textbf{Claim:} $\Pr\Big( \LL_+^A(r,s) \cap \K(r+s-1) \Big) \le \ds\exp\left( - \frac{\delta^3}{8 \lambda^2 n^{3/2}} \max_{m \in [r,r+s]} \frac{1}{\alpha(t) \beta(t)} \right)$ for every $r,s \in \N$. 

\begin{proof}[Proof of claim]
Set $s_0 = \min\{ s, n^3 / r \}$ and $M(r) = A^*(r)$, and for each $r \le m < r + s_0$, define 
\begin{equation}\label{eq:self:mart:Mdef}
M(m+1) \, := \, M(m) + A^*(m+1) - A^*(m) + \frac{\delta r}{n^{3}}
\end{equation}
if $\K(m)$ holds and $A^*(m) \ge 1/3$, and $M(m+1) := M(m)$ otherwise. Note that $M(m)$ depends only on the process up to step $m$, and so is $\F_m$-measurable, where $\F_m$ is the $\sigma$-algebra of information obtained by observing the sequence $(G_0,\ldots,G_m)$.

We claim that $(M(m))_{m = r}^{r+s}$ is a super-martingale with respect to the filtration $(\F_m)_{m = r}^{r+s}$. Indeed, since the function $A$ is $(g,h;\K)$-self-correcting on $[a,b]$ with respect to $\At$, we have
$$\Ex \big[ M(m+1) - M(m) \;\big|\; \F_m \big] \, = \, \Ex \big[ \Delta A^*(m) \big] + \frac{\delta r}{n^3}  \, \le \, - h(t) + \frac{\delta t}{n^{3/2}} \, \le \, 0$$  
if $\K(m)$ holds and $A^*(m) \ge 1/3$, by~\eqref{eq:self:mart:Mdef} and Definition~\ref{def:selfcorrecting}, and $M(m+1) = M(m)$ if not. 

In order to apply Lemma~\ref{mart}, we need to check that~\eqref{eq:martlemma} holds for a suitable pair $(\alpha,\beta)$. To do so, observe first that either $M(m+1) = M(m)$, or $\K(m)$ holds and
$$| M(m+1) - M(m)| \, \le \, | \Delta A^*(m) \big| + \frac{\delta r}{n^3} \, \le \, 2 \cdot \alpha(t),$$ 
since $|\Delta A^*(m)| \le \alpha(t)$ whenever $\K(m)$ holds, and $\alpha(t) \ge \delta t / n^{3/2} \ge \delta r / n^3$, since $m \ge r$. Similarly, if $\K(m)$ holds then
$$\Ex\big[ |M(m+1) - M(m)| \;\big|\; \F_m \big] \, \le \, \Ex\big[ | \Delta A^*(m) \big| \big] + \frac{\delta r}{n^3} \, \le \, 2 \cdot \beta(t),$$
since $\Ex\big[ | \Delta A^*(m) \big| \big] \le \beta(t)$ whenever $\K(m)$ holds, and $\beta(t) \ge \delta t / n^{3/2} \ge \delta r / n^3$, as before. Thus it follows that~\eqref{eq:martlemma} holds with 
$$\alpha = 2 \cdot \max_{r \le m < r + s_0} \alpha(t) \qquad \text{and} \qquad  \beta = \frac{1}{s_0} + 2 \cdot \max_{r \le m < r + s_0} \beta(t),$$
and moreover $\beta \cdot s_0 \ge 1$. 

We claim next that if $\LL_+^A(r,s) \cap \K(r+s-1)$ holds, then $M(r+s_0) \ge M(r) + \delta$. Indeed, the event $\K(r+s-1)$ implies that $\K(m)$ holds for every $r \le m < r + s_0$ (since $s_0 \le s$), and also that $A^*(r) \ge A^*(r+1) - \delta$. Combining this with the event $\LL_+^A(r,s)$, it follows that $A^*(m) \ge 1/3$ for every $r \le m < r + s_0$. Thus, if $s_0 = s \le n^3 / r$, then 
$$M(r+s_0) \, = \, A^*(r+s) + \frac{\delta rs}{n^3} \, > \, 1 \, > \, M(r) + \frac{1}{2},$$ 
\pagebreak where the second and third follow from the event $\LL_+^A(r,s)$. On the other hand, if $s_0 = n^3 / r \le s$, then we similarly obtain
$$M(r+s_0) \, = \, A^*(r+s_0) + \frac{\delta r s_0}{n^3} \, \ge \, \frac{1}{2} + \delta \, > \, M(r) + \delta,$$ 
since $\LL_+^A(r,s)$ implies that $A^*(r+s_0) \ge 1/2$. Hence, by Lemma~\ref{mart}, it follows that
\begin{equation}\label{eq:applyingmart}
\Pr\Big( \LL_+^A(r,s) \cap \K(r+s-1) \Big) \, \le \,  \exp\left( - \frac{\delta^2}{4\alpha \beta s_0} \right).
\end{equation}

Finally, we claim that 
\begin{equation}\label{eq:boundingabs}
\alpha \beta s_0 \, \le \, \frac{4\lambda^2}{\delta} \alpha(t) \beta(t) n^{3/2},
\end{equation}
for every $r \le m < r + s_0$. To see this, recall first that $\alpha(t)$ and $\beta(t)$ are $\lambda$-slow on $[a,b]$ and $s_0 \le n^3/r$, so each varies by at most a factor of $\lambda$ in the range $r \le m < r + s_0$. Now, if $\beta s_0 \ge 2$, then $\alpha \beta \le 8\lambda^2 \alpha(t) \beta(t)$, and so in this case~\eqref{eq:boundingabs} follows, since $s_0 \le n^{3/2}$ and $\delta \le 1/4$. On the other hand, if $\beta s_0 \le 2$, then $\alpha \beta s_0 \le 4\lambda \alpha(t)$, and so~\eqref{eq:boundingabs} again follows, since $\beta(t) \ge \delta t / n^{3/2}$ and $\lambda, t \ge 1$. Combining~\eqref{eq:applyingmart} and~\eqref{eq:boundingabs}, the claim follows.  
\end{proof}

It follows immediately from the claim that 
\begin{align*}
\Pr\bigg( \bigcup_{m = a}^b \big( |A^*(m)| > 1 \big) \cap \K(m-1) \bigg) & \le \, \sum_{r,s} \Pr\Big( \LL^A(r,s) \cap \K(r+s-1) \Big) \\
& \le \, n^4 \ds\exp\left( - \frac{\delta'}{n^{3/2}} \min_{m \in [a,b]} \bigg\{ \frac{1}{\alpha(t) \beta(t)} \bigg\} \right),
\end{align*}
where $\delta' = \delta^3 / (4\lambda)^2$, as required.
\end{proof}

\subsection{The events $\X(m)$, $\Y(m)$, $\Z(m)$ and $\Q(m)$} \label{Zsec}

To finish this section, we shall motivate the martingale technique introduced above by using it to track the variables $X_e$. This is a particularly simple special case of the argument used in Section~\ref{EEsec}, and should help prepare the reader for the more involved application performed there, where we use it to control~$N_\phi(F)$. 

In order to track the variables $X_e$, we shall need to introduce some more notation, which will be used in a similar way throughout the paper. Set
\begin{equation}\label{def:fy} 
f_y(t) \, = \, e^{Ct^2} n^{-1/4} (\log n)^{5/2}
\end{equation}
and $f_x(t) \, = \, e^{-4t^2} f_y(t)$, where $C = C(\eps) > 0$ is the large constant chosen above. Recall that $g_y(t) = e^{2t^2} n^{-1/4} (\log n)^4$, and set $g_x(t) = C g_y(t)$. The following simple definitions will play a crucial role in the proof of Theorems~\ref{triangle} and~\ref{R3k}. 

\pagebreak

\begin{defn}\label{def:events:XYQ}
For each $0 \le m' \le m^*$, we define events $\X(m')$, $\Y(m')$ and  $\Q(m')$  as follows: 
\begin{itemize}
\item[$(a)$] $\X(m')$ denotes the event that 
$$X_e(m) \, \in \, \Xt(m) \pm f_x(t) \Xt(n^{3/2})$$
for every open edge $e \in O(G_m)$, and every $m \le \min\big\{ \omega \cdot n^{3/2}, m' \big\}$, and
$$X_e(m) \, \in \, \big( 1 \pm g_x(t) \big) \Xt(m)$$
for every open edge $e \in O(G_m)$, and every $\omega \cdot n^{3/2} < m \le m'$.
\item[$(b)$] $\Y(m')$ denotes the event that 
$$Y_e(m) \, \in \, \Yt(m) \pm f_y(t) \Yt(n^{3/2})$$
for every open edge $e \in O(G_m)$, and every $m \le \min\big\{ \omega \cdot n^{3/2}, m' \big\}$, and 
$$Y_e(m) \, \in \, \big( 1 \pm g_y(t) \big) \Yt(m)$$
for every open edge $e \in O(G_m)$, and every $\omega \cdot n^{3/2} < m \le m'$.
\item[$(c)$] $\Q(m')$ denotes the event that
$$Q(m) \, \in \, \Qt(m) \pm \eps \cdot f_y(t) \Qt(n^{3/2})$$
for every $m \le \min\big\{ \omega \cdot n^{3/2}, m' \big\}$, and 
$$\frac{\Xb(m)}{\Xt(m)} \in 1 \pm g_q(t), \quad \frac{\Yb(m)}{\Yt(m)} \in 1 \pm g_q(t) \quad \text{and} \quad \frac{Q(m)}{\Qt(m)} \in 1 \pm g_q(t)$$
for every $\omega \cdot n^{3/2} < m \le m'$. 
\end{itemize}
\end{defn}

Similarly, we shall later define $\E(m')$ to be the event that the conclusion of Theorem~\ref{EEthm}, below, holds for all $m \le m'$ (for the precise definition, see Section~\ref{EEsec}). In this section, we shall only need the following two special cases of that event: 
\begin{itemize}
\item[1.] With a slight abuse of notation, if $f = \{u,v\} \in E(K_n)$, then let us write $\hat{Y}_f(m)$ for the number of vertices $w \in V(G_m)$ such that $\{u,w\} \in O(G_m)$ and $\{v,w\} \in E(G_m)$, or vice-versa, even if $f \not\in O(G_m)$. The event $\E(m')$ implies that, for every $f \in E(K_n)$, 
$$\hat{Y}_f(m) \, \le \, \big( 1 + \eps \big) \Yt(m)$$ 
if $\omega \cdot n^{3/2} < m \le m'$, and moreover that $\hat{Y}_f(m) \le \sqrt{n}$ if $m \le \min\{ \omega \cdot n^{3/2}, m'\}$.
\item[2.] Let $(W,A)$ denote the graph structure pair with $v(W) = 4$, $v_A(W) = 1$, $e(W) = 1$ and $o(W) = 2$. The event $\E(m')$ implies that 
$$N_\phi(W)(m) \, \le \, \max\big\{ 4t e^{-8t^2} \sqrt{n}, (\log n)^\omega \big\}$$
for every $\omega \cdot n^{3/2} < m \le m'$ and every faithful map $\phi \colon A \to V(G_m)$. 
\end{itemize}
 
Finally, given $e = \{u,v\} \in E(K_n)$, set
$$Z_e(m) \, = \, \big| N_{G_m}(u) \cap N_{G_m}(v) \big|,$$
and let $\Z(m')$ denote the event that $Z_e(m) \le (\log n)^2$ for every $e \in E(K_n)$ and every $m \le m'$. We shall use the following bound frequently throughout the proof of Theorem~\ref{triangle}.

\begin{prop}\label{Zprop}
Let $m \in [m^*]$. With probability at least $1 - n^{-C\log n}$, either the event $\big( \E(m-1) \cap \Q(m-1) \big)^c$ holds, or
$$Z_e(m) \, \le \, (\log n)^2$$
for every $e \in E(K_n)$. In other words,
$$\Pr\Big( \E(m-1) \cap \Z(m)^c \cap \Q(m-1) \Big) \, \le \, n^{-C\log n}.$$
\end{prop}

We would like to emphasize the appearance of the event $\E(m)$ in the statement above, and the absence of $\Y(m)$. This is because the event $\Y(m)$ only gives us a bound on $Y_e(m)$ for \emph{open} edges $e \in O(G_m)$, whereas we shall require a bound on the number of (edge, open edge) pairs which form a triangle with $e$ for both open and non-open pairs in ${[n] \choose 2} \setminus E(G_m)$. As noted above, the event $\E(m)$ gives us such a bound on $\hat{Y}_e(m)$ of the form $\big(1 + \eps \big)\Yt(m)$ (see Section~\ref{EEsec} for the details). Although the error term in this bound is larger than that in Theorem~\ref{Ythm}, it easily suffices for our current purposes. 

\begin{proof}[Proof of Proposition~\ref{Zprop}]
We imitate the proof of the corresponding statement from~\cite{Boh}. Let $m' \in [m^*]$, and let $e = \{u,v\} \in E(K_n)$. As noted above, if the event $\E(m')$ holds, then  
\begin{equation}\label{eq:N+N}
\hat{Y}_e(m) \, \le \, \big( 1 + \eps \big) 4t e^{-4t^2} \sqrt{n}
\end{equation}
for every $\omega \cdot n^{3/2} < m \le m'$, and $\hat{Y}_e(m) \le \sqrt{n}$ for every $m \le \min\{ \omega \cdot n^{3/2}, m'\}$. Now, note that $0 \le \Delta Z_e(m) \le 1$ (deterministically), and that, if $\E(m) \cap \Q(m)$ holds, then
$$\Pr\big( Z_e(m+1) > Z_e(m) \big) \, = \, \frac{\hat{Y}_e(m)}{Q(m)}  \, < \, \frac{\log n}{n^{3/2}},$$ 
where the second inequality follows from~\eqref{eq:N+N} and the event $\Q(m)$. It follows that
$$\Pr\Big( \E(m'-1) \cap \Z(m')^c \cap \Q(m'-1) \Big) \, \le \, {n \choose 2} {m^* \choose (\log n)^2}  \bigg( \frac{\log n}{n^{3/2}} \bigg)^{(\log n)^2} \, \le \, e^{-C(\log n)^2}$$
for every $m' \le m^*$, as required.
\end{proof}

\subsection{Tracking $X_e$}\label{SecX}

As a concrete example to aid the reader's understanding, we shall now show how to track the variables $X_e$; in order to avoid unnecessary distractions, we postpone a couple of the (straightforward, but somewhat technical) calculations to the Appendix~\cite{App}. 

Set $a = \omega \cdot n^{3/2}$ and define, for each $a \le m \le m^*$,  
$$\K^\X(m) = \E(m) \cap \X(a) \cap \Y(m) \cap \Q(m).$$
We shall prove the following bounds on $X_e(m)$, which are slightly stronger than those given by Theorem~\ref{NFthm}. We shall need exactly this strengthening in Section~\ref{Ysec}. 

\begin{prop}\label{Xprop}
Let $\omega \cdot n^{3/2} < m \le m^*$. With probability at least $1 - n^{-C\log n}$ either $\K^\X(m-1)^c$ holds, or
\begin{equation}\label{eq:Xprop}
X_e(m) \, \in \, \Big( 1 \pm C e^{2t^2} n^{-1/4} (\log n)^4 \Big) \cdot 2 e^{-8t^2} n \, = \, \big( 1 \pm g_x(t) \big) \Xt(m)
\end{equation}
for every open edge $e \in O(G_m)$. 
\end{prop}

The first step in tracking a variable will always be as follows: we define the normalized error, in this case
$$X_e^*(m) \, = \, \frac{X_e(m) - \Xt(m)}{g_x(t) \Xt(m)},$$
and show that, while everything is still tracking, it is self-correcting. More precisely, we prove the following lemma.

\begin{lemma}\label{selfX}
Let $\omega \cdot n^{3/2} < m \le m^*$. If $\E(m) \cap \X(m) \cap \Y(m) \cap \Q(m)$ holds, then 
$$\Ex \big[ \Delta X^*_e(m) \big] \, \in \, \frac{4t}{n^{3/2}} \Big( - X_e^*(m) \pm \eps \Big)$$
for every $e \in O(G_m)$.
\end{lemma}

Observe (or see the Appendix) that the evolution of the variable $X_e$ is controlled by the following equation:
\begin{equation}\label{eq:Xeq}
\Ex\big[ \Delta X_e(m) \big] \, = \, - \frac{2}{Q(m)} \sum_{f \in X_e(m)} \Big( Y_f(m) + 1 \Big).
\end{equation}
Lemma~\ref{selfX} follows from~\eqref{eq:Xeq} via a straightforward, but somewhat technical calculation, and also follows from the proof of (the much more general) Lemma~\ref{selfN*}, which will be given in the next section; however, for completeness we provide a proof in the Appendix. The key idea is that, since $g_x(t) = C \cdot g_y(t)$ and $C = C(\eps)$ is large, the error in our knowledge of the variables $Y_f$ is small compared with the error we allow in $X_e$. 

Our next lemma bounds $|\Delta X_e(m)|$. Although simple, our application of it in the martingale bound illustrates one of the key ideas in the proof of Theorem~\ref{NFthm}.  

\begin{lemma}\label{Xe_alpha}
Let $\omega \cdot n^{3/2} < m \le m^*$. If $\E(m)$ holds, then
$$|\Delta X_e(m)| \, \le \, 2 \cdot \max\big\{ 4t e^{-8t^2} \sqrt{n}, (\log n)^\omega \big\}$$
for every $e \in O(G_m)$.
\end{lemma}

\begin{proof}
The key observation is that, if edge $f$ is chosen in step $m+1$ of the triangle-free process, then $|\Delta X_e(m)| = X_e(m) - X_e(m+1)$ is bounded by one of the variables controlled by the event $\E(m)$. More precisely, note that if an open triangle $T$ containing $e$ is destroyed by $f$, then $f$ must close one of the open edges of $T$. If more than one such open triangle is destroyed, then we must have $|e \cap f| = 1$, in which case the number of open triangles which are destroyed is exactly $N_\phi(W)(m)$, where $\textup{Im}(\phi) = e \cup f$. As noted above, the event $\E(m)$ implies that
$$|N_\phi(W)(m)| \, \le \, \max\big\{ 4t e^{-8t^2} \sqrt{n}, (\log n)^\omega \big\}$$
for every faithful $\phi$, as required. 
\end{proof}

Finally, using the two lemmas above, we can easily bound $|\Delta X^*_e(m)|$ and $\Ex\big[ |\Delta X_e^*(m)| \big]$, see the Appendix for the (somewhat tedious) details. 

\begin{lemma}\label{Xe_gamma}
Let $\omega \cdot n^{3/2} < m \le m^*$. If $\E(m) \cap \X(m) \cap \Y(m) \cap \Q(m)$ holds, then 
$$|\Delta X^*_e(m)| \, \le \, \frac{C}{g_x(t)} \cdot \frac{e^{8t^2}}{n} \max\big\{ t e^{-8t^2} \sqrt{n} , (\log n)^\omega \big\} \quad \text{and} \quad \Ex \big[ |\Delta X^*_e(m)| \big] \, \le \, \frac{C}{g_x(t)} \cdot \frac{\log n}{n^{3/2}}$$
for every $e \in O(G_m)$.
\end{lemma}

We are now ready to prove Proposition~\ref{Xprop}.

\begin{proof}[Proof of Proposition~\ref{Xprop}]
We begin by choosing a family of parameters as in Definition~\ref{def:reasonable}. Set $\K(m) = \K^\X(m) \cap \X(m)$ and $I = [a,b] = [\omega \cdot n^{3/2},m^*]$, and let
$$\alpha(t) \, = \, \frac{C}{g_x(t)} \cdot \frac{e^{8t^2}}{n} \max\big\{ t e^{-8t^2} \sqrt{n}, (\log n)^\omega \big\} \qquad \text{and} \qquad \beta(t) \, = \, \frac{C}{g_x(t)} \cdot \frac{\log n}{n^{3/2}}.$$
Moreover, set $\lambda = C$, $\delta = \eps$ and $h(t) = t \cdot n^{-3/2}$. We claim that $(\lambda,\delta;g_x,h;\alpha,\beta;\K)$ is a reasonable collection, and that $X_e$ satisfies the conditions of Lemma~\ref{lem:self:mart} if $e \in O(G_m)$.

To prove the first statement, we need to show that $\alpha$ and $\beta$ are $\lambda$-slow on $[a,b]$, and that 
$$\min\big\{ \alpha(t), \, \beta(t), \, h(t) \big\} \, \ge \, \ds\frac{\eps t}{n^{3/2}}$$ 
and $\alpha(t) \le \eps$ for every $\omega < t \le t^*$, each of which is easy to check.\footnote{To spell out the details, recall that $g_x(t) = C e^{2t^2} n^{-1/4} (\log n)^4$, so both $\alpha$ and $\beta$ are (piecewise) of the form $c t^k e^{\ell t^2}$. The lower bound on $\min\{ \alpha(t), \beta(t), h(t) \}$ holds since $g_x(t) \le 1$ for all $t \le t^*$, and the upper bound on $\alpha(t)$ holds since $(g_x(t) \sqrt{n})^{-1} \le n^{-1/4}$, and $e^{8t^2} (g_x(t) n)^{-1} \le e^{6t^2} n^{-3/4} \le n^{-\eps}$ for all $t \le t^*$.} To prove the second, we need to show that $X_e$ is $(g_x,h;\K)$-self-correcting, which follows from Lemma~\ref{selfX}, and that, for every $\omega \cdot n^{3/2} < m \le m^*$, if $\K(m)$ holds then 
$$|\Delta X_e^*(m)| \le \alpha(t) \qquad  \text{and} \qquad \Ex\big[ |\Delta X_e^*(m)| \big] \le \beta(t),$$
which follows from Lemma~\ref{Xe_gamma}. Note also that the bound $|X^*_e(a)| < 1/2$ follows from the event $\X(a)$, since $f_x(\omega) \Xt(n^{3/2}) \ll g_x(\omega) \Xt(a)$ if $\omega(n) \to \infty$ sufficiently slowly. 

Finally, observe that
$$\alpha(t) \beta(t) n^{3/2} \, \le \, \frac{e^{4t^2}}{\sqrt{n} \cdot (\log n)^7} \cdot \max\big\{ t e^{-8t^2} \sqrt{n}, (\log n)^\omega \big\} \, \le \, \frac{1}{(\log n)^3}$$
for every $\omega < t \le t^*$, since $e^{4t^2} \le n^{1/2 - \eps}$. By Lemma~\ref{lem:self:mart}, and summing over edges $e \in E(K_n)$ the probability that $e \in O(G_m)$ and $X_e^*(m) > 1$, it follows that
$$\Pr\Big( \X(m)^c \cap \K(m-1) \text{ for some $m \in [a,b]$} \Big) \, \le \, n^6 \exp\Big( - \delta' (\log n)^3 \Big) \, \le \, n^{-C \log n},$$
as required.
\end{proof}

\section{Tracking Everything Else}\label{EEsec}

In this section we shall generalize the method introduced in Section~\ref{MartSec} in order to prove Theorem~\ref{NFthm} under the assumption that the variables $Y_e$, $Z_e$ and $Q$ are tracking. In other words, as long as the events $\Y(m)$, $\Z(m)$ and $\Q(m)$ all hold, we shall give close to best possible bounds on the number of copies of an arbitrary graph structure~$F$ in~$G_m$. 

We begin by recalling that a graph structure triple $(F,A,\phi)$ consists of a permissible graph structure $F$, an independent set $A \subseteq V(F)$, and an injective map $\phi \colon A \to V(G_m)$. Recall that, given such a triple, we set 
\begin{equation}\label{def:NtF}
\Nt_A(F)(m) \, = \, \big( e^{-4t^2} \big)^{o(F)} \left( \frac{2t}{\sqrt{n}} \right)^{e(F)} n^{v_A(F)},
\end{equation}
and observe that this is roughly the expected number of (labelled) copies of $F$ rooted at $\phi(A)$ in the Erd\H{o}s-R\'enyi random graph $G(n,m)$. Recall that $t^* = \big( \frac{1}{2\sqrt{2}} - \eps \big) \sqrt{\log n}$, and that the tracking time of the pair $(F,A)$ is defined to be
$$t_A(F) \, = \, \min\left\{ \min \big\{ t_A^*(H) : A \subsetneq H \subseteq F \big\}, \,t^* \right\},$$
where $t_A^*(H)$ is defined so that $\Nt_A(H)(m) = (2t)^{e(H)}$ at $t = t_A^*(H)$. We shall track $N_\phi(F)$ up to time $t_A(F)$, or until $\phi$ is no longer faithful, for \emph{every} graph structure triple $(F,A,\phi)$. In other words, we shall control the number of (rooted) copies of every permissible graph structure, for every root $\phi(A)$, and for (essentially) as long as it is possible to do so. 

Recall that if $\phi$ is faithful, then 
\begin{align*}
& N_\phi(F)(m) \, = \, \Big| \Big\{ \psi \colon V(F) \to V(G_m) \,:\, \psi \textup{ is an injective homomorphism and } \psi|_A = \phi \Big\} \Big|
\end{align*}
denotes the number of labelled copies of the graph structure $F$ in $G_m$ which agree with $\phi$ on $A$, and that $N_\phi(F)(m) = N_\phi(F)(m-1)$ otherwise; as noted in Section~\ref{sketchSec}, we shall often suppress the dependence of $N_\phi(F)$ on $m$. We shall also occasionally write $N_\phi(F)$ for the \emph{collection} of copies as above, and trust that this will not cause confusion. Recall that $C = C(\eps) > 0$, chosen earlier, is a sufficiently large constant.  Now define
\begin{equation}\label{def:del}
\Delta(F,A) = \big( C^3 v_A(F)^2 + 2e(F) + o(F) \big)^C
\end{equation}
and, for each $A \subseteq H \subseteq F$, set
\begin{equation}\label{def:gam}
\Delta(F,H,A) = \Delta(F,H) + \Delta(H,A) \quad \text{and} \quad \gamma(F,A) = \Delta(F,A) - e(F) - 2.
\end{equation}
We remark that the choice of these parameters is quite delicate, and that they will play an important role in the proofs of Lemmas~\ref{selfN*},~\ref{gammaNF} and~\ref{maxalphabeta}, below, which (respectively) imply that $N_\phi(F)$ is self-correcting, and control its single-step changes. 
Define
\begin{equation}\label{def:ffat}
f_{F,A}(t) \, = \, e^{C(o(F)+1)(t^2 + 1)} n^{-1/4} (\log n)^{\Delta(F,A) - \sqrt{\Delta(F,A)}}, 
\end{equation}
and if $t_A(F) > 0$, then set
\begin{equation}\label{def:gfat}
g_{F,A}(t) \, = \, e^{ct^2} n^{-1/4} (\log n)^{\gamma(F,A)},
\end{equation}
where $c = c(F,A)$ is chosen (as defined in~\eqref{def:cFA}) so that $e^{ct^2} = n^{1/4}$ when $t = t_A(F)$ if $t_A(F) < t^*$, and so that $c(F,A) \ge 2$.  
We shall bound $N_\phi(F)$ up to a multiplicative factor  of $\big( 1 \pm g_{F,A}(t) \big)$ if $\omega < t \le t_A(F)$, and up to an additive factor of $f_{F,A}(t) \Nt_A(F)(n^{3/2})$ if $t \le \omega < t_A(F)$. Finally, set
\begin{equation}\label{def:event:KE}
\K^\E(m) \, = \, \Y(m) \cap \Z(m) \cap \Q(m)
\end{equation}
for each $m \le m^*$. As noted earlier, we assume throughout that $n \in \N$ is sufficiently large. 

The following theorem is the main result of this section. It says that we can track all graph structures up to their tracking time, and that after this time we can bound them up to a poly-log factor.

\begin{thm}\label{EEthm}
Let $(F,A)$ be a graph structure pair with no isolated vertices\footnote{The condition that $F$ has no isolated vertices is necessary in order to disallow the possibility that $A$ is extremely large, in which case $N_\phi(F) \le (n - |A|)^{v_A(F)}$, and so the conclusion of the theorem does not hold. However, one can easily deduce bounds for an arbitrary triple $(F,A,\phi)$, see the discussion after the theorem.}, and let $m \in [m^*]$. Then, with probability at least $1 - n^{-2\log n}$, either the event $\K^\E(m-1)^c$ holds, or the following holds for every faithful injection $\phi \colon A \to V(G_m)$:
\begin{itemize}
\item[$(a)$] If $0 < t \le \omega < t_A(F)$, then 
$$N_\phi(F)(m) \, \in \, \Nt_A(F)(m) \pm f_{F,A}(t) \Nt_A(F)(n^{3/2}).$$ 
\item[$(b)$] If $\omega < t \le t_A(F)$, then 
$$N_\phi(F)(m) \, \in \, \big( 1 \pm g_{F,A}(t) \big) \Nt_A(F)(m).$$ 
\item[$(c)$] If $t > t_A(F)$, then
$$N_\phi(F)(m) \, \le \, ( \log n )^{\Delta(F,H,A)} \Nt_H(F)(m^+),$$
where $A \subsetneq H \subseteq F$ is minimal such that $t < t_H(F)$, and $m^+ = \max\{ m, n^{3/2} \}$.
\end{itemize}
\end{thm}

We emphasize that the theorem holds for \emph{all} structures $F$, not just for those of bounded size. However, the statement becomes trivial if $F$ has more than about $(\log n)^{1/C}$ vertices (outside $A$), edges, or open edges, since then both $(\log n)^{\gamma(F,A)}$ and $(\log n)^{\Delta(F,H,A)}$ become larger\footnote{More precisely, if $(\log n)^{\gamma(F,A)} \ge n^{v_A(F)+e(F)+1}$ then $f_{F,A}(t) \Nt_A(F)(n^{3/2}) \gg n^{v_A(F)}$ for every $t \le \omega$, and also $g_{F,A}(t) \Nt_A(F)(m) \gg n^{v_A(F)}$ for every $\omega < t \le t_A(F)$, since $\Nt_A(F)(m) \ge 1$ for every such $t$.} than the total possible number of labelled copies of $F$ in $G_m$. We shall therefore assume throughout this section that $v_A(F) + e(F) + o(F) \le (\log n)^{1/5}$. In particular, this implies that either $t_A(F) > \omega$ or $t_A(F) = 0$, see Observation~\ref{obs:zeroomega} below. 

Using the assumption that $F$ has no isolated vertices, it follows that $|A| \le (\log n)^{1/5}$. This assumption will be useful later on, since it will allow us to bound the number of graph structure triples $(F,A,\phi)$ in our applications of the union bound. However, this is the only point at which we will need to assume that $F$ does not have isolated vertices; indeed, it is easy to deduce a bound for general $F$, since the isolated vertices in $A$ impose no restrictions (except that their vertices cannot be reused), and because we may embed the remaining isolated vertices arbitrarily. If $v(F)$ is small, then we have about $n$ choices for each. 

\pagebreak

In what follows, we'll need to assume frequently that the conclusion of the theorem holds at all earlier times, so for each $0 \le m' \le m^*$, define an event $\E(m')$ as follows.

\begin{defn}\label{def:events:E}
$\E(m')$ is the event that parts $(a)$, $(b)$ and $(c)$ of Theorem~\ref{EEthm} all hold for:
\begin{itemize}
\item[$(i)$] every graph structure triple $(F,A,\phi)$ with no isolated vertices, and 
\item[$(ii)$] every $m \le m'$ such that the map~$\phi$ is faithful at time~$t$.
\end{itemize}
\end{defn}


Thus, we are required to prove that
$$\Pr \Big( \E(m)^c \cap \Y(m-1) \cap \Z(m-1) \cap \Q(m-1) \Big) \, \le \, n^{-2\log n}$$
for every $m \le m^*$. 


\begin{rmk}\label{Eremark}
In fact, the event $\E(m)$ implies that $N_\phi(F)(m)$ satisfies the conclusion of Theorem~\ref{EEthm} (i.e., either $(a)$, $(b)$ or $(c)$) for \emph{every} graph structure pair with $v(F) = n^{o(1)}$, every $m \in [m^*]$ and every faithful injection $\phi \colon A \to V(G_m)$, see the Appendix~\cite{App}. 
\end{rmk}

The proof of Theorem~\ref{EEthm} is roughly as follows. Proving that part $(a)$ is not the first to go wrong is the most straightforward, and follows by Bohman's method from~\cite{Boh}, suitably generalized (see Section~\ref{landbeforetimeSec}). 
For part $(b)$, we shall count the expected number of copies of $F$ which are created or destroyed in a single step of the process, and deduce that $N_\phi(F)$ exhibits a certain self-correction, of the type considered in Section~\ref{MartSec}. We also analyse the \emph{maximum} possible number of copies of $F$ rooted at $\phi(A)$ which can be created or destroyed in a single step. Together with the method outlined in Section~\ref{MartSec}, this will be enough to deduce that part $(b)$ of the theorem does not go wrong before part $(c)$. 

In order to prove part $(c)$, we break up the structure $F$ into its \emph{building sequence}  
$$A \subseteq H_0 \subsetneq \dots \subsetneq H_\ell = F,$$
which has the property that the graph structure pair $(H_{i+1},H_i)$ is balanced\footnote{For graph structure pairs $(F,A)$ with $t_A(F) > 0$, we shall say that the pair $(F,A)$ is \emph{balanced} if and only if $t_A^*(H) \ge t_A^*(F) > 0$ for every $A \subsetneq H \subseteq F$. If $t_A(F) = 0$ then the definition is slightly more complicated, see Section~\ref{BSsec}.} for every $0 \le i < \ell$, see Section~\ref{BSsec}. Since $\Delta(F,H,A)$ is super-additive with respect to this sequence, this will allow us (see Section~\ref{EEproofSec}) to restrict our attention to balanced pairs $(F,A)$. The bound in $(c)$ at time $t = t_A(F)$ follows for such structures by part $(b)$; in order to prove it for larger values of $t$ we apply a slight variation of our martingale method.\footnote{In fact, in Section~\ref{balancedsec} we shall prove a slight strengthening of this bound for balanced pairs $(F,A)$. We shall need this strengthening in order to bound $|\Delta N_\phi(F)(m)|$ when $t \le t_A(F)$, see Section~\ref{deltaNsec}.}

Finally, we note that the proof of Theorem~\ref{EEthm} is extremely long and technical. We encourage those readers who wish to avoid getting bogged down in such technical detail to skip forward to Section~\ref{Ysec}, and reassure them that the remaining sections are essentially independent of this one. The one crucial thing that they will need in order to follow the proofs in Sections~\ref{Ysec} and~\ref{indepSec} is the definition of the event $\E(m)$.

\subsection{Building sequences}\label{BSsec}

We begin the proof by introducing our key tool for controlling unbalanced graph structures. The following crucial definition allows us to break the process into a sequence of balanced steps, each of which we can control more easily.

\begin{defn}[Building sequences] \label{def:BS}
A \emph{building sequence} of a pair $(F,A)$, where $F$ is a graph structure and $A \subseteq V(F)$ is an arbitrary 
subset\footnote{Note that the edges and open edges of $F[A]$ do not affect the definition, so the reader can think of $A$ as being an independent set, as usual.} of the vertices of $F$, is a collection 
$$A \subseteq H_0 \subsetneq \dots \subsetneq H_\ell = F$$
of induced graph sub-structures of $F$ such that the following conditions hold:
\begin{itemize}
\item $H_0$ is maximal over structures with $t_A^*(H_0) = 0$ and $\Nt_A(H_0)$ at $t = 1/2$ minimal. \smallskip
\item $H_{j+1}$ is maximal over structures $H_j \subsetneq H \subseteq F$ with $t^*_{H_j}(H_{j+1})$ minimal. 
\end{itemize}
Given such a sequence, we define $t_0 = 0$ and $t_j = t_{H_{j-1}}^*(H_j)$ for each $j \in [\ell]$.
\end{defn}

We begin by stating the various properties of building sequences which we shall prove in this subsection; we assume throughout that $v_A(F) + e(F) + o(F) \le (\log n)^{1/5}$. Our first lemma shows that building sequences exist, and are unique. 

\begin{lemma}\label{BS}
Every graph structure pair $(F,A)$ has a unique building sequence.
\end{lemma}

The times $t_j$ are not only well-defined, they are increasing. This fact will also follow from the proof of Lemma~\ref{BS}. 

\begin{lemma}\label{lem:crashtimes}
For every graph structure pair $(F,A)$, we have $0 = t_0 < t_1 < \cdots < t_\ell \le \infty$.
\end{lemma}

Lemmas~\ref{BS} and~\ref{lem:crashtimes} allow us to describe the graph structure in part $(c)$ of Theorem~\ref{EEthm}. 

\begin{lemma}\label{lem:Hj}
Let $0 < t < t^*$, and let $0 \le j \le \ell$ be such that $t_j \le t < t_{j+1}$, where $t_{\ell+1} := \infty$. Then $H_j$ is the minimal $A \subseteq H \subseteq F$ such that $t < t_H(F)$.
\end{lemma}

The next lemma gives an alternative, and perhaps more intuitive characterization of the time $t_\ell$. We remark that if $t < t_\ell$ then the conclusion of the lemma does not hold.\footnote{To see this, simply set $H = H_{\ell-1}$. Note also that Lemma~\ref{lem:tell:NHFsmall} implies the case $j = \ell$ of Lemma~\ref{lem:Hj}.} 

\begin{lemma}\label{lem:tell:NHFsmall}
Let $(F,A)$ be a graph structure pair with building sequence $A \subseteq H_0 \subsetneq \dots \subsetneq H_\ell = F$. If $t \ge t_\ell$, then 
$$\Nt_H(F)(m) \, \le \, (2t)^{e(F) - e(H)}$$
for every $A \subseteq H \subseteq F$.  
\end{lemma} 

We now make an important definition, already mentioned above. 

\begin{defn}[Balanced graph structure pairs] \label{def:balanced}
A graph structure pair $(F,A)$ is said to be \emph{balanced} if the building sequence of $(F,A)$ is either
$$A \subseteq H_0 = F \qquad \text{or} \qquad A = H_0 \subsetneq H_1 = F.$$
Thus $(F,A)$ is balanced if and only if one of the following holds:
\begin{itemize}
\item[$(a)$] $t_A^*(H) \ge t_A^*(F)$ for every $A \subsetneq H \subseteq F$, and $t_A(F) > 0$.
\item[$(b)$]  $e(H) - 2v_A(H) \le e(F) - 2v_A(F)$ for every $A \subseteq H \subseteq F$, and $t_A(F) = 0$.
\end{itemize}
\end{defn}

It follows from Lemma~\ref{lem:crashtimes} that each pair $(H_{i+1},H_i)$ in the building sequence of each $(F,A)$ is balanced. Indeed, we have $t^*_{H_i}(H) \ge t^*_{H_i}(H_{i+1}) > 0$ for every $H_i \subsetneq H \subseteq F$, by Definition~\ref{def:BS} and Lemma~\ref{lem:crashtimes}. The pair $(H_0,A)$ is also balanced; this follows directly from the definitions. 
We record this important property in the following lemma.

\begin{lemma}\label{lem:HHbalanced}
Let $(F,A)$ be a graph structure pair with building sequence $A \subseteq H_0 \subsetneq \dots \subsetneq H_\ell = F$. Then the pairs $(H_0,A)$ and $(H_i,H_{i-1})$ for each $1 \le i \le \ell$ are all balanced.
\end{lemma}


We shall use Lemma~\ref{lem:HHbalanced} in conjunction with the following lemma to deduce bounds on $N_\phi(F)$ for unbalanced pairs $(F,A)$ from bounds for balanced pairs.

\begin{lemma}\label{lem:crucial}
Let $(F,A,\phi)$ be a graph structure triple with building sequence $A \subseteq H_0 \subsetneq \dots \subsetneq H_\ell = F$, and suppose that $\phi$ is faithful at time~$t$. Then, for each $A \subseteq H \subseteq F$, 
$$N_\phi(F)(m) \, \le \, N_{\phi}(H)(m) \cdot \max_{\phi' : \, H \to V(G_m)} N_{\phi'}(F)(m),$$
and hence, for each $0 \le j \le \ell$, 
$$N_\phi(F)(m) \, \le \, N_{\phi}(H_0)(m) \cdot \bigg( \prod_{i = 0}^{j-1} \max_{\phi_i : \, H_i \to V(G_m)} N_{\phi_i}(H_{i+1})(m) \bigg) \cdot \max_{\phi_j : \, H_j \to V(G_m)} N_{\phi_j}(F)(m),$$
where for each $0 \le i \le j$ the maximum is taken over injective maps $\phi_i \colon V(H_i) \to V(G_m)$ which are faithful at time $t$. 
\end{lemma}

For example, we can use this idea to improve (under certain circumstances) our upper bound on $N_\phi(F)$ for unbalanced pairs $(F,A)$. We will use the following lemma in Section~\ref{deltaNsec}, below, to bound the maximum possible size of $|\Delta N_\phi(F)(m)|$.

\begin{lemma}\label{lem:unbalanced}
Let $(F,A,\phi)$ be an unbalanced graph structure triple with building sequence $A \subseteq H_0 \subsetneq \dots \subsetneq H_\ell = F$, and suppose that $t \ge t_\ell$, and that $\phi$ is faithful at time~$t$. If $\E(m)$ holds, then
$$N_\phi(F)(m) \, \le \, (\log n)^{\Delta(F,H_{\ell-1},A)}.$$
\end{lemma}



Having stated the main results of this subsection, let us now turn to the proofs. We begin with a few straightforward observations, which we shall use on numerous occasions throughout the section. Recall that we write $\hat{F}^A$ for the graph structure obtained by removing the edges and open edges from $F[A]$, and that we define $\Nt_A(F) := \Nt_A(\hat{F}^A)$. 
 
\begin{obs}\label{obs:Ntchain}
For every graph structure pair $(F,A)$ and every $A \subseteq H \subseteq F$, 
$$\Nt_A(H) \cdot \Nt_H(F) \, = \, \Nt_A(F).$$
\end{obs}

\begin{proof}
This follows easily from the definition of $\Nt_A(F)$. Indeed, simply note that $v_A(F) = v_A(H) + v_H(F)$, $e(\hat{F}^H) = e(F) - e(H)$ and $o(\hat{F}^H) = o(F) - o(H)$, and use~\eqref{def:NtF}.
\end{proof}

Given a graph structure pair $(F,A)$, and two substructures $A \subseteq H \subseteq F$ and $A \subseteq H' \subseteq F$, we shall write $H \cup H'$  and $H \cap H'$ to denote the substructures of $F$ induced by $V(H) \cup V(H')$ and $V(H) \cap V(H')$, respectively. 

\begin{obs}\label{obs:inclexcl}
For any $A \subseteq H \subseteq F$ and $A \subseteq H' \subseteq F$,
$$\Nt_A(H \cup H') \cdot \Nt_A(H \cap H') \, \le \, \Nt_A(H) \cdot \Nt_A(H').$$
\end{obs}

\begin{proof}
This also follows easily from the definition. Indeed, simply note that 
$$v_A(H \cup H') + v_A(H \cap H') = v_A(H) + v_A(H'),$$
$e(H \cup H') + e(H \cap H') \ge e(H) + e(H')$ and $o(H \cup H') + o(H \cap H') \ge o(H) + o(H')$.
\end{proof}

Note that we do not necessarily have equality in the observation above, since there may be edges of $F$ between $V(H)$ and $V(H')$ which are not in either $H$ or $H'$. The next observation follows easily from the previous two. 

\begin{obs}\label{obs:tHH}
For any $A \subseteq H \subseteq F$ and $A \subseteq H' \subseteq F$, we have
\begin{equation}\label{eq:obs:tHH}
\Nt_H(H \cup H') \, \le \, \Nt_{H \cap H'}(H') \qquad \text{and} \qquad t^*_H(H \cup H') \le t^*_{H \cap H'}(H').
\end{equation}
\end{obs}

\begin{proof}
By Observation~\ref{obs:Ntchain}, we have $\Nt_A(H \cup H') = \Nt_A(H) \cdot \Nt_H(H \cup H')$ and $\Nt_A(H') = \Nt_A(H \cap H') \cdot \Nt_{H \cap H'}(H')$. By Observation~\ref{obs:inclexcl}, it follows that $\Nt_H(H \cup H') \le \Nt_{H \cap H'}(H')$ for every $m \in [m^*]$, and hence $t^*_H(H \cup H') \le t^*_{H \cap H'}(H')$ by Definition~\ref{def:tAF}. 
\end{proof}

Our next observation is slightly more technical. It is also the point at which we use our assumption that $v_A(F) + e(F) + o(F) \le (\log n)^{1/5}$. 

\begin{obs}\label{obs:twots}
Let $(F,A)$ be a graph structure pair, and let $A \subseteq H_1 \subseteq H_2 \subseteq F$ and $A \subseteq H_3 \subseteq H_4 \subseteq F$. If $0 \le t^*_{H_1}(H_2) < t^*_{H_3}(H_4) < \infty$, then
$$t^*_{H_3}(H_4)^2 -  t^*_{H_1}(H_2)^2 \, \gg \, \sqrt{\log n} \, \gg \, e(F) \log \log n.$$ 
\end{obs}

\begin{proof}
This follows easily from the definition of $t_A^*(F)$. Indeed, it follows from~\eqref{def:t*} that if $0 < t_A^*(F) < \infty$, then 
\begin{equation}\label{eq:tstardef}
t_A^*(F) \, = \, \left( \frac{2v_A(F) - e(F)}{8o(F)} \right)^{1/2} \sqrt{\log n},
\end{equation}
and so if $0 \le t^*_{H_1}(H_2) < t^*_{H_3}(H_4) < \infty$, then 
$$t^*_{H_3}(H_4)^2 - t^*_{H_1}(H_2)^2 \, \ge \, \frac{\log n}{64 \cdot o(F)^2} \, \gg \, \sqrt{\log n} \, \gg \, e(F) \log \log n,$$
where the final two inequalities follow since $e(F) + o(F) \le (\log n)^{1/5}$. 
\end{proof}

Finally, let us note an easy consequence of the previous observation, which was already made earlier in the text.

\begin{obs}\label{obs:zeroomega}
Let $(F,A)$ be a graph structure pair. If $t_A(F) > 0$, then 
$$t_{A}(F) \gg (\log n)^{1/4} > \omega.$$ 
Moreover, $o(F) / v_A(F) \le c(F,A) \le 2o(F) + 2$. 
\end{obs}

\begin{proof}
Simply apply Observation~\ref{obs:twots} with $H_1 = H_2 = H_3 = A$ and $A \subsetneq H_4 \subseteq F$ arbitrary. Since $t_A^*(A) = 0$, it follows that $t^*_{A}(H_4) \gg (\log n)^{1/4}$ for every $A \subsetneq H_4 \subseteq F$, and hence $t_A(F) \gg (\log n)^{1/4}$, as required. The bounds on $c(F,A)$ follow immediately from~\eqref{def:cFA}, noting that $2v_A(H) > e(H)$ for every $A \subsetneq H \subseteq F$, since $t_A(F) > 0$. 
\end{proof}

We are now ready to prove Lemmas~\ref{BS} and~\ref{lem:crashtimes}.

\begin{proof}[Proof of Lemmas~\ref{BS} and~\ref{lem:crashtimes}]
We begin by showing that the graph structure $H_0$ is well-defined, and that $t_1 \ge t_{H_0}(F) > 0$. Set
$$\HH_0 \, = \, \Big\{ A \subseteq H \subseteq F \, : \, \Nt_A(H) \textup{ at $t = 1/2$ is minimal} \Big\}.$$
We claim that the collection $\HH_0$ is union-closed,\footnote{Note that if $t_A(F) > 0$ then $\HH_0 = \{A\}$, and so the claim holds trivially.} and hence that $H_0 = \bigcup_{H \in \HH_0} H$. Indeed, let $H, H' \in \HH_0$, and recall that
\begin{equation}\label{eq:obs:inclexcl:repeat}
\Nt_A(H \cup H') \cdot \Nt_A(H \cap H') \, \le \, \Nt_A(H) \cdot \Nt_A(H')
\end{equation}
by Observation~\ref{obs:inclexcl}. Moreover, by the minimality of $\Nt_A(H) = \Nt_A(H')$, we have 
$$\Nt_A(H \cup H') \ge \Nt_A(H) \qquad \textup{and} \qquad \Nt_A(H \cap H') \ge \Nt_A(H')$$ 
at time $t = 1/2$. It follows that the four terms in~\eqref{eq:obs:inclexcl:repeat} are all equal at $t = 1/2$. In particular, we have $\Nt_A(H \cup H') = \Nt_A(H)$, and hence $H \cup H' \in \HH_0$, as claimed. 

Now, suppose that $t^*_{H_0}(H) = 0$ for some $H_0 \subsetneq H \subseteq F$. Then $e(H) - e(H_0) \ge 2v_{H_0}(H)$, and so $\Nt_{H_0}(H) \le 1$ at time $t = 1/2$. By Observation~\ref{obs:Ntchain} it follows that 
$$\Nt_A(H) \, = \, \Nt_{A}(H_0) \cdot \Nt_{H_0}(H) \, \le \, \Nt_{A}(H_0)$$
at $t = 1/2$, which contradicts the maximality of $H_0$. Thus $t_{H_0}(F) > 0$ as claimed.

Now suppose that we have already constructed $A \subseteq H_0 \subsetneq \cdots \subsetneq H_i \neq F$ in a unique way; we claim that there exists a unique $H_i \subsetneq H_{i+1} \subseteq F$ which is maximal over structures with $t^*_{H_i}(H_{i+1})$ minimal, and that $t_{i+1} \ge t_{H_i}(F) > t_i$. The argument is almost the same as that above. Indeed, setting 
$$\HH_{i+1} \, = \, \Big\{ H_i \subsetneq H \subseteq F \, : \, t^*_{H_i}(H) \textup{ is minimal} \Big\},$$
we make the following claim.

\medskip

\noindent \textbf{Claim:} $\HH_{i+1}$ is union-closed and $t_{i+1} > t_i$.

\begin{proof}[Proof of claim]
We first show that $t_{H_i}(F) > t_i$, i.e., that $t^*_{H_i}(H) > t_i$ for each $H_i \subsetneq H \subseteq F$. When $i = 0$ this was proved above, so let $i \ge 1$ and suppose that $t^*_{H_i}(H) \le t_i$ for some $H_i \subsetneq H \subseteq F$. Then, at $t = t_i$, we have
$$\Nt_{H_{i-1}}(H) \, = \, \Nt_{H_{i-1}}(H_i) \cdot \Nt_{H_i}(H) \, \le \, (2t)^{e(H) - e(H_{i-1})},$$
and so $t^*_{H_{i-1}}(H) \le t_i$, which contradicts the maximality of $H_i$. Hence $t_{H_i}(F) > t_i$, as claimed.

Next, let $H,H' \in \HH_{i+1}$, and note that we have $\Nt_{H_i}(H \cap H') \ge 1$ at $t = t_{i+1}$, since either $H \cap H' = H_i$, or $t_{H_i}^*(H \cap H') \ge t_{i+1} > \omega$, by Observation~\ref{obs:zeroomega}. It follows that
$$\Nt_{H_i}(H \cup H') \, \le \, \Nt_{H_i}(H \cup H') \cdot \Nt_{H_i}(H \cap H') \, \le \, \Nt_{H_i}(H) \cdot \Nt_{H_i}(H') \, \le \, (2t)^{e(H) + e(H') - 2e(H_i)}$$
at time $t = t_{i+1}$, by Observation~\ref{obs:inclexcl} and since $t_{H_i}^*(H) = t_{H_i}^*(H') = t_{i+1}$. If $e(H \cup H') = e(H) + e(H') - e(H_i)$, then it follows 
that $t^*_{H_i}(H \cup H') \le t_{i+1}$, and hence that $H \cup H' \in \HH_{i+1}$, as required. On the other hand, if $e(H \cup H') > e(H) + e(H') - e(H_i)$ then we gain a factor of $\sqrt{n} / 2t$ in our application of Observation~\ref{obs:inclexcl}. Since $\sqrt{n} \gg (2t)^{2e(F)}$, it follows that $t^*_{H_i}(H \cup H') \le t_{i+1}$ in this case also, as claimed.
\end{proof}

The claim implies that $H_{i+1} = \bigcup_{H \in \HH_{i+1}} H$, which completes the proof of the lemma.
\end{proof}

We next prove Lemmas~\ref{lem:Hj} and~\ref{lem:tell:NHFsmall}. Both are easy consequences of the following lemma.

\begin{lemma}\label{sublem:Hj}
Let $(F,A)$ be a graph structure pair, and let $A \subseteq H \subsetneq F$. If $0 \le j \le \ell$ is minimal such that $H_j \not\subseteq H$, then 
$$t_H(F) \, \le \, t^*_H(H \cup H_j) \, \le \, t_j.$$
\end{lemma}

\begin{proof}
We shall in fact prove that 
\begin{equation}\label{eq:sublem:tHF}
t_H(F) \, \le \, t^*_H(H \cup H_j) \, \le \, t^*_{H \cap H_j}(H_j)  \, \le \, t_j.
\end{equation}
The first inequality follows by the definition of $t_H(F)$, and the second follows by Observation~\ref{obs:tHH}. In order to prove the third, note that $H_{j-1} \subseteq H \cap H_j$, and that therefore
\begin{equation}\label{eq:sublem:NHH}
\Nt_{H_{j-1}}(H_j) \, = \, \Nt_{H_{j-1}}(H \cap H_j) \cdot \Nt_{H \cap H_j}(H_j)
\end{equation}
by Observation~\ref{obs:Ntchain}, where we set $H_{-1} = A$. Now, if $j = 0$ then, by the definition of $H_0$, we have $\Nt_A(H \cap H_0) \ge \Nt_A(H_0)$ at time $t = 1/2$. By~\eqref{eq:sublem:NHH}, it follows that $\Nt_{H \cap H_0}(H_0) \le 1$, and hence $t^*_{H \cap H_0}(H_0) = 0$, as claimed. On the other hand, if $j \ge 1$ then
$$\Nt_{H_{j-1}}(H_j) = (2t)^{e(H_j) - e(H_{j-1})} \qquad \text{and} \qquad \Nt_{H_{j-1}}(H \cap H_j) \ge (2t)^{e(H \cap H_j) - e(H_{j-1})}$$ 
at time~$t = t_j$, since $H_j$ minimizes $t^*_{H_{j-1}}(H_j) = t_j$. By~\eqref{eq:sublem:NHH}, it follows that 
$$\Nt_{H \cap H_j}(H_j) \le (2t)^{e(H_j) - e(H \cap H_j)}$$ 
at time $t = t_j$, which implies $t^*_{H \cap H_j}(H_j) \le t_j$, as claimed. Hence~\eqref{eq:sublem:tHF} holds, as required.
\end{proof}

We next prove Lemma~\ref{lem:tell:NHFsmall}, which gives a natural alternative definition of the time $t_\ell$, and deals with the case $j = \ell$ of Lemma~\ref{lem:Hj}. We will also use it later on, see Section~\ref{balancedsec}. 

\begin{proof}[Proof of Lemma~\ref{lem:tell:NHFsmall}]
Recalling that $H_\ell = F$, and this time setting $H_{-1} = \emptyset$, we have
\begin{equation}\label{eq:buildNHF}
\Nt_H(F) \, = \, \prod_{j=0}^\ell \Nt_{H \cup H_{j-1}}(H \cup H_j), 
\end{equation}
by Observation~\ref{obs:Ntchain}. In order to bound $\Nt_{H \cup H_{j-1}}(H \cup H_j)$, note that if $H \cup H_{j-1} \neq H \cup H_j$ then $H_j \not\subseteq H \cup H_{j-1}$. 
Thus, applying Lemma~\ref{sublem:Hj} to $H \cup H_{j-1}$, we obtain
$$t^*_{H \cup H_{j-1}}(H \cup H_j) \, \le \, t_j,$$
and hence
$$\Nt_{H \cup H_{j-1}}(H \cup H_j)(m) \, \le \, (2t)^{e(H \cup H_j) - e(H \cup H_{j-1})}$$ 
for every $t \ge t_j$. Since $t \ge t_\ell \ge t_j$ for each $0 \le j \le \ell$, this (together with~\eqref{eq:buildNHF}) implies that $\Nt_H(F)(m) \le (2t)^{e(F) - e(H)}$ for every $t \ge t_\ell$, as required. 
\end{proof}

We can now easily deduce Lemma~\ref{lem:Hj}. 

\begin{proof}[Proof of Lemma~\ref{lem:Hj}]
The case $j = \ell$ follows by Lemma~\ref{lem:tell:NHFsmall}, since it implies that $t^*_H(F) \le t_\ell$ for every $A \subseteq H \subsetneq F$, and we have $t_F(F) = t^*$. So let $0 \le j \le \ell - 1$ and suppose that $t_j \le t < t_{j+1}$. Note first that $t_{H_j}(F) = \min\{t_{j+1},t^*\}$, by the definition of $H_{j+1}$, so $t < t_{H_j}(F)$, as required. On the other hand, by Lemma~\ref{sublem:Hj} we have $t_H(F) \le t_j \le t$ for every $A \subseteq H \subseteq F$ with $H \not\supseteq H_j$, since the $t_j$ are increasing, by Lemma~\ref{lem:crashtimes}. Hence $H_j$ is the minimal $A \subseteq H \subseteq F$ such that $t < t_H(F)$, as claimed.
\end{proof}

Finally, let us prove Lemma~\ref{lem:unbalanced}. 

\begin{proof}[Proof of Lemma~\ref{lem:unbalanced}]
By Lemma~\ref{lem:crashtimes} we have $t_{\ell-1} < t_\ell$. Thus, by Lemma~\ref{lem:crucial} and the event $\E(m)$, and using Remark~\ref{Eremark}, we have 
$$N_\phi(F)(m) \, \le \, N_\phi(H_{\ell-1})(m) \cdot \max_{\phi' : \, H_{\ell-1} \to V(G_m)} N_{\phi'}(F)(m) \, \le \, (\log n)^{\Delta(F,H_{\ell-1}) + \Delta(H_{\ell-1},A)},$$
as claimed.
\end{proof}

\subsection{Self-correction}\label{selfsec}

In this section we shall prove that, for each graph structure triple $(F,A,\phi)$, the random variable $N_\phi(F)$ is self-correcting (in the sense of Section~\ref{MartSec}) on the interval $\omega < t \le t_A(F)$, as long as the events $\E(m)$,  $\Y(m)$, $\Z(m)$ and $\Q(m)$ all hold, and the map $\phi$ is faithful. This will be a crucial tool in our proof that these variables track the functions $\Nt_A(F)$ on this interval. We begin by defining $N^*_\phi(F)$ to be the normalized error in $N_\phi(F)$ at time $t$, i.e.,
\begin{equation}\label{def:Nstar}
N^*_\phi(F)(m) \, = \, \frac{N_\phi(F)(m) - \Nt_A(F)(m)}{g_{F,A}(t) \Nt_A(F)(m)}.
\end{equation}
Since $N_\phi(F) = \big( 1 + g_{F,A}(t) N^*_\phi(F) \big) \Nt_A(F)$, in order to prove Theorem~\ref{EEthm}$(b)$ we are required to prove that $|N^*_\phi(F)(m)| \le 1$ for every $\omega < t \le t_A(F)$. Recall from Section~\ref{sketchSec} that $\eps > 0$ and $C = C(\eps) > 0$ are constants (with $\eps$ sufficiently small and $C$ sufficiently large) that are fixed throughout the proof, and recall from the discussion before Theorem~\ref{EEthm} that $c(F,A) \ge 2$ for every graph structure pair $(F,A)$. We will prove the following key lemma.

\begin{lemma}\label{selfN*}
Let $(F,A,\phi)$ be a graph structure triple, let $\omega < t \le t_A(F)$, and suppose that $\phi \colon A \to V(G_m)$ is faithful at time $t$. If $\E(m) \cap \Y(m) \cap \Z(m) \cap \Q(m)$ holds, then 
$$\Ex\big[ \Delta N^*_\phi(F)(m) \big] \, \in \, \bigg( c(F,A) + \frac{e(F)}{2t^2} \bigg) \cdot \frac{2t}{n^{3/2}} \cdot \Big( - N_\phi^*(F)(m) \pm \eps \Big).$$
\end{lemma}

We begin by calculating the expected change in $N_\phi(F)$. We will need the following family of graph structures.

\begin{defn}
Given a permissible graph structure $F$, we define $\F^o_F$ to be the family of (labelled) graph structures $F^o$ obtained by changing an edge of $F$ into an open edge. We call this (changed) edge \emph{$F$-vulnerable}.
\end{defn}

We make a quick observation, which follows immediately from the definition, and which we shall use frequently in the proofs below.

\begin{obs}\label{obs:NtF-}
Let $(F,A)$ be a graph structure pair. Then $|\F^o_F| = e(F)$, and 
$$2te^{4t^2} \cdot \Nt_A(F^o)(m) \, = \, \sqrt{n} \cdot \Nt_A(F)(m)$$
for every $F^o \in \F_F^o$.
\end{obs}

\begin{proof}
Since $v_A(F) = v_A(F^o)$, $e(F) = e(F^o) + 1$ and $o(F) = o(F^o) - 1$, the observation follows immediately from the definition~\eqref{def:NtF}.
\end{proof}

Recall that we use $N_\phi(F)$ to denote both the collection of copies of $F$ rooted at $\phi(A)$ in $G_m$, and the size of this collection. 

\begin{lemma}\label{selfN}
Let $(F,A,\phi)$ be a graph structure triple, and suppose that $\phi$ is faithful at time~$t$, where $0 < t \le t^*$. If $\Z(m)$ holds, then
$$\Ex\big[ \Delta N_\phi(F) \big] \, \in \, \frac{1}{Q(m)} \bigg( - \sum_{F^* \in N_\phi(F)} \sum_{f \in O(F^*)} Y_f(m) \,+ \sum_{F^o \in \F^o_F} N_{\phi}(F^o) \, \pm \, o(F)^2 (\log n)^2 N_\phi(F) \bigg).$$
\end{lemma}

\begin{proof}
Let $e$ be the edge chosen in step $m+1$ of the triangle-free process. A copy of $F$ rooted at $\phi(A)$ is created when $e$ is the $F$-vulnerable edge of a copy of some $F^o \in \F^o_F$, rooted at $\phi(A)$. Moreover, each copy of $F$ is created by exactly one such structure $F^o$. Thus the expected number of such copies of $F$ created in a single step is exactly 
\begin{equation}\label{eq:number:created}
\sum_{F^o \in \F^o_F} \frac{N_{\phi}(F^o)}{Q(m)}.
\end{equation}
Similarly, a copy $F^*$ of $F$ rooted at $\phi(A)$ is destroyed when an open edge of that copy is closed by the addition of $e$, which occurs with probability 
\begin{equation}\label{eq:selfNdestroyed}
\frac{1}{Q(m)} \bigg| \bigcup_{f \in O(F^*)} Y_f(m) \bigg| \, \in \, \frac{1}{Q(m)} \sum_{f \in O(F^*)} Y_f(m) \,\pm\, \frac{1}{Q(m)} \sum_{\substack{f,f' \in O(F^*)\\ f \neq f'}} \big| Y_f(m) \cap Y_{f'}(m) \big|.
\end{equation}

Now, if $f$ and $f'$ are disjoint then $|Y_f(m) \cap Y_{f'}(m)| \le 2$, so suppose that $e = \{u,v\}$ closes both $f = \{v,w\}$ and $f' = \{v,z\}$; then $\{u,w\}$ and $\{u,z\}$ must both be edges of $G_m$, and so (assuming the event $\Z(m)$ holds) there are at most $(\log n)^2$ such edges $e$. Combined with~\eqref{eq:selfNdestroyed}, this implies that the expected number of copies of $F$ destroyed by $e$ is contained in the interval
\begin{equation}\label{eq:selfNapprox}
\frac{1}{Q(m)} \sum_{F^* \in N_\phi(F)} \bigg( \sum_{f \in O(F^*)} Y_f(m) \,\pm\, o(F)^2 (\log n)^2 \bigg),
\end{equation}
as required.
\end{proof}

In the proof of Lemma~\ref{selfN*}, and several times later in the paper, we shall need to use the Product Rule, which we state here for later reference.

\begin{lemma}[The Product Rule]\label{chain}
For any random variables $a(m)$ and $b(m)$,  
\begin{equation}\label{eq:chain}
\Ex \big[ \Delta \big( a(m) b(m) \big) \big] \, = \, a(m) \Ex \big[ \Delta b(m) \big] + b(m) \Ex \big[ \Delta a(m) \big] + \Ex\big[ \big( \Delta a(m) \big) \big( \Delta b(m) \big) \big].
\end{equation}
In particular, if $a(m)$ is deterministic, then
\begin{equation}\label{eq:chain2}
\Ex \big[ \Delta \big( a(m) b(m) \big) \big] \, = \, a(m) \Ex \big[ \Delta b(m) \big] + \Delta a(m) \Big( b(m) + \Ex\big[ \Delta b(m) \big] \Big).
\end{equation}
\end{lemma}

\begin{proof}
Simply expand the right-hand side, and use linearity of expectation.
\end{proof}

We remark that for all of the random variables $A(m)$ which we shall need to consider, the single step change $\Delta A(m)$ will be \emph{much} smaller than $A(m)$, and hence the final term in~\eqref{eq:chain} will be negligible. In order to bound $| \Delta N_\phi^*(F)(m) |$ and $\Ex\big[ | \Delta N_\phi^*(F)(m) | \big]$, we shall also need the following related inequality.


\begin{lemma}\label{lem:chainstar}
Let $A(m)$ be a random variable, let $\At(m)$ and $g(t)$ be functions, and set  
$$A^*(m) \, = \, \frac{A(m) - \At(m)}{g(t) \At(m)}.$$
If $|A(m)| \le \big( 1 + g(t) \big) \At(m)$,
\begin{equation}\label{eq:deltaAgA}
|\Delta \At(m) | \, \ll \, \frac{\log n}{n^{3/2}} \cdot \At(m)  \quad \text{and} \quad |\Delta \big( g(t) \At(m) \big) | \ll \frac{\log n}{n^{3/2}} \cdot g(t) \At(m),
\end{equation}
then
$$| \Delta A^*(m) | \, \le \, 2 \cdot \left( \frac{| \Delta A(m) |}{g(t) \At(m)} \,+\, \frac{1 + g(t)}{g(t)} \cdot \frac{\log n}{n^{3/2}} \right).$$
\end{lemma}


 
We postpone the (straightforward) proof to the Appendix~\cite{App}, and remark that the condition~\eqref{eq:deltaAgA} is satisfied by the functions $\Nt_A(F)$ and $g_{F,A}(t)$. We shall also use the following easy observation.  

\begin{obs}\label{obs:F-}
If $(F,A)$ is a graph structure pair and $F^o \in \F_F^o$, then $t_A(F) \le t_A(F^o)$.
\end{obs}

\begin{proof}
This follows easily from the definitions, using Observations~\ref{obs:zeroomega} and~\ref{obs:NtF-}. See the Appendix for the details.
\end{proof}


Finally, we need the following relations between different error terms. 

\begin{obs}\label{obs:gfat}
Let $(F,A)$ be a graph structure pair, and let $F^o \in \F^o_F$. Then 
$$g_q(t) \ll o(F) g_y(t) \ll g_{F,A}(t) \qquad \text{and} \qquad g_{F^o,A}(t) \ll g_{F,A}(t)$$ 
as $n \to \infty$.
\end{obs}

\begin{proof}
These also both follow easily from the definitions, using Observation~\ref{obs:F-}. We again postpone the details to the Appendix.  
\end{proof}


We are now ready to prove Lemma~\ref{selfN*}. In the proof below, in order to simplify the calculations we shall use the symbol $\approx$ to indicate equality up to a multiplicative factor of at most $1 \pm O(1/n)$. We emphasize that this error term will never play an important role.

\begin{proof}[Proof of Lemma~\ref{selfN*}]
By Lemma~\ref{selfN}, and since $\Z(m)$ holds, we have
$$\Ex\big[ \Delta N_\phi(F) \big] \, \in \, - \frac{1}{Q(m)} \sum_{F^* \in N_\phi(F)} \sum_{f \in O(F^*)} Y_f(m) \,+\, \sum_{F^o \in \F^o_F} \frac{N_{\phi}(F^o)}{Q(m)} \,\pm\, \frac{o(F)^2 (\log n)^2 N_\phi(F)}{Q(m)}.$$
Moreover, differentiating~\eqref{def:NtF}, we obtain
\begin{align*}
\Delta \Nt_A(F) & \, \approx \, \frac{1}{n^{3/2}} \bigg( \frac{e(F)}{t} - 8t o(F) \bigg) \Nt_A(F) \, \approx \, - \, o(F) \cdot \frac{\Yt(m)}{\Qt(m)} \cdot  \Nt_A(F) + \sum_{F^o \in \F^o_F} \frac{\Nt_A(F^o)}{\Qt(m)},
\end{align*}
since $\frac{\Yt(m)}{\Qt(m)} \approx \frac{8t}{n^{3/2}}$, $|\F^o_F| = e(F)$ and $t n^{3/2} \cdot \Nt_A(F^o) = \Nt_A(F) \cdot e^{-4t^2}  n^2 / 2 \approx \Nt_A(F) \cdot \Qt(m)$. Subtracting, and using the event $\Y(m)$ and our assumption that $o(F) \ll n^{o(1)}$, we obtain 
\begin{multline}\label{eq:selfN:line1}
\Ex\big[ \Delta N_\phi(F) \big] - \Delta \Nt_A(F) \, \in \,- \big( 1 \pm g_y(t) \big) \cdot o(F) \cdot \frac{N_\phi(F) \cdot \Yt(m)}{Q(m)} \,+\, o(F) \cdot \frac{ \Nt_A(F) \cdot \Yt(m)}{\Qt(m)} \\
\, + \, \sum_{F^o \in \F^o_F} \bigg( \frac{N_{\phi}(F^o)}{Q(m)} - \frac{\Nt_A(F^o)}{\Qt(m)} \bigg) \,\pm\, \frac{n^{o(1)} N_\phi(F)}{Q(m)}.
\end{multline}
Note that we have $\omega < t \le t_A(F) \le t_A(F^o)$ for each $F^o \in \F^o_F$, by Observation~\ref{obs:F-}. Thus, using~\eqref{def:Nstar}, the event $\Q(m)$ and the fact that $g_q(t) \ll g_y(t)$ to bound the first term, and the event $\E(m) \cap \Q(m)$ to bound the third and fourth terms, it follows that the right-hand side of~\eqref{eq:selfN:line1} is contained in
\begin{multline*}
\Big( 1 \,-\, \big( 1 \pm 2g_y(t) \big) \big( 1 + g_{F,A}(t) N^*_\phi(F) \big) \Big) \cdot o(F) \cdot \frac{ \Nt_A(F) \cdot \Yt(m)}{\Qt(m)} \\
\pm \, \frac{1}{\Qt(m)} \sum_{F^o \in \F^o_F} \bigg( \frac{1 \pm g_{F^o,A}(t)}{1 \pm g_q(t)} - 1 \bigg)  \Nt_A(F^o) \,\pm\, n^{o(1)} \frac{\big( 1 + g_{F,A}(t) \big) \Nt_A(F)}{\Qt(m)}.
\end{multline*}
Now, dividing both sides by $g_{F,A}(t) \Nt_A(F)$, and using Observation~\ref{obs:gfat}, we obtain\footnote{Here we again use the fact that $\frac{\Yt(m)}{\Qt(m)} \approx \frac{8t}{n^{3/2}}$, $|\F^o_F| = e(F)$ and $t n^{3/2} \cdot \Nt_A(F^o) \approx \Nt_A(F) \cdot \Qt(m)$. To bound the final term, recall that $c(F,A) \ge 2$, and so $g_{F,A}(t) \Qt(m) \ge n^{7/4} e^{-2t^2} \ge n^{3/2 + \eps}$ for every $t \le t^*$.}
\begin{equation}\label{eq:deltaNstar1}
\frac{\Ex\big[ \Delta N_\phi(F) \big] - \Delta \Nt_A(F)}{g_{F,A}(t) \Nt_A(F)} \, \in \, - \, \frac{8t  \cdot o(F) }{n^{3/2}} \cdot N^*_\phi(F) \,\pm\, \frac{o(1)}{n^{3/2}} \bigg( t + \frac{e(F)}{t} \bigg).
\end{equation}

The proof is almost complete; all that remains is a little simple analysis. Indeed, since $N_\phi(F) - \Nt_A(F) = g_{F,A}(t)  \Nt_A(F) \cdot N^*_\phi(F)$, the Product Rule (Lemma~\ref{chain}) gives
\begin{equation}\label{eq:Nphichain}
\frac{\Ex\big[ \Delta N_\phi(F) \big] - \Delta \Nt_A(F)}{g_{F,A}(t) \Nt_A(F)} \, = \, \Ex\big[ \Delta N^*_\phi(F) \big] \,+\, \frac{\Delta \big( g_{F,A}(t) \Nt_A(F) \big)}{g_{F,A}(t) \Nt_A(F)} \Big( N^*_\phi(F) \,+\,  \Ex\big[ \Delta N^*_\phi(F) \big] \Big),
\end{equation}
and since $g_{F,A}(t) \Nt_A(F)$ is equal to $(2t)^{e(F)} e^{(c(F,A) - 4o(F))t^2}$ times some function of $n$, we have 
\begin{equation}\label{eq:deltagNt}
\Delta \big( g_{F,A}(t) \Nt_A(F) \big) \, \approx \, \frac{1}{n^{3/2}} \left( \frac{e(F)}{t} + 2t \Big( c(F,A) - 4o(F) \Big) \right) \cdot g_{F,A}(t) \Nt_A(F).
\end{equation}
Combining the last three displayed lines, and observing that the terms involving $8t \cdot o(F)$ in~\eqref{eq:deltaNstar1} and~\eqref{eq:deltagNt} cancel one another\footnote{Note also that, by~\eqref{eq:deltagNt}, 
the final term in~\eqref{eq:Nphichain} is swallowed by the error term.}, we obtain
$$\Ex\big[ \Delta N^*_\phi(F) \big] \in - \,  \frac{N^*_\phi(F)}{n^{3/2}} \left( \frac{e(F)}{t} + 2t \cdot c(F,A) \right) \,\pm\, \frac{o(1)}{n^{3/2}} \cdot \bigg( t + \frac{e(F)}{t} \bigg).$$
Noting again that $c(F,A) \ge 2$ for every pair $(F,A)$, it follows that
$$\Ex\big[ \Delta N^*_\phi(F) \big] \, \in \, \bigg( c(F,A) + \frac{e(F)}{2t^2} \bigg) \cdot \frac{2t}{n^{3/2}} \cdot \Big( - N_\phi^*(F) \pm \eps \Big),$$
as required.
\end{proof}

We finish the subsection by deducing the following easy consequence of the observations above, which will be necessary in the martingale argument to follow.

\begin{lemma}\label{gammaNF}
Let $(F,A,\phi)$ be a graph structure triple, let $\omega < t \le t_A(F)$, and suppose that $\phi$ is faithful at time $t$.  If $\E(m) \cap \Y(m) \cap \Z(m) \cap \Q(m)$ holds, then
$$\Ex\big[ | \Delta N_\phi^*(F)(m) | \big] \, \le \, \frac{C \cdot \log n}{n^{3/2}} \cdot \left(  \frac{1 + g_{F,A}(t)}{g_{F,A}(t)} \right).$$
\end{lemma}

\begin{proof}
Observe first that $\Ex\big[ | \Delta N_\phi(F)(m) | \big]$ is at most the expected number of copies of $F$ rooted at $\phi(A)$ created in step $m+1$ of the triangle-free process, plus the expected number of copies destroyed. By~\eqref{eq:number:created}, 
and since $\E(m) \cap \Q(m)$ holds and $\omega < t \le t_A(F) \le t_A(F^o)$, the expected number of copies created is 
$$\sum_{F^o \in \F^o_F} \frac{N_{\phi}(F^o)}{Q(m)} \, \le \, \sum_{F^o \in \F^o_F} \frac{1 + g_{F^o,A}(t)}{1 - g_q(t)} \cdot \frac{\Nt_A(F^o)}{\Qt(m)}  \, \le \, \big( 1 + g_{F,A}(t) \big) \cdot \frac{e(F)}{t \cdot n^{3/2}} \cdot \Nt_A(F),$$
where the second inequality follows\footnote{In particular, we use the fact that $g_q(t) + g_{F^o,A}(t) + n^{-1} \ll g_{F,A}(t)$.} using Observations~\ref{obs:NtF-} and~\ref{obs:gfat}.

Similarly, by~\eqref{eq:selfNapprox} 
and the event $\E(m) \cap \Y(m) \cap \Z(m) \cap \Q(m)$, the expected number of copies destroyed in step $m+1$ is at most
$$\frac{1}{Q(m)} \sum_{F^* \in N_\phi(F)} \bigg( \sum_{f \in O(F^*)} Y_f(m) \, + \, o(F)^2 (\log n)^2 \bigg) \, \le \, 2 \cdot \big( 1 + g_{F,A}(t) \big) \cdot o(F) \cdot \frac{8t}{n^{3/2}} \cdot \Nt_A(F),$$
since $g_q(t) \ll g_y(t) \ll 1$ and $o(F)^2 (\log n)^2 \ll n^{o(1)} \ll g_y(t) \Yt(m)$ for all $t \le t^*$. Thus
$$\Ex\big[ | \Delta N_\phi(F)(m) | \big] \, \le \, 16t \cdot \big( 1 + g_{F,A}(t) \big) \bigg(  \frac{e(F)}{t^2} + o(F) \bigg) \frac{\Nt_A(F)(m)}{n^{3/2}}.$$
Now, by Lemma~\ref{lem:chainstar}, we have\footnote{The conditions in~\eqref{eq:deltaAgA} follow from~\eqref{eq:deltagNt}, the event $\E(m)$ and the fact that $e(F) + o(F) + c(F,A) \ll (\log n)^{1/4}$, which holds by Observation~\ref{obs:zeroomega}.}
$$\Ex\big[ | \Delta N_\phi^*(F)(m) | \big] \, \le \, 2 \cdot \left( \frac{\Ex\big[ | \Delta N_\phi(F)(m) | \big]}{g_{F,A}(t) \Nt_A(F)(m)} \,+\, \frac{1 + g_{F,A}(t)}{g_{F,A}(t)} \cdot \frac{\log n}{n^{3/2}} \right),$$
and hence it follows that
$$\Ex\big[ | \Delta N_\phi^*(F)(m) | \big] \, \le \, \frac{C \cdot \log n}{n^{3/2}} \cdot \left(  \frac{1 + g_{F,A}(t)}{g_{F,A}(t)} \right),$$
as required.
\end{proof}

\subsection{Creating and destroying copies of $F$}\label{createsec}

In order to apply our martingale technique to the self-correcting variables $N_\phi(F)$, we shall also need to bound the maximum possible step size of each of these variables, under the assumption that all of the other variables are still tracking. We shall do so by showing that each copy of $F$ rooted at $\phi(A)$ which is created or destroyed in step $m+1$ corresponds to another graph structure $F'$ in $G_m$. We shall thus be able to bound the number of such copies of $F$ using the event $\E(m)$.

Let us first consider the number of copies of $F$ rooted at $\phi(A)$ which can be created by the addition of a single edge $e$. Note that this is exactly the number of copies (in $G_m$) of graphs in $\F^o_F$ whose $F$-vulnerable edge is $e$, and which are rooted at $\phi(A)$. This observation suggests the following definition. 

\begin{defn}
Given a graph structure $F$ and an independent set $A \subseteq V(F)$, define the family $\F_{F,A}^+$ to be the collection of (labelled) pairs $(F^+,A^+)$ obtained by absorbing the endpoints of an edge of $F$ into $A$ (to form $A^+$), and removing the edges inside $A^+$.
\end{defn}

Note that there are at most $e(F)$ pairs $(F^+,A^+)$ in $\F_{F,A}^+$. The following lemma motivates the definition above.

\begin{lemma}\label{lem:F+}
Let $(F,A,\phi)$ be a graph structure triple. The number of copies of $F$ rooted at $\phi(A)$ created by the addition of a single edge to $G_m$ is at most
$$\sum_{(F^+,A^+) \in \F_{F,A}^+} \max_{\phi^+ : \, A^+ \to V(G_m)} N_{\phi^+}(F^+)(m),$$
where the maximum is over faithful maps $\phi^+ \colon A^+ \to V(G_m)$.  
\end{lemma}

\begin{proof}
Let $e \in E(K_n)$ be the edge added in step $m+1$, and consider the family of copies of $F$ rooted at $\phi(A)$ which are created by the addition of $e$. We partition this family according to the endpoints of $e$ in $V(F) \setminus A$, i.e., according to which edge of $F$ was added last. 

Now, each part of this partition corresponds\footnote{Note that if two edges share an endpoint and have the other endpoint in $A$ then we obtain the same pair $(F^+,A^+)$. However, in that case one of the two parts corresponding to $(F^+,A^+)$ is empty.} to a pair $(F^+,A^+) \in \F_{F,A}^+$; let us consider one such pair. The number of copies of $F$ in the corresponding part is exactly $N_{\phi^+}(F^+)$, where $\phi^+ \colon A^+ \to V(G_m)$ satisfies $\phi^+|_A = \phi$ and maps the extra vertex (or vertices) of $A^+$ to the endpoint(s) of $e$. The lemma follows immediately. 
\end{proof}

We next turn to destroying copies of $F$. Given a pair $(F,A)$, we would like to define a family $\F^-_{F,A}$ of pairs ($F^-,A^-)$ in such a way that the appearance of such a pair in $G_m$ corresponds to the destruction, by a given edge $e \in E(K_n)$, of a copy of $F$ rooted at some given $\phi(A)$. After some thought, this leads to the following, somewhat convoluted definition.  

\begin{defn}\label{def:F-}
$\F^-_{F,A}$ consists of all graph structure pairs $(F^-,A^-)$ which are obtained as follows:\footnote{We write $F + v$ and $F \cup \{e\}$ for the graph structures with (vertex, edge, open edge) sets $\big( V(F) \cup \{v\}, E(F), O(F) \big)$ and $\big( V(F), E(F) \cup \{e\}, O(F) \setminus \{e\} \big)$ respectively.}
\begin{itemize}
\item[$(a)$] Set $A^- = A$ and $F^- = F \cup \{e\}$, where $e \not\in E(F) \cup O(F)$ is an edge from a vertex of $A$ to a vertex $v \in V(F) \setminus A$, where $v$ has an open $F$-neighbour in $A$.
\item[$(b)$] Set $A^- = A \cup \{v\}$ for some $v \in V(F) \setminus A$, and let $F^- = \hat{F}^{A^-}$ be obtained from $F$ by removing the edges inside $A^-$. 
\item[$(c)$] Set $A^- = A \cup \{v\}$ for some $v \not\in V(F)$, and let $F^- = (F + v) \cup \{e\}$ be obtained by adding to $F$ the vertex $v$ and an edge $e$ from $v$ to some vertex of $V(F) \setminus A$.
\item[$(d)$] Set $A^- = A \cup \{u,v\}$ for some $u,v \in V(F) \setminus A$, and let $F^- = \hat{F}^{A^-}$ be obtained from $F$ by removing the edges inside $A^-$.
\item[$(e)$] Set $A^- = A \cup \{u,v\}$ for some $u \in V(F) \setminus A$ and $v \not\in V(F)$, and let $F^- = \big(\hat{F} + v \big)^{A^-}$ be obtained from $F$ by adding the vertex $v$ and removing the edges inside~$A^-$.
\item[$(f)$] Set $A^- = A \cup \{u,v\}$ for some $u \in V(F) \setminus A$ and $v \not\in V(F)$, and let the structure $F^- = \big(\hat{F}^{A^-} + v \big) \cup \{e\}$ be obtained from $F$ by adding the vertex $v$, adding an edge $e$ from $v$ to some vertex of $V(F) \setminus A^-$, and removing the edges inside $A^-$.
\end{itemize} 
\end{defn}

The following table will be useful in the calculations below.

\vskip0.5cm
\begin{center}
\begin{tabular}{c|c|c|c|c|c|c}
& $(a)$ & $(b)$ & $(c)$ & $(d)$ & $(e)$ & $(f)$ \\[+0.6ex] 
\hline &&&&& \\[-2.1ex]
$v_{A^-}(F^-) - v_A(F)$ \, & \; 0 \; & \; $-1$ \; & \; 0 \; & \; $-2$ \; & \; $-1$ \; & \; $-1$ \;  \\[+0.6ex] 
\hline &&&&&& \\[-2.1ex]
$e(F^-) - e(F)$ \, & \; 1 \; & $\le 0$ & \; 1 \; & $\le 0$ & $\le 0$ & $\le 1$ \\[+0.6ex] 
\hline &&&&& \\[-2.1ex]
$o(F^-) - o(F)$ \, & $0$ 
& $\le 0$ & \; 0 \; & $\le 0$  & $\le 0$ & $\le 0$ 
\end{tabular}\\\
\end{center}
\begin{center} 
Table~4.1
\end{center}
\vskip0.2cm

Note in particular that 
\begin{equation}\label{eq:veoF'A'FA}
v_{A'}(F') \le v_A(F), \qquad o(F') \le o(F) \qquad \text{and} \qquad e(F') \le e(F) + 1
\end{equation}
for every $(F',A') \in \F_{F,A}^-$. The next lemma motivates the definition above. 

\begin{lemma}\label{lem:F-}
Let $(F,A,\phi)$ be a graph structure triple, and suppose that $\phi$ is faithful in $G_{m} \cup \{e\}$. Then the number of copies (in $G_m$) of $F$ rooted at $\phi(A)$ destroyed by the addition of the edge $e$ to $G_m$ is at most
$$\sum_{(F^-,A^-) \in \F_{F,A}^-} \max_{\phi^- :\, A^- \to V(G_m)} N_{\phi^-}(F^-)(m).$$ 
where the maximum is over faithful maps $\phi^- \colon A^- \to V(G_m)$.  
\end{lemma}

\begin{proof}
Let $e = \{u,v\}$ be the edge added in step $m+1$, and suppose that $e$ destroys a copy $F^*$ of $F$, rooted at $\phi(A)$, in $G_m$. Note that this implies that either $e \in O(F^*)$ or $e$ closes an open edge of~$F^*$. We claim that there is a graph structure $H \subseteq G_m\big[V(F^*) \cup \{u,v\} \big]$ such that
\begin{equation}\label{Fcondition}
H \in N_{\phi^-}(F^-) \quad \text{for some} \quad (F^-,A^-) \in \F_{F,A}^- \quad\text{and} \quad \phi^- \colon A^- \to V(G_m)
\end{equation}
with $\phi^-|_A = \phi$ and $\Im(\phi^-) = \phi(A) \cup \{u,v\}$. There are various cases to consider. 

Suppose first that $e \subseteq \phi(A)$; we claim that~\eqref{Fcondition} holds with $(F^-,A^-)$ as in case $(a)$ of Definition~\ref{def:F-}. Indeed, $e$ closes an open edge $f \in O(F^*)$, and so this open edge must have one endpoint ($u$, say) inside $A$ and the other $w$ outside. Since adding $e$ closes $f$, it follows that $\{v,w\}$ must be an edge of $G_m$. Hence there must exist a copy of $F^-$ in $G_m$ on the same vertex set, where $F^-$ is obtained from $F$ by adding the edge ($h$, say) corresponding to $\{v,w\}$. Clearly $h \not\in O(F)$, since $\{v,w\} \in E(G_m)$, and moreover $h \not\in E(F)$, since otherwise $\phi$ would not be faithful in $G_{m} \cup \{e\}$.

Suppose next that $e \cap \phi(A) = \{u\}$. If $v \in V(F^*)$ then it is immediate that~\eqref{Fcondition} holds with $(F^-,A^-)$ as in case $(b)$, so suppose not. Then adding $e$ must close an open edge $\{u,w\}$ of~$F^*$. It follows that $\{v,w\}$ must be an edge of $G_m$, and hence~\eqref{Fcondition} holds with $(F^-,A^-)$ as in case $(c)$. 

Finally, suppose that $e \cap \phi(A) = \emptyset$. If $\{u,v\} \subseteq V(F^*)$ then~\eqref{Fcondition} holds with $(F^-,A^-)$ as in case $(d)$. On the other hand, if $\{u,v\} \cap V(F^*) = \emptyset$ then $e$ cannot destroy $F^*$. Hence we may assume that $u \in V(F^*)$ and $v \not\in V(F^*)$, and that the addition of $e$ closes an open edge $\{u,w\}$ in $F^*$. It follows that the edge $\{v,w\}$ is an edge of $G_m$, and hence if $w \in A$ then~\eqref{Fcondition} holds with $(F^-,A^-)$ as in case $(e)$, and if $w \not\in A$ then~\eqref{Fcondition} holds with $(F^-,A^-)$ as in case $(f)$ of Definition~\ref{def:F-}, as required. 
\end{proof}

\begin{rmk}
The observant reader will have noticed that the graph structures $F^+$ and $F^-$ may have isolated vertices. However, using Remark~\ref{Eremark}, the event $\E(m)$ implies bounds on $N_{\phi'}(F')$ for every pair $(F',A') \in \F_{F,A}^+ \cup  \F_{F,A}^-$, and every faithful $\phi' \colon A' \to V(G_m)$. 
\end{rmk}


We finish this section by proving some straightforward lemmas and observations which will be useful in later sections. 

\begin{obs}\label{obs:F+inF-}
$\F_{F,A}^+ \subseteq \F_{F,A}^-$ for every graph structure pair $(F,A)$. 
\end{obs}

\begin{proof}
This follows immediately from the definitions, since if $(F',A') \in \F_{F,A}^+$ then we are in either case $(b)$ or $(d)$ of Definition~\ref{def:F-}. 
\end{proof}

\begin{obs}\label{obs:AcapF}
Let $(F,A)$ be a graph structure pair, and let $(F',A') \in \F_{F,A}^-$.  
\begin{itemize}
\item[$(a)$] If $A' \cap F = A$ then $\Nt_{A'}(F')(m) = \frac{2t}{\sqrt{n}} \cdot \Nt_A(F)(m)$.
\item[$(b)$] If $A' \subsetneq H' \subseteq F'$ then $H' \cap F \neq A$.
\end{itemize}
\end{obs}

\begin{proof}
Both statements follow easily from Definition~\ref{def:F-}. Indeed, if $A' \cap F = A$ then we must be in either case $(a)$ or $(c)$ of that definition, and in both cases we have $v_A(F) = v_{A'}(F')$, $o(F) = o(F')$ and $e(F) = e(F') - 1$. For the second statement, simply note that $v_{A'}(F') = v_A(F)$, i.e., every new vertex of $F'$ is included in $A'$. 
\end{proof}

\begin{obs}\label{obs:NH'F'NHF}
Let $(F,A)$ be a graph structure pair, and let $(F',A') \in \F_{F,A}^-$.  For every $A' \subseteq H' \subseteq F'$ and every $m \in [m^*]$, we have $\Nt_{H'}(F')(m) \le \Nt_{H' \cap F}(F)(m)$.
\end{obs}

\begin{proof}
This also follows easily from Definition~\ref{def:F-}. Indeed, setting $H = H' \cap F$, we have 
$$v_{H'}(F') = v_{H}(F), \quad e(F') - e(H') \ge e(F) - e(H) \quad \text{and} \quad o(F') - o(H') = o(F) - o(H)$$
in each of the cases $(a)$-$(f)$.  
\end{proof}

We make one more simple observation, which will play a crucial role in Section~\ref{balancedsec}. Let $\F_{F,A}^* \subset \F_{F,A}^-$ denote the graph structure pairs in $\F_{F,A}^-$ with $v_{A'}(F') < v_A(F)$. 

\begin{obs}\label{obs:destroy:disjoint}
Let $(F,A,\phi)$ be a graph structure triple, and suppose that the edge $e$ which is added in step $m+1$ of the triangle-free process 
is disjoint from $\phi(A)$. Then the number of copies of $F$ rooted at $\phi(A)$ which are destroyed by $e$ is at most
\begin{equation}\label{eq:destroy:disjoint}
\sum_{(F',A') \in \F_{F,A}^*} \max_{\phi' :\, A' \to V(G_m)} N_{\phi'}(F')(m),
\end{equation}
where the maximum is over faithful maps $\phi' \colon A' \to V(G_m)$.  

Moreover, we have $\F_{F,A}^+ \subset \F_{F,A}^*$, and so the number of copies of $F$ rooted at $\phi(A)$ created in step $m+1$ of the triangle-free process is also bounded above by~\eqref{eq:destroy:disjoint}.
\end{obs}

\begin{proof}
Proceed as in the proof of Lemma~\ref{lem:F-}, noting that if $e$ is disjoint from $\phi(A)$ then the pair $(F^-,A^-)$ in~\eqref{Fcondition} was obtained via either case $(d)$, $(e)$ or $(f)$ of Definition~\ref{def:F-}, and that in each of these cases we have $v_{A^-}(F^-) < v_A(F)$, and so $(F^-,A^-) \in \F_{F,A}^*$, as claimed. For the second part, simply note that if $(F',A') \in \F_{F,A}^+$ then we are in either case $(b)$ or $(d)$ of Definition~\ref{def:F-}. 
\end{proof}

We next introduce a piece of notation which will be extremely useful in Sections~\ref{balancedsec} and~\ref{deltaNsec}, below. Given a graph structure pair $(F,A)$, set 
$$\delta(F,A) = C^3 v_A(F)^2 + 2e(F) + o(F)$$ 
and recall that $\Delta(F,A) = \delta(F,A)^C$. We shall write 
\begin{equation}\label{def:deltaF-vF+e}
\Delta(F-v,A) := \big( \delta(F,A) - C \big)^C. 
\end{equation}
Note that this is an upper bound on the value of $\Delta(F,A)$ one obtains by decreasing $v_A(F)$ and increasing $e(F)$ by one. 

Recall that $g_{F,A}(t) \, = \, e^{ct^2} n^{-1/4} (\log n)^{\gamma(F,A)}$, where $\gamma(F,A) = \Delta(F,A) - e(F) - 2$, see~\eqref{def:gam} and~\eqref{def:gfat}. We will need the following properties of $\Delta(F-v,A)$.\footnote{See the Appendix~\cite{App} for more detailed proofs.} 

\begin{lemma}\label{obs:gF'A'gFA}
Let $(F,A)$ be a graph structure pair, and let $(F',A') \in \F_{F,A}^-$. If $A' \cap F \neq A$, then
$$g_{F',A'}(t) \, \le \, (\log n)^{\Delta(F-v,A)}$$
for every $0 < t \le t_{A'}(F')$.
\end{lemma}

\begin{proof}
The condition $A' \cap F \neq A$ implies that we are in neither case~$(a)$ nor case~$(c)$ of Definition~\ref{def:F-}. Since $v_{A'}(F') < v_A(F)$ in each of the other cases, the result follows easily from the definition~\eqref{def:gfat}, since $g_{F',A'}(t) \, \le \, (\log n)^{\gamma(F',A')}$ for every $0 < t \le t_{A'}(F')$.
\end{proof}

\begin{obs}\label{obs:deltaF'H'A'deltaFA}
Let $(F,A)$ be a graph structure pair, and let $(F',A') \in \F_{F,A}^-$. Then 
$$\Delta(F',H',A') \, \le \, \Delta(F-v,A) \, \le \, \Delta(F,A) - 3\sqrt{\Delta(F,A)}$$
for every $A' \subsetneq H' \subsetneq F'$. Moreover, the same bounds hold if $H' = F'$ and $A'  \cap F \neq A$.
\end{obs}

\begin{proof}
Note that if $A' \neq H' \neq F'$, then $1 \le v_{A'}(H') = v_{A'}(F') - v_{H'}(F') \le v_{A'}(F') - 1$. Both inequalities now follow easily from~\eqref{eq:veoF'A'FA}, using the convexity of the function $x \mapsto x^C$. On the other hand, if $H' = F'$ and $A'  \cap F \neq A$ then we just repeat the proof of Lemma~\ref{obs:gF'A'gFA}.
\end{proof}

The next observation also follows by the same argument.

\begin{obs}\label{obs:deltaF'vA'deltaFA}
Let $(F,A)$ be a graph structure pair, and let $(F',A') \in \F_{F,A}^-$. Then 
$$\Delta(F'-v,A') \, \le \, \Delta(F,A) - 3\sqrt{\Delta(F,A)}.$$
\end{obs}

\begin{obs}\label{obs:deltaaddlittle}
Let $(F,A)$ be a graph structure pair, and let $(F',A') \in \F_{F,A}^-$. Then 
$$\Delta(F',A') \, \le \, (1 + \eps) \Delta(F,A) \, \le \, (1 + 2\eps) \Delta(F-v,A).$$
\end{obs}

\begin{proof}
Note that $\Delta(F',A') \le \big( \delta(F,A) + 2 \big)^C$, by~\eqref{eq:veoF'A'FA}, and recall that $\Delta(F-v,A) = \big( \delta(F,A) - C \big)^C$, by definition. Since $\delta(F,A) \ge C^3$, the claimed bounds follow. 
\end{proof}

Finally, we give three bounds which rely on one of the following assumptions\footnote{We remark that when either of these inequalities is reversed, the conclusion of Theorem~\ref{EEthm} (in the corresponding case) will hold trivially, see below.}: either 
\begin{equation}\label{eq:Delta:assumption}
(\log n)^{\gamma(F,A)} \, \le \, n^{v_A(F)+e(F)+1}
\end{equation}
and $t_A(F) > 0$, or $(\log n)^{\Delta(F-v,A)} \le n^{v_A(F)}$.
 
\begin{lemma}\label{obs:deltaF'A'Fv}
Let $(F,A)$ be a graph structure pair. If $(\log n)^{\Delta(F-v,A)} \le n^{v_A(F)}$, then
$$\max\Big\{ \Delta(F,A), \, \Delta(F',A') \Big\} - \Delta(F-v,A) \, \le \, \frac{\eps}{v_A(F)} \cdot \frac{\log n}{\log\log n}$$
for every $(F',A') \in \F_{F,A}^-$.
\end{lemma}

\begin{proof}
It follows from the definitions and~\eqref{eq:veoF'A'FA} that 
\begin{multline*} 
\max\Big\{ \Delta(F,A), \, \Delta(F',A') \Big\} - \Delta(F-v,A) \, \le \, \big( \delta(F,A) + 2 \big)^C - \big( \delta(F,A) - C \big)^C \\[+0.5ex]
 \, \le \, 2C^2 \cdot \delta(F,A)^{C - 1} \, = \, \frac{2C^2 \cdot \Delta(F,A)}{\delta(F,A)} \, \le \, \frac{4}{C} \cdot \frac{\Delta(F-v,A)}{v_A(F)^2},
\end{multline*}
since $\delta(F,A) \ge C^3 v_A(F)^2$ and by Observation~\ref{obs:deltaaddlittle}. Hence, if $(\log n)^{\Delta(F-v,A)} \le n^{v_A(F)}$, then 
$$\max\Big\{ \Delta(F,A), \, \Delta(F',A') \Big\} - \Delta(F-v,A) \, \le \, \frac{4}{C \cdot v_A(F)} \cdot \frac{\log n}{\log \log n} \, \le \, \frac{\eps}{v_A(F)} \cdot \frac{\log n}{\log\log n},$$
as claimed.
\end{proof}
 
The proof of the next lemma is almost identical, see the Appendix~\cite{App} for the details.
 
\begin{lemma}\label{obs:deltaFdeltaFv}
Let $(F,A)$ be a graph structure pair with $t_A(F) > 0$, and let $(F',A') \in \F_{F,A}^-$. If $(\log n)^{\gamma(F,A)} \le n^{v_A(F)+e(F)+1}$, then
$$\max\Big\{ \Delta(F',A') - \Delta(F,A), \sqrt{\Delta(F,A)} \Big\} \, \le \, \frac{\eps^2}{v_A(F)} \cdot \frac{\log n}{\log\log n}.$$
\end{lemma}

Finally, the following bound follows easily from Lemma~\ref{obs:deltaFdeltaFv}.

\begin{lemma}\label{lem:gF'A'gFA}
Let $(F,A)$ be a graph structure pair with $t_A(F) > 0$, and let $(F',A') \in \F_{F,A}^-$. If $(\log n)^{\gamma(F,A)} \le n^{v_A(F)+e(F)+1}$, then
$$g_{F',A'}(t) \, \le \, n^{1/4 + \eps} \cdot g_{F,A}(t)$$
for every $0 < t \le t_{A'}(F')$. 
\end{lemma}

\begin{proof}
Simply note that, by Lemma~\ref{obs:deltaFdeltaFv},
$$\gamma(F',A') - \gamma(F,A) \, \le \, \Delta(F',A') - \Delta(F,A) + e(F) \, \le \, \frac{\eps}{v_A(F)} \cdot \frac{\log n}{\log\log n}.$$
Since $t \le t_{A'}(F')$ and $v_A(F) \ge 1$, it follows that
$$g_{F',A'}(t) \, \le \, (\log n)^{\gamma(F',A')} \, \le \, n^{1/4 + \eps} \cdot g_{F,A}(t),$$
as required.
\end{proof}

\subsection{Balanced non-tracking graph structures}\label{balancedsec}

In this subsection we shall use the tools developed above to prove a slight strengthening of the bound in Theorem~\ref{EEthm}$(c)$ for balanced pairs $(F,A)$, i.e., pairs whose building sequence is either
$$A \subseteq H_0 = F \qquad \text{or} \qquad A = H_0 \subsetneq H_1 = F,$$
see Definition~\ref{def:balanced}. We shall prove the following proposition. 

\begin{prop}\label{prop:balanced:late}
Let $(F,A)$ be a balanced graph structure pair, and let $t_A(F) < t \le t^*$. With probability at least $1 - n^{- 3\log n}$, either $\big( \E(m) \cap \Z(m) \cap \Q(m) \big)^c$ holds, or
\begin{equation}\label{eq:balanced:late}
N_\phi(F)(m) \, \le \, \max\Big\{ e^{- o(F)(t^2 - t_A(F)^2)} (\log n)^{\Delta(F,A)}, (\log n)^{\Delta(F-v,A)} \Big\} 
\end{equation}
for every $\phi \colon A \to V(G_m)$ which is faithful at time~$t$. 
\end{prop}

In the case $t_A(F) = 0$, we shall prove the following slightly stronger bound.

\begin{prop}\label{prop:balanced:late:tAFzero}
Let $(F,A)$ be a balanced graph structure pair with $t_A(F) = 0$, and let $0 < t \le t^*$. With probability at least $1 - n^{- 3\log n}$, either $\big( \E(m) \cap \Z(m) \cap \Q(m) \big)^c$ holds, or
\begin{equation}\label{eq:balanced:late:tAFzero}
N_\phi(F)(m) \, \le \, (\log n)^{\Delta(F-v,A)}
\end{equation}
for every $\phi \colon A \to V(G_m)$ which is faithful at time~$t$. 
\end{prop}

Let us begin by sketching the proof of Propositions~\ref{prop:balanced:late} and~\ref{prop:balanced:late:tAFzero}, and discussing why they are necessary for the proof of Theorem~\ref{EEthm}. When $t_A(F) = 0$ we shall use induction on $v_A(F)$, combined with Bohman's martingale method (i.e., Lemma~\ref{Bohmart}). 
In fact, when $v_A(F)$ is bounded we shall prove a much stronger bound, see Proposition~\ref{ZpropforgeneralF}, below. 

The harder case is when $t_A(F) > 0$. First observe that, since $t_A(F) \not\in \{0,t^*\}$, it follows that $o(F) > 0$, and therefore we need to show that $N_\phi(F)$ is decreasing for a while after time $t = t_A(F)$. We would like to use our usual martingale method (i.e., Lemma~\ref{mart}), but there is a problem: the single step changes in $N_\phi(F)$ can be \emph{very} large. In order to get around this problem, we define a new variable (see~\eqref{def:bal:mart}, below), which counts the number of copies of $F$ rooted at $\phi(A)$ which are destroyed \emph{by the addition of an edge which is disjoint from $\phi(A)$}. Using Observation~\ref{obs:destroy:disjoint}, we will be able to give a sufficiently strong upper bound on the single-step change in this variable; moreover, we shall be able to bound its expected change, using the trivial observation that every open edge of~$F$ has at most one endpoint in~$A$. 

Finally, we remark that the bounds~\eqref{eq:balanced:late} and~\eqref{eq:balanced:late:tAFzero} will be used to bound $N_{\phi'}(F')(m)$ for pairs $(F',A') \in \F_{F,A}^-$ which are balanced, when $t_{A'}(F') < t \le t_A(F)$ and $g_{F,A}(t) \ge 1$. The following definition will allow us to assume that this bound holds when we need it.

\begin{defn}\label{def:event:M}
For each $m' \in [m^*]$, let $\M(m')$ denote the event that the bound~\eqref{eq:balanced:late} holds for every balanced graph structure pair $(F,A)$, every $t_A(F) \cdot n^{3/2} < m \le m'$ and every map $\phi \colon A \to V(G_m)$ which is faithful at time~$t$, and that moreover~\eqref{eq:balanced:late:tAFzero} holds if $t_A(F) = 0$.
\end{defn}

As noted above, the event $\M(m)$ will be a crucial tool in Section~\ref{deltaNsec}, where we shall give an upper bound on $|\Delta N^*_\phi(F)(m)|$ in the case $\omega < t \le t_A(F)$. Observe that~\eqref{eq:balanced:late} and~\eqref{eq:balanced:late:tAFzero} both hold trivially if $(\log n)^{\Delta(F-v,A)} \ge n^{v_A(F)}$, so we may assume otherwise. 

We begin by showing that when $t_A(F) = 0$ and $v_A(F)$ is bounded, we can obtain a much sharper result, which (almost) generalizes Proposition~\ref{Zprop}.

\begin{prop}\label{ZpropforgeneralF}
Let $(F,A)$ be a balanced graph structure pair, and suppose that $v_A(F) \le \omega$ and $t_A(F) = 0$. Then, with probability at least $1 - n^{- 3\log n}$, for every $m \in [m^*]$ either $\big( \E(m-1) \cap \Q(m-1) \big)^c$ holds, or
\begin{equation}\label{eq:ZpropforgeneralF}
N_\phi(F)(m) \, \le \, (\log n)^{3v_A(F)}
\end{equation}
for every $\phi \colon A \to V(G_m)$ which is faithful at time~$t$. 
\end{prop}

Let's begin by deducing Proposition~\ref{ZpropforgeneralF} from Lemma~\ref{Bohmart} and Proposition~\ref{Zprop}. 

\begin{proof}[Proof of Proposition~\ref{ZpropforgeneralF}]
We shall use induction on $v_A(F)$ to prove that the event in the statement holds with probability at least $1 - n^{v_A(F)} \cdot n^{-4\log n}$, which implies the claimed bound since $v_A(F) \le \omega \ll \log n$. This holds for $v_A(F) = 1$ by Proposition~\ref{Zprop}, so suppose that $v_A(F) \ge 2$, and let us assume that the induction hypothesis holds for all smaller values of $v_A(F)$. Recall from Definition~\ref{def:balanced} that the conditions that $(F,A)$ is balanced and $t_A(F) = 0$ imply (and in fact are equivalent to)
\begin{equation}\label{eq:baltAFzero:condition}
e(H) - 2v_A(H) \le e(F) - 2v_A(F) \qquad \textup{for every } A \subseteq H \subseteq F,
\end{equation}
and note in particular (setting $H = A$) that $t_A^*(F) = 0$. 

We claim first that, without loss of generality, we have $e(F) = 2v_A(F)$ and $o(F) = 0$. Indeed, if $o(F) > 0$ then we can simply remove all open edges from $F$; in doing so we only increase $N_\phi(F)$, and we retain the condition~\eqref{eq:baltAFzero:condition}. 
If $e(F) > 2v_A(F)$ then there are two cases: either there exists a substructure $A \subsetneq H \subsetneq F$ with
$$e(H) - 2v_A(H) \, = \, e(F) - 2v_A(F),$$
or there does not. In the former case, observe that~\eqref{eq:baltAFzero:condition} holds for the pairs $(F,H)$ and $(H,A)$. 
It follows by the induction hypothesis and Lemma~\ref{lem:crucial} that
$$N_\phi(F) \, \le \, N_\phi(H) \cdot \max_{\phi' \colon H \to V(G_m)} N_{\phi'}(F) \, \le \, (\log n)^{3(v_A(H) + v_H(F))}  \, = \, (\log n)^{3v_A(F)},$$
with probability at least $1 - \big( n^{v_A(H)} + n^{v_H(F)} \big) n^{-4\log n}$, as required. In the latter case, i.e., no such $H$ exists, then we may remove an arbitrary edge from $F$. Note that $N_\phi(F') \ge N_\phi(F)$ for the resulting graph structure pair $(F',A)$, and moreover that~\eqref{eq:baltAFzero:condition} holds for $(F',A)$, since $e(F) > 2v_A(F)$. Hence $(F',A)$ is balanced and $t_A(F') = 0$, as required.

We are left to deal with the case $e(F) = 2v_A(F)$ and $o(F) = 0$. If there exists a substructure $A \subsetneq H \subsetneq F$ with
$e(H) = 2v_A(H)$ then we are easily done, exactly as above, by applying the induction hypothesis to the pairs $(F,H)$ and $(H,A)$. So assume not, and recall from~\eqref{eq:number:created} that,
$$\Ex\big[ \Delta N_\phi(F)(m) \big] \, = \, \sum_{F^o \in \F^o_F} \frac{N_{\phi}(F^o)}{Q(m)},$$
We claim that $\Nt_A(F^o) = (2t)^{e(F) - 1} e^{-4t^2} \sqrt{n}$ for every $F^o \in \F^o_F$, and that $c(F^o,A) = 2$. 
Indeed, these statements follow immediately from the assumptions that $e(F) = 2v_A(F)$ and $o(F) = 0$, and that $e(H) < 2v_A(H)$ for every $A \subsetneq H \subsetneq F$, respectively. 
Thus, while the event $\E(m) \cap \Q(m)$ holds, we have\footnote{Note that we must consider separately the cases $t \le \omega$ and $\omega < t \le t^*$, and observe that $\textbf{1}_{t \le \omega} \cdot f_{F^o,A}(t)  + \textbf{1}_{t > \omega} \cdot g_{F^o,A}(t)  \le n^{-\eps} \cdot (\log n)^{\gamma(F,A)} \ll 1$ since $v(F) \le \omega$ and $c(F^o,A) = 2$.}
\begin{equation}\label{eq:bal:NFooverQ}
\sum_{F^o \in \F^o_F} \frac{N_\phi(F^o)(m)}{Q(m)} \, \le \, \frac{e(F) \cdot (\log n)^{(e(F) - 1)/2} e^{-4t^2} \sqrt{n}}{\Qt(m)} \, \le \, \frac{(\log n)^{v_A(F)}}{n^{3/2}},
\end{equation}
where we used the bounds $2t < \sqrt{\log n}$ and $e(F) = 2v_A(F) \le 2\omega$. 

Now, set $\ell = v_A(F)$ and, for each $m' \in [m^*]$, let $\hat{\R}_\ell(m')$ denote the event that~\eqref{eq:ZpropforgeneralF} holds for every balanced graph structure pair $(F',A')$ with $t_A(F) = 0$, $v_{A'}(F') < \ell$ and $v(F) \le \omega$, every $m \le m'$ and every $\phi \colon A \to V(G_m)$ which is faithful at time~$t$. 
We shall bound, for each $m \in [m^*]$, the probability of the event $\hat{\LL}_\phi^F(m)$, defined as follows:
$$\hat{\LL}_\phi^F(m) \, := \, \E(m-1) \cap \Q(m-1) \cap \hat{\R}_\ell(m-1) \cap \big( N_\phi(F)(m) > (\log n)^{3v_A(F)} \big).$$
Fix $m_0 \in [m^*]$; we shall bound the probability of $\hat{\LL}_\phi^F(m_0)$. Subtracting the right-hand side of~\eqref{eq:bal:NFooverQ} from $\Delta N_\phi(F)(m)$ and summing over $m$, we obtain a function
$$M_\phi^F(m) \, = \, N_\phi(F)(m) \,-\, \frac{m \cdot (\log n)^{v_A(F)}}{n^{3/2}}$$
defined on $[m_0]$, which is super-martingale while $\E(m) \cap \Q(m)$ holds. Our plan is to apply Lemma~\ref{Bohmart} to $M_\phi^F$; in order to do so, we first claim that if $\hat{\R}_\ell(m)$ holds then 
\begin{equation}\label{eq:bal:bounds:on:deltaNF}
0 \, \le \, \Delta N_\phi(F)(m) \, \le \, e(F) \cdot (\log n)^{3v_A(F) - 3}.
\end{equation}
The lower bound is trivial, since $o(F) = 0$. To prove the upper bound, recall from the previous section the definition of $\F_{F,A}^+$, and Lemma~\ref{lem:F+}. We claim that if $\hat{\R}_\ell(m)$ holds then 
\begin{equation}\label{eq:bal:deltaNF+}
N_{\phi^+}(F^+)(m) \, \le \, (\log n)^{3v_A(F) - 3}
\end{equation}
for each $(F^+,A^+) \in \F_{F,A}^+$, which clearly implies~\eqref{eq:bal:bounds:on:deltaNF}. 

To prove~\eqref{eq:bal:deltaNF+}, simply note that $e(A^+) \le 2v_A(A^+)$, by~\eqref{eq:baltAFzero:condition} and since $e(F) = 2v_A(F)$, and that $F^+ = \hat{F}^{A^+}$. It follows that~\eqref{eq:baltAFzero:condition} holds for the pair $(F^+,A^+)$, and hence $(F^+,A^+)$ is balanced and $t_{A^+}(F^+) = 0$. The bound~\eqref{eq:bal:deltaNF+} now follows from the event $\hat{\R}_\ell(m)$, and the fact that $A^+ \neq A$, as claimed. 


Set $\alpha = e(F) \cdot (\log n)^{3v_A(F) - 3}$ and $\beta = (\log n)^{3v_A(F)} / m_0 \ge (\log n)^{v_A(F)} / n^{3/2}$, and observe that
$$\alpha \cdot \beta \cdot m_0 \, = \, e(F) \cdot (\log n)^{6v_A(F) - 3}.$$
Writing $\K(m) = \E(m) \cap \Q(m) \cap \hat{\R}_\ell(m)$ and applying Lemma~\ref{Bohmart}, it follows that
$$\Pr\big( \hat{\LL}_\phi^F(m_0) \big) \, \le \, \Pr\left( \bigg( M_\phi^F(m_0) > \frac{(\log n)^{3v_A(F)}}{2} \bigg) \cap \K(m_0-1) \right) \, \le \, \exp\left( - \frac{(\log n)^3}{16 \cdot e(F)} \right).$$ 
Summing over choices for $m_0$, and adding this to the probability that $\hat{\R}_\ell(m^*)$ fails to hold, it follows (using the induction hypothesis)  that the probability that 
$$\E(m-1) \cap \Q(m-1) \cap \big( N_\phi(F)(m) > (\log n)^{3v_A(F)} \big)$$
holds for some $m \in [m^*]$ is at most 
$$m^* \cdot \exp\left( - \frac{(\log n)^3}{16 \cdot e(F)} \right) \,+\, \sum_{\ell' = 1}^{\ell - 1}\sum_{\substack{(F',A')\\ v_{A'}(F') = \ell'}} n^{\ell'} \cdot n^{-4\log n} \, \le \, n^{v_A(F) - 1/2} \cdot n^{-4\log n}.$$ 
Summing over the (at most most $\sqrt{n}$) choices of~$\phi$, we obtain the claimed bound on the probability of the event in the statement.
\end{proof}

Using almost the same proof, we obtain Proposition~\ref{prop:balanced:late:tAFzero}.

\begin{proof}[Proof of Proposition~\ref{prop:balanced:late:tAFzero}]
We use induction on $v_A(F)$ to prove that if $(F,A)$ is balanced and $t_A(F) = 0$, then with probability at least $1 - n^{v_A(F)} \cdot n^{- 4\log n}$, for each $m \in [m^*]$ either $\big( \E(m-1) \cap \Q(m-1) \big)^c$ holds, or
\begin{equation}\label{eq:balanced:late:tAF=0}
N_\phi(F)(m') \, \le \, (\log n)^{\Delta(F-v,A)}
\end{equation}
for every $\phi \colon A \to V(G_m)$ which is faithful at time~$t$. The proof is almost identical\footnote{Note that our applications of the induction hypothesis still work in exactly the same way, using the convexity of the function $x \mapsto x^C$.} to that of Proposition~\ref{ZpropforgeneralF}, the main difference being that the bound~\eqref{eq:bal:NFooverQ} becomes
\begin{equation}\label{eq:bal:NFooverQlargeF}
\sum_{F^o \in \F^o_F} \frac{N_\phi(F^o)(m)}{Q(m)} \, \le \, \left( 1 + \frac{(\log n)^{\gamma(F,A)}}{n^\eps} \right) \frac{(\log n)^{v_A(F)}}{n^{3/2}},
\end{equation}
if $\E(m) \cap \Q(m)$ holds, using the fact that $(F^o,A)$ is balanced, and noting that 
$$\textbf{1}_{t \le \omega} \cdot f_{F^o,A}(t) + \textbf{1}_{t > \omega} \cdot g_{F^o,A}(t) \, \le \, n^{-\eps} \cdot (\log n)^{\gamma(F,A)},$$ 
which follows since $c(F^o,A) = 2$. Note also that the bounds in~\eqref{eq:bal:bounds:on:deltaNF} become
$$0 \, \le \, \Delta N_\phi(F)(m) \, \le \, \sum_{(F^+,A^+) \in \F_{F,A}^+} (\log n)^{\Delta(F^+-v,A^+)} \, < \, (\log n)^{\Delta(F-v,A) - 3},$$
since $F^+ = \hat{F}^{A^+}$ and $A^+ \neq A$. 

Now, let $\R_\ell(m')$ denote the event that~\eqref{eq:balanced:late:tAFzero} holds for every balanced graph structure pair $(F',A')$ with $t_A(F) = 0$ and $v_{A'}(F') < \ell = v_A(F)$, every $m \le m'$ and every $\phi \colon A \to V(G_m)$ which is faithful at time~$t$. Define
$$\LL_\phi^F(m) \, := \, \E(m-1) \cap \Q(m-1) \cap \R_\ell(m-1) \cap \big( N_\phi(F)(m) > (\log n)^{\Delta(F-v,A)} \big)$$
and fix $m_0 \in [m^*]$; we shall bound the probability of the event $\LL_\phi^F(m_0)$. Indeed, subtracting the right-hand side of~\eqref{eq:bal:NFooverQlargeF} from $\Delta N_\phi(F)(m)$ and summing over $m$, we obtain a function
$$M_\phi^F(m) \, = \, N_\phi(F)(m) \,-\, \left( 1 + \frac{(\log n)^{\gamma(F,A)}}{n^\eps} \right) \frac{m \cdot (\log n)^{v_A(F)}}{n^{3/2}}$$
defined on $[m_0]$, which is super-martingale while $\E(m) \cap \Q(m)$ holds. Let $\alpha = (\log n)^{\Delta(F-v,A) - 3}$ and
$$\beta \,=\, \frac{(\log n)^{\Delta(F-v,A)}}{m_0} \, \ge \, \frac{(\log n)^{\Delta(F-v,A) - 1}}{n^{3/2}} \, \gg \, \left( 1 + \frac{(\log n)^{\gamma(F,A)}}{n^\eps} \right) \frac{(\log n)^{v_A(F)}}{n^{3/2}},$$
where the final inequality holds since $(\log n)^{\Delta(F,A)} \le (\log n)^{\Delta(F-v,A)} \cdot n^{\eps/2}$, by Lemma~\ref{obs:deltaF'A'Fv}, since $(\log n)^{\Delta(F-v,A)} \le n^{v_A(F)}$, by assumption.

Writing $\K(m) = \E(m) \cap \Q(m) \cap \R_\ell(m)$ and applying Lemma~\ref{Bohmart}, it follows that
$$\Pr\big( \LL_\phi^F(m_0) \big) \, \le \, \Pr\left( \bigg( M_\phi^F(m_0) > \frac{(\log n)^{\Delta(F-v,A)}}{2} \bigg) \cap \K(m_0-1) \right) \, \le \, n^{-C \log n}.$$ 
The claimed bound now follows by the induction hypothesis, exactly as in the proof of Proposition~\ref{ZpropforgeneralF}.
\end{proof}

We now turn to the more substantive part of this subsection: the proof of Proposition~\ref{prop:balanced:late} when $t_A(F) > 0$. Note that in this case we have $o(F) > 0$, since $0 < t_A(F) < t^*$. We begin by noting that at time $t = t_A(F)$, the bound on $N_\phi(F)(m)$ given by part~$(b)$ of Theorem~\ref{EEthm} implies that required in part~$(c)$. 

\begin{lemma}\label{lem:t=tAF}
Let $(F,A)$ be a balanced graph structure pair with $0 < t_A(F) < t^*$, and set $t = t_A(F)$. For each injection $\phi \colon A \to V(G_m)$, if
$$N_\phi(F)(m) \, \in \, \big( 1 \pm g_{F,A}(t) \big) \Nt_A(F)(m),$$
then
\begin{equation}\label{eq:bal:bound:at:tAF}
N_\phi(F)(m) \, \le \, \big( \log n \big)^{\Delta(F,A)-1}.
\end{equation}
\end{lemma}
 
\begin{proof} 
We claim that, at time $t = t_A(F)$,
$$\big( 1 + g_{F,A}(t) \big) \Nt_A(F)(m) \, = \, \big( 1 + (\log n)^{\gamma(F,A)} \big) \cdot (2t)^{e(F)} \, \le \, \big( \log n \big)^{\Delta(F,A)-1},$$
from which the lemma follows immediately. The equality follows from the definitions~\eqref{def:t*} and~\eqref{def:gfat} of $t_A^*(F)$ and $g_{F,A}(t)$ respectively, the fact that $(F,A)$ is balanced, which implies that $t_A^*(F) = t_A(F)$, and the fact that $e^{ct^2} = n^{1/4}$ when $t = t_A(F)$, since we are assuming that $t_A(F) < t^*$. The inequality follows from the definition~\eqref{def:gam} of $\gamma(F,A)$, i.e.,
$$\gamma(F,A) = \Delta(F,A) - e(F) - 2,$$
and the fact that $2t < \sqrt{\log n}$, since $t < t^*$. 
\end{proof}

In order to see that~\eqref{eq:bal:bound:at:tAF} implies the bound in Theorem~\ref{EEthm}$(c)$, simply note that the building sequence for $(F,A)$ is 
$$A = H_0 \subseteq H_1 = F$$
since $(F,A)$ is balanced and $t_A(F) > 0$. Thus $t_1 = t_A(F)$ (since $t_A(F) < t^*$), and so, by Lemma~\ref{lem:Hj}, the minimal $A \subseteq H \subseteq F$ with $t < t_H(F)$ is equal to $F$ for every $t \ge t_A(F)$.


We are left with the task of showing that $N_\phi(F)$ continues to decrease for some time after $t_A(F)$, and moreover does not increase again too much later on; we shall do so using a slight variant of the martingale method of Section~\ref{MartSec}. In particular, we shall work with the original (un-renormalised) random variables, and we will prove only an upper bound on $N_\phi(F)$.  

More precisely, let us define for each balanced pair $(F,A)$ a corresponding Line of Death:
$$L_A^D(F)(m) \, = \, \max\Big\{ e^{- o(F)(t^2 - t_A(F)^2)} (\log n)^{\Delta(F,A)}, (\log n)^{\Delta(F-v,A)} \Big\},$$
and a Line of Peril, $L^P_A(F)(m) = \frac{3}{4} \cdot L^D_A(F)(m)$. For each $t_A(F) \cdot n^{3/2} \le m \le m^*$, let 
$$\LL_\phi^F(m) \, = \, \Big\{ L^P_A(F)(m) \le N_\phi(F)(m) \le L^D_A(F)(m) \Big\}$$
denote the event that $N_\phi(F)(m)$ lies between these two lines, and for each pair $(r,s) \in \N^2$ with $t_A(F) \cdot n^{3/2} \le r \le r+s \le m^*$, define
\begin{multline*}
\LL_\phi^F(r,s) \, = \, \big\{ N_\phi(F)(r) < L^P_A(F)(r) \big\} \cap \bigcap_{m = r+1}^{r+s-1} \LL_\phi^F(m) \cap \big\{ N_\phi(F)(r+s) > L^D_A(F)(r+s) \big\},
\end{multline*}
so $\LL_\phi^F(r,s)$ holds if $N_\phi(F)(m)$ crosses the Line of Death in step $r+s$, and crossed the Line of Peril for the last time (before step $r+s$) in step $r+1$. It follows immediately from the definition and Lemma~\ref{lem:t=tAF} that if $\E(m)$ holds but~\eqref{eq:balanced:late} does not, for some $t_A(F) < t \le t^*$, then the event $\LL_\phi^F(r,s)$ holds for some pair $(r,s)$. Given such a pair $(r,s)$, set\footnote{Note that $r \ge \omega \cdot n^{3/2}$, since $t_A(F) \ge \omega$.}  
$$s_0 \,=\, \min\bigg\{ s, \, \frac{n^3}{o(F) \cdot r} \bigg\},$$
and observe that $L^D_A(F)(r) \le C \cdot L^D_A(F)(r+s_0)$, since we have $s_0 \le n^{3/2} \le r$, which implies that $o(F) \big( (r+s_0)^2 - r^2 \big) = O(n^3)$. We define a random variable\footnote{Note that, to simplify the notation, we suppress the dependence of $M_\phi^F$ on the pair $(r,s)$.} as follows:
\begin{equation}\label{def:bal:mart}
M_\phi^{F}(m') \, = \, \sum_{m= r}^{m'-1} \Big( C_\phi(F)(m) - \hat{D}_\phi(F)(m) + \big| \Delta L_A^D(F)(m) \big| + \frac{r}{4 n^3} \cdot L_A^D(F)(m) \Big)
\end{equation}
for each $r \le m' \le r+s_0$, where 
\begin{itemize}
\item $C_\phi(F)(m)$ denotes the number of copies of $F$ rooted at $\phi(A)$ which are created in step $m+1$ of the triangle-free process, and 
\item $\hat{D}_\phi(F)(m)$ denotes the number of copies of $F$ rooted at $\phi(A)$ which are destroyed in step $m+1$, if the edge $e_{m+1}$ (that is, the edge which is added to $G_m$ in step $m+1$ of the triangle-free process) is disjoint from $\phi(A)$, and $\hat{D}_\phi(F)(m) = 0$ otherwise. 
\end{itemize}
The motivation for this slightly convoluted definition is given by the following lemma, which follows easily from the fact (see Observation~\ref{obs:destroy:disjoint}) that if $e_{m+1}$ is disjoint from $\phi(A)$, then the single-step change $|\Delta N_\phi(F)(m)|$ cannot be too large.

\begin{lemma}\label{lem:DeltaMF}
Let $(F,A)$ be a balanced graph structure pair with $t_A(F) > 0$, let 
$$t_A(F) \cdot n^{3/2} \, \le \, r \, \le \, m \, < \, r+s_0 \, \le \, m^*,$$ 
and suppose that $\phi \colon A \to V(G_m)$ is faithful at time~$t$. If $\E(m)$ holds, then
$$\big| \Delta M_\phi^{F}(m) \big| \, \le \, (\log n)^{\Delta(F-v,A) - C}.$$ 
\end{lemma}

\begin{proof}
Observe first that 
\begin{equation}\label{eq:deltaLoD}
\big| \Delta L_A^D(F)(m) \big| \, \le \, \big( 2o(F) + \eps \big) \cdot \frac{t}{n^{3/2}} \cdot L_A^D(F)(m) \, \le \, \frac{(\log n)^{\Delta(F-v,A)}}{n},
\end{equation}
by Lemma~\ref{obs:deltaF'A'Fv}. Thus it will suffice to bound $C_\phi(F)(m) + \hat{D}_\phi(F)(m)$. Recall that $\F_{F,A}^*$ denotes the graph structure pairs in $\F_{F,A}^-$ with $v_{A'}(F') < v_A(F)$. We claim first that 
\begin{equation}\label{eq:C+Dlessthan}
C_\phi(F)(m) + \hat{D}_\phi(F)(m) \, \le \, \sum_{(F',A') \in \F_{F,A}^*} \max_{\phi' \colon A' \to V(G_m)} N_{\phi'}(F')(m)
\end{equation}
where the maximum is taken over faithful maps $\phi' \colon A' \to V(G_m)$. Indeed, this follows immediately by Observation~\ref{obs:destroy:disjoint} and the definition of $\hat{D}_\phi(F)(m)$. 



We will show that 
\begin{equation}\label{eq:DeltaMFappofE}
N_{\phi'}(F')(m) \, \le \, (\log n)^{\Delta(F',A')} \, \le \, (\log n)^{\Delta(F-v,A) - 2C}.
\end{equation}
To prove~\eqref{eq:DeltaMFappofE}, we claim first that $t \ge t_{H'}(F')$ for every $A' \subseteq H' \subsetneq F'$. To see this, observe that since $(F,A)$ is balanced, by Lemma~\ref{lem:tell:NHFsmall} and Observation~\ref{obs:NH'F'NHF} it follows that
$$\Nt_{H'}(F')(m) \, \le \, \Nt_{H' \cap F}(F)(m) \, \le \, (2t)^{e(F) - e(H' \cap F)} \, \le \, (2t)^{e(F') - e(H')}$$ 
for every $t \ge t_A(F)$, since every edge of $E(H') \setminus E(H)$ is also in $E(F') \setminus E(F)$. Therefore $t \ge t_{H'}(F')$, as claimed. Using the event $\E(m)$, the definitions (see~\eqref{def:deltaF-vF+e}) of $\Delta(F,A)$ and $\Delta(F-v,A)$, and the fact that $v_{A'}(F') < v_A(F)$ for every $(F',A') \in \F_{F,A}^*$, we obtain~\eqref{eq:DeltaMFappofE}.

Finally, recalling that $| \F_{F,A}^- | \le 5v_A(F)^2 \le (\log n)^{2/5}$, the claimed bound now follows from~\eqref{eq:C+Dlessthan} and~\eqref{eq:DeltaMFappofE}. 
\end{proof}

We shall need two more properties of $M_\phi^{F}(m)$. Together with Lemmas~\ref{mart} and~\ref{lem:DeltaMF}, they will easily imply Proposition~\ref{prop:balanced:late}. The first connects the event $\LL_\phi^F(r,s)$ and the function~$M_\phi^{F}$. 
Note that $L_A^D(F)(m)$ is decreasing on $t \ge t_A(F)$.

\begin{lemma}\label{lem:LLimpliesMF}
If $\LL_\phi^F(r,s)$ holds, then $M_\phi^{F}(r+s_0) \ge \frac{1}{4o(F)} \cdot L_A^D(F)(r+s_0)$.
\end{lemma}

\begin{proof}
Suppose first that $s_0 = s$. Then by~\eqref{def:bal:mart} we have
\begin{align*}
M_\phi^{F}(r+s_0) & \, \ge \, \big( N_\phi(F)(r+s) - N_\phi(F)(r) \big) + \big( L_A^D(F)(r) - L_A^D(F)(r+s) \big)\\
& \, \ge \, \big( L_A^D(F)(r) - N_\phi(F)(r) \big) \, \ge \, \frac{1}{4} \cdot L_A^D(F)(r),
\end{align*}
since the inequalities $N_\phi(F)(r+s) \ge L_A^D(F)(r+s)$ and $N_\phi(F)(r) \le \frac{3}{4} \cdot L_A^D(F)(r)$ follow from the event $\LL_\phi^F(r,s)$. On the other hand, if $s_0 = n^3 / ( o(F) \cdot r )$, then we have
\begin{align*}
\big( N_\phi(F)(r+s_0) - N_\phi(F)(r) \big) & + \big( L_A^D(F)(r) - L_A^D(F)(r+s_0) \big)\\
& \ge \, \frac{1}{4} \cdot \big( L_A^D(F)(r) - L_A^D(F)(r+s_0) \big) \, \ge \, 0,
\end{align*}
since $N_\phi(F)(r+s_0) \ge \frac{3}{4} \cdot L_A^D(F)(r+s_0)$ and $N_\phi(F)(r) \le \frac{3}{4} \cdot L_A^D(F)(r)$, by $\LL_\phi^F(r,s)$. Hence
$$M_\phi^{F}(r+s_0) \, \ge \, \frac{r}{4n^3} \cdot \sum_{m = r}^{r+s_0-1} L_A^D(F)(m) \, \ge \, \frac{1}{4o(F)} \cdot L_A^D(F)(r+s_0),$$
as claimed. 
\end{proof}

The next lemma implies that $M_\phi^{F}$ is a super-martingale on $[r,r+s_0]$, and gives bounds on its expected step-size. Recall that~\eqref{eq:balanced:late} holds trivially if $(\log n)^{\Delta(F-v,A)} \ge n^{v_A(F)}$. 


\begin{lemma}\label{lem:ExDeltaMF}
Let $(F,A)$ be a balanced graph structure pair with $t_A(F) > 0$, let 
$$t_A(F) \cdot n^{3/2}  \, < \, m \, \le \, m^*,$$ 
and suppose that $\phi \colon A \to V(G_m)$ is faithful at time~$t$. If $(\log n)^{\Delta(F-v,A)} \le n^{v_A(F)}$ and $\LL_\phi^F(m) \cap \E(m) \cap \Z(m) \cap \Q(m)$ holds, then 
$$\Ex\big[ \Delta M_\phi^{F}(m) \big] \, \le \, 0$$
and
$$\Ex\big[ | \Delta M_\phi^{F}(m) | \big] \, \le \, \frac{8 \cdot o(F) \cdot t}{n^{3/2}} \cdot L_A^D(F)(m).$$  
\end{lemma}

In the proof of Lemma~\ref{lem:ExDeltaMF} we shall need the following bound on $N_\phi(F^o)(m)$. 

\begin{lemma}\label{lem:bal:Fo}
Let $(F,A)$ be a balanced graph structure pair with $t_A(F) > 0$, let 
$$t_A(F) \cdot n^{3/2} \, < \, m \, \le \, m^*,$$ 
and suppose that $\phi \colon A \to V(G_m)$ is faithful at time~$t$. If $(\log n)^{\Delta(F-v,A)} \le n^{v_A(F)}$ and the event $\E(m)$ holds, then 
$$N_\phi(F^o)(m) \, \le \, \frac{\Yt(m) \cdot L_A^D(F)(m)}{\log n}$$
for every $F^o \in \F^o_F$. 
\end{lemma}

\begin{proof}
Observe first that, since $(F,A)$ is balanced and $t_A(F) < t^*$, we have $t_A^*(F) = t_A(F)$ and therefore $e^{4o(F)t_A(F)^2} = n^{v_A(F) - e(F)/2}$. Similarly, we have $t_A^*(H) \ge t_A(F)$, and hence $e^{4o(H)t_A(F)^2} \le n^{v_A(H) - e(H)/2}$, for every $A \subseteq H \subseteq F$. It follows that
$$\Nt_A(F)(m) \, = \, (2t)^{e(F)} e^{-4o(F)( t^2 - t_A(F)^2)} \, \le \, \frac{L_A^D(F)(m)}{(\log n)^{\Delta(F,A) - e(F)}}$$
and hence\footnote{Here we use Observation~\ref{obs:NtF-} and the fact that $2t e^{4t^2} \le \sqrt{n}$ for all $t \le t^*$.} 
\begin{equation}\label{eq:bal:Fo:NAH}
\Nt_A(H^o)(m) \, \ge \, \Nt_A(H)(m) \, \ge \, (2t)^{e(H)} e^{-4o(H)( t^2 - t_A(F)^2)} \, \ge \, (2t)^{e(H) - e(F)} \Nt_A(F)(m)
\end{equation}
for every $A \subseteq H \subseteq F$, every $H^o \in \F^o_H$, and every $t_A(F) \le t \le t^*$. 

Next, recall that $8t^2 \cdot \Nt_A(F^o) = \Yt(m) \cdot \Nt_A(F)$, by Observation~\ref{obs:NtF-}, and that $t_A(F) > \omega$, by Observation~\ref{obs:zeroomega}. Hence, if $t \le t_A(F^o)$ then, since the event $\E(m)$ holds, we have
\begin{align*}
N_\phi(F^o)(m) & \, \le \, \big( 1 + g_{F^o,A}(t) \big) \Nt_A(F^o)(m) \, = \, \big( 1 + g_{F^o,A}(t) \big) \cdot \frac{\Yt(m) \cdot \Nt_A(F)(m)}{8t^2} \\
& \, \le \, \frac{1 + g_{F^o,A}(t)}{(\log n)^{\Delta(F,A) - e(F)}} \cdot \frac{\Yt(m) \cdot L_A^D(F)(m)}{8t^2} \, \le \, \frac{\Yt(m) \cdot L_A^D(F)(m)}{\log n}
\end{align*}
as claimed. On the other hand, if $t > t_A(F^o)$ then let $A \subsetneq H \subseteq F^o$ be minimal such that $t < t_H(F)$, and suppose first that $H \neq F^o$. Then, using~\eqref{eq:bal:Fo:NAH}, we have
\begin{align*}
 N_\phi(F^o)(m) & \, \le \, (\log n)^{\Delta(F^o,H,A)} \Nt_H(F^o)(m) \, = \, (\log n)^{\Delta(F^o,H,A)} \cdot \frac{\Yt(m) \cdot \Nt_A(F)(m)}{8t^2 \cdot \Nt_A(H)(m)} \\
& \, \le \, \Yt(m) \cdot (\log n)^{\Delta(F^o,H,A)} \cdot (2t)^{e(F) - e(H)} \, \le \, \frac{\Yt(m) \cdot L_A^D(F)(m)}{\log n}
\end{align*}
since $A \neq H \neq F^o$, and hence $L_A^D(F)(m) \ge (\log n)^{\Delta(F-v,A)} \gg (\log n)^{\Delta(F^o,H,A) + e(F) + 3}$, using the convexity of the function $x \mapsto x^C$. 
Finally, if $t > t_A(F^o)$ and $H = F^o$ then we have
$$N_\phi(F^o)(m) \, \le \, (\log n)^{\Delta(F^o,A)} \, \le \, \frac{\Yt(m) \cdot L_A^D(F)(m)}{n^{\eps/2}},$$
since $\Yt(m) \ge n^{\eps}$ and $(\log n)^{\Delta(F-v,A)} \le n^{v_A(F)}$, by assumption, which implies that
$$L_A^D(F)(m) \, \ge \, (\log n)^{\Delta(F-v,A)} \, \ge \,  \frac{(\log n)^{\Delta(F,A)}}{n^{\eps/2}}  \, \ge \,  \frac{(\log n)^{\Delta(F^o,A)}}{n^{\eps/2}},$$ 
by Lemma~\ref{obs:deltaF'A'Fv}. This completes the proof of the lemma.
\end{proof}

We can now prove Lemma~\ref{lem:ExDeltaMF}.

\begin{proof}[Proof of Lemma~\ref{lem:ExDeltaMF}]
We begin by recalling~\eqref{eq:deltaLoD}, and that
\begin{equation}\label{eq:ExDeltaMF:basic}
\Delta M_\phi^{F}(m) \, = \, C_\phi(F)(m) - \hat{D}_\phi(F)(m) + \big| \Delta L_A^D(F)(m) \big| + \frac{r}{4n^3} \cdot L_A^D(F)(m).
\end{equation}
It follows that our task is to bound $\Ex\big[ C_\phi(F)(m) \big]$ and $\Ex\big[ \hat{D}_\phi(F)(m) \big]$ in terms of $L_A^D(F)(m)$. In order to do so, we shall essentially repeat the proof of Lemma~\ref{selfN}, except ignoring copies of $F$ which are destroyed by edges which intersect $\phi(A)$. 

Indeed, recall first, from~\eqref{eq:number:created}, that the expected number of such copies of $F$ created in a single step is exactly 
$$\sum_{F^o \in \F^o_F} \frac{N_{\phi}(F^o)}{Q(m)}.$$
Next, similarly as in~\eqref{eq:selfNapprox}, and using the event $\Z(m)$, we claim that the expected number of copies of $F$ rooted at $\phi(A)$ which are destroyed by the addition of an edge disjoint from $\phi(A)$ in step $m+1$ is at least
$$\frac{1}{Q(m)} \sum_{F^* \in N_\phi(F)} \bigg( \sum_{f \in O(F^*)} \bigg( \frac{1 - \eps}{2} \bigg) \Yt(m) \,-\, o(F)^2 (\log n)^2 \bigg).$$
To see this, simply note that each open edge $f \in O(F)$ has at most one endpoint in $A$, and hence there are at least $\big( 1/2 + o(1) \big)\Yt(m)$ open edges which close $f$ and are disjoint from~$\phi(A)$, since the event $\E(m)$ holds. 

Noting that 
$\Yt(m) \gg o(F)^2 (\log n)^2$, we obtain
$$\Ex\big[ C_\phi(F)(m) - \hat{D}_\phi(F)(m) \big] \, \le \, \frac{1}{Q(m)} \bigg( \sum_{F^o \in \F^o_F} N_{\phi}(F^o) \,-\, \left( \frac{1}{2} - \eps \right) o(F) \cdot \Yt(m) \cdot N_\phi(F) \bigg),$$
and hence, by Lemma~\ref{lem:bal:Fo}, 
$$\Ex\big[ C_\phi(F)(m) - \hat{D}_\phi(F)(m) \big] \, \le \, \frac{\Yt(m) \cdot L_A^D(F)(m)}{Q(m)} \bigg( \frac{e(F)}{\log n} \,- \frac{o(F)}{3} \bigg),$$
since the event $\LL_\phi^F(m)$ implies that $N_\phi(F)(m) \ge \big( \frac{3}{4} - \eps \big) L_A^D(F)(m)$. 


Finally, recalling that $o(F) > 0$ (since $0 < t_A(F) < t^*$) and that $e(F) \ll \log n$, observe that
$$\frac{\Yt(m) \cdot L_A^D(F)(m)}{Q(m)} \bigg( \frac{e(F)}{\log n} \,- \frac{o(F)}{3} \bigg)  \, \le \, - \, \frac{5o(F)}{2} \cdot \frac{t}{n^{3/2}} \cdot L_A^D(F)(m),$$
since the event $\Q(m)$ holds and $8/3 > 5/2$. 
By~\eqref{eq:deltaLoD} and~\eqref{eq:ExDeltaMF:basic}, and since $m \ge r$, it follows that
$$\Ex\big[ \Delta M_\phi^{F}(m) \big] \, \le \, \bigg( - \frac{5o(F)}{2} \cdot \frac{t}{n^{3/2}} + \big( 2o(F) + \eps \big) \cdot \frac{t}{n^{3/2}} + \frac{r}{4n^3} \bigg) \cdot L_A^D(F)(m) \, < \, 0,$$
as required. 

The proof of the claimed bound on $\Ex\big[ | \Delta M_\phi^{F}(m) | \big]$ is almost identical. Indeed, since the event $\LL_\phi^F(m)$ implies that $N_\phi(F)(m) \le L_A^D(F)(m)$, repeating the calculation above gives
\begin{multline*}
\Ex\big[ | \Delta M_\phi^{F}(m) | \big] \, \le \, \Ex\big[ C_\phi(F)(m) + \hat{D}_\phi(F)(m) \big] + \big| \Delta L_A^D(F)(m) \big| + \frac{r}{4n^3} \cdot L_A^D(F)(m)\\
 \, \le \, \bigg( \Big( (4+\eps) + (2 +\eps) \Big) \cdot \frac{o(F) \cdot t}{n^{3/2}} + \frac{r}{4n^3} \bigg) \cdot L_A^D(F)(m) \, \le \, \frac{8 \cdot o(F) \cdot t}{n^{3/2}} \cdot L_A^D(F)(m)
\end{multline*}
as required. 
\end{proof}

Finally, we are ready to prove Proposition~\ref{prop:balanced:late}. 


\begin{proof}[Proof of Proposition~\ref{prop:balanced:late}]
We are required to prove that, for every $t_A(F) \cdot n^{3/2} \le m \le m^*$,
\begin{equation}\label{eq:prop:bal:restatement}
\Pr\Big( \M(m)^c \cap \E(m) \cap \Z(m) \cap \Q(m) \Big) \, \le \, n^{-3\log n}.
\end{equation}
Recall that the event $\M(m)^c \cap \E(m)$ implies that the event $\LL_\phi^F(r,s)$ holds for some pair $(r,s)$. We claim that, for each faithful $\phi \colon A \to V(G_m)$ and each triple $(m',r,s) \in \N^3$ with $t_A(F) \cdot n^{3/2} \le r \le r+s \le m' \le m^*$, we have 
\begin{equation}\label{eq:proof:Prop:bal}
\Pr\Big( \LL_\phi^F(r,s) \cap \E(m') \cap \Z(m') \cap \Q(m') \Big) \, \le \, n^{-(\log n)^2}.
\end{equation}
By the union bound, this implies~\eqref{eq:prop:bal:restatement}, and so will be sufficient to prove the proposition.

In order to prove~\eqref{eq:proof:Prop:bal}, we shall apply Lemma~\ref{mart} to $M_\phi^F$. Recall that the event $\LL_\phi^F(r,s)$ implies that $M_\phi^{F}(r+s_0) \ge \frac{1}{4o(F)} \cdot L_A^D(F)(r+s_0)$, by Lemma~\ref{lem:LLimpliesMF}, and set 
$$\alpha = (\log n)^{\Delta(F-v,A) - C} \qquad \text{and} \qquad \beta = \frac{C}{s_0} \cdot L_A^D(F)(r).$$
Set $\K(m) := \ds\bigcap_{r \le m' \le m} \LL_\phi^F(m') \cap \E(m) \cap \Z(m) \cap \Q(m)$, and recall that if $\K(m)$ holds then 
\begin{itemize}
\item[$(a)$] $| \Delta M_\phi^{F}(m) | \le \alpha$, by Lemma~\ref{lem:DeltaMF}, 
\item[$(b)$] $\Ex\big[ | \Delta M_\phi^{F}(m) | \big] \le \beta$, by Lemma~\ref{lem:ExDeltaMF}, since $L_A^D(F)$ is decreasing and $s_0 \le \frac{n^{3/2}}{t \cdot o(F)}$,
\item[$(c)$] $L_A^D(F)(r+s_0) \le \beta s_0$, since $L_A^D(F)$ is decreasing and $o(F) > 0$. 
\end{itemize}
Moreover, $M_\phi^F$ is a super-martingale on $[r,r+s_0]$, by Lemma~\ref{lem:ExDeltaMF}, and
$$\alpha \cdot \beta \cdot s_0 \, \le \, C^2 \cdot (\log n)^{\Delta(F-v,A) - C} \cdot L_A^D(F)(r+s_0),$$
since $L_A^D(F)(r) \le C \cdot L_A^D(F)(r+s_0)$, as noted earlier. Hence, by Lemma~\ref{mart}, we have
$$\Pr\bigg( \Big( M_\phi^F(r+s_0) \ge \frac{1}{4o(F)} \cdot L_A^D(F)(r+s_0) \Big) \cap \K(m-1) \bigg) \, \le \, \exp\Big( - (\log n)^3 \Big),$$
which implies~\eqref{eq:proof:Prop:bal}. Finally, summing over choices\footnote{Note that, since $|A| \le (\log n)^{1/5}$, we have at most $n$ choices for $\phi$.} of $\phi$ and $(r,s)$, the proposition follows.
\end{proof}

\subsection{Bounding the maximum change in $N^*_\phi(F)$}\label{deltaNsec}

 
We now return to the range $\omega < t \le t_A(F)$, and use the event $\M(m)$ (see Definition~\ref{def:event:M}), together with the tools developed in Section~\ref{createsec}, in order to bound the maximum possible single-step change of the function $N_\phi^*(F)$. The aim of this subsection is to prove the following lemma. 

\begin{lemma}\label{maxalphabeta}
Let $(F,A)$ be a graph structure pair, let $\omega < t \le t_A(F)$, and suppose that $(\log n)^{\gamma(F,A)} \le n^{v_A(F)+e(F)+1}$. If $\E(m) \cap \M(m)$ holds, then
$$\big| \Delta N_\phi^*(F)(m) \big| \, \le \,  (\log n)^{- \sqrt{\Delta(F,A)}} \cdot \frac{g_{F,A}(t)}{1 + g_{F,A}(t)}$$
for every $\phi \colon A \to V(G_m)$ which is faithful at time $t$.
\end{lemma}

We shall need one more straightforward lemma. 

\begin{lemma}\label{lem:F4}
Let $(F,A)$ be a graph structure pair. If $A \subsetneq H \subseteq F$ and $\omega < t \le t_A(F)$, then 
$$\Nt_A(H)(m) \, \ge \, \left( \frac{(\log n)^{\gamma(F,A)}}{g_{F,A}(t)} \right)^2.$$
\end{lemma}

\begin{proof}
Set $a = 2v_A(H) - e(H)$, $b = o(H)$ and $c = c(F,A)$. Note that 
$$ac \, \ge \, 2b$$
by the definition~\eqref{def:cFA} of $c(F,A)$, that $a \ge 1$, since $t_A(F) > 0$ and $H \neq A$, and that
$$\Nt_A(H)(m) \, = \, (2t)^{e(H)} n^{a/2} e^{-4 b t^2},$$
by~\eqref{def:NtF}. Since $g_{F,A}(t) = n^{-1/4} e^{ct^2} (\log n)^{\gamma(F,A)}$ and $n^{-1/4} e^{c t_A(F)^2} \le 1$, we obtain
$$\Nt_A(H)(m) \, \ge \, \big( n^{1/4} e^{-ct^2} \big)^{2a}  \, \ge \, \big( n^{1/4} e^{-ct^2} \big)^{2}  \, = \, \left( \frac{(\log n)^{\gamma(F,A)}}{g_{F,A}(t)} \right)^{2}$$
for every $\omega < t \le t_A(F)$, as required. 
\end{proof}

We are now ready for a key calculation.

\begin{lemma}\label{lem:F3}
Let $(F,A)$ be a graph structure pair, let $\omega < t \le t_A(F)$, and suppose that $(\log n)^{\gamma(F,A)} \le n^{v_A(F)+e(F)+1}$. If $\E(m) \cap \M(m)$ holds, then
\begin{equation}\label{eq:lem:F3}
N_{\phi'}(F')(m) \, \le \, (\log n)^{- 2\sqrt{\Delta(F,A)}} \cdot \frac{g_{F,A}(t)^2}{1 + g_{F,A}(t)} \cdot \Nt_A(F)(m).
\end{equation}
for every $(F',A') \in \F_{F,A}^-$ and every $\phi' \colon A' \to V(G_m)$ which is faithful at time~$t$.
\end{lemma}

\begin{proof}
We split the proof into three cases. 

\medskip
\noindent{\textbf{Case 1:}} $\omega < t \le t_{A'}(F')$ and $A' \cap F = A$.
\medskip

The condition $A' \cap F = A$ implies that $\sqrt{n} \cdot \Nt_{A'}(F') \le 2t \cdot \Nt_A(F)$, by Observation~\ref{obs:AcapF}. Suppose first that $g_{F,A}(t) \le 1$. Then, since $\omega < t \le t_{A'}(F')$, by the event $\E(m)$ we have
\begin{multline*}
N_{\phi'}(F') \, \le \, \big( 1 + g_{F',A'}(t) \big) \cdot \Nt_{A'}(F')  \, \le \, (\log n)^{\gamma(F',A') + 1} \cdot \frac{2t}{\sqrt{n}} \cdot \Nt_{A}(F)\\
\, \le \, (\log n)^{(1 + 2\eps)\gamma(F,A)} \cdot \frac{g_{F,A}(t)^2}{(\log n)^{2\gamma(F,A)}} \cdot \Nt_{A}(F) \, \le \, (\log n)^{-2 \sqrt{\Delta(F,A)}} \cdot \frac{g_{F,A}(t)^2}{1 + g_{F,A}(t)} \cdot \Nt_A(F).
\end{multline*}
Note that we used Observation~\ref{obs:deltaaddlittle} and the fact that $g_{F,A}(t)^2 \sqrt{n} \ge (\log n)^{2\gamma(F,A)}$ in the third step, and our assumption that $g_{F,A}(t) \le 1$ in the fourth.

On the other hand, if $g_{F,A}(t) \ge 1$ then, since $(\log n)^{\gamma(F,A)} \le n^{v_A(F)+e(F)+1}$ and $\omega < t \le \min\big\{ t_A(F), t_{A'}(F') \big\}$, we have
$$N_{\phi'}(F') \, \le \, \big( 1 + g_{F',A'}(t) \big) \cdot \Nt_{A'}(F')  \, \le \, n^{1/4 + 2\eps} \cdot g_{F,A}(t) \cdot \frac{2t}{\sqrt{n}} \cdot \Nt_{A}(F),$$
by Lemma~\ref{lem:gF'A'gFA} and the event $\E(m)$. Noting that $(\log n)^{\sqrt{\Delta(F,A)}} \le n^\eps$, by Lemma~\ref{obs:deltaFdeltaFv}, it follows that
$$N_{\phi'}(F')  \, \le \, n^{-1/4 + 3\eps} \cdot g_{F,A}(t) \cdot \Nt_{A}(F)  \, \le \, (\log n)^{-2\sqrt{\Delta(F,A)}}  \cdot \frac{g_{F,A}(t)^2}{1 + g_{F,A}(t)} \cdot \Nt_A(F)$$
since $g_{F,A}(t) \ge 1$, as required.

\bigskip
\noindent{\textbf{Case 2:}} $\omega < t \le t_{A'}(F')$ and $A' \cap F \neq A$.
\medskip

By Observation~\ref{obs:NH'F'NHF} and Lemma~\ref{lem:F4}, we have  
\begin{equation}\label{eq:F3:case21}
\Nt_{A'}(F') \, \le \, \Nt_{A' \cap F}(F) \, = \, \frac{\Nt_A(F)}{\Nt_A(A' \cap F)} \, \le \, \left( \frac{g_{F,A}(t)}{(\log n)^{\gamma(F,A)}} \right)^{2} \cdot \Nt_A(F),
\end{equation}
since $A' \cap F \neq A$. Moreover, $g_{F',A'}(t) \le (\log n)^{\Delta(F,A) - 3\sqrt{\Delta(F,A)}}$, by Lemma~\ref{obs:gF'A'gFA} and Observation~\ref{obs:deltaF'H'A'deltaFA}. Using the event $\E(m)$ and the fact that $\omega < t \le t_{A'}(F')$, it follows that
\begin{equation}\label{eq:F3:case22}
N_{\phi'}(F') \, \le \, \big( 1 + g_{F',A'}(t) \big) \cdot \Nt_{A'}(F')  \, \le \, \frac{g_{F,A}(t)^2}{1 + g_{F,A}(t)} \cdot (\log n)^{-2 \sqrt{\Delta(F,A)}} \cdot \Nt_A(F),
\end{equation}
as required, since $1 + g_{F,A}(t) \le (\log n)^{\gamma(F,A) + 1}$ and $\gamma(F,A) + \sqrt{\Delta(F,A)} > \Delta(F,A) + 1$. 

\bigskip
\noindent{\textbf{Case 3:}} $t > t_{A'}(F')$.
\medskip

Let $A' \subsetneq H' \subseteq F'$ be minimal such that $t < t_{H'}(F')$, set $H = H' \cap F$, and suppose first that either $H' \neq F'$ or $A' \cap F \neq A$. Then, since $\E(m)$ holds and $t > t_{A'}(F')$, we have
$$N_{\phi'}(F') \, \le \, (\log n)^{\Delta(F',H',A')} \Nt_{H'}(F') \, \le \, (\log n)^{\Delta(F,A) - 3\sqrt{\Delta(F,A)}} \Nt_H(F).$$
where the second inequality follows from by Observations~\ref{obs:NH'F'NHF} and~\ref{obs:deltaF'H'A'deltaFA}. Since $H \neq A$, by Observation~\ref{obs:AcapF}, and $\Nt_A(F) = \Nt_A(H) \cdot \Nt_H(F)$, we may apply Lemma~\ref{lem:F4} to $\Nt_A(H)$, as in~\eqref{eq:F3:case21}, and hence obtain
$$N_{\phi'}(F')(m) \, \le \, \frac{g_{F,A}(t)^2}{1 + g_{F,A}(t)} \cdot (\log n)^{- 2 \sqrt{\Delta(F,A)}} \cdot \Nt_A(F)(m),$$
exactly as in~\eqref{eq:F3:case22}.

So suppose now that $H' = F'$ and $A' \cap F = A$, let $A' \subseteq H_0 \subsetneq \dots \subsetneq H_\ell = F'$ be the building sequence of $(F',A')$, and note that since $H' = F'$, we have $t \ge t_\ell$, by Lemma~\ref{lem:Hj}.  If $(F',A')$ is unbalanced then, by Lemma~\ref{lem:unbalanced}, Observation~\ref{obs:deltaF'H'A'deltaFA} and the event $\E(m)$, 
$$N_{\phi'}(F')(m) \, \le \, (\log n)^{\Delta(F',H_{\ell-1},A')} \, \le \, (\log n)^{\Delta(F,A) - 3\sqrt{\Delta(F,A)}},$$
which again implies~\eqref{eq:lem:F3}, using Lemma~\ref{lem:F4} to bound $\Nt_A(F)$. On the other hand, if $(F',A')$ is balanced then, since the event $\M(m)$ holds, we have
$$N_{\phi'}(F')(m) \, \le \, \max\Big\{ e^{- o(F')(t^2 - t_{A'}(F')^2)} (\log n)^{\Delta(F',A')}, (\log n)^{\Delta(F'-v,A')} \Big\}.$$
By Observation~\ref{obs:deltaF'vA'deltaFA}, if
$$N_{\phi'}(F')(m) \, \le \, (\log n)^{\Delta(F'-v,A')} \, \le \, (\log n)^{\Delta(F,A) - 3\sqrt{\Delta(F,A)}}$$
then we are done as before, so let's assume that $t_{A'}(F') > 0$, and that
\begin{equation}\label{eq:case3:balancedbound}
N_{\phi'}(F')(m) \, \le \, e^{- o(F')(t^2 - t_{A'}(F')^2)} (\log n)^{\Delta(F',A')}.
\end{equation}
Note that moreover $o(F) = o(F') > 0$, since $A' \cap F = A$ (see Definition~\ref{def:F-}).

Next, we claim that if either $g_{F,A}(t) \le 1$ or $c(F,A) = 2$, then 
\begin{equation}\label{eq:gsmallortAFiststar}
N_{\phi'}(F')(m) \, \le \, (\log n)^{\Delta(F',A')} \, \le \, \frac{(\log n)^{2 \gamma(F,A) - 2\sqrt{\Delta(F,A)} } }{1 + g_{F,A}(t)},
\end{equation}
by Observation~\ref{obs:deltaaddlittle} and the event $\E(m)$. Indeed, if $g_{F,A}(t) \le 1$ then this is immediate, and otherwise we have
$$1 \, \le \, g_{F,A}(t) \, \le \, n^{-\eps} (\log n)^{\gamma(F,A)} \, \le \, n^{-\eps/2} (\log n)^{\gamma(F,A) - 3\sqrt{\Delta(F,A)}},$$
where the second inequality follows (for every $t \le t_A(F)$) since $c(F,A) = 2$, and the third holds because $(\log n)^{\sqrt{\Delta(F,A)}} \le n^{\eps/6}$, by Lemma~\ref{obs:deltaFdeltaFv}. Using Lemma~\ref{obs:deltaFdeltaFv} once again, it follows that
$$(\log n)^{\Delta(F',A')} \, \le \, n^{\eps/2} (\log n)^{\Delta(F,A)} \, \le \, \frac{(\log n)^{2 \gamma(F,A) - 2\sqrt{\Delta(F,A)} } }{1 + g_{F,A}(t)},$$
as claimed. The bound~\eqref{eq:lem:F3} now follows immediately from~\eqref{eq:gsmallortAFiststar}, using Lemma~\ref{lem:F4}.

\smallskip 

It remains to deal with the case in which~\eqref{eq:case3:balancedbound} holds, $(F',A')$ is balanced and\footnote{Recall that $e^{c(F,A)t_A(F)^2} \le n^{1/4}$.}
\begin{equation}\label{eq:gFAt:alternative}
1 \, \le \, g_{F,A}(t) \, \le \, e^{c(F,A) (t^2 - t_A(F)^2)} (\log n)^{\gamma(F,A)}.
\end{equation}
where $0 < t_{A'}(F') < t \le t_A(F) \le t_A^*(F)$ and $c(F,A) > 2$. We claim that
\begin{equation}\label{eq:case3:tminust}
 t_A^*(F)^2 - t_{A'}(F')^2 \, \ge \, \frac{\log n}{8 \cdot o(F)}.
\end{equation}
To see this, recall first that $t^*_{A'}(F') = t_{A'}(F')$, since $(F',A')$ is balanced and $t_{A'}(F') < t^*$, and that $o(F) = o(F')$. Recall also from~\eqref{eq:tstardef} that $8o(F)t_A^*(F)^2 / \log n$ and $8o(F')t^*_{A'}(F')^2 / \log n$ are both integers. Since $t_{A'}(F') < t_A^*(F)$, it follows that these integers are distinct, and hence~\eqref{eq:case3:tminust} holds.

Suppose first that $o(F) \big( t_A(F)^2 - t_{A'}(F')^2 \big) \ge \eps \cdot \log n$, and recall from Observation~\ref{obs:zeroomega} that $c(F,A) \ge o(F) / v_A(F)$. It follows that 
\begin{equation}\label{eq:case3:oplusc}
o(F) \big( t^2 - t_{A'}(F')^2 \big) + c(F,A) \big( t_A(F)^2 - t^2 \big) \, \ge \, \frac{o(F)}{v_A(F)} \cdot \big( t_A(F)^2 -  t_{A'}(F')^2 \big) \, \ge \, \frac{\eps \cdot \log n}{v_A(F)},
\end{equation}
and by Lemma~\ref{obs:deltaFdeltaFv} we have
\begin{equation}\label{eq:case3:nepsovervAF}
n^{\eps / v_A(F)} \, \ge \, (\log n)^{\Delta(F',A') - \Delta(F,A) + 3 \sqrt{\Delta(F,A)} }.
\end{equation}
Now, by~\eqref{eq:case3:balancedbound} and~\eqref{eq:gFAt:alternative} we have
\begin{align*}
N_{\phi'}(F')(m) & \, \le \, e^{- o(F)(t^2 - t_{A'}(F')^2)} (\log n)^{\Delta(F',A')} \\
&  \, \le \, e^{- o(F)(t^2 - t_{A'}(F')^2) - c(F,A) ( t_A(F)^2 - t^2 )} \cdot \frac{(\log n)^{\Delta(F',A') + \gamma(F,A)}}{g_{F,A}(t)}\end{align*}
and hence, by~\eqref{eq:case3:oplusc},~\eqref{eq:case3:nepsovervAF} and Lemma~\ref{lem:F4}, and since $g_{F,A}(t) \ge 1$, 
$$N_{\phi'}(F')(m) \, \le \, \frac{(\log n)^{2\gamma(F,A) - 2 \sqrt{\Delta(F,A)} } }{1 + g_{F,A}(t)} \, \le \, (\log n)^{- 2\sqrt{\Delta(F,A)}} \cdot \frac{g_{F,A}(t)^2}{1 + g_{F,A}(t)} \cdot \Nt_A(F)(m),$$
as required.

Finally, suppose that $o(F) \big( t_A(F)^2 - t_{A'}(F')^2 \big) \le \eps \cdot \log n$. It follows from~\eqref{eq:case3:tminust} that
$$o(F) \big( t_A^*(F)^2 - t_A(F)^2 \big) \, \ge \, \left( \frac{1}{8} - \eps \right) \log n,$$
and hence, since $8 o(F) t_A^*(F)^2 = \big( 2 v_A(F) - e(F) \big) \log n$, by~\eqref{eq:tstardef}\footnote{Note that $0 < t_A^*(F) < \infty$, since $t_A(F) > 0$ and $o(F) > 0$.}, we have 
$$\Nt_A(F)(m) \, = \, (2t)^{e(F)} \cdot \exp\Big( 4 o(F) \big( t_A^*(F)^2 - t^2 \big) \Big) \, \ge \, n^{1/2 - 4\eps}.$$
Since $g_{F,A}(t) \ge n^{-1/4} (\log n)^{\gamma(F,A)}$, it follows, using~\eqref{eq:case3:balancedbound},~\eqref{eq:case3:nepsovervAF} and Lemma~\ref{obs:deltaFdeltaFv}, that 
\begin{align*}
N_{\phi'}(F')(m) & \, \le \, (\log n)^{\Delta(F',A')} \, \le \, n^\eps \cdot (\log n)^{\gamma(F,A)} \, \le \, n^{1/4+\eps} \cdot g_{F,A}(t) \\
&  \, \le \, n^{-1/4+5\eps} \cdot g_{F,A}(t)  \cdot \Nt_A(F)(m) \, \le \,  (\log n)^{- 2\sqrt{\Delta(F,A)}} \cdot \frac{g_{F,A}(t)^2}{1 + g_{F,A}(t)} \cdot \Nt_A(F)(m),
\end{align*}
since $g_{F,A}(t) \ge 1$ and $(\log n)^{\sqrt{\Delta(F,A)}} \le n^\eps$, as required.
\end{proof}

Using Lemma~\ref{lem:F3}, we can now easily bound $|\Delta N_\phi^*(F)(m)|$.

\begin{proof}[Proof of Lemma~\ref{maxalphabeta}]
By Lemmas~\ref{lem:F+}, ~\ref{lem:F-} and~\ref{lem:F3}, the maximum number of copies of $F$ rooted at $\phi(A)$ which can be either created or destroyed by the addition of a single edge is at most
$$|\F_{F,A}^- \cup \F_{F,A}^+| \cdot (\log n)^{- 2\sqrt{\Delta(F,A)}} \cdot \frac{g_{F,A}(t)^2}{1 + g_{F,A}(t)} \cdot \Nt_A(F)(m).$$
Moreover, by Lemma~\ref{lem:chainstar}, we have
$$| \Delta N_\phi^*(F)(m) |  \, \le \, 2 \cdot \left( \frac{ | \Delta N_\phi(F)(m) | }{g_{F,A}(t) \Nt_A(F)(m)} \,+\, \frac{1 + g_{F,A}(t)}{g_{F,A}(t)} \cdot \frac{\log n}{n^{3/2}} \right).$$
Since $g_{F,A}(t) \ge n^{-1/4}$ and $v(F) \ll \log n$, it follows that
\begin{align*}
| \Delta N_\phi^*(F)(m) | & \, \le \, \left( \frac{C \cdot v(F)^2}{g_{F,A}(t)} \right) \cdot (\log n)^{- 2\sqrt{\Delta(F,A)}} \cdot \frac{g_{F,A}(t)^2}{1 + g_{F,A}(t)}\\
& \, \le \, (\log n)^{- \sqrt{\Delta(F,A)}} \cdot \frac{g_{F,A}(t)}{1 + g_{F,A}(t)}
\end{align*}as claimed.
\end{proof}

\subsection{The land before time $t = \omega$}\label{landbeforetimeSec}

When $t$ is bounded, the variables $N_\phi(F)$ are not self-correcting, and so we cannot use the martingale technique introduced in Section~\ref{MartSec}. Fortunately for us, however, the faster-growing bounds given by the method of Bohman~\cite{Boh} suffice for our purposes. In this subsection we shall state the bounds we obtain in the case $0 < t \le \omega < t_A(F)$, and give an extended sketch of their proof. Since the ideas used in this section are not new, 
we postpone the details to the Appendix. 

Recall from~\eqref{def:ffat} that
$$f_{F,A}(t) \, = \, e^{C(o(F)+1)(t^2 + 1)} n^{-1/4} (\log n)^{\Delta(F,A) - \sqrt{\Delta(F,A)}}$$
for each graph structure pair $(F,A)$, and that $\K^\E(m) = \Y(m) \cap \Z(m) \cap \Q(m)$. The following proposition shows that Theorem~\ref{EEthm}$(a)$ is unlikely to be the first of our constraints to fail.

\begin{prop}\label{NFa}
Let $(F,A)$ be a graph structure pair, and let $0 < t \le \omega < t_A(F)$. Then, with probability at least $1 - n^{- 3\log n}$, either $\big( \E(m - 1) \cap \M(m - 1) \cap \K^\E(m-1) \big)^c$ holds, or
\begin{equation}\label{eq:propNFa}
N_\phi(F)(m) \, \in \, \Nt_A(F)(m) \,\pm\, f_{F,A}(t) \cdot \Nt_A(F)(n^{3/2})
\end{equation}
for every $\phi \colon A \to V(G_m)$ which is faithful at time~$t$. 
\end{prop}

The proof of Proposition~\ref{NFa} relies heavily on the fact that the event $\Y(m)$ gives us stronger bounds (in the range $t \le \omega$) on the variables $Y_e$ than those given by the event $\E(m)$. Recall from~\eqref{def:fy} that 
$$f_y(t) = e^{Ct^2} n^{-1/4} (\log n)^{5/2} \qquad \text{and} \qquad f_x(t) = e^{-4t^2} f_y(t).$$
The following proposition is essentially due to Bohman~\cite{Boh}, although he stated only a slightly weaker version of it; for completeness, we give a proof in the Appendix~\cite{App}. 

\begin{prop}[Bohman~\cite{Boh}]\label{lem:landbeforetime}
Let $m \le \omega \cdot n^{3/2}$. With probability at least $1 - n^{- 4\log n}$, either $\big( \Z(m-1) \cap \Q(m-1) \big)^c$ holds, or we have
$$X_e(m) \in \Xt(m) \pm f_x(t) \Xt(n^{3/2}) \qquad \text{and} \qquad Y_e(m) \in \Yt(m) \pm f_y(t) \Yt(n^{3/2})$$
for every $e \in O(G_m)$.
\end{prop}

Finally, let us note that $\Y(m-1) \Rightarrow \Q(m)$ in the range $m \le \omega \cdot n^{3/2}$.

\begin{prop}\label{prop:YimpliesQ}
For every $m \le \omega \cdot n^{3/2}$, if $\Y(m-1)$ holds then
$$Q(m) \, \in \, e^{-4t^2} {n \choose 2} \, \pm \, \eps \cdot f_y(t) {n \choose 2}.$$
\end{prop}

\begin{proof}
Recall that $\Delta Q(m) = - Y_e(m) - 1$, where $e$ is the edge chosen in step $m+1$ of the triangle-free process. Noting that, for every $m' \le \omega \cdot n^{3/2}$, we have
$$\sum_{m=0}^{m'-1} \Yt(m) \, \in \, \big( 1 - e^{-4t'^2} \big) {n \choose 2} \pm n \qquad \textup{and} \qquad \sum_{m=0}^{m'-1} f_y(t) \, \le \, \frac{1}{C} \cdot f_y(t').$$ 
It follows that if $\Y(m'-1)$ holds, then
$$Q(m') \, \in \, {n \choose 2} \,-\, \sum_{m=0}^{m'-1} \Big( \Yt(m) \pm f_y(t) \Yt(n^{3/2}) \Big) \, \subseteq \,  \Qt(m') \, \pm \, \eps \cdot f_y(t') {n \choose 2},$$
as claimed.
\end{proof}

We shall next use Bohman's method to control the variables $N_\phi(F)$ in the range $t \le \omega$. Let us fix a graph structure triple $(F,A,\phi)$ with $t_A(F) > 0$. The first step  is to break up $N_\phi(F)$ as follows: 
\begin{equation}\label{eq:N=C+D}
N_\phi(F)(m') \, = \, \sum_{m = 0}^{m'-1} \Big( C_\phi(F)(m) - D_\phi(F)(m) \Big),
\end{equation}
where $C_\phi(F)(m)$ denotes the number of copies of $F$ rooted at $\phi(A)$ which are created at step $m+1$ of the triangle-free process, and $D_\phi(F)(m)$ denotes the number of such copies which are destroyed in that step. We shall need bounds on the expected and maximum possible single-step changes in $C_\phi(F)(m)$ and $D_\phi(F)(m)$. Since the proofs of these bounds are straightforward, and somewhat technical, we defer the details to the Appendix~\cite{App}. 

\begin{lemma}\label{Cmart}
Let $(F,A,\phi)$ be a graph structure triple, and suppose that $0 < t \le \omega < t_A(F)$, and that $\phi$ is faithful at time~$t$. If $\E(m) \cap \Q(m)$ holds, then
$$\Ex\big[ C_\phi(F)(m) \,|\, G_m \big] - \frac{e(F) \cdot \Nt_A(F)(m)}{t \cdot n^{3/2}} \, \in \, \pm \, \frac{f_{F,A}(t) \cdot \Nt_A(F)(n^{3/2})}{n^{3/2}}.$$
\end{lemma}

\begin{lemma}\label{Dmart}
Let $(F,A,\phi)$ be a graph structure triple, and suppose that $0 < t \le \omega < t_A(F)$, and that $\phi$ is faithful at time~$t$. If $\E(m) \cap \Y(m) \cap \Z(m) \cap \Q(m)$ holds, then
$$\Ex\big[ D_\phi(F)(m) \,|\, G_m \big] - \frac{8t \cdot o(F) \cdot \Nt_A(F)(m)}{n^{3/2}} \, \in \, \pm \, \frac{C \cdot o(F) \cdot ( t + 1 )}{n^{3/2}} \cdot  f_{F,A}(t) \Nt_A(F)(n^{3/2}).$$
\end{lemma}

In order to apply Lemma~\ref{Bohmart}, we shall also need bounds on $C_\phi(F)(m)$ and $D_\phi(F)(m)$ which hold deterministically for all $0 < t \le \omega$. 

\begin{lemma}\label{Cbound}
Let $(F,A,\phi)$ be a graph structure triple, and suppose that $0 < t \le \omega < t_A(F)$, and that $\phi$ is faithful at time~$t$. If $\E(m) \cap \M(m)$ holds, then 
$$0 \, \le \, C_\phi(F)(m) \, \le \, \min\Big\{ n^\eps, \, (\log n)^{\Delta(F,A)/2} \Big\} \cdot \frac{(\log n)^{\Delta(F,A) - 2\sqrt{\Delta(F,A)}}}{\sqrt{n}} \cdot \Nt_A(F)(n^{3/2}).$$
Moreover, the same bounds also hold for $D_\phi(F)(m)$.
\end{lemma}

We can now apply Lemma~\ref{Bohmart} to the variables $C_\phi(F)$ and $D_\phi(F)$; we again refer the reader to the Appendix for the full details. 

\begin{proof}[Sketch proof of Proposition~\ref{NFa}]
For each $m \in [m^*]$, set $\K(m) = \E(m) \cap \M(m) \cap \K^\E(m)$. We shall bound, for each $m_0 \le \omega \cdot n^{3/2}$, the probability that $m_0$ is the minimal $m \in \N$ such that $\K(m-1)$ holds, and
$$N_\phi(F)(m) \, \not\in \, \Nt_A(F)(m) \,\pm\, f_{F,A}(t) \cdot \Nt_A(F)(n^{3/2})$$
for some $\phi$ which is faithful at time~$t = m \cdot n^{-3/2}$. Note that the event in the statement of the proposition implies that this event holds for some $m \le \omega \cdot n^{3/2}$.

Fix $m_0 \le \omega \cdot n^{3/2}$, and for each $m' \le m_0$, define random variables
$$M_C^\pm(m') = \sum_{m=0}^{m'-1} \bigg[ C_\phi(F)(m) - \frac{e(F) \cdot \Nt_A(F)(m)}{t \cdot n^{3/2}} \pm \, \frac{f_{F,A}(t) \cdot  \Nt_A(F)(n^{3/2})}{n^{3/2}} \bigg]$$
and
$$M_D^\pm(m') = \sum_{m=0}^{m'-1} \bigg[ D_\phi(F)(m) \,-\, \frac{8t \cdot o(F) \cdot \Nt_A(F)(m)}{n^{3/2}} \, \pm \, \frac{C \cdot o(F) \cdot ( t + 1 )}{n^{3/2}} \,  f_{F,A}(t) \Nt_A(F)(n^{3/2}) \bigg].$$
It follows from Lemmas~\ref{Cmart} and~\ref{Dmart} that, while the event $\E(m) \cap \Y(m) \cap \Z(m) \cap \Q(m)$ holds, $M_C^\pm$ and $M_D^\pm$ are both super-/sub-martingale pairs. Now, set 
$$\alpha \, = \, \bigg(  \min\big\{ n^\eps, \, (\log n)^{\Delta(F,A)/2} \big\} \cdot \frac{(\log n)^{\Delta(F,A) - 2\sqrt{\Delta(F,A)}}}{\sqrt{n} } + \frac{f_{F,A}(\omega)}{m_0} \bigg) \cdot \Nt_A(F)(n^{3/2})$$
and
$$\beta \, = \, \bigg( \frac{(\log n)^{e(F) + o(F)}}{n^{3/2}} + \frac{f_{F,A}(\omega)}{m_0} \bigg) \cdot \Nt_A(F)(n^{3/2}).$$
By Lemma~\ref{Cbound}, we have 
$$- \beta \, \le \, \Delta M^\pm_C(m) + \Delta M^\pm_D(m) \, \le \, \alpha$$
while $\E(m) \cap \M(m)$ holds. Moreover, since $f_{F,A}(t_0) \ge n^{-1/4} (\log n)^{\Delta(F,A) - \sqrt{\Delta(F,A)}}$, and we may assume that $m_0 \ge n^\eps$, we have
$$\frac{\alpha \cdot \beta \cdot m_0}{\Nt_A(F)(n^{3/2})^2} \, \le \, \frac{f_{F,A}(t_0)^2}{(\log n)^4}.$$
Hence, by Lemma~\ref{Bohmart}, we obtain 
$$\Pr\bigg( \bigg( M_C^-(m_0) > \frac{1}{4} f_{F,A}(t_0) \Nt_A(F)(n^{3/2}) \bigg) \cap \K(m_0-1) \bigg) \, \le \, e^{-(\log n)^3},$$
and similarly for $M_C^+$, $M_D^-$ and $M_D^+$.


Note that the number of choices for $\phi$ is negligible, since $|A| \le (\log n)^{1/5}$. Via a straightforward calculation it follows that, with probability at most $n^{-C\log n}$, 
$$N_\phi(F)(m_0) \, = \, \sum_{m = 0}^{m_0-1} \Big( C_\phi(F)(m) - D_\phi(F)(m) \Big) \, \not\in \, \Nt_A(F)(m_0) \,\pm\, f_{F,A}(t_0) \Nt_A(F)(n^{3/2}),$$
for some faithful $\phi \colon A \to V(G_m)$, as required.
\end{proof}

\subsection{Proof of Theorem~\ref{EEthm}}\label{EEproofSec} 

Finally, let us put the pieces together and prove Theorem~\ref{EEthm} using the martingale method introduced in Section~\ref{MartSec}. We shall break up the event that the theorem fails to hold into four sub-events, depending on whether the first\footnote{More precisely, a first, since there could be several which go astray at the same step.} graph structure triple which goes astray is tracking and/or balanced. Recall from~\eqref{def:event:KE}, and from Definitions~\ref{def:events:XYQ},~\ref{def:events:E} and~\ref{def:event:M}, the definitions of the events $\E(m)$, $\M(m)$ and $\K^\E(m)$. 


\begin{defn}
For each $m \in [m^*]$, we define events $\E_1(m)$, $\E_2(m)$, $\E_3(m)$ and $\E_4(m)$ as follows:
\begin{itemize}
\item[$(a)$] $\E_1(m)$ denotes the event that $\E(m - 1) \cap \M(m - 1) \cap \K^\E(m - 1)$ holds, and that there exists a graph structure triple $(F,A,\phi)$ with $0 < t \le \omega < t_A(F)$ such that
$$N_\phi(F)(m) \, \not\in \, \Nt_A(F)(m) \pm f_{F,A}(t) \Nt_A(F)(n^{3/2}),$$
and $\phi$ is faithful at time $t$.
\item[$(b)$] $\E_2(m)$ denotes the event that $\E(m - 1) \cap \M(m - 1) \cap \K^\E(m - 1)$ holds, and that there exists a graph structure triple $(F,A,\phi)$ with $\omega < t \le t_A(F)$ such that
$$N_\phi(F)(m) \, \not\in \, \big( 1 \pm g_{F,A}(t) \big) \Nt_A(F)(m),$$
and $\phi$ is faithful at time $t$.
\item[$(c)$] $\E_3(m)$ denotes the event that $\E(m - 1) \cap \K^\E(m - 1)$ holds, and that there exists a balanced graph structure triple $(F,A,\phi)$ with $t > t_A(F)$, such that either
$$N_\phi(F)(m) \, > \, \max\Big\{ e^{- o(F)(t^2 - t_A(F)^2)} (\log n)^{\Delta(F,A)}, (\log n)^{\Delta(F-v,A)} \Big\}$$
and $\phi$ is faithful at time $t$, or $t_A(F) = 0$,
$$N_\phi(F)(m) \, > \, (\log n)^{\Delta(F-v,A)}$$
and $\phi$ is faithful at time $t$.
\item[$(d)$] $\E_4(m)$ denotes the event that none of the events $\E_1(m)$, $\E_2(m)$ and $\E_3(m)$ holds, but the event $\E(m - 1) \cap \M(m - 1) \cap \K^\E(m-1)$ holds, and there exists an unbalanced graph structure triple $(F,A,\phi)$ with $t > t_A(F)$, such that\footnote{Here, as in the statement of Theorem~\ref{EEthm}, we set $m^+ = \max\{m,n^{3/2}\}$.}
$$N_\phi(F)(m) \, > \, (\log n)^{\Delta(F,H,A)} \Nt_H(F)(m^+),$$
where $A \subsetneq H \subseteq F$ is minimal such that $t < t_H(F)$, and $\phi$ is faithful at time $t$. 
\end{itemize}
\end{defn}

It is easy to see\footnote{Indeed, simply consider the minimum $m$ such that $\E(m)^c \cap \K^\E(m - 1)$ holds, and observe that if $\M(m - 1)^c$ holds then $\E_3(m - 1)$ does too.} that 
$$\bigcup_{m = 1}^{m^*} \E(m)^c \cap \K^\E(m - 1) \, \subseteq \, \bigcup_{m = 1}^{m^*} \E_1(m) \cup \E_2(m) \cup \E_3(m) \cup \E_4(m),$$
i.e., if the conclusion of Theorem~\ref{EEthm} fails, then one of $\E_1(m)$, $\E_2(m)$, $\E_3(m)$ and $\E_4(m)$ must hold for some $m \in [m^*]$. It will therefore suffice to bound the probabilities of these events. The next three lemmas do so; the first follows immediately from Propositions~\ref{prop:balanced:late} and~\ref{NFa}.

\begin{lemma}\label{lem:Eam}
For every $m \in [m^*]$,
$$\Pr\big( \E_1(m) \big) + \Pr\big( \E_3(m) \big) \, \le \, 2 \cdot n^{-3\log n}.$$
\end{lemma}

In the next lemma, we use the martingale argument from Section~\ref{MartSec} to control $N_\phi(F)$ in the range $\omega < t \le t_A(F)$.

\begin{lemma}\label{lem:Ebm}
For every $m \in [m^*]$,
\begin{equation}\label{eq:lem:Ebm}
\Pr\big( \E_2(m) \big) \, \le \, n^{-C\log n}.
\end{equation}
\end{lemma}

\begin{proof}
Noting that $\Pr\big( \E_2(m) \big) = 0$ if $t \le \omega$, let $\omega < t' \le t^*$, set $m' = t' \cdot n^{3/2}$ and suppose that $\E_2(m')$ holds. Let $(F,A,\phi)$ be the graph structure triple for which $|N^*_\phi(F)(m')| > 1$, and for which $\omega < t' \le t_A(F)$ and $\phi$ is faithful at time $t'$. Moreover, observe that
\begin{equation}\label{gamma:assumption}
(\log n)^{\gamma(F,A)} \le n^{v_A(F)+e(F)+1},
\end{equation}
since otherwise the bound in Theorem~\ref{EEthm}$(b)$ (and hence $\E_2(m')$) holds trivially, and that the event $\E(m' - 1) \cap \M(m' - 1) \cap \K^\E(m' - 1)$ holds. We shall apply the martingale method introduced in Section~\ref{MartSec} to the self-correcting variable $N_\phi(F)$, and then sum over choices of $(F,A,\phi)$ and $m'$. 

We begin by choosing a family of parameters as in Definition~\ref{def:reasonable}. Set $\K(m) = \E(m) \cap \M(m) \cap \Y(m) \cap \Z(m) \cap \Q(m)$ and $I = [a,b] = [\omega \cdot n^{3/2}, t_A(F) \cdot n^{3/2} ]$, and let
$$\alpha(t) \, = \, (\log n)^{- \sqrt{\Delta(F,A)}} \cdot \frac{g_{F,A}(t)}{1 + g_{F,A}(t)}$$
and
$$\beta(t) \, = \, \frac{C \cdot \log n}{n^{3/2}} \cdot \frac{1 + g_{F,A}(t)}{g_{F,A}(t)}.$$
Moreover, set $\lambda = C \cdot c(F,A)$, $\delta = \eps$ and $h(t) = t \cdot n^{-3/2}$, and note that $\alpha$ and $\beta$ are $\lambda$-slow on $[a,b]$, and that $\alpha(t) \le \eps$ and\footnote{Here we use Lemma~\ref{obs:deltaFdeltaFv} together with~\eqref{gamma:assumption} to show that $\alpha(t) \ge n^{-\eps} g_{F,A}(t) \ge n^{-1/4 - \eps}$.}
$$\min\big\{ \alpha(t), \, \beta(t), \, h(t) \big\} \, \ge \, \ds\frac{\eps t}{n^{3/2}}$$ 
for every $\omega < t \le t_A(F)$. Hence $(\lambda,\delta;g_{F,A},h;\alpha,\beta;\K)$ is a reasonable collection.

We claim that $N_\phi(F)$ satisfies the conditions of Lemma~\ref{lem:self:mart} as long as $\phi$ is faithful. Indeed, Lemma~\ref{selfN*} implies that $N_\phi(F)$ is $(g_{F,A},h;\K)$-self-correcting, since $c(F,A) \ge 2$, and Lemmas~\ref{gammaNF} and~\ref{maxalphabeta} imply that  
$$|\Delta N^*_\phi(F)(m)| \le \alpha(t) \qquad  \text{and} \qquad \Ex\big[ |\Delta N^*_\phi(F)(m)| \big] \le \beta(t)$$
for every $\omega < t \le t_A(F)$ for which $\K(m)$ holds. Moreover, the bound $|N^*_\phi(F)(a)| < 1/2$ follows from the event $\E(a)$, since $f_{F,A}(\omega) \Nt_A(F)(n^{3/2}) \ll g_{F,A}(\omega) \Nt_A(F)(a)$.

Finally, observe that, since $e(F) + o(F) \ll \log n$, we have 
$$\alpha(t) \beta(t) n^{3/2} \, \le \, (\log n)^2 \cdot (\log n)^{- \sqrt{\Delta(F,A)}} \, \le \, \frac{1}{(\log n)^5}$$
for every $\omega < t \le t_A(F)$. By Lemma~\ref{lem:self:mart}, it follows that
$$\Pr\Big( \big( |N^*_\phi(F)(m')| > 1 \big) \cap \K(m'-1) \Big) \, \le \, n^4 \cdot e^{- (\log n)^4} \, \le \, n^{-(\log n)^2}.$$
Summing over our choices\footnote{Once again, by our assumed bound on $v(F) + e(F) + o(F)$, the number of choices is negligible.} for $(F,A,\phi)$ and $m'$, we obtain~\eqref{eq:lem:Ebm}, as required.
\end{proof}

Finally, let's use the building sequences introduced in Section~\ref{BS} to show that an unbalanced non-tracking graph structure cannot be the first to go astray.

\begin{lemma}\label{lem:Edm}
For every $m \in [m^*]$,
$$\Pr\big( \E_4(m) \big) = 0.$$
\end{lemma}

\begin{proof}
Let $m \in [m^*]$, and let $(F,A)$ be an unbalanced graph structure pair with $t > t_A(F)$. We claim that if $\big( \E_1(m) \cup \E_2(m) \cup \E_3(m) \big)^c$ and $\E(m - 1) \cap \M(m-1) \cap \K^\E(m-1)$ both hold, and if $\phi \colon A \to V(G_m)$ is faithful at time $t$, then
\begin{equation}\label{eq:theorem41c}
N_\phi(F)(m) \, \le \, (\log n)^{\Delta(F,H,A)} \Nt_H(F)(m^+),
\end{equation}
where $A \subsetneq H \subseteq F$ is minimal such that $t < t_H(F)$, and $m^+ = \max\{m,n^{3/2}\}$. Since $(F,A,\phi)$ was arbitrary, this implies immediately that $\Pr\big( \E_4(m) \big) = 0$, as required.

In order to prove~\eqref{eq:theorem41c}, let the building sequence of $(F,A)$ be
$$A \subseteq H_0 \subsetneq \dots \subsetneq H_\ell = F,$$
and recall that $t_i = t^*_{H_{i-1}}(H_{i})$ for each $i \in [\ell]$, and that $0 = t_0 < t_1 < \cdots < t_\ell \le \infty$, by Lemma~\ref{lem:crashtimes}. Suppose that $t_j \le t < t_{j+1}$, and recall that $H_j$ is the minimal $A \subseteq H \subseteq F$ such that $t < t_H(F)$, by Lemma~\ref{lem:Hj}. Recall also that each pair $(H_0,A)$ and $(H_{i+1},H_i)$ is balanced, by Lemma~\ref{lem:HHbalanced}. Thus, since $\E(m - 1) \cap \K^\E(m-1) \cap \E_3(m)^c$ holds, we have
$$N_{\phi}(H_0)(m) \, \le \, (\log n)^{\Delta(H_0,A)} \qquad \text{and} \qquad N_{\phi_i}(H_{i+1})(m) \, \le \, (\log n)^{\Delta(H_{i+1},H_i)}$$
for every $0 \le i \le j-1$, and every map $\phi_i : \, H_i \to V(G_m)$ which is faithful at time $t$. Since 
$$N_\phi(F) \, \le \, N_{\phi}(H_0) \cdot \bigg( \prod_{i = 0}^{j-1} \max_{\phi_i : \, H_i \to V(G_m)} N_{\phi_i}(H_{i+1}) \bigg) \cdot \max_{\phi_j : \, H_j \to V(G_m)} N_{\phi_j}(F)$$
by Lemma~\ref{lem:crucial}, where the maxima are over faithful maps $\phi_i$, and
$$\Delta(H_0,A) + \sum_{i = 0}^{j-1} \Delta(H_{i+1},H_i) \, \le \, \Delta(H_j,A),$$
by the definition~\eqref{def:del} of $\Delta(F,A)$ and the convexity of the function $x \mapsto x^C$, it follows that
\begin{equation}\label{eq:thm41c:nontracking}
N_\phi(F)(m) \, \le \, (\log n)^{\Delta(H_j,A)} \cdot \max_{\phi_j : \, H_j \to V(G_m)} N_{\phi_j}(F)(m).
\end{equation}
Finally, we claim that, 
\begin{equation}\label{eq:thm41c:tracking}
N_{\phi_j}(F)(m) \, \le \, (\log n)^{\Delta(F,H_j)} \Nt_{H_j}(F)(m^+).
\end{equation}
To prove~\eqref{eq:thm41c:tracking}, suppose first that $t \le \omega$. Then, since $\E(m - 1) \cap \K^\E(m-1) \cap \E_1(m)^c$ holds and $t \le t_{j+1} = t_{H_j}^*(F)$, we have\footnote{To see this, simply note that $\Nt_{H_j}(F)(m) + \Nt_{H_j}(F)(n^{3/2}) \le \big( e^{4 \omega^2 o(F)} + 1 \big) \Nt_{H_j}(F)(m^+)$ for every $t \le \omega$.}
$$N_{\phi_j}(F)(m) \, \le \, \Nt_{H_j}(F)(m) + f_{F,H_j}(t) \Nt_{H_j}(F)(n^{3/2}) \, \le \, (\log n)^{\Delta(F,H_j) - \sqrt{\Delta(F,H_j)}} \Nt_{H_j}(F)(m^+).$$
On the other hand, if $t > \omega$ then, since $\E(m - 1) \cap \M(m - 1) \cap \K^\E(m-1) \cap \E_2(m)^c$ holds and $t \le t_{j+1} = t_{H_j}^*(F)$, we have
$$N_{\phi_j}(F)(m) \, \le \, \big( 1 + g_{F,H_j}(t) \big) \Nt_{H_j}(F)(m) \, \le \, (\log n)^{\Delta(F,H_j)} \Nt_{H_j}(F)(m^+),$$
as claimed. Combining~\eqref{eq:thm41c:nontracking} and~\eqref{eq:thm41c:tracking}, we obtain 
$$N_\phi(F)(m) \, \le \, (\log n)^{\Delta(H_j,A) + \Delta(F,H_j)} \Nt_{H_j}(F)(m^+) \, = \, (\log n)^{\Delta(F,H_j,A)} \Nt_{H_j}(F)(m^+),$$
as required.
\end{proof}

\begin{proof}[Proof of Theorem~\ref{EEthm}]
Since, as noted above, we have
$$\bigcup_{m = 1}^{m^*} \E(m)^c \cap \K^\E(m-1) \, \subseteq \, \bigcup_{m = 1}^{m^*} \E_1(m) \cup \E_2(m) \cup \E_3(m) \cup \E_4(m),$$
the theorem follows immediately from Lemmas~\ref{lem:Eam},~\ref{lem:Ebm} and~\ref{lem:Edm}.
\end{proof}

\section{Tracking $Y_e$, and mixing in the $Y$-graph}\label{Ysec}

In this section we shall show how to use the events $\E(m)$ and $\X(m)$ (see Definitions~\ref{def:events:XYQ} and~\ref{def:events:E}) to track $Y_e$ for each open edge $e$. Set $a = \omega \cdot n^{3/2}$ and 
$$\K^\Y(m) \, = \, \E(m) \cap \X(m) \cap \Y(a) \cap \Z(m) \cap \Q(m)$$
for each $m \in [m^*]$, and recall that $\Yt(m) = 4t \sqrt{n} e^{-4t^2}$, that $g_y(t) \, = \, e^{2t^2} n^{-1/4} (\log n)^4$ and that $n \ge n_0(\eps,C,\omega)$ is chosen sufficiently large. The main aim of this section is to prove the following key proposition, cf.~Theorem~\ref{Ythm}.

\begin{prop}\label{Yprop}
Let $\omega \cdot n^{3/2} < m \le m^*$. With probability at least $1 - n^{-C\log n}$, either the event $\K^\Y(m-1)^c$ holds, or 
\begin{equation}\label{eq:Yprop}
Y_e(m) \, \in \, \big( 1 \pm g_y(t) \big) \Yt(m)
\end{equation}
for every open edge $e \in O(G_m)$. 
\end{prop}

We begin by defining a collection of variables $V_e^\sigma$ which are averages of the variables $Y_f$ over various multi-sets of edges. Given two open edges $e,f \in O(G_m)$, recall that $e$ and $f$ are said to be \emph{$Y$-neighbours} in $G_m$ if $f \in Y_e(m)$ and (equivalently) $e \in Y_f(m)$. 

\begin{defn}[The $Y$-graph]
Given the graph $G_m$ at time $t$, define the \emph{$Y$-graph} of $G_m$ to be the graph with vertex set $O(G_m)$, whose edges are pairs of $Y$-neighbours in $G_m$.
\end{defn}

Note that when we take a step in the $Y$-graph, one of our endpoints remains the same, and the other changes. We shall therefore imagine ourselves walking through the $Y$-graph, with a foot on each endpoint of the currently occupied open edge, and a step being taken either with the left or the right foot. Let us call the sequence of left / rights the \emph{type} of a walk; we shall need to differentiate between paths in the $Y$-graph of different types.  

Let $k := \lceil 3 / \eps \rceil$ be a constant fixed throughout the proof. We say that a sequence $\sigma = (\sigma_1,\ldots,\sigma_\ell)$ of lefts and rights\footnote{We shall typically use the letter $\ell$ to denote the length $|\sigma|$ of such a sequence $\sigma$.} is \emph{$k$-short} if every string of consecutive lefts or rights has length at most $k$, and there are at most $k$ `changes of foot', i.e., 
$$\sigma_{i+1} = \ldots = \sigma_{i+j} \;\Rightarrow \; j \le k \qquad \text{and} \qquad s(\sigma) :=  \big| \big\{ i \in [\ell-1] : \sigma_i \neq \sigma_{i+1} \big\} \big | \le k.$$ 
Note that if $\sigma$ is $k$-short then $|\sigma| \le k(k+1)$. The idea behind this definition is that if we take more than $k$ steps in a row with the same foot, then we will be very well `mixed' in the open neighbourhood of our planted foot, while if we alternate feet more than $k$ times, we will be well mixed in the whole graph. 

Before continuing, let us make an \textbf{important remark}. In this section, we shall write $e$ to denote an open edge of $G_m$ \emph{with an orientation}; that is, we assign the symbols $L$ and $R$ to the endpoints of $e$. This orientation\footnote{In particular, we emphasize that the same edge may have different orientations at different points in the proof, and even in the same $\sigma$-walk. In other words, we denote by $e$ an (edge, orientation)-pair.} will be inherited by the $Y$-neighbours of $e$ in the obvious way: if we write $f \in Y_e(m)$, then the label of the vertex $v \in e \cap f$ (as an endpoint of $f$) is the same as its label as an endpoint of $e$. Moreover, let us write 
$$Y^L_e(m) \, = \, \Big\{ f \in Y_e(m) \,:\, \textup{the vertex $v \in e \setminus f$ has label $L$} \Big\},$$ 
and set $Y^R_e(m)  = Y_e(m) \setminus Y^L_e(m)$, and define $X_e^L(m)$ and $X_e^R(m)$ similarly.\footnote{So, for example, $X^L_e(m) \, = \, \big\{ f \in X_e(m) : \textup{the vertex $v \in e \setminus f$ has label $L$} \big\}$. As noted above, if $f \in X_e(m) \cup Y_e(m)$ then the vertex $v \in e \setminus f$ has label $L$ if and only if the vertex $v' \in f \setminus e$ has label $L$.}

Given open edges $e_1,e_2 \in O(G_m)$ (with orientations) and a sequence $\sigma \in \{L,R\}^*$, where $\{L,R\}^*$ denotes the set of finite strings of lefts and rights, we shall say that a sequence $W$ of open edges in $G_m$ is a \emph{$\sigma$-walk} from $e_1$ to $e_2$ if the following hold:
$$W = \big(f_0, f_1,\ldots,f_{|\sigma|}\big), \quad f_0 = e_1, \quad f_{|\sigma|} = e_2 \quad \text{and} \quad f_{i} \in Y^{\sigma_i}_{f_{i-1}}(m) \text{ for every } 1 \le i \le |\sigma|.$$
That is, $f_i$ and $f_{i-1}$ are $Y$-neighbours in $G_m$, and the vertex $v \in f_i \setminus f_{i-1}$ has label $\sigma_i$.

We can now define the random variables which we shall need to track.

\begin{defn}\label{def:UandV}
Let $\sigma \in \{L,R\}^*$, let $m \in [m^*]$ and let $e \in O(G_m)$. We shall write
$$U^{\sigma}_e(m) \,=\, \sum_{f_1 \in Y^{\sigma_1}_e(m)} \sum_{f_2 \in Y^{\sigma_2}_{f_1}(m)} \dots \sum_{f_\ell \in Y^{\sigma_\ell}_{f_{\ell-1}}(m)} 1,$$
to denote the number\footnote{Later, we shall also write $U_e^\sigma(m)$ to denote the multi-set of open edges reached via a $\sigma$-walk from $e$.} of $\sigma$-walks starting from $e$, and\footnote{As usual, if $e \not\in O(G_m)$ then we set $U_e^\sigma(m) = U_e^\sigma(m-1)$ and $V_e^\sigma(m) = V_e^\sigma(m-1)$.}
$$V_e^\sigma(m) \, = \, \frac{1}{U_e^\sigma(m)} \sum_{f_1 \in Y^{\sigma_1}_e(m)} \sum_{f_2 \in Y^{\sigma_2}_{f_1}(m)} \dots \sum_{f_\ell \in Y^{\sigma_\ell}_{f_{\ell-1}}(m)} Y_{f_\ell}(m),$$
for the average of the $Y$-values reached via such walks. 
\end{defn}

In particular, we note that if $|\sigma| = 0$ then the unique $\sigma$-path starting from $e$ is $W = (e)$, and hence  $U^{\sigma}_e(m) = 1$ and $V^\sigma_e(m) = Y_e(m)$. For each $\sigma \in \{L,R\}^*$, set
$$g_\sigma(t) =\eps^{|\sigma|} g_y(t).$$
We emphasize that this error bound improves exponentially as $|\sigma|$ increases; this property will be crucial in the proof below.

We shall prove the following generalization of Proposition~\ref{Yprop}. It says that, for any $k$-short sequence $\sigma$, with (very) high probability $V^\sigma_e$ is tracking up to time $t^*$.

\begin{prop}\label{Vprop}
Let $\omega \cdot n^{3/2} < m \le m^*$. Then, with probability at least $1 - n^{-C \log n}$, either $\K^\Y(m-1)^c$ holds, or  
\begin{equation}\label{eq:Vprop}
V^\sigma_e(m) \, \in \, \big( 1 \pm g_\sigma(t) \big) \Yt(m) 
\end{equation}
for every $e \in O(G_m)$, and every $k$-short sequence $\sigma \in \{L,R\}^*$.
\end{prop}

In order to prove Proposition~\ref{Vprop}, we shall need the following bounds on $U_e^\sigma(m)$, which may easily be proved using the method of Section~\ref{SecX}, see the Appendix for the details.

\begin{prop}\label{Uprop}
Let $\omega \cdot n^{3/2} < m \le m^*$. Then, with probability at least $1 - n^{-C\log n}$, either $\K^\Y(m-1)^c$ holds, or 
\begin{equation}\label{eq:YL}
Y^L_e(m) \, \in \, \big( 1 \pm g_x(t) \big) \cdot \big( 2t e^{-4t^2} \sqrt{n}  \big)
\end{equation}
for every $e \in O(G_m)$, and hence 
\begin{equation}\label{eq:Uprop}
U^\sigma_e(m) \, \in \, \big( 1 \pm g_x(t) \big)^{|\sigma|} \cdot \big( 2t e^{-4t^2} \sqrt{n} \big)^{|\sigma|}
\end{equation}
for every $e \in O(G_m)$, and every sequence $\sigma \in \{L,R\}^*$.
\end{prop}

As in Section~\ref{EEsec}, we'll often need to assume that the bounds above hold at all earlier times, so for each $0 \le m' \le m^*$, define events $\U(m')$ and $\V(m')$ as follows.

\begin{defn}\label{def:events:UV}
$\U(m')$ and $\V(m')$ are the events that that the inequalities~\eqref{eq:Uprop} (resp.~\eqref{eq:Vprop}) hold for every $\omega \cdot n^{3/2} < m \le m'$, every $e \in O(G_m)$ and every $k$-short sequence $\sigma \in \{L,R\}^*$. 
\end{defn}

Recall that $V_e^\emptyset(m) = Y_e(m) = U_e^L(m) + U_e^R(m)$, and note that therefore Proposition~\ref{Vprop} immediately implies Proposition~\ref{Yprop}. It will also be convenient in this section to write $\Yh(m')$ for the event that~\eqref{eq:Yprop} holds for every $\omega \cdot n^{3/2} < m \le m'$; this is because the event $\V(m)$ implies $\Yh(m)$ (but \emph{not} $\Y(m)$, since it does not control $Y_e(m)$ when $t \le \omega$), and we shall need to assume that $\Yh(m)$ holds in several of the lemmas below. 

The proof of Proposition~\ref{Vprop} is roughly as follows. First we show that by taking $k$ steps with the left foot (say), we `mix well' in the open neighbourhood of our right foot. More precisely, writing $Q_u(m)$ for the collection of open edges incident to vertex $u$, and assuming that our right foot is fixed at $u$, we shall prove that $V_e^{L^k}(m)$ (that is, $V_e^{\sigma}(m)$ for the sequence $\sigma = (L,\ldots,L)$ of length $k$) is within a factor of $1 \pm o\big( g_y(t) \big)$ of the average of $Y_f(m)$ over $f \in Q_u(m)$. We will also prove a similar bound on $V_e^{\sigma}(m)$ for any (bounded length) sequence $\sigma$ which `changes foot' at least $k$ times. The only difference is that in this case we mix well in the entire $Y$-graph, i.e., we will prove that $V_e^{\sigma}(m)$ is within a factor of $1 \pm o\big( g_y(t) \big)$ of $\Yb(m)$. Both mixing results follow from the event $\E(m)$ (see Theorem~\ref{EEthm}), applied to the graph structures which correspond to $\sigma$-walks from $e$ to $f$.  

We next control the one-step changes in $V_e^\sigma$ for each $k$-short sequence $\sigma$, and show that the variable $V_e^\sigma$ is self-correcting, assuming that none of the other variables have yet gone astray. The mixing results above are a crucial tool in this calculation, since the rate of change of $V_e^\sigma$ depends on $V_e^{\sigma L}$ and $V_e^{\sigma R}$, and these longer sequences\footnote{If $\sigma = (\sigma_1,\ldots,\sigma_\ell)$ then $\sigma L = (\sigma_1,\ldots,\sigma_\ell,L)$ and $\sigma R = (\sigma_1,\ldots,\sigma_\ell,R)$.} may not be $k$-short. Finally, we apply the method of Section~\ref{MartSec} to bound, for each $k$-short sequence $\sigma \in \{L,R\}^*$, the probability that $V_e^\sigma$ is the first variable to cross its Line of Death.

\subsection{Mixing inside open neighbourhoods}\label{MixSec1}

Our aim over the next few subsections is to show that $V_e^\sigma$ is self-correcting; we do so by writing $\Ex\big[ \Delta V_e^\sigma(m) \big]$ in terms of $V_e^{\sigma L}(m)$ and $V_e^{\sigma R}(m)$, which we can control as long as both $\sigma L$ and $\sigma R$ are also $k$-short. In the next two subsections we shall deal with the other case. 

The key lemma of this subsection is as follows; it says (roughly) that a walk of length $k$ in the open neighbourhood of a vertex is well-mixed.

\begin{lemma}\label{Vmix}
Let $\omega \cdot n^{3/2} < m \le m^*$. If $\E(m) \cap \Yh(m)$ holds, then  
\begin{equation}\label{eq:Vmix} 
\big| V_e^{L^k}(m) - V_e^{L^{k+1}}(m) \big| \, = \, o\big( g_y(t) \Yt(m) \big)
\end{equation}
for every open edge $e \in O(G_m)$.  
\end{lemma}

Note that the $o(\cdot)$-term in the lemma indicates that the left-hand side of~\eqref{eq:Vmix} divided by $g_y(t) \Yt(m)$ is (uniformly over times $\omega < t \le t^*$) at most some function of $n$ which tends to zero as $n \to \infty$. In particular, this implies that if $n \ge n_0(\eps,C,\omega)$, then the left-hand side of~\eqref{eq:Vmix} is smaller than $g_\sigma(t) \Yt(m)$ for every $k$-short sequence $\sigma$. 

We shall deduce Lemma~\ref{Vmix} from the following lemma. Recall that $Q_u(m)$ denotes the open neighbourhood of the vertex $u$ in $G_m$, i.e., $Q_u(m) = \{ e \in O(G_m) : u \in e \}$. 

\begin{lemma}\label{Vmix2}
Let $\omega \cdot n^{3/2} < m \le m^*$. If $\E(m) \cap \Yh(m)$ holds, then  
$$V_e^{L^k}(m) \, \in \, \frac{1 \pm o\big( g_y(t) \big)}{Q_u(m)} \sum_{f \in Q_u(m)} Y_f(m)$$
for every open edge $e \in Q_u(m)$. 
\end{lemma}

In order to prove Lemma~\ref{Vmix2}, we shall use the following simple facts.\footnote{In both observations, all variables are assumed to be positive real numbers.} 

\begin{obs}\label{avobs1}
Let $a_1,\ldots,a_{r+s} \in (1 \pm \gamma)\tilde{a}$, and suppose that $s \le \delta r$. Then
$$\frac{r+s}{r} \sum_{j = 1}^r a_j \, \in \, \big( 1 \pm O(\gamma \delta) \big) \sum_{j=1}^{r+s} a_j.$$
\end{obs}

\begin{obs}\label{avobs2}
Let $a_1,\ldots,a_r \in (1 \pm \gamma)\tilde{a}$ and $b_1,\ldots,b_r \in (1 \pm \delta)\tilde{b}$. Then
$$r \cdot \sum_{j = 1}^r a_j b_j \, \in \, \big( 1 \pm O(\gamma \delta ) \big) \bigg( \sum_{j=1}^r a_j \bigg) \bigg( \sum_{j=1}^r b_j \bigg).$$
\end{obs}

We shall also need a little extra terminology. 

\begin{defn}[$\sigma$-paths and $\sigma$-cycles]
Let $\sigma \in \{L,R\}^*$ and $e,f \in O(G_m)$. A $\sigma$-walk from $e$ to $f$ which leaves each endpoint of $e$ at most once, and arrives at every other vertex at most once, is called a \emph{$\sigma$-path}. If moreover $e = f$, then it is also called a \emph{$\sigma$-cycle}. 
\end{defn}

Let us write $P_e^\sigma(m)$ for the number of $\sigma$-paths starting at an open edge $e$, and recall that $U_e^\sigma(m)$ denotes the number of $\sigma$-walks starting at $e$.  We shall use the event $\E(m)$ to bound the number of $\sigma$-paths (or cycles) between two edges $e,f \in Q_u(m)$. 

Finally, we make an observation which will be useful several times in this section. 

\begin{obs}\label{obs:easyUbounds}
Let $\omega \cdot n^{3/2} < m \le m^*$. If $\E(m)$ holds, then
$$\frac{1}{3^\omega} \cdot \Yt(m)^{|\sigma|} \, \le \, U_e^\sigma(m) \, \le \, \Yt(m)^{|\sigma|}$$
for every $\sigma \in \{L,R\}^*$ with $|\sigma| \le \omega$.
\end{obs}

\begin{proof}
Simply note that, since the event $\E(m)$ holds, the variables $Y_f^L(m)$ and $Y_f^R(m)$ are each at most $\Yt(m)$, and at least $\Yt(m)/3$, for every $f \in O(G_m)$. 
\end{proof}

\begin{proof}[Proof of Lemma~\ref{Vmix2}]
Let $\omega \cdot n^{3/2} < m \le m^*$, recall that $k = \lceil 3 / \eps \rceil$, and set $\sigma = L^k$. Moreover, let $u \in V(G_m)$ and $e \in Q_u(m)$. We claim that 
\begin{itemize}
\item[$(a)$] the number of $\sigma$-paths from $e$ to $f$ is the same, up to a factor of $1 \pm o(1)$, for each open edge $f \in Q_u(m)$, and 
\item[$(b)$] almost all $\sigma$-walks starting at $e$ are in fact $\sigma$-paths. 
\end{itemize}
Using Observations~\ref{avobs1} and~\ref{avobs2}, the lemma will follow easily from $(a)$ and $(b)$.

Suppose first that $e \neq f$, and consider the graph structure triple $(F,A,\phi)$ which represents a $\sigma$-path from $e$ to $f$, see Figure~5.1$(a)$. Note that $v(F) = k + 2$, $|A| = |e \cup f| = 3$, $e(F) = k$ and $o(F) = k-1$. We claim that $t_A(F) = t^*$. 

\begin{figure}[h]
\includegraphics[scale=1]{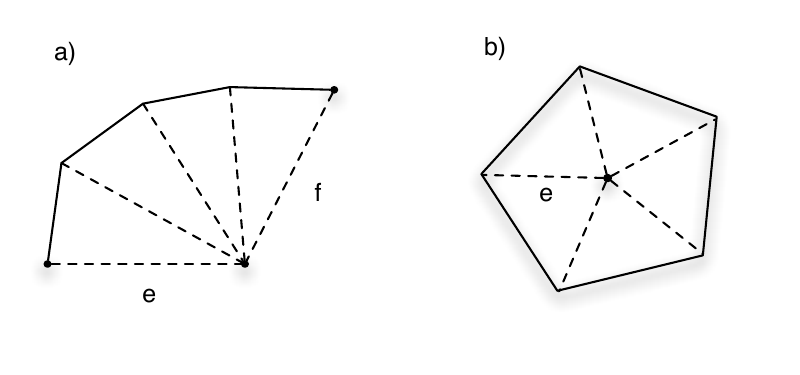}
\caption{$L^k$-paths from $e$ to $f$.}
\end{figure}

To prove the claim, recall from Definition~\ref{def:tAF} and~\eqref{eq:tstardef} that if $t_A(F) < t^*$, then there exists an induced sub-structure $A \subsetneq H \subseteq F$ such that either $e(H) \ge 2v_A(H)$ or 
\begin{equation}\label{eq:tstardefSec5}
t_A^*(H) \, = \, \left( \frac{2v_A(H) - e(H)}{8o(H)} \right)^{1/2} \sqrt{\log n} \, < \, \left( \frac{1}{2\sqrt{2}} - \eps \right) \sqrt{\log n} \, = \, t^*.
\end{equation}
Note that $e(H) \le v_A(H) + \one_{H = F}$ and $o(H) = v_A(H)$. Since $v_A(F) = k - 1 > 1/\eps > 1$, it follows that $e(H) < 2v_A(H)$ and
$$\frac{2v_A(H) - e(H)}{8o(H)} \, \ge \, \frac{1}{8} \left( 1 - \frac{1}{k - 1} \right) \, > \, \left( \frac{1}{2\sqrt{2}} - \eps \right)^2,$$
and so $t_A(F) = t^*$, as claimed. Hence, assuming $\E(m)$ holds, and noting that $\phi$ is faithful, and that $\gamma(F,A) \le \omega$ (since $v_A(F) \le k = O(1)$, see~\eqref{def:del} and~\eqref{def:gam}), we have  
\begin{equation}\label{eq:mixfan}
N_\phi(F) \, \in \, \big( 1 \pm o(1) \big) \Nt_A(F),
\end{equation}
and so $g_{F,A}(t) \le n^{-\eps}$ for every $\omega < t \le t^*$. 

On the other hand, if $e = f$ then consider the graph structure triple $(F',A',\phi')$ which represents a $\sigma$-cycle from $e$ to itself, see Figure~5.1$(b)$. Since $k > 2$, the structure $(F',A')$ is obtained from $(F,A)$ by identifying two vertices of $A$ with no common neighbours, and it follows that $\Nt_{A'}(F') = \Nt_A(F)$, and that 
$$N_{\phi'}(F') \, \in \, \big( 1 \pm o(1) \big) \Nt_{A'}(F')  \, \in \, \big( 1 \pm o(1) \big) \Nt_A(F),$$ 
exactly as above. 

It follows that the number of $\sigma$-paths from $e$ to $f$ is within a factor of $1 \pm o(1)$ of $\Nt_A(F)$ for every $f \in Q_u(m)$. It remains to observe that, assuming $\E(m)$ holds, 
\begin{equation}\label{eq:UP}
U_e^\sigma(m) - P_e^\sigma(m) \, < \, \omega \cdot \Yt(m)^{k-1} \, \ll \, \frac{\Yt(m)^k}{3^{\omega}} \, \le \, U_e^\sigma(m).
\end{equation}
To see the first inequality, simply sum over the (less than) $k^2 \le \omega^2$ choices for the two steps $1 \le i < j \le k$ which use the same vertex, and note that, since the event $\E(m)$ holds, we have at most $\big( 1 + o(1) \big) \frac{\Yt(m)}{2} < \Yt(m)$ choices\footnote{We shall use this bound, $Y_e^L(m) \le \Yt(m)$, several times below without further comment.} for each of the steps except $j$, and at most one choice for step $j$ (given step $i$). The second inequality holds since $\omega^2 \cdot 3^\omega \ll n^\eps \le \Yt(m)$, and the third holds by Observation~\ref{obs:easyUbounds}. 

Hence, writing $\hat{V}_e^{L^k}(m)$ for the average of $Y_f(m)$ over the multi-set of open edges $f$ reached by $\sigma$-paths from $e$, we claim that 
$$V_e^{L^k}(m) \, \in \, \big( 1 \pm o(g_y(t)) \big) \cdot \hat{V}_e^{L^k}(m) \, \subseteq \, \frac{1 \pm o\big( g_y(t) \big)}{Q_u(m)} \sum_{f \in Q_u(m)} Y_f(m),$$
as required. Indeed, the first inclusion follows from Observation~\ref{avobs1}, using~\eqref{eq:UP} and the event $\Yh(m)$, by enumerating the $\sigma$-walks from $e$, and setting $a_j = Y_f(m)$ if the $j^{th}$ walk ends at edge $f$. The second inclusion follows by Observation~\ref{avobs2} and the event $\Yh(m)$, by setting (for each $f \in Q_u(m)$) $a_f$ equal to the number of $\sigma$-walks from $e$ to $f$, and $b_f = Y_f(m)$. 
\end{proof}

Noting that Lemma~\ref{Vmix2} also holds for $\sigma = L^{k+1}$, it is now easy to deduce Lemma~\ref{Vmix}.

\begin{proof}[Proof of Lemma~\ref{Vmix}]
By Lemma~\ref{Vmix2} we have 
$$\big| V_e^{L^k}(m) - V_e^{L^{k+1}}(m) \big| \, \le \, \frac{o\big( g_y(t) \big)}{Q_u(m)} \sum_{f \in Q_u(m)} Y_f(m).$$
Now $Y_f(m) \in \big( 1 \pm o(1) \big) \Yt(m)$ for every $f \in Q_u(m)$, by $\Yh(m)$, so the lemma follows.
\end{proof}

\subsection{Mixing in the whole $Y$-graph}\label{MixSec2}

The other way in which $\sigma L$ or $\sigma R$ might not be $k$-short is if $\sigma$ changes foot more than $k$ times. The following lemma follows from the fact that such a $\sigma$-walk `mixes fast' in the entire $Y$-graph. Recall that we write $s(\sigma) = |\{ i : \sigma_i \neq \sigma_{i+1}\}|$ for the number of `changes of foot' during a $\sigma$-walk. 

\begin{lemma}\label{VmixQ}
Let $\omega \cdot n^{3/2} < m \le m^*$, and let $\sigma \in \{L,R\}^*$ be a sequence with $|\sigma| \le \omega$ and $s(\sigma) \ge k$. If $\E(m) \cap \Yh(m)$ holds, then  
$$\big| V_e^{\sigma}(m) - V_e^{\sigma L}(m) \big| \, = \, o\big( g_y(t) \Yt(m) \big)$$
for every open edge $e \in O(G_m)$. 
\end{lemma}

We shall deduce Lemma~\ref{VmixQ} from the following lemma. 

\begin{lemma}\label{VmixQ2}
Let $\omega \cdot n^{3/2} < m \le m^*$, and let $\sigma \in \{L,R\}^*$ be a sequence with $|\sigma| \le \omega$ and $s(\sigma) \ge k$. If $\E(m) \cap \Yh(m)$ holds, then  
$$V_e^{\sigma}(m) \, \in \, \big( 1 \pm o\big( g_y(t) \big) \big) \Yb(m)$$
for every open edge $e \in O(G_m)$. 
\end{lemma}

\begin{proof}
The proof is similar to that of Lemma~\ref{Vmix2}, above. Indeed, let $\omega \cdot n^{3/2} < m \le m^*$, let $e \in O(G_m)$, and let $\sigma \in \{L,R\}^*$ be a sequence with $|\sigma| \le \omega$ and $s(\sigma) \ge k$. We claim that 
\begin{itemize}
\item[$(a)$] the number of $\sigma$-paths from $e$ to $f$ is the same, up to a factor of $1 \pm o(1)$, for each open edge $f \in O(G_m)$, and 
\item[$(b)$] almost all $\sigma$-walks starting at $e$ are in fact $\sigma$-paths. 
\end{itemize}
The second claim again follows by~\eqref{eq:UP}, so we shall need to prove only the first.

Suppose first that $e$ and $f$ are disjoint, and consider the graph structure triple $(F,A,\phi)$ which corresponds to a $\sigma$-path from $e$ to $f$, see Figure~5.2$(a)$. We claim that $t_A(F) = t^*$. 

\begin{figure}[h]
\includegraphics[scale=0.8] {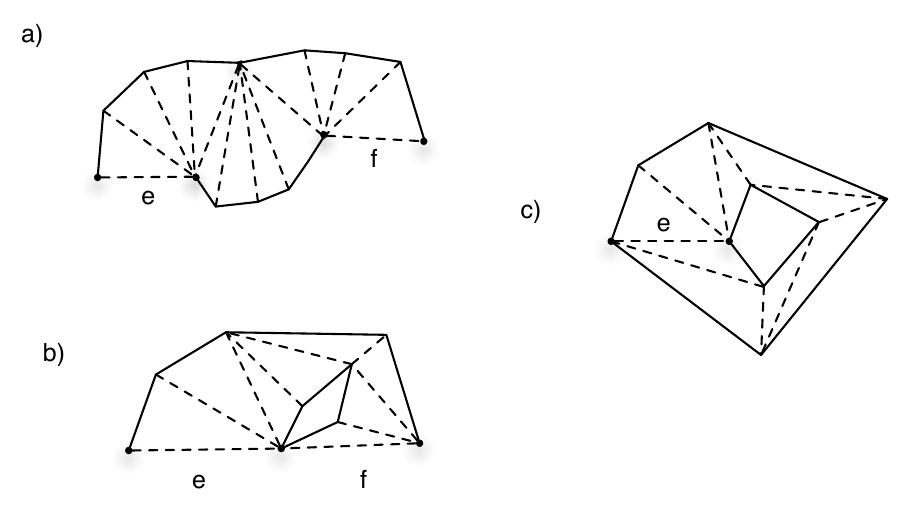}
\caption{$\sigma$-paths from $e$ to $f$.}
\end{figure}

To prove the claim, recall from the proof of Lemma~\ref{Vmix2} that if $t_A(F) < t^*$, then there exists an induced sub-structure $A \subsetneq H \subseteq F$ such that either $e(H) \ge 2v_A(H)$ or~\eqref{eq:tstardefSec5} holds. We claim first that if $e$ and $f$ are in the same component of $H$, i.e., that there is a path in $E(H) \cup O(H)$ from an endpoint of $e$ to an endpoint of $f$, then
$$v_A(H) \ge \frac{s(\sigma) - 3}{2} \ge \frac{k}{3}, \qquad e(H) \le v_A(H) + 2 \qquad \text{and} \qquad o(H) \le v_A(H) + 2.$$
To prove the bound on $v_A(H)$, simply note that each consecutive pair of `pivot' vertices in $F$ form a cut-set which divides $e$ from $f$, and so $H$ must contain at least one of each such pair. To prove the bounds on $e(H)$ and $o(H)$, consider the vertices of $H$ according to the order in which they are reached by the $\sigma$-walk. Each vertex sends at most one edge and one open edge backwards, and thus $e(H) \le v_A(H) + 2$ and $o(H) \le v_A(H) + 2$, as claimed\footnote{Note that the $+2$ term corresponds to the two endpoints of $f$.}. It follows that $e(H) < 2v_A(H)$, since $v_A(H) \ge k/3 > 2$, and moreover
\begin{equation}\label{eq:VMixQ:tstarH}
\frac{2v_A(H) - e(H)}{8o(H)} \, \ge \, \frac{1}{8} \left( \frac{v_A(H) - 2}{v_A(H) + 2} \right) \, \ge \, \frac{1}{8} \left( 1 - \frac{12}{k + 6} \right) \, > \, \left( \frac{1}{2\sqrt{2}} - \eps \right)^2
\end{equation}
for all such $H$, since $k + 6 > 3 / \eps$. 

On the other hand, suppose that $e$ and $f$ are in different components of $H$. Then we may apply exactly the same argument to each component, and obtain the stronger bounds $e(H) \le v_A(H)$ and $o(H) \le v_A(H)$ (component-wise, and thus also in $H$). Thus $2v_A(H) - e(H) \ge o(H)$, and so~\eqref{eq:VMixQ:tstarH} still holds, i.e., $t_A^*(H) > t^*$. Hence $t_A(F) = t^*$, as claimed. 

Now, since $v_A(F) \le |\sigma| + 2 \le \omega + 2$, it follows from the definition~\eqref{def:gam} (and the fact that $\omega \to \infty$ sufficiently slowly) that $\gamma(F,A) \le \log \log n$ (say), and so $g_{F,A}(t) \le n^{-\eps}$ for every $\omega < t \le t^*$. Since $\E(m)$ holds and $\phi$ is faithful at time $t$, it follows that  
\begin{equation}\label{eq:mixfan}
N_\phi(F)(m) \, \in \, \big( 1 \pm o(1) \big) \Nt_A(F)(m).
\end{equation}
The proof when $e$ and $f$ are not disjoint is similar. Indeed, consider the graph structure triples $(F',A',\phi')$ and $(F'',A'',\phi'')$ which represents a $\sigma$-path from $e$ to $f$ when $|e \cap f| = 1$ and when $e = f$, respectively, see Figure~5.2$(b)$ and~$(c)$. We claim that $t_{A'}(F') = t_{A''}(F'') = t^*$. Indeed, since $s(\sigma) > 2$, it follows that $(F',A')$ and $F'',A'')$ are obtained from $(F,A)$ by identifying vertices of $A$ with no common neighbours. Thus $\Nt_{A'}(F') = \Nt_{A''}(F'') = \Nt_A(F)$, and hence we obtain 
$$N_{\phi'}(F') \in \big( 1 \pm o(1) \big) \Nt_A(F) \qquad \text{and} \qquad N_{\phi''}(F'')  \in \big( 1 \pm o(1) \big) \Nt_A(F),$$
exactly as above. 

We have proved that the number of $\sigma$-paths from $e$ to $f$ is within a factor of $1 \pm o(1)$ of $\Nt_A(F)$ for every $f \in O(G_m)$.  Hence, writing $\hat{V}_e^{\sigma}(m)$ for the average of $Y_f(m)$ over the multi-set of open edges $f$ reached by $\sigma$-paths from $e$, we claim that 
$$V_e^{\sigma}(m) \, \in \, \big( 1 \pm o\big( g_y(t) \big) \big) \hat{V}_e^{\sigma}(m) \, \subseteq \,  \big( 1 \pm o\big( g_y(t) \big) \big)  \Yb(m),$$ 
as required. Indeed, the first inclusion again follows from Observation~\ref{avobs1}, using~\eqref{eq:UP} and the event $\Yh(m)$, and the second inclusion follows by Observation~\ref{avobs2} and the event $\Yh(m)$, exactly as before. 
\end{proof}

It is now easy to deduce Lemma~\ref{VmixQ}.

\begin{proof}[Proof of Lemma~\ref{VmixQ}]
By Lemma~\ref{VmixQ2}, we have 
$$\big| V_e^{\sigma}(m) - V_e^{\sigma L}(m) \big| \, \le \, o\big( g_y(t) \Yb(m) \big),$$
and since $\Yh(m)$ holds, we have $\Yb(m) \in \big( 1 \pm o(1) \big) \Yt(m)$. The lemma follows immediately.
\end{proof}

\subsection{Creating and destroying $\sigma$-walks}\label{VcountSec}

In this subsection we shall bound the maximum possible size of $| \Delta V_e^\sigma(m) |$; not only will we need this bound in the martingale calculation, but the walk-counting lemmas below will be useful in proving self-correction. 

The main result of the subsection is as follows.

\begin{lemma}\label{deltaV}
Let $\omega \cdot n^{3/2} < m \le m^*$. If $\E(m) \cap \Z(m)$ holds, then
$$\big| \Delta V_e^\sigma(m) \big| \, \le \, (\log n)^4$$
for every $e \in O(G_m)$ and every $\sigma \in \{L,R\}^*$ with $|\sigma| \le \omega$.
\end{lemma}

We begin with a straightforward observation, which will allow us to control $\Delta U_e^\sigma(m)$. 

\begin{obs}\label{obs:twostep}
Let $m \le m^*$. If $\Z(m)$ holds then, for every $e,f \in O(G_m)$ with $e \neq f$, there are at most $(\log n)^2$  walks of length two in the $Y$-graph from $e$ to $f$.
\end{obs}

\begin{proof}
If $e$ and $f$ are disjoint, then there are at most four walks of length two from $e$ to $f$, so assume that  $e = \{u,v\}$ and $f = \{u,w\}$. Then the first step must take us to an open edge $h = \{u,x\}$ with $\{v,x\} \in E(G_m)$ (since $h \in Y_e(m)$) and $\{w,x\} \in E(G_m)$ (since $h \in Y_f(m)$), i.e., $x$ is a common neighbour of $v$ and $w$. Since event $\Z(m)$ holds, there are at most $(\log n)^2$ such vertices, and so there are at most $(\log n)^2$ such walks, as claimed. 
\end{proof}

We can now bound the number of $\sigma$-walks between any two open edges of $G_m$.

\begin{lemma}\label{lem:countpaths}
Let $\omega \cdot n^{3/2} < m \le m^*$, let $e,f \in O(G_m)$ with $e \neq f$, and let $\sigma \in \{L,R\}^*$ with $|\sigma| \ge 2$. If $\E(m) \cap \Z(m)$ holds, then there are at most
$$2 \cdot (\log n)^2 \cdot \Yt(m)^{|\sigma| - 2}$$
$\sigma$-walks from $e$ to $f$.
\end{lemma}

\begin{proof}
Since $\E(m)$ holds, it follows that $Y_e^L(m) \le \Yt(m)$ for every $e \in O(G_m)$, and hence we have at most $\Yt(m)$ choices for each step of a $\sigma$-walk on $G_m$. Let $h$ be the open edge we have reached after $|\sigma| - 2$ steps, and consider two cases: either $h = f$, or $h \neq f$.

If $h = f$, then the bound is easy: we have only one choice for steps $|\sigma| - 2$ and $|\sigma|$ (since both must land on $f$), and so there are at most $\Yt(m)^{|\sigma| - 2}$ such walks. On the other hand, if $h \neq f$ then by Observation~\ref{obs:twostep} there are at most $(\log n)^2$ choices for the final two steps, since $\Z(m)$ holds. Hence there are at most $(\log n)^2 \cdot \Yt(m)^{|\sigma| - 2}$ such walks, as required.
\end{proof}

We can now easily bound $\big| \Delta U_e^\sigma(m) \big|$. 

\begin{lemma}\label{deltaU}
Let $\omega \cdot n^{3/2} < m \le m^*$. If $\E(m) \cap \Z(m)$ holds, then 
$$\big| \Delta U_e^\sigma(m) \big| \, \le \, |\sigma| \cdot (\log n)^3 \cdot \Yt(m)^{|\sigma| - 1}$$
for every $e \in O(G_m)$ and every $\sigma \in \{L,R\}^*$. 
\end{lemma}

\begin{proof}
Note first that if $|\sigma| = 0$ then the lemma is trivial, since $\Delta U_e^\sigma(m) = 0$. Moreover, if $|\sigma| = 1$ then we have $U_e^\sigma = Y_e^L$ or $U_e^\sigma = Y_e^R$, and it is easy to see (cf. Section~\ref{createsec}) that $- (\log n)^2 \le \Delta Y_e^L(m) \le 1$, since $\Z(m)$ holds, and so the lemma follows in this case also. Therefore, let us assume from now on that $|\sigma| \ge 2$.

Let $f$ denote the edge added in step $m+1$ of the triangle-free process, and note that if either $e = f$ or $f \in Y_e(m)$ then we are done, since\footnote{Recall from Definition~\ref{def:UandV} that if $e \not\in O(G_{m+1})$, then $U^{\sigma}_e(m+1) = U^{\sigma}_e(m)$.} in that case $\Delta U_e^\sigma(m) = 0$. So assume that $e \neq f$ and $e \not\in Y_f(m)$, and observe that a $\sigma$-walk is destroyed by the addition of $f$ only if it passes through a $Y$-neighbour of $f$, and is created only if it passes (in consecutive steps) through \emph{both} of a pair of $X$-neighbours of $f$. 

We claim that at most $|\sigma| \cdot \Yt(m)^{|\sigma| - 1}$ $\sigma$-walks from $e$ are created in step $m+1$, and that at most
\begin{equation}\label{eq:destroyedwalks}
\sum_{r = 1}^{|\sigma|} \Big( 4 \cdot (\log n)^2 \cdot \Yt(m)^{r - 1} \cdot \Yt(m)^{|\sigma| - r} \Big) \, \le \, |\sigma| \cdot (\log n)^3 \cdot \Yt(m)^{|\sigma| - 1}
\end{equation}
such walks are destroyed in the same step. To prove the former bound, simply note that one of the steps (the one between the $X$-neighbours of $f$) is pre-chosen and that, since $\E(m)$ holds, we have at most $\Yt(m)$ choices for each of the other steps. To prove that~\eqref{eq:destroyedwalks} is an upper bound on the number of destroyed walks, let $r$ denote the number of steps after which we first reach a $Y$-neighbour of $f$, and note that $r \ge 1$ since $e \not\in Y_f(m)$. Now, by Lemma~\ref{lem:countpaths}, since $e \neq f$ and $\E(m) \cap \Z(m)$ holds, we have at most $4 \cdot (\log n)^2 \cdot \Yt(m)^{r - 1}$ choices for the walk\footnote{Note that such walks, of length $r$,  are in bijection with the walks from $e$ to $f$ of length $r+1$.} from $e$ to some $Y$-neighbour of $f$. 

This proves the claimed bounds on the number of $\sigma$-walks created or destroyed in a single step. The lemma follows by taking the maximum of the two bounds.
\end{proof}

Putting these lemmas together, we obtain Lemma~\ref{deltaV}. To simplify the notation, we shall write $U_e^\sigma(m)$ to denote the multi-set of open edges reached via a $\sigma$-walk from $e$, as well as for the size of this multi-set. Note that, in this notation, if $|\sigma| = 0$ then $U_e^\sigma(m) = \{e\}$. 

\begin{proof}[Proof of Lemma~\ref{deltaV}]
Fix $\omega \cdot n^{3/2} < m \le m^*$, an open edge $e \in O(G_m)$ and a sequence $\sigma \in \{L,R\}^*$, and assume that $\E(m) \cap \Z(m)$ holds. The lemma holds if $V_e^\sigma = Y_e$, as noted above, so let us assume that $|\sigma| > 0$. We begin by bounding
$$\Big| \Delta \Big( V_e^\sigma(m) \cdot U_e^\sigma(m) \Big) \Big| \, = \, \bigg| \Delta \bigg( \sum_{f \in U_e^\sigma(m)} Y_f(m) \bigg) \bigg| \, \le \, 2 \cdot \Yt(m) \cdot \big| \Delta U_e^\sigma(m) \big| + \sum_{f \in U_e^\sigma(m)} \big| \Delta Y_f(m) \big|,$$
where the inequality holds since $Y_f(m) \le 2 \cdot \Yt(m)$, which follows from the event $\E(m)$. We claim that 
$$2 \cdot \Yt(m) \cdot \big| \Delta U_e^\sigma(m) \big| + \sum_{f \in U_e^\sigma(m)} \big| \Delta Y_f(m) \big| \, \le \, 3 \cdot |\sigma| \cdot (\log n)^3 \cdot \Yt(m)^{|\sigma|}.$$
Indeed, this follows by Lemma~\ref{deltaU}, and the facts that $U_e^\sigma(m) \le \Yt(m)^{|\sigma|}$ (since $\E(m)$ holds) and that $| \Delta Y_f(m) | \le 2 \cdot (\log n)^2$ (since $\Z(m)$ holds), as above. 

Now, by the product rule, we have
$$\Delta \Big( V_e^\sigma(m) \cdot U_e^\sigma(m) \Big) \, = \, U_e^\sigma(m) \Delta \big( V_e^\sigma(m) \big) + V_e^\sigma(m) \Delta \big( U_e^\sigma(m) \big) + \Delta \big( V_e^\sigma(m) \big) \Delta \big( U_e^\sigma(m) \big).$$
Moreover, we have $V_e^\sigma(m) \le 2\Yt(m)$ (since $\E(m)$ holds), and since $\Yt(m) \ge n^\eps \gg |\sigma| \cdot (\log n)^3$ for every $m \le m^*$, it follows from Lemma~\ref{deltaU} that $| \Delta U_e^\sigma(m) | \ll U_e^\sigma(m)$. Hence, by the triangle inequality and Lemma~\ref{deltaU} (again), we have 
\begin{align*}
\big| \Delta V_e^\sigma(m) \big| & \, \le \, \frac{4}{U_e^\sigma(m)} \Big( |\sigma| \cdot (\log n)^3 \cdot \Yt(m)^{|\sigma|} \,+\, \Yt(m) \cdot | \Delta U_e^\sigma(m) | \Big) \\
& \, \le \, \frac{8}{U_e^\sigma(m)} \Big( |\sigma| \cdot (\log n)^3 \cdot \Yt(m)^{|\sigma|} \Big) \, \le \, (\log n)^4,
\end{align*}
as required.
\end{proof}

\subsection{Self-correction}\label{YselfSec}

In this subsection we shall prove, using the results above, that the variables $V_e^\sigma$ exhibit a kind of self-correction. The calculation is somewhat lengthy, and requires careful counting of the $\sigma$-paths created and destroyed in a typical step of the triangle-free process. For each $m \in [m^*]$, each $k$-short sequence $\sigma \in \{ L,R \}^*$ and each open edge $e \in O(G_m)$, 
define
$$(V_e^\sigma)^*(m) \, = \, \frac{V_e^\sigma(m) - \Yt(m)}{g_\sigma(t) \cdot \Yt(m)}.$$
Recall that $\V(m')$ denotes the event that the variables~$V_e^\sigma(m)$ are tracking up to step $m'$, i.e., that~$|(V_e^\sigma)^*(m)| \le 1$ for every $\omega \cdot n^{3/2} < m \le m'$, every $e \in O(G_m)$ and every $k$-short sequence $\sigma$.  

Recall that $k = \lceil 3 / \eps \rceil$ and $n \ge n_0(\eps,C,\omega)$. The aim of this subsection is to prove the following key lemma, which implies that the variables $V_e^\sigma$ are self-correcting. 

\begin{lemma}\label{Vself}
Let $\omega \cdot n^{3/2} < m \le m^*$. If $\E(m) \cap \U(m) \cap \V(m) \cap \X(m) \cap \Z(m) \cap \Q(m)$ holds, then 
\begin{equation}\label{eq:Vself}
\Ex\big[ \Delta (V^{\sigma}_e)^*(m) \big] \, \in \, \big( 8 \pm 5 \big) \cdot \frac{t}{n^{3/2}} \Big( - (V_e^{\sigma})^*(m) \pm O(\eps) \Big)
\end{equation}
for every $e \in O(G_m)$ and every $k$-short sequence $\sigma \in \{ L,R \}^*$.  
\end{lemma}


Our main task will be to prove the following bounds on $\Ex\big[ \Delta V^{\sigma}_e(m) \big]$. We shall then deduce Lemma~\ref{Vself} using the mixing results from Sections~\ref{MixSec1} and~\ref{MixSec2}.

\begin{lemma}\label{lem:Veq}
Let $\omega \cdot n^{3/2} < m \le m^*$. If $\E(m) \cap \U(m) \cap \X(m) \cap \Z(m)$ holds, then
$$\Ex\big[ \Delta V^{\sigma}_e(m) \big] \, \in \, - \frac{1}{Q(m)} \bigg( \frac{ U_e^{\sigma L}(m) V^{\sigma L}_e(m) + U_e^{\sigma R}(m) V^{\sigma R}_e(m)}{U_e^{\sigma}(m)} \bigg) \,+\, \big( 1 \pm C g_x(t) \big) \frac{\Xt(m)}{Q(m)}$$
for every $e \in O(G_m)$, and every $\sigma \in \{L,R\}^*$ with $|\sigma| \le \omega$.
\end{lemma}

We begin by controlling $\Ex\big[ \Delta U_e^\sigma(m) \big]$. For each $\sigma \in \{L,R\}^*$, define $\sigma(j)$ to be the sequence formed by the first $j$ elements of $\sigma$.\footnote{Thus if $\sigma = (\sigma_1,\ldots,\sigma_\ell)$, then $\sigma(j) = (\sigma_1,\ldots,\sigma_j)$.} Lemma~\ref{lem:Veq} is a straightforward consequence of the following bounds.


\begin{lemma}\label{Ueq}
Let $\omega \cdot n^{3/2} < m \le m^*$. If $\E(m) \cap \U(m) \cap \X(m) \cap \Z(m)$ holds, then
$$\Ex\big[ \Delta U^{\sigma}_e(m) \big] \, \in \, - \frac{1}{Q(m)} \sum_{j=1}^{|\sigma|} U_e^{\sigma}(m) V_e^{\sigma(j)}(m) \,+\, |\sigma| \left( \frac{1}{2^{|\sigma|}} \pm g_x(t) \right) \frac{\Xt(m) \Yt(m)^{|\sigma|-1}}{Q(m)},$$
for every $e \in O(G_m)$, and every $\sigma \in \{L,R\}^*$ with $|\sigma| \le \omega$.
\end{lemma}

In order to prove Lemma~\ref{Ueq}, we shall first need to prove a series of simpler lemmas. We begin with two easy observations.

\begin{obs}\label{obs:Ygrows}
Let $m \le m^*$ and $e \in O(G_m)$. Then $\big| Y^L_e(m+1) \,\setminus\, Y^L_e(m) \big| \le 1$, and moreover
$$\Pr\Big( f \in Y^L_e(m+1) \,\setminus\, Y^L_e(m) \,\big|\, G_m \Big) \, = \, \frac{1}{Q(m)},$$ 
for each $f \in X^L_e(m)$. If $f \not\in X^L_e(m)$ then the probability is zero.
\end{obs}

\begin{proof}
Let $h$ be the edge selected in step $m+1$, and suppose that $f \in Y^L_e(m+1) \,\setminus\, Y^L_e(m)$. Then $f \in X^L_e(m)$ and $h \in X^R_e(m)$ form an open triangle with $e$ in $G_m$, and moreover $Y^L_e(m+1) \,\setminus\, Y^L_e(m) = \{f\}$, as required.
\end{proof}

Recall that if $e \not\in O(G_{m+1})$ then (by convention) we set $Y^L_e(m+1) = Y^L_e(m)$. 

\begin{obs}\label{obs:Yshrinks}
Let $m \le m^*$ and $e \in O(G_m)$. Then
$$\Pr\Big( f \in Y^L_e(m) \,\setminus\, Y^L_e(m+1)  \,\big|\, G_m \Big) \, = \, \frac{Y_f(m) - 1}{Q(m)},$$ 
for each $f \in Y^L_e(m)$. If $f \not\in Y^L_e(m)$ then the probability is zero. 
\end{obs}

\begin{proof}
Let $h$ be the edge selected in step $m+1$, and suppose that $f \in Y^L_e(m) \,\setminus\, Y^L_e(m+1)$. Then either $e$ or $f$ was closed by $h$. But if $e \not\in O(G_{m+1})$ then $Y^L_e(m+1) = Y^L_e(m)$, so we must have $h \in Y_f(m) \setminus \{e\}$. Any such $h$ will suffice, and hence we have $f \in Y^L_e(m) \,\setminus\, Y^L_e(m+1)$ with probability exactly $\big( Y_f(m) - 1 \big) / Q(m)$, as claimed.
\end{proof}

The observations above imply the following two identities.

\begin{lemma}\label{lem:ExsumUf}
Let $m \le m^*$, $e \in O(G_m)$ and $\sigma \in \{L,R\}^*$. Then
\begin{equation}\label{eq1:sumUf}
\Ex \bigg[ \sum_{f \in Y^L_e(m+1) \,\setminus\, Y^L_e(m)} U^{\sigma}_f(m) \,\Big|\, G_m \bigg] \, = \, \frac{1}{Q(m)} \sum_{f \in X^L_e(m)} U_f^{\sigma}(m),
\end{equation}
and 
\begin{equation}\label{eq2:sumUf}
\Ex \bigg[ \sum_{f \in Y^L_e(m) \,\setminus\, Y^L_e(m+1)} U^{\sigma}_f(m) \,\Big|\, G_m \bigg] \, = \, \frac{1}{Q(m)} \sum_{f \in Y^L_e(m)} \big( Y_f(m) - 1 \big) U_f^{\sigma}(m).
\end{equation}
\end{lemma}

\begin{proof}
The first identity follows immediately from Observation~\ref{obs:Ygrows}, since the sum on the left is equal to $U_f^{\sigma}(m)$ with probability $1 / Q(m)$ for each $f \in X^L_e(m)$, and zero otherwise. The second identity follows immediately from Observation~\ref{obs:Yshrinks}, since the sum on the left contains $f$ with probability $( Y_f(m) - 1 ) / Q(m)$ for each $f \in Y^L_e(m)$.
\end{proof}


Combining Observation~\ref{obs:Ygrows} with Lemma~\ref{deltaU}, we also obtain the following bounds.

\begin{lemma}\label{Ueq3}
For every $\omega \cdot n^{3/2} < m \le m^*$, if $\E(m) \cap \Z(m)$ holds, then
$$\Ex\bigg[ \sum_{f \in Y^{L}_e(m+1)} \Delta U^{\sigma}_f(m) \,- \sum_{f \in Y^{L}_e(m)} \Delta U^{\sigma}_f(m) \,\Big|\, G_m \bigg] \, \in \, \pm \, (\log n)^4 \cdot  \frac{\Xt(m) \Yt(m)^{|\sigma|-1}}{Q(m)}$$
for every $e \in O(G_m)$, and every $\sigma \in \{L,R\}^*$ with $|\sigma| \le \omega$.
\end{lemma}

\begin{proof}
Observe first that, by linearity of expectation, 
$$\Ex\bigg[ \sum_{f \in Y^{L}_e(m+1)} \Delta U^{\sigma}_f(m) \,\Big|\, G_m \bigg] -  \sum_{f \in Y^{L}_e(m)} \Ex\big[ \Delta U^{\sigma}_f(m) \big] \, = \, \Ex\bigg[ \sum_{f \in Y^{L}_e(m+1) \,\setminus\, Y^{L}_e(m)} \Delta U^{\sigma}_f(m) \,\Big|\, G_m \bigg],$$
since if $f \in Y^{L}_e(m) \setminus Y^{L}_e(m+1)$ then $f \not\in O(G_{m+1})$, and thus $\Delta U^{\sigma}_f(m) = 0$.

Now, since $\E(m) \cap \Z(m)$ holds, by Lemma~\ref{deltaU} we have
$$\big| \Delta U_f^\sigma(m) \big| \, \le \, |\sigma| \cdot (\log n)^3 \cdot \Yt(m)^{|\sigma| - 1}.$$
Moreover, by Observation~\ref{obs:Ygrows}, we have $f \in Y^L_e(m+1) \,\setminus\, Y^L_e(m)$ with probability $1 / Q(m)$ for each $f \in X^L_e(m)$. Hence, using the event $\E(m)$ to bound $X_e(m)$, we have
$$\Ex\bigg[ \sum_{f \in Y^{L}_e(m+1) \,\setminus\, Y^{L}_e(m)} \Delta U^{\sigma}_f(m) \,\Big|\, G_m \bigg] \, \in \, \pm \, (\log n)^4 \cdot \Yt(m)^{|\sigma| - 1} \cdot \frac{\Xt(m)}{Q(m)},$$
as required.
\end{proof}

We can now prove our first bounds on $\Ex\big[ \Delta U^{\sigma}_e(m) \big]$. For each $\sigma \in \{L,R\}^*$, let us write $\sigma^-$ for the sequence obtained by removing the \emph{first} entry of $\sigma$.\footnote{Thus if $\sigma = (\sigma_1,\ldots,\sigma_\ell)$, then $\sigma^- = (\sigma_2,\ldots,\sigma_\ell)$.} 

\begin{lemma}\label{Ueq2}
Let $\omega \cdot n^{3/2} < m \le m^*$. If $\E(m) \cap \U(m) \cap \X(m) \cap \Z(m)$ holds, then
$$
\Ex\big[ \Delta U^{\sigma}_e(m) \big] \in \sum_{f \in Y^{\sigma_1}_e(m)} \bigg( \Ex\big[ \Delta U^{\sigma^-}_f(m) \big] - \frac{Y_f(m) U_f^{\sigma^-}(m)}{Q(m)} \bigg) + \left( \frac{1}{2^{|\sigma|}} \pm g_x(t) \right) \frac{\Xt(m) \Yt(m)^{|\sigma|-1}}{Q(m)}
$$
for every $e \in O(G_m)$, and every $\sigma \in \{L,R\}^*$ with $0 < |\sigma| \le \omega$.
\end{lemma}

\begin{proof}
Since
$$
\Delta U^{\sigma}_e(m) \, = \, U^{\sigma}_e(m+1) - U^{\sigma}_e(m) \, = \, \sum_{f \in Y^{\sigma_1}_e(m+1)} U^{\sigma^-}_f(m+1) -  \sum_{f \in Y^{\sigma_1}_e(m)} U^{\sigma^-}_f(m),
$$
it follows that
\begin{multline*}
\Ex\big[ \Delta U^{\sigma}_e(m) \big] \, = \,  \Ex\bigg[ \sum_{f \in Y^{\sigma_1}_e(m+1)} U^{\sigma^-}_f(m+1) -  \sum_{f \in Y^{\sigma_1}_e(m)} U^{\sigma^-}_f(m)  \,\Big|\, G_m \bigg] \\ 
\, = \, \sum_{f \in Y^{\sigma_1}_e(m+1)} \Ex\big[  \Delta U^{\sigma^-}_f(m) \big] \,+\, \Ex \bigg[ \sum_{f \in Y^{\sigma_1}_e(m+1)} U^{\sigma^-}_f(m) -  \sum_{f \in Y^{\sigma_1}_e(m)} U^{\sigma^-}_f(m) \,\Big|\, G_m \bigg].
\end{multline*}
Now, by Lemma~\ref{lem:ExsumUf}, and using the event $\U(m) \cap \X(m)$, we have 
$$
\Ex \bigg[ \sum_{f \in Y^{\sigma_1}_e(m+1) \,\setminus\, Y^{\sigma_1}_e(m)} U^{\sigma^-}_f(m) \,\Big|\, G_m \bigg] \, = \, \sum_{f \in X^{\sigma_1}_e(m)} \frac{U_f^{\sigma^-}(m)}{Q(m)} \, \in \, \big( 1 \pm g_x(t) \big)^{|\sigma|} \cdot \frac{\Xt(m) \Yt(m)^{|\sigma|-1}}{2^{|\sigma|} \cdot Q(m)},
$$
since $|\sigma^-| = |\sigma| - 1$ and $X_e^L(m) = X_e^R(m)$. 
Similarly, we have
$$
\Ex \bigg[ \sum_{f \in Y^{\sigma_1}_e(m) \,\setminus\, Y^{\sigma_1}_e(m+1)} U^{\sigma^-}_f(m) \,\Big|\, G_m \bigg] \, \in \, \sum_{f \in Y^{\sigma_1}_e(m)} \frac{Y_f(m) U_f^{\sigma^-}(m)}{Q(m)} \,\pm\, \frac{\Yt(m)^{|\sigma|}}{Q(m)},
$$
and by Lemma~\ref{Ueq3}, 
$$
\Ex \bigg[ \sum_{f \in Y^{\sigma_1}_e(m+1)} \Delta U^{\sigma^-}_f(m) -  \sum_{f \in Y^{\sigma_1}_e(m)}  \Delta U^{\sigma^-}_f(m)  \,\Big|\, G_m \bigg] \, \in \, \pm \, (\log n)^4 \cdot \frac{\Xt(m) \Yt(m)^{|\sigma|-2}}{Q(m)}.
$$
Combining the last four displayed equations, and noting that $g_x(t) \Xt(m) \gg \Yt(m) \gg (\log n)^4$, we obtain
$$
\Ex\big[ \Delta U^{\sigma}_e(m) \big] \in \sum_{f \in Y^{\sigma_1}_e(m)} \bigg( \Ex\big[  \Delta U^{\sigma^-}_f(m) \big]  \, - \, \frac{Y_f(m) U_f^{\sigma^-}(m)}{Q(m)} \bigg) + \left( \frac{1}{2^{|\sigma|}} \pm g_x(t) \right) \frac{\Xt(m) \Yt(m)^{|\sigma|-1}}{Q(m)},
$$
as required.
\end{proof}

Recall that $\sigma(j)$ denotes the sequence formed by the first $j$ elements of $\sigma$ (so if $\sigma = (\sigma_1,\ldots,\sigma_\ell)$, then $\sigma(j) = (\sigma_1,\ldots,\sigma_j)$). In order to prove Lemma~\ref{Ueq}, we shall need the following bound on the covariance of $U_f^{\sigma^-}(m)$ and $V_f^{\sigma^-(j)}(m)$ over the set $Y_e^{\sigma_1}(m)$.

\begin{lemma}\label{Ucov2}
Let $\omega \cdot n^{3/2} < m \le m^*$. If $\U(m)$ holds, then 
\begin{equation}\label{eq:Ucov2}
\sum_{f \in Y^{\sigma_1}_e(m)} U_f^{\sigma^-}(m) V_f^{\sigma^-(j)}(m) \, \in \, U_e^{\sigma}(m) V_e^{\sigma(j+1)}(m) \,\pm\, \eps \cdot g_x(t) \cdot \Xt(m) \Yt(m)^{|\sigma| - 1}
\end{equation}
for every $e \in O(G_m)$ and every $\sigma \in \{L,R\}^*$ with $0 \le j < |\sigma| \le \omega$.
\end{lemma}

In order to prove Lemma~\ref{Ucov2}, we shall use the following corollary of Observation~\ref{avobs2}. 

\begin{obs}\label{avobs4}
Let $\omega \cdot n^{3/2} < m \le m^*$. If $\U(m)$ holds, then 
$$\frac{1}{Y^{\sigma_1}_e(m)} \sum_{f \in Y^{\sigma_1}_e(m)} V_f^{\sigma^-}(m) \, \in \, \Big( 1 \pm \omega^2 g_x(t)^2 \Big) \cdot V_e^{\sigma}(m)$$
for every $e \in O(G_m)$ and every $\sigma \in \{L,R\}^*$ with $0 < |\sigma| \le \omega$.
\end{obs}

\begin{proof}
Set $a_f = V_f^{\sigma^-}(m)$ and $b_f = U_f^{\sigma^-}(m)$ for each $f \in Y^{\sigma_1}_e(m)$. By the event $\U(m)$, we have
$$a_f \in \big( 1 \pm g_x(t) \big) \Yt(m) \qquad \text{and} \qquad b_f \in \big( 1 \pm g_x(t) \big)^\omega \bigg( \frac{\Yt(m)}{2} \bigg)^{|\sigma|-1},$$
and so, by Observation~\ref{avobs2},
$$V_e^{\sigma}(m) \, = \, \frac{\sum_{f \in Y^{\sigma_1}_e(m)} V_f^{\sigma^-}(m) U_f^{\sigma^-}(m)}{\sum_{f \in Y^{\sigma_1}_e(m)} U_f^{\sigma^-}(m)} \, \in \, \frac{1 \pm \omega^2 g_x(t)^2}{Y^{\sigma_1}_e(m)} \sum_{f \in Y^{\sigma_1}_e(m)} V_f^{\sigma^-}(m),$$
as required.
\end{proof}

We can now easily deduce Lemma~\ref{Ucov2} from Observations~\ref{avobs2} and~\ref{avobs4}. 

\begin{proof}[Proof of Lemma~\ref{Ucov2}]
Set $a_f = V_f^{\sigma^-(j)}(m)$ and $b_f = U^{\sigma^-}_f(m)$ for each $f \in Y^{\sigma_1}_e(m)$. Noting that,
$$a_f \in \big( 1 \pm g_x(t) \big) \Yt(m) \qquad \text{and} \qquad b_f \in \big( 1 \pm g_x(t) \big)^\omega \bigg( \frac{\Yt(m)}{2} \bigg)^{|\sigma|-1},$$
which both follow from the event $\U(m)$, we obtain
$$Y^{\sigma_1}_e(m) \sum_{f \in Y^{\sigma_1}_e(m)} U_f^{\sigma^-}(m) V_f^{\sigma^-(j)}(m) \, \in \,  \big( 1 \pm \omega^2 g_x(t)^2 \big) \sum_{f \in Y^{\sigma_1}_e(m)} U_f^{\sigma^-}(m) \sum_{f \in Y^{\sigma_1}_e(m)} V_f^{\sigma^-(j)}(m),$$
by Observation~\ref{avobs2}. Applying Observation~\ref{avobs4} to the sequence $\sigma(j+1)$, it follows that
$$\sum_{f \in Y^{\sigma_1}_e(m)} U_f^{\sigma^-}(m) V_f^{\sigma^-(j)}(m) \, \in \, U_e^{\sigma}(m) V_e^{\sigma(j+1)}(m) \,\pm\, \omega^2 g_x(t)^2 \Yt(m)^{|\sigma| + 1},$$
which implies~\eqref{eq:Ucov2}, since $\Xt(m) \gg \omega^2 \cdot g_x(t) \cdot \Yt(m)^2$ for every $m \le m^*$. 
\end{proof}

We are now ready to prove the claimed bounds on $\Ex\big[ \Delta U^{\sigma}_e(m) \big]$. 

\begin{proof}[Proof of Lemma~\ref{Ueq}]
Let $\omega \cdot n^{3/2} < m \le m^*$, and assume that $\E(m) \cap \U(m) \cap \X(m) \cap \Z(m)$ holds. We shall prove that
\begin{equation}\label{eq:UIH}
\Ex\big[ \Delta U^{\sigma}_e(m) \big] \, \in \, - \frac{1}{Q(m)} \sum_{j=1}^{|\sigma|} U_e^{\sigma}(m) V_e^{\sigma(j)}(m) \,+\, |\sigma| \left( \frac{1}{2^{|\sigma|}} \pm g_x(t) \right) \frac{\Xt(m) \Yt(m)^{|\sigma|-1}}{Q(m)},
\end{equation}
for every $e \in O(G_m)$ and every $\sigma \in \{L,R\}^*$ with $|\sigma| \le \omega$, by induction on $\ell = |\sigma|$. The result is trivial if $\ell = 0$, since in that case $U_e^\sigma(m) = 1$ for every $e \in O(G_m)$ and every $m \le m^*$. For $\ell = 1$, assume for simplicity that $\sigma = L$, and (recalling that $U_e^L = Y_e^L$) observe that 
\begin{equation}\label{eq:deltaYL}
\Ex\big[ \Delta Y^L_e(m) \big] \, = \, - \frac{1}{Q(m)} \sum_{f \in Y_e^L(m)} \big( Y_f(m) - 1 \big) \,+\, \frac{X^R_e(m)}{Q(m)}.
\end{equation}
The bounds in~\eqref{eq:UIH} now follow, since $U_e^L(m) V_e^L(m) = \sum_{f \in Y_e^L(m)} Y_f(m)$, by Definition~\ref{def:UandV}, and $\U(m) \cap \X(m)$ implies that $Y^L_e(m) + X^R_e(m) \in \big( \frac{1}{2} \pm g_x(t) \big) \Xt(m)$, since $X_e^L(m) = X^R_e(m)$.

So let $\ell \ge 2$ and assume that~\eqref{eq:UIH} holds for every open edge in $G_m$ and every sequence of length $\ell - 1$. Let $e \in O(G_m)$, and let $\sigma \in \{L,R\}^\ell$. By Lemma~\ref{Ueq2}, we have
$$
\Ex\big[ \Delta U^{\sigma}_e(m) \big] \in \sum_{f \in Y^{\sigma_1}_e(m)} \bigg( \Ex\big[ \Delta U^{\sigma^-}_f(m) \big] - \frac{Y_f(m) U_f^{\sigma^-}(m)}{Q(m)} \bigg) \, + \, \left( \frac{1}{2^{|\sigma|}} \pm g_x(t) \right) \frac{\Xt(m) \Yt(m)^{|\sigma|-1}}{Q(m)},
$$
and by the induction hypothesis, applied to the pair $(f,\sigma^-)$, we have
$$
\Ex\big[ \Delta U^{\sigma^-}_f(m) \big] \in - \frac{1}{Q(m)} \sum_{j=1}^{|\sigma|-1} U_f^{\sigma^-}(m) V_f^{\sigma^-(j)}(m) \,+\, \big( |\sigma| - 1 \big) \left( \frac{1}{2^{|\sigma|-1}} \pm g_x(t) \right) \frac{\Xt(m) \Yt(m)^{|\sigma|-2}}{Q(m)}
$$
for each $f \in Y^{\sigma_1}_e(m)$. Combining these bounds, we obtain\footnote{Note that the term $Y_f(m) U_f^{\sigma^-}(m)$ is absorbed into the sum as the term $j = 0$.}
\begin{align*}
\Ex\big[ \Delta U^{\sigma}_e(m) \big] & \, \in \, - \, \frac{1}{Q(m)}  \sum_{j=0}^{|\sigma|-1} \sum_{f \in Y^{\sigma_1}_e(m)} U_f^{\sigma^-}(m) V_f^{\sigma^-(j)} \\
& \, + \, \frac{\Xt(m) \Yt(m)^{|\sigma|-1}}{Q(m)}  \bigg( \frac{1}{2^{|\sigma|}} \pm g_x(t) + \big( |\sigma| - 1 \big) \left( \frac{1}{2^{|\sigma|-1}} \pm g_x(t) \right) \frac{Y_e^{\sigma_1}(m)}{\Yt(m)} \bigg).
\end{align*}
Now, by Lemma~\ref{Ucov2}, we have
$$
\sum_{j=0}^{|\sigma|-1} \sum_{f \in Y^{\sigma_1}_e(m)} U_f^{\sigma^-}(m) V_f^{\sigma^-(j)} \, \in \, \sum_{j=1}^{|\sigma|} U_e^{\sigma}(m) V_e^{\sigma(j)}(m) \, \pm \, \eps \cdot |\sigma| \cdot g_x(t) \Xt(m) \Yt(m)^{|\sigma|-1},
$$
and by the event $\U(m)$, we have $Y_e^{\sigma_1}(m) \in \big( 1 \pm g_x(t) \big) \Yt(m) / 2$. Since $\ell \ge 2$, we obtain\footnote{This follows since $\frac{1}{2} \big( 1 \pm g_x(t) \big) \big( 2^{-\ell+1} \pm g_x(t) \big) \in \big( 2^{-\ell} \pm g_x(t) \big)$.}
$$\Ex\big[ \Delta U^{\sigma}_e(m) \big] \, \in \, - \frac{1}{Q(m)} \sum_{j=1}^{|\sigma|} U_e^{\sigma}(m) V_e^{\sigma(j)}(m) \,+\, |\sigma| \left( \frac{1}{2^{|\sigma|}} \pm g_x(t) \right) \frac{\Xt(m) \Yt(m)^{|\sigma|-1}}{Q(m)},$$
as required.
\end{proof}

We note the following easy identity, which is proved in the Appendix~\cite{App}.

\begin{lemma}\label{A/B}
$$\Ex \bigg[ \Delta \left( \frac{A(m)}{B(m)} \right) \bigg] \, = \, \frac{\Ex \big[ \Delta A(m) \big]}{B(m)} \, - \, \frac{1}{B(m)}\Ex\Bigg[\frac{A(m+1) \Delta B(m)}{B(m+1)} \,\Big| \, G_m \Bigg].$$
\end{lemma}

Our bound on $\Ex\big[ \Delta V_e^\sigma(m) \big]$ now follows via a straightforward calculation.

\begin{proof}[Proof of Lemma~\ref{lem:Veq}]
Let $\sigma \in \{L,R\}^*$ be a sequence with $|\sigma| \le \omega$, let $e \in O(G_m)$ for some $\omega \cdot n^{3/2} < m \le m^*$, and suppose that $\E(m) \cap \U(m) \cap \X(m) \cap \Z(m)$ holds. Since
$$V_e^\sigma(m)  \, = \, \frac{U_e^{\sigma L}(m) + U_e^{\sigma R}(m)}{U_e^{\sigma}(m)},$$
it follows, by Lemma~\ref{A/B}, that
\begin{equation}\label{eq:Vdelta}
\Ex\big[ \Delta V^{\sigma}_e(m) \big] \, = \, \frac{1}{U_e^\sigma(m)} \bigg( \Ex\Big[ \Delta \big( U_e^{\sigma L}(m) + U_e^{\sigma R}(m) \big) \Big] \,-\, \Ex\Big[ V_e^\sigma(m+1) \cdot \Delta U_e^\sigma(m) \,\big| \, G_m \Big] \bigg).
\end{equation}
Now, by Lemma~\ref{deltaV}, and since $\E(m) \cap \Z(m)$ holds, we have
$$V_e^\sigma(m+1) \, \in \, V_e^\sigma(m) \pm (\log n)^4,$$
and since $\U(m)$ holds, $U_e^{\sigma}(m) V_e^{\sigma(j)}(m) \le 2 \cdot \Yt(m)^{|\sigma|}$. Moreover, note that $g_x(t)\Xt(m) \gg (\log n)^4 \cdot \Yt(m)$. By Lemma~\ref{Ueq}, it follows that 
\begin{multline} \label{eq:Vdelta2}
\Ex\Big[ V_e^\sigma(m+1) \cdot \Delta U_e^\sigma(m) \,\big| \, G_m \Big] \, \in \, - \frac{1}{Q(m)} \sum_{j=1}^{|\sigma|} U_e^{\sigma}(m) V_e^{\sigma(j)}(m) V_e^\sigma(m)  \\
\,+\, |\sigma| \left( \frac{1}{2^{|\sigma|}} \pm 2g_x(t) \right) \frac{\Xt(m) \Yt(m)^{|\sigma|}}{Q(m)}.
\end{multline}
Moreover, applying Lemma~\ref{Ueq} to $U_e^{\sigma L}(m)$ and $U_e^{\sigma R}(m)$, we obtain
\begin{multline*}
\Ex\Big[ \Delta \big( U_e^{\sigma L}(m) + U_e^{\sigma R}(m) \big) \Big] \, \in \, - \, \frac{1}{Q(m)} \sum_{j=1}^{|\sigma|} \Big( U_e^{\sigma L}(m) + U_e^{\sigma R}(m) \Big) V_e^{\sigma(j)}(m) \\
- \frac{1}{Q(m)} \Big( U_e^{\sigma L}(m) V^{\sigma L}_e(m) + U_e^{\sigma R}(m) V^{\sigma R}_e(m) \Big) \,+\, \big( |\sigma| + 1 \big) \left( \frac{1}{2^{|\sigma|}} \pm 2g_x(t) \right) \frac{\Xt(m) \Yt(m)^{|\sigma|}}{Q(m)}.
\end{multline*}
Now, combining this with~\eqref{eq:Vdelta} and~\eqref{eq:Vdelta2}, and recalling that $U_e^{\sigma L}(m) + U_e^{\sigma R}(m) = U_e^{\sigma}(m) V_e^{\sigma}(m)$ and $2^{|\sigma|} U_e^\sigma(m) \in \big( 1 \pm g_x(t) \big)^{|\sigma|} \Yt(m)^{|\sigma|}$, 
it follows that
$$\Ex\big[ \Delta V^{\sigma}_e(m) \big] \, \in \, - \frac{1}{Q(m)} \bigg( \frac{ U_e^{\sigma L}(m) V^{\sigma L}_e(m) + U_e^{\sigma R}(m) V^{\sigma R}_e(m)}{U_e^{\sigma}(m)} \bigg) + \big( 1 \pm C g_x(t) \big) \frac{\Xt(m)}{Q(m)},$$
as required.
\end{proof}

Combining Lemma~\ref{lem:Veq} with mixing results from Sections~\ref{MixSec1} and~\ref{MixSec2}, we can now prove that the variables $(V_e^\sigma)^*(m)$ are self-correcting. 

\begin{proof}[Proof of Lemma~\ref{Vself}]
Let $\omega \cdot n^{3/2} < m \le m^*$, let $\sigma \in \{L,R\}^*$ be a $k$-short sequence, and suppose that the event $\E(m) \cap \U(m) \cap \V(m) \cap \X(m) \cap \Z(m) \cap \Q(m)$ holds. We claim first that, if exactly $r \in \{0,1,2\}$ of the sequences $\sigma L$ and $\sigma R$ are not $k$-short, then\footnote{Here, and throughout this proof, we write $g_{\sigma + 1}(t)$ to denote the function $g_{\sigma L}(t) = g_{\sigma R}(t)$.}  
\begin{multline}\label{eq:Vself1}
W_e^\sigma(m) \, := \, \frac{U_e^{\sigma L}(m) V^{\sigma L}_e(m) + U_e^{\sigma R}(m) V^{\sigma R}_e(m)}{U_e^{\sigma}(m)} \, \in \,  \big( 1 \pm g_{\sigma + 1}(t) \big) \Yt(m) V_e^\sigma(m) \\
\, + \,\big( 1 \pm g_x(t) \big) \cdot  \frac{r}{2} \cdot g_\sigma(t) \Yt(m)^2 \cdot (V_e^\sigma)^*(m).
\end{multline}
for every $e \in O(G_m)$. 

To prove~\eqref{eq:Vself1}, suppose first that the sequences $\sigma L$ and $\sigma R$ are both $k$-short, i.e., $r = 0$. Then, since the event $\V(m)$ holds, we have
\begin{align*}
\frac{ U_e^{\sigma L}(m) V^{\sigma L}_e(m) + U_e^{\sigma R}(m) V^{\sigma R}_e(m)}{U_e^{\sigma}(m)} & \, \in \, \big( 1 \pm g_{\sigma + 1}(t) \big) \Yt(m) \cdot \frac{ U_e^{\sigma L}(m) + U_e^{\sigma R}(m)}{U_e^{\sigma}(m)} \\
& \, = \, \big( 1 \pm g_{\sigma + 1}(t) \big) \Yt(m) V_e^\sigma(m),
\end{align*}
as required. Next, suppose that $r = 1$, i.e., that $\sigma L$ is not $k$-short (say), but $\sigma R$ is $k$-short. Then either $s(\sigma) \ge k$, or $\sigma$ ends with $k$ consecutive $L$s, so (recalling that $\V(m) \Rightarrow \Yh(m)$) we can apply either Lemma~\ref{VmixQ} (in the former case) or Lemma~\ref{Vmix} (in the latter case) to obtain
\begin{equation}\label{eq:Vselfmix}
\big| V_e^\sigma(m) - V_e^{\sigma L}(m) \big| \, = \, o\big( g_y(t) \Yt(m) \big) \, \le \, g_{\sigma + 1}(t) \Yt(m).
\end{equation}
Hence, since the event $\U(m) \cap \V(m)$ holds, 
\begin{multline*}
W_e^\sigma(m) \, \in \, \frac{U_e^{\sigma L}(m) V^{\sigma}_e(m) + U_e^{\sigma R}(m) \Yt(m)}{U_e^{\sigma}(m)} \pm g_{\sigma + 1}(t) \Yt(m) V_e^\sigma(m)\\
 \, = \, \big( 1 \pm g_{\sigma + 1}(t) \big)\Yt(m) V_e^\sigma(m) \,+\, \frac{U_e^{\sigma L}(m)}{U_e^{\sigma}(m)} \big( V_e^\sigma(m) - \Yt(m) \big) \hspace{1cm} \\
 \, \subseteq \, \big( 1 \pm g_{\sigma + 1}(t) \big)\Yt(m) V_e^\sigma(m) + \frac{\big( 1 \pm g_x(t) \big) \Yt(m)}{2} \cdot g_\sigma(t) (V_e^\sigma)^*(m) \Yt(m),
\end{multline*}
as claimed. Finally, if $r = 2$, i.e., neither $\sigma L$ nor $\sigma R$ is $k$-short, then by Lemmas~\ref{Vmix} and~\ref{VmixQ} it follows that~\eqref{eq:Vselfmix} holds for $V_e^{\sigma L}(m)$, and also for $V_e^{\sigma R}(m)$,  and so
\begin{multline*}
W_e^\sigma(m) \, \in \, \frac{U_e^{\sigma L}(m) V^{\sigma}_e(m) + U_e^{\sigma R}(m) V^{\sigma}_e(m)}{U_e^{\sigma}(m)} \pm g_{\sigma + 1}(t) \Yt(m) V_e^\sigma(m)\\
 \, = \, \big( 1 \pm g_{\sigma + 1}(t) \big)\Yt(m) V_e^\sigma(m) \,+\, \frac{U_e^{\sigma L}(m) + U_e^{\sigma R}(m)}{U_e^{\sigma}(m)} \cdot \big( V_e^\sigma(m) - \Yt(m) \big) \\
 \, \subseteq \, \big( 1 \pm g_{\sigma + 1}(t) \big)\Yt(m) V_e^\sigma(m) + \big( 1 \pm g_x(t) \big) \Yt(m) \cdot g_\sigma(t) (V_e^\sigma)^*(m) \Yt(m),
\end{multline*}
which proves~\eqref{eq:Vself1}.

Next, combining Lemma~\ref{lem:Veq} with~\eqref{eq:Vself1}, we shall prove that
\begin{equation}\label{eq:Vself2}
\frac{1}{g_\sigma(t)} \bigg( \frac{\Ex\big[ \Delta V^{\sigma}_e(m) \big]}{\Yt(m)} + \frac{8t^2 - 1}{t \cdot n^{3/2}} \bigg) \, \in \, \big( 1 \pm 2g_x(t) \big) \big( 8 + 4r \big) \cdot \frac{t}{n^{3/2}} \Big( - (V_e^\sigma)^*(m) \pm O(\eps) \Big)
\end{equation}
for every $e \in O(G_m)$. 

To obtain~\eqref{eq:Vself2}, recall first that, by Lemma~\ref{lem:Veq} and the event $\Q(m)$, we have
\begin{equation}\label{eq:Veqrecall}
\Ex\big[ \Delta V^{\sigma}_e(m) \big] \, \in \, - \frac{W_e^\sigma(m)}{Q(m)} + \big( 1 \pm C g_x(t) \big) \frac{\Xt(m)}{\Qt(m)}.
\end{equation}
Next, note that, using the event $\Q(m)$, the final term in~\eqref{eq:Vself1} can be simplified as follows:
$$\frac{\big( 1 \pm g_x(t) \big) \cdot  \frac{r}{2} \cdot g_\sigma(t) \Yt(m)^2 \cdot (V_e^\sigma)^*(m)}{g_\sigma(t) \Yt(m) Q(m)} \, \in \, \big( 1 \pm 2g_x(t) \big) \cdot \frac{4rt}{n^{3/2}} \cdot (V_e^\sigma)^*(m).$$
Moreover, observe that
$$\frac{\Xt(m)}{\Yt(m) \Qt(m)} \, = \, \frac{1}{t \cdot n^{3/2}}.$$
Thus, using~\eqref{eq:Vself1}, it follows that the left-hand side of~\eqref{eq:Vself2} is contained in
\begin{equation}\label{eq:Vself3}
\frac{1}{g_\sigma(t)} \bigg( - \frac{\big( 1 \pm g_{\sigma + 1}(t) \big) V_e^\sigma(m)}{Q(m)} +  \frac{8t}{n^{3/2}} \pm \frac{Cg_x(t)}{t \cdot n^{3/2}} \bigg) - \big( 1 \pm 2g_x(t) \big) \cdot \frac{4rt}{n^{3/2}} \cdot (V_e^\sigma)^*(m).
\end{equation}
Now, observe that\footnote{To simplify the calculation slightly, we ignore in this step a multiplicative error of order $1 + O(1/n)$, which in any case is swallowed by the (much larger) error term in the next step.}  
$$\frac{1}{g_\sigma(t)} \bigg( \frac{V_e^\sigma(m)}{\Qt(m)} - \frac{8t}{n^{3/2}} \bigg) \, = \, \frac{8t}{n^{3/2}} \bigg( \frac{V_e^\sigma(m) - \Yt(m)}{g_\sigma(t) \Yt(m)} \bigg) \, = \, \frac{8t}{n^{3/2}} \cdot (V_e^\sigma)^*(m).$$
It follows that~\eqref{eq:Vself3} is contained in 
\begin{align*}
& \, -\, \frac{8t}{n^{3/2}} \cdot (V_e^\sigma)^*(m) \pm \frac{2\eps \cdot \Yt(m)}{\Qt(m)} - \big( 1 \pm 2g_x(t) \big) \cdot \frac{4rt}{n^{3/2}} \cdot (V_e^\sigma)^*(m) \pm \frac{\omega}{t \cdot n^{3/2}} \\
& \hspace{5cm} \, \subseteq \, \big( 1 \pm 2g_x(t) \big) \big( 8 + 4r \big) \cdot \frac{t}{n^{3/2}} \Big( - (V_e^\sigma)^*(m) \pm O(\eps) \Big)
\end{align*}
as claimed, since 
$g_{\sigma + 1}(t) = \eps \cdot g_\sigma(t)$, $t \ge \omega$ and $g_x(t) \ll \omega \cdot g_\sigma(t)$.

Using~\eqref{eq:Vself2}, we can now easily prove~\eqref{eq:Vself}. Indeed, by Lemma~\ref{A/B} we have 
$$\Ex\big[ \Delta (V^{\sigma}_e)^*(m) \big] = \frac{\Ex\big[ \Delta V_e^{\sigma}(m) \big] - \Delta \Yt(m) - \Ex\big[ (V_e^\sigma)^*(m+1) \cdot \Delta \big( g_\sigma(t) \Yt(m) \big) \,\big| \, G_m \big]}{g_\sigma(t) \Yt(m)},$$
and differentiating gives
$$\Delta \Yt(m) \approx - \, \frac{1}{n^{3/2}} \bigg( \frac{8t^2 - 1}{t} \bigg) \Yt(m) \quad \text{and} \quad \Delta \big( g_\sigma(t) \Yt(m) \big) \approx -\,\frac{1}{n^{3/2}} \bigg( \frac{4t^2 - 1}{t} \bigg) g_\sigma(t) \Yt(m).$$
Moreover, by Lemma~\ref{deltaV}, together with Lemma~\ref{lem:chainstar}, we have 
\begin{equation}\label{eq:deltaVstar}
\big| \Delta (V^{\sigma}_e)^*(m) \big| \, \le \, \frac{3}{g_\sigma(t)} \cdot \left( \frac{| \Delta V^{\sigma}_e(m) |}{\Yt(m)} \,+\, \frac{\log n}{n^{3/2}} \right) \, = \, o(1).
\end{equation}
Combining these bounds, it follows that
$$\Ex\big[ \Delta (V^{\sigma}_e)^*(m) \big] \in \frac{1}{g_\sigma(t)} \bigg( \frac{\Ex\big[ \Delta V_e^\sigma(m) \big]}{\Yt(m)} \,+\, \frac{8t^2 - 1}{t \cdot n^{3/2}} \bigg) + \big( (V_e^\sigma)^*(m) \pm o(1) \big) \cdot \frac{1}{n^{3/2}} \bigg( \frac{4t^2 - 1}{t} \bigg),$$
and hence, by~\eqref{eq:Vself2}, 
$$\Ex\big[ \Delta (V^{\sigma}_e)^*(m) \big] \in \big( 4 + 4r \pm \eps \big) \cdot \frac{t}{n^{3/2}} \Big( - (V_e^\sigma)^*(m) \pm O(\eps) \Big),$$
as required.
\end{proof}

\subsection{The Lines of Peril and Death}

In order to apply the method of Section~\ref{MartSec} to the variables $V^\sigma_e$, we need one more lemma, which bounds the maximum and expected absolute single-step changes in $(V^\sigma_e)^*$.

\begin{lemma}\label{UVgamma}
Let $\omega \cdot n^{3/2} < m \le m^*$. If $\E(m) \cap \U(m) \cap \V(m) \cap \X(m) \cap \Z(m) \cap \Q(m)$ holds, then
$$|\Delta (V^\sigma_e)^*(m)| \, \le \, \ds\frac{C \cdot (\log n)^4}{g_\sigma(t) \Yt(m)} \qquad \textup{and} \qquad \Ex \big[ |\Delta (V^\sigma_e)^*(m)| \big] \, \le \, \ds\frac{C \cdot \log n}{g_\sigma(t) \cdot n^{3/2}}$$
for every $e \in O(G_m)$ and every $k$-short sequence $\sigma \in \{ L,R \}^*$.  
\end{lemma}

Lemma~\ref{UVgamma} is an easy consequence of the following simple observation, combined with the results above.

\begin{lemma}\label{lem:deltaUtrivialbound}
Let $\omega \cdot n^{3/2} < m \le m^*$. If $\E(m) \cap \U(m) \cap \X(m) \cap \Z(m) \cap \Q(m)$ holds, then
$$\Ex\big[ | \Delta U_e^{\sigma}(m) | \big] \, \le \, \frac{C \cdot \Yt(m)^{|\sigma| + 1}}{\Qt(m)}$$
for every $e \in O(G_m)$ and every $k$-short sequence $\sigma \in \{ L,R \}^*$.  
\end{lemma}

\begin{proof}
If $|\sigma| = 0$ then the claimed bound holds trivially, so assume that $|\sigma| \ge 1$. To prove the lemma, observe that a $\sigma$-walk $W$ is destroyed in step $m+1$ of the triangle-free process only if the edge $h$ added in that step is a $Y$-neighbour of an edge of $W$. 
Moreover, since the event $\E(m) \cap \U(m) \cap \X(m) \cap \Z(m)$ holds, and since $\Yt(m)^2 \gg \Xt(m)$ for $t \ge \omega$, it follows from Lemma~\ref{Ueq} that $\Ex\big[ \Delta U_e^{\sigma}(m) \big] < 0$, and hence that the expected number of $\sigma$-walks from $e$ created in step $m+1$ is fewer than the expected number destroyed. Thus
$$\Ex\big[ | \Delta U_e^{\sigma}(m) | \big] \, \le \, \frac{2|\sigma| \cdot U^\sigma_e(m) \cdot \max_f Y_f(m)}{Q(m)} \, \le \, \frac{C \cdot \Yt(m)^{|\sigma| + 1}}{\Qt(m)},$$
by the event $\U(m) \cap \Q(m)$, as required. 
\end{proof}

\begin{proof}[Proof of Lemma~\ref{UVgamma}]
The first inequality follows easily from Lemmas~\ref{lem:chainstar} and~\ref{deltaV} (cf. the proof of~\eqref{eq:deltaVstar}). Indeed, since $\E(m) \cap \V(m) \cap \Z(m)$ holds, we have 
\begin{equation}\label{eq:deltaVstar4gamma}
\big| \Delta (V^{\sigma}_e)^*(m) \big| \, \le \, \frac{3}{g_\sigma(t)} \cdot \left( \frac{| \Delta V^{\sigma}_e(m) |}{\Yt(m)} \,+\, \frac{\log n}{n^{3/2}} \right) \, \le \, \frac{C \cdot (\log n)^4}{g_\sigma(t) \Yt(m)},
\end{equation}
by Lemmas~\ref{lem:chainstar} and~\ref{deltaV}, and since $\sigma$ is $k$-short, as claimed. 

To bound $\Ex \big[ |\Delta (V^\sigma_e)^*(m)| \big]$, observe that (cf. Lemma~\ref{A/B}), for any positive functions $A$ and $B$, 
$$\bigg| \Delta \left( \frac{A(m)}{B(m)} \right) \bigg| \, \le \, \frac{\big| \Delta A(m) \big|}{B(m)} \, + \, \frac{A(m+1) |\Delta B(m)|}{B(m) \cdot B(m+1)}.$$
Applying this to $V_e^\sigma(m) = \big( U_e^{\sigma L}(m) + U_e^{\sigma R}(m) \big) / U_e^{\sigma}(m)$, we obtain
$$\big| \Delta V^{\sigma}_e(m) \big| \, \le \, \frac{\big| \Delta U_e^{\sigma L}(m) \big| + \big| \Delta U_e^{\sigma R}(m) \big| + V_e^\sigma(m+1) \cdot \big| \Delta U_e^\sigma(m) \big|}{U_e^\sigma(m)}.$$
Combining this bound with Lemma~\ref{lem:deltaUtrivialbound}, and using the event $\U(m)$, we obtain
$$\Ex\big[ | \Delta V^{\sigma}_e(m) | \big] \, \le \, \frac{C^2 \cdot \Yt(m)^2}{\Qt(m)},$$
Hence, by~\eqref{eq:deltaVstar4gamma}, we have
$$\Ex\big[ | \Delta (V^{\sigma}_e)^*(m) | \big] \, \le \, \frac{2}{g_\sigma(t)} \cdot \left( \frac{\Ex\big[ | \Delta V^{\sigma}_e(m) | \big]}{\Yt(m)} \,+\, \frac{\log n}{n^{3/2}} \right) \, \le \, \frac{C \cdot \log n}{g_\sigma(t) \cdot n^{3/2}},$$ 
as required.
\end{proof}

We are now ready to prove Proposition~\ref{Vprop}. 

\begin{proof}[Proof of Proposition~\ref{Vprop}]
We shall bound the probabilities of the event 
\begin{equation}\label{eq:UVfailingfirst}
\U(m-1) \cap \V(m)^c \cap \K^\Y(m-1)
\end{equation}
for each $\omega \cdot n^{3/2} < m \le m^*$. In order to do so, we shall apply our usual martingale method to the variables $V_e^\sigma$ for each edge $e$ and each $k$-short sequence $\sigma \in \{L,R\}^*$.  

As usual, we begin by choosing a family of parameters as in Definition~\ref{def:reasonable}. Set $\K(m) = \E(m) \cap \U(m) \cap \V(m) \cap \X(m) \cap \Y(a) \cap \Z(m) \cap \Q(m)$ and $I = [a,b] = [\omega \cdot n^{3/2},m^*]$, let 
$$\alpha(t) \, = \, \frac{C \cdot (\log n)^4}{g_\sigma(t) \Yt(m)} \qquad \text{and} \qquad \beta(t) \, = \, \frac{C \cdot \log n}{g_\sigma(t) \cdot n^{3/2}},$$
and set $\lambda = C$ and $\delta = \eps$ and $h(t) = t \cdot n^{-3/2}$. We claim that $(\lambda,\delta;g_\sigma,h;\alpha,\beta;\K)$ is a reasonable collection, and that $V_e^\sigma$ satisfy the conditions of Lemma~\ref{lem:self:mart}.

The first statement follows easily, since $\alpha$ and $\beta$ are clearly $\lambda$-slow, and the bounds $\min\big\{ \alpha(t), \, \beta(t), \, h(t) \big\} \ge \frac{\delta t}{n^{3/2}}$ and $\alpha(t) \le \eps$ follow from the fact that $n^{-1/4} \le g_\sigma(t) \le 1$ for every $t \le t^*$. To prove the second, we need to show that the variables $V_e^\sigma$ are $(g_\sigma,h;\K)$-self-correcting, which follows from Lemma~\ref{Vself}, and that, for every $\omega \cdot n^{3/2} < m \le m^*$, if $\K(m)$ holds then 
$$|\Delta (V^\sigma_e)^*(m)| \le \alpha(t) \qquad  \text{and} \qquad \Ex\big[ |\Delta (V^\sigma_e)^*(m)| \big] \le \beta(t),$$
which follows from Lemma~\ref{UVgamma}. Moreover, the bound $|(V^\sigma_e)^*(a)| < 1/2$ follows 
from the event $\Y(a)$, since $f_y(\omega) \Yt(n^{3/2}) \ll g_y(\omega) \Yt(a)$ if $\omega(n) \to \infty$ sufficiently slowly.

Finally, observe that
$$\alpha(t) \beta(t) n^{3/2} \, \le \, \frac{C \cdot (\log n)^4}{g_\sigma(t) \Yt(m)} \cdot \frac{C \cdot \log n}{g_\sigma(t)} \, \le \, \frac{1}{(\log n)^3}$$
for every $\omega < t \le t^*$, since $g_\sigma(t)^2 \Yt(m) \ge \eps^{3k^2} \cdot \omega \cdot (\log n)^8$. By Lemma~\ref{lem:self:mart}, it follows that
$$\Pr\Big( \V(m)^c \cap \K(m-1) \text{ for some $m \in [a,b]$} \Big) \, \le \, n^7 \exp\Big( - \delta' (\log n)^3 \Big) \, \le \, n^{-2C \log n},$$
where we summed over edges $e \in E(K_n)$ and $k$-short sequences $\sigma \in \{L,R\}^*$ the probability that $e \in O(G_m)$ and $(V_e^\sigma)^*(m) > 1$. This gives us a bound on the probability that one of the events in~\eqref{eq:UVfailingfirst} occurs for the first time at step $m'$ of the triangle-free process; summing over choices of $m' \le m$, this gives a bound on the probability of the event 
$$\U(m-1) \cap \V(m)^c \cap \K^\Y(m-1).$$ 
Using Proposition~\ref{Uprop} to bound the probability of the event $\U(m-1)^c \cap \K^\Y(m-1)$, we obtain the claimed bound  on the probability of the event $\V(m)^c \cap \K^\Y(m-1)$, as required.
\end{proof}

Finally, note that Proposition~\ref{Yprop} follows immediately from Proposition~\ref{Vprop}.

\pagebreak

\section{Whirlpools and Lyapunov functions}\label{XYQsec}

In this section we shall prove the following theorem, which together with the results of the previous three sections, implies Theorems~\ref{Qthm} and~\ref{XbYbthm} and hence the lower bound in Theorem~\ref{triangle}. Recall that the events $\X(m)$, $\Y(m)$ and $\Q(m)$ were defined in Definition~\ref{def:events:XYQ}. 

\begin{thm}\label{XYQthm}
For every $\omega \cdot n^{3/2} < m \le m^*$, with probability at least $1 - n^{- 2\log n}$ either $\big( \X(m-1) \cap \Y(m-1) \cap \Q(m-1) \big)^c$ or the following holds:
\begin{equation}\label{eq:XYQthm}
(a) \;\; \frac{Q(m)}{\Qt(m)} \in 1 \pm g_q(t), \quad (b) \;\; \frac{\Xb(m)}{\Xt(m)} \in 1 \pm g_q(t), \quad (c) \;\; \frac{\Yb(m)}{\Yt(m)} \in 1 \pm g_q(t).
\end{equation}
\end{thm}

The proof of Theorem~\ref{XYQthm} is roughly as follows. We shall first show how the normalized errors of $\Xb$, $\Yb$ and $Q$ depend on one another; in particular, we shall show that $\Xb$ is self-correcting (given bounds on $\Yb$ and $Q$), whereas $\Yb$ and $Q$ have a more complicated two-dimensional interaction which resembles a whirlpool, see Figure~6.1.  

Since the eigenvalues of this interaction turn out to be negative (in fact, both are equal to $-1$), we will be able to define a  Lyapunov function, $\Lambda(m)$, which is self-correcting and is bounded if and only if the normalized errors of both $\Yb$ and $Q$ are bounded. The required bounds on $\Xb$, $\Yb$ and $Q$ then follow easily by our usual Line of Peril / Line of Death argument.

\subsection{Whirlpools}\label{WhirlSec}

The first step in the proof of Theorem~\ref{XYQthm} is the following lemma, which shows how the normalized errors of $\Xb(m)$, $\Yb(m)$ and $Q(m)$ depend on one another. 
Set 
$$\Xs(m) \, = \,  \frac{\Xb(m) - \tilde{X}(m)}{g_q(t) \tilde{X}(m)}, \quad \Ys(m) \, = \,  \frac{\Yb(m) - \Yt(m)}{g_q(t) \Yt(m)} \quad \text{and} \quad \Qs(m) \, = \,  \frac{Q(m) - \Qt(m)}{g_q(t) \Qt(m)}$$
for each $m \in \N$.

\begin{lemma}\label{whirlpool}
Let $\omega \cdot n^{3/2} < m \le m^*$. If $\X(m) \cap \Y(m) \cap \Q(m)$ holds, then
\begin{itemize}
\item[$(a)$] $\Ex \big[ \Delta \Qs(m) \big] \, \in \, \ds\frac{4t}{n^{3/2}} \Big( - 2 \Ys(m) + \Qs(m) \pm o(1) \Big)$.
\item[$(b)$] $\Ex \big[ \Delta \Ys(m) \big] \, \in \, \ds\frac{4t}{n^{3/2}} \Big( - 3 \Ys(m) + 2\Qs(m) \pm o(1) \Big)$.
\item[$(c)$] $\Ex \big[ \Delta \Xs(m) \big] \, \in \, \ds\frac{4t}{n^{3/2}} \bigg( - \Xs(m) - 4\Ys(m) + 4\Qs(m) \pm o(1) \Big)$.
\end{itemize}
\end{lemma}

We shall only sketch the proof of Lemma~\ref{whirlpool} here, and defer the (lengthy, but fairly straightforward) details to the Appendix. We first recall the following lemmas, which were already stated in Section~\ref{sketchSec}. 

\begin{lemma}\label{Qeq}
For every $m \in \N$, 
$$\Ex \big[ \Delta Q(m) \big] \, = \, - \, \Yb(m) - 1.$$
\end{lemma}

\begin{lemma}\label{Ybeq}
Let $\omega \cdot n^{3/2} < m \le m^*$. If $\X(m) \cap \Y(m) \cap \Q(m)$ holds, then
$$\Ex \big[ \Delta \Yb(m) \big] \, \in \,  \frac{1}{Q(m)} \Big( - \Yb(m)^2 + \Xb(m) - 2 \cdot \Var\big(Y_e(m)\big) \pm  O\big( \Yt(m) \big) \Big).$$
\end{lemma}

\begin{lemma}\label{Xbeq}
Let $\omega \cdot n^{3/2} < m \le m^*$. If $\X(m) \cap \Y(m) \cap \Q(m)$ holds, then
$$\Ex \big[ \Delta \Xb(m) \big] \, \in \, \frac{1}{Q(m)} \bigg( - 2 \cdot \Xb(m) \Yb(m) - 3 \cdot \Cov(X,Y) \pm O\big(  \Yt(m)^2 \big) \bigg).$$
\end{lemma}

Lemma~\ref{Qeq} is trivial, since if edge $e$ is chosen in step $m+1$, then $\Delta Q(m) = - Y_e(m) - 1$. In order to prove Lemmas~\ref{Ybeq} and~\ref{Xbeq}, we use the variables
$$\YY(m) = \sum_{e \in Q(m)} Y_e(m) \qquad \text{and} \qquad \XX(m) = \sum_{e \in Q(m)} X_e(m),$$
which are exactly twice the number of edges in the $Y$-graph, and six times the number of open triangles in $G_m$, respectively. Using the fact that the $Y$-graph is triangle-free (which follows since $G_m$ is triangle-free), it is not hard to show that if edge $e$ is added at step $m+1$, then
\begin{equation}\label{eq:deltaYY}
\Delta \YY(m) \, = \, X_e(m) - 2 \sum_{f \in Y_e(m)} Y_f(m),
\end{equation}
and moreover
\begin{equation}\label{eq:deltaXX}
\Delta \XX(m) \, \in \, - \, 3 \sum_{f \in Y_e(m)} X_f(m) \, \pm \, O\big(  \Yt(m)^2 \big)
\end{equation}
assuming $\X(m) \cap \Y(m) \cap \Q(m)$ holds. Summing over open edges $e \in Q(m)$ gives 
$$\Ex\big[ \Delta \YY(m) \big] \, = \, \frac{\XX(m)}{Q(m)} \, - \, \frac{2}{Q(m)}\sum_{e \in Q(m)} Y_e(m)^2$$
and
$$\Ex\big[ \Delta \XX(m) \big] \, \in \, - \frac{3}{Q(m)} \sum_{f \in Q(m)} X_f(m) \cdot Y_f(m) \, \pm \, O\big(  \Yt(m)^2 \big).$$
Finally, a straightforward calculation shows that if $\X(m) \cap \Y(m) \cap \Q(m)$, then 
$$\Ex \big[ \Delta \Yb(m) \big] \, \in \, \frac{\Ex \big[ \Delta \YY(m) \big]}{Q(m)} \, + \, \frac{\Yb(m)^2 \pm O\big( \Yt(m) \big)}{Q(m)}$$
and 
$$\Ex \big[ \Delta \Xb(m) \big] \, \in \, \frac{\Ex \big[ \Delta \XX(m) \big]}{Q(m)} \, + \, \frac{\Xb(m)\Yb(m) \pm O\big( \Xt(m) \big)}{Q(m)}.$$
Lemmas~\ref{Ybeq} and~\ref{Xbeq} now follow easily, see the Appendix~\cite{App} for details.

To deduce Lemma~\ref{whirlpool}, we simply use Lemma~\ref{A/B} to differentiate $\Xs$, $\Ys$ and $\Qs$, and observe that $\Xt(m) \ll \Yt(m)^2$, since $t \ge \omega$, and that 
$$\Var\big( Y(m) \big) \ll g_q(t) \Yt(m)^2 \quad \text{and} \quad \Cov\big( X(m), Y(m) \big) \ll g_q(t) \Xt(m) \Yt(m),$$ 
since $\X(m)$ and $\Y(m)$ hold. The lemma now follows via a straightforward calculation. 

\begin{figure}[h]
\includegraphics[scale=0.75]{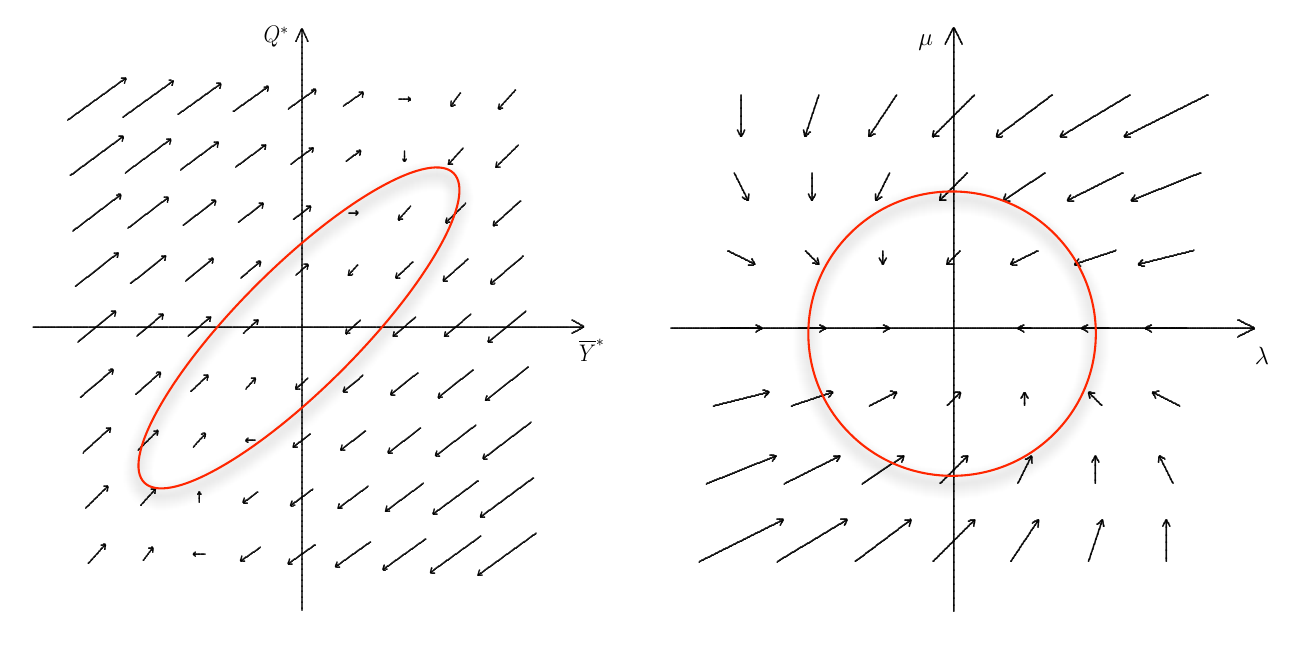}
\caption{The whirlpools of $(\Ys,\Qs)$ and $(\lambda,\mu)$.}
\end{figure}

\subsection{Lyapunov functions}

Note that, by Lemma~\ref{whirlpool}, although $\Xs$ and $\Ys$ are self-correcting, $\Qs$ is not. We shall therefore have to define a slightly more complicated martingale than in the previous sections. Fortunately, the classical work of Lyapunov~\cite{Lyap} provides us with exactly the function of $\Ys$ and $\Qs$ which we need. Indeed, let us rewrite the vector $\big( \Ys, \Qs \big)$ using the change of basis
\begin{equation}\label{eq:def:lambda}
\left( \begin{array}{c}
\Ys \\
\Qs
\end{array} \right) \, = \, 
\eps \left( \begin{array}{cc}
4 & 5 \\
4 & 3
\end{array} \right)
\left( \begin{array}{c}
\lambda \\
\mu
\end{array} \right),
\end{equation}
and define a new parameter 
$$\Lambda(m) \, = \, \lambda(m)^2 + \mu(m)^2.$$
The following result will imply the bounds on $\Yb(m)$ and $Q(m)$ in Theorem~\ref{XYQthm}.

\begin{prop}\label{Lambprop}
For each $\omega \cdot n^{3/2} < m \le m^*$, set $\K(m) = \X(m) \cap \Y(m) \cap \Q(m)$. Then
$$\Pr\Big( \big( \Lambda(m) > 1 \big) \cap \K(m-1) \textup{ for some $\omega \cdot n^{3/2} < m \le m^*$} \Big) \, \le \, n^{-C \log n}.$$
\end{prop}

We will prove Proposition~\ref{Lambprop} using the our usual martingale method; as always, we shall need to bound the maximum possible and expected single-step changes.

\begin{lemma}\label{XYQalpha}
Let $\omega \cdot n^{3/2} < m \le m^*$, and suppose that $\X(m) \cap \Y(m) \cap \Q(m)$ holds. Then
\begin{equation}\label{eq:XYQstar1}
|\Delta \Xs(m)| + |\Delta \Ys(m)| + | \Delta \Qs(m)| \, \le \, \frac{(\log n)^3}{g_q(t) \cdot n^{3/2}},
\end{equation}
and hence
\begin{equation}\label{eq:XYQstar2}
|\Delta \Lambda(m) | \, \le \, \frac{(\log n)^4}{g_q(t) \cdot n^{3/2}} \qquad \text{and} \qquad  \Ex\big[ | \Delta \Lambda(m) | \big] \, \le \, \frac{(\log n)^4}{g_q(t) \cdot n^{3/2}}.
\end{equation}
\end{lemma}

Once again, we only sketch the proof, and postpone the (easy) details to the Appendix~\cite{App}. 

\begin{proof}[Sketch of proof]
The inequality~\eqref{eq:XYQstar1} follows easily from~\eqref{eq:deltaYY} and~\eqref{eq:deltaXX}, together with bounds given by Lemma~\ref{lem:chainstar} (cf.~\eqref{eq:deltaVstar}) for $\Xb$, $\Yb$ and $Q$, via a straightforward calculation (see the Appendix). To deduce~\eqref{eq:XYQstar2}, simply observe that
$$|\Delta \Lambda(m) | \, \le \, 2 \Big( \big| \lambda(m) \cdot \Delta \lambda(m) \big| + \big| \mu(m) \cdot \Delta \mu(m) \big| \Big) + |\Delta \lambda(m)|^2 + |\Delta \mu(m)|^2,$$
and use~\eqref{eq:XYQstar1} and the event $\Q(m)$ to bound the various terms. 
\end{proof}

We also need to show that $\Lambda(m)$ is self-correcting.

\begin{lemma}\label{Lambself}
If $\X(m) \cap \Y(m) \cap \Q(m)$ holds, then 
\begin{equation}\label{eq:lamb}
\Ex\big[ \Delta \Lambda(m) \big] \, \le \, \frac{4t}{n^{3/2}} \Big( - \Lambda(m) + \eps \Big).
\end{equation}
\end{lemma}

\begin{proof}
Note first that 
$$\frac{1}{\eps} \left( \begin{array}{c}
-3\Ys + 2\Qs \\
-2\Ys + \Qs
\end{array} \right)  \, = \, 
\left( \begin{array}{cc}
-3 & 2 \\
-2 & 1
\end{array} \right)
\left( \begin{array}{cc}
 4 & 5 \\
4 & 3
\end{array} \right) 
\left( \begin{array}{c}
\lambda \\
\mu
\end{array} \right) \, = \, 
\left( \begin{array}{c}
-4\lambda - 9\mu \\
-4\lambda - 7\mu
\end{array} \right),$$
and hence, by~\eqref{eq:def:lambda} and Lemma~\ref{whirlpool},
\begin{align*}
4 \cdot \Ex\big[ \Delta \lambda(m) \big] + 5 \cdot \Ex\big[ \Delta \mu(m) \big] & = \, \frac{1}{\eps} \cdot \Ex \big[ \Delta \Ys(m) \big] \, \in \, \ds\frac{4t}{n^{3/2}} \Big( - 4 \lambda(m) - 9\mu(m) \pm o(1) \Big)\\
4 \cdot\Ex\big[ \Delta \lambda(m) \big] + 3 \cdot \Ex\big[ \Delta \mu(m) \big] & = \, \frac{1}{\eps} \cdot \Ex \big[ \Delta \Qs(m) \big] \, = \, \ds\frac{4t}{n^{3/2}} \Big( - 4 \lambda(m) - 7\mu(m) + o(1) \Big).
\end{align*}
Note that $|\lambda| + |\mu| = O(1/\eps)$, since the event $\Q(m)$ holds. It follows that
$$\frac{n^{3/2}}{4t} \cdot \Ex\big[ \Delta \lambda(m) \big] \, \in \, - \lambda(m) - \mu(m) \pm o(1) \quad \text{and} \quad \frac{n^{3/2}}{4t} \cdot  \Ex\big[ \Delta \mu(m) \big] \, \in \, - \mu(m) \pm o(1),$$
and hence
\begin{multline*}
 \frac{n^{3/2}}{2t} \cdot \left( \begin{array}{c}
\Ex\big[ \Delta \lambda(m) \big] \\[+0.5ex]
\Ex\big[ \Delta \mu(m) \big] 
\end{array} \right) \cdot
\left( \begin{array}{c}
\lambda(m) \\
\mu(m) 
\end{array} \right)
 \, \le \, - 2\lambda(m)^2 - 2\lambda(m) \mu(m) - 2\mu(m)^2 + o(1) \\
 \, = \, - \Big( \lambda(m)^2 + \mu(m)^2 + \big(\lambda(m) + \mu(m) \big)^2 \Big) + o(1) \, \le \, - \Lambda(m) + o(1).
\end{multline*}
Noting that
$$\Delta \Lambda(m) \, = \, \big( \Delta \lambda(m) \big)^2 + \big( \Delta \mu(m) \big)^2 + 2 \Big( \lambda(m) \cdot \Delta \lambda(m) + \mu(m) \cdot \Delta \mu(m) \Big),$$
and using~\eqref{eq:XYQstar1} to bound $\Delta \lambda(m)$ and $\Delta \mu(m)$, the bound~\eqref{eq:lamb} follows.
\end{proof}

We can now prove Proposition~\ref{Lambprop}, using the method of Section~\ref{MartSec}. To be precise, we shall show that $\Lambda$ has all of the properties required of $A^*$ in the statement of Lemma~\ref{lem:self:mart}, and deduce that therefore the conclusion of the lemma holds with $A^* = \Lambda$. 

\begin{proof}[Proof of Proposition~\ref{Lambprop}]
Let $\K(m) = \X(m) \cap \Y(m) \cap \Q(m)$ for each $m \in [m^*]$, set $I = [a,b] = [\omega \cdot n^{3/2},m^*]$ and
$$\alpha(t) \, = \, \beta(t) \, = \, \frac{(\log n)^4}{g_q(t) \cdot n^{3/2}},$$
and observe that $\alpha$ and $\beta$ are $C$-slow and satisfy $\frac{\eps t}{n^{3/2}} \le \alpha(t) = \beta(t) \le \eps^2$. By Lemma~\ref{Lambself}, if $\K(m)$ holds and $\Lambda(m) \ge 1/2$ then  
$$\Ex\big[ \Delta \Lambda(m) \big] \, \le \, -\,\frac{t}{n^{3/2}},$$
which (since $\Lambda(m) \ge 0$) is exactly the required self-correction condition. Moreover, by Lemma~\ref{XYQalpha}, we have 
$$|\Delta \Lambda(m)| \le \alpha(t) \qquad  \text{and} \qquad \Ex\big[ |\Delta \Lambda(m)| \big] \le \beta(t)$$
for every $m \in [a,b]$ for which $\K(m)$ holds. It now simply remains to observe that
\begin{equation}
\left( \begin{array}{c}
\lambda \\
\mu
\end{array} \right) \, = \, 
\frac{1}{8\eps} \left( \begin{array}{rr}
-3 & 5 \\
4 & -4
\end{array} \right)
\left( \begin{array}{c}
\Ys\\
\Qs
\end{array} \right),
\end{equation}
and that if $\K(a)$ holds then $|\Ys(a)| + |\Qs(a)| \le \eps^2$, since $f_y(\omega) \Yt(n^{3/2}) \ll g_q(\omega) \Yt(a)$, and similarly $f_y(\omega) \Qt(n^{3/2}) \ll g_q(\omega) \Qt(a)$. Hence, 
$$\Lambda(a) \, = \, \lambda(a)^2 + \mu(a)^2  \, \le \, \frac{O(1)}{\eps^2} \cdot \Big( | \Ys(a) | + | \Qs(a) | \Big)^2 < \, \frac{1}{2},$$
as required. Finally, observe that
$$\alpha(t) \beta(t) n^{3/2} \, = \, \frac{(\log n)^8}{g_q(t)^2 \cdot n^{3/2}} \, \le \, \frac{(\log n)^2}{n}$$
for every $\omega < t \le t^*$. By Lemma~\ref{lem:self:mart}, it follows that
$$\Pr\Big( \big( \Lambda(m) > 1 \big) \cap \K(m-1) \text{ for some $m \in [a,b]$} \Big) \, \le \, n^4 \exp\bigg( - \frac{\delta' n}{(\log n)^2} \bigg) \, \le \, n^{-C \log n},$$
as required.
\end{proof}

We are finally ready to prove Theorem~\ref{XYQthm}.

\begin{proof}[Proof of Theorem~\ref{XYQthm}]
The bounds on $\Yb(m)$ and $Q(m)$ follow easily from Proposition~\ref{Lambprop}, since if $\Lambda(m) \le 1$ then $|\Ys(m)| + |\Qs(m)| \le 20\eps$. Set
$$\K(m) \, = \, \X(m) \cap \Y(m) \cap \Q(m) \cap \Big( |\Ys(m)| + |\Qs(m)| \le 20 \eps \Big)$$
for each $m \in [m^*]$, and let $I = [a,b] = [\omega \cdot n^{3/2},m^*]$. We claim that
$$\Pr\Big( \big( \Xs(m) > 1 \big) \cap \K(m-1) \text{ for some $m \in [a,b]$} \Big) \, \le \, n^{-C \log n}.$$
As usual, we shall apply the method of Section~\ref{MartSec}. Indeed, set $h(t) = t \cdot n^{-3/2}$ and
$$\alpha(t) \, = \, \beta(t) \, = \, \frac{(\log n)^3}{g_q(t) \cdot n^{3/2}},$$
and observe that $\alpha$ and $\beta$ are $C$-slow and satisfy $\frac{\eps t}{n^{3/2}} \le \alpha(t) = \beta(t) \le \eps^2$. Observe that $\Xb$ is $(g_q,h;\K)$-self-correcting on $[a,b]$ since if $\K(m)$ holds then 
$$\Ex \big[ \Delta \Xs(m) \big] \, \in \, \ds\frac{4t}{n^{3/2}} \Big( - \Xs(m) \pm \eps \Big).$$
by Lemma~\ref{whirlpool}, and moreover
$$|\Delta \Xs(m)| \le \alpha(t) \qquad  \text{and} \qquad \Ex\big[ |\Delta \Xs(m)| \big] \le \beta(t)$$
by Lemma~\ref{XYQalpha}. Note also that  if $\K(a)$ holds then $|\Xs(a)| < 1/2$, since $f_x(\omega) \Xt(n^{3/2}) \ll g_q(\omega) \Xt(a)$ if $\omega(n) \to \infty$ sufficiently slowly. Finally, observe that
$$\alpha(t) \beta(t) n^{3/2} \, \le \, \frac{(\log n)^6}{g_q(t)^2 \cdot n^{3/2}} \, \ll \, \frac{1}{n}$$
for every $\omega < t \le t^*$. By Lemma~\ref{lem:self:mart}, it follows that
$$\Pr\Big( \big( \Xs(m) > 1 \big) \cap \K(m-1) \text{ for some $m \in [a,b]$} \Big) \, \le \, n^4 e^{-n} \, \le \, n^{-C \log n},$$
as claimed. Combining this bound with Proposition~\ref{Lambprop}, the theorem follows. 
\end{proof}

\subsection{The proof of Theorems~\ref{Qthm},~\ref{XbYbthm},~\ref{Ythm},~\ref{Vthm} and~\ref{NFthm}}

We end this section by deducing the main results of Section~\ref{sketchSec}.

\begin{thm}\label{finalthm}
With high probability, the events $\E(m)$, $\V(m)$, $\X(m)$, $\Y(m)$, $\Z(m)$ and $\Q(m)$ all hold for every $m \le m^*$. Or, more precisely,
\begin{equation}\label{eq:finalthm}
\Pr\Big( \E(m^*) \cap \V(m^*) \cap \X(m^*) \cap \Y(m^*) \cap \Z(m^*) \cap \Q(m^*) \Big) \, \ge \, 1 - n^{-\log n}
\end{equation}
for all sufficiently large $n \in \N$.
\end{thm}

\begin{proof}
Consider the\footnote{Or, more precisely, \emph{a} first event, since more than one might fail in the same step.} first of the events to go astray, and suppose that it does so in step $m$ of the triangle-free process. The probability of this event is controlled by:
\begin{itemize}
\item[--] Theorem~\ref{EEthm} if $\E(m)^c$ holds, see Definition~\ref{def:events:E};\smallskip
\item[--] Proposition~\ref{lem:landbeforetime} if either $\X(m)^c$ or $\Y(m)^c$ holds and $t \le \omega$, see Definition~\ref{def:events:XYQ};\smallskip
\item[--] Proposition~\ref{prop:YimpliesQ} if $\Q(m)^c$ holds and $t \le \omega$, see Definition~\ref{def:events:XYQ};\smallskip
\item[--] Proposition~\ref{Xprop} if $\X(m)^c$ holds and $t > \omega$, see Definition~\ref{def:events:XYQ};\smallskip
\item[--] Proposition~\ref{Yprop} $\Y(m)^c$ holds and $t > \omega$, see Definition~\ref{def:events:XYQ}; \smallskip
\item[--] Proposition~\ref{Vprop} $\V(m)^c$ holds and $t > \omega$, see Definition~\ref{def:events:UV}; \smallskip
\item[--] Proposition~\ref{Zprop} if $\Z(m)^c$ holds; and, finally, \smallskip
\item[--] Theorem~\ref{XYQthm} if $\Q(m)^c$ holds and $t > \omega$, see Definition~\ref{def:events:XYQ}.
\end{itemize} 
Summing the probabilities in these statements, we obtain~\eqref{eq:finalthm}, as required.
\end{proof}

Theorems~\ref{Qthm},~\ref{XbYbthm},~\ref{Ythm},~\ref{Vthm} and~\ref{NFthm} all follow immediately from Theorem~\ref{finalthm}.

\section{Independent sets and maximum degrees in $G_{n,\triangle}$}\label{indepSec}

In this section we shall control the maximum degree and the independence number of the graph $G_{n,\triangle}$. Recall that $\eps > 0$ is an arbitrary, sufficiently small constant, set $\gamma = 10 \sqrt{\eps}$, and choose $\delta = \delta(\eps) > 0$ sufficiently small.\footnote{Since $C = C(\eps)$ was arbitrary, we may assume that $C = C(\eps,\delta)$ is sufficiently large.} Moreover, let $n_1(\eps,\delta,C,\omega) \in \N$ be sufficiently large.

We shall prove the following propositions. Together with Theorem~\ref{finalthm}, they imply Theorem~\ref{thm:indepsets}, and hence complete the proofs of Theorems~\ref{triangle} and~\ref{R3k}. 

\begin{prop}\label{prop:maxdeg}
If $n \ge n_1(\eps,\delta,C,\omega)$, then with probability at least $1 - e^{-\sqrt{n}}$, either   
$$\Delta\big( G_{n,\triangle} \big) \, \le \, \bigg( \frac{1}{\sqrt{2}} + \gamma \bigg) \sqrt{ n \log n },$$
or $\E(m) \cap \Y(m) \cap \Z(m) \cap \Q(m)$ fails to hold for some $m \le m^*$.
\end{prop}

\begin{prop}\label{prop:indep}
If $n \ge n_1(\eps,\delta,C,\omega)$, then with probability at least $1 - e^{-\sqrt{n}}$ either     
$$\alpha\big( G_{n,\triangle} \big) \, \le \, \big( \sqrt{2} + \gamma \big) \sqrt{ n \log n },$$
or $\E(m) \cap \Y(m) \cap \Z(m) \cap \Q(m)$ fails to hold for some $m \le m^*$.
\end{prop}

Both propositions will follow by essentially the same argument, with a few key changes. We begin by outlining the main ideas of the proof.

\subsection{A sketch of the proof}

The basic idea behind Proposition~\ref{prop:indep} is that a \emph{typical} set of vertices $S$ will contain roughly ${|S| \choose 2} e^{-4t^2}$ open edges at time $t$. If this were to hold for every set $S$, then the proof would be easy, since the probability of choosing an edge inside $S$ would be (roughly) $|S|^2 / n^2$ in each step, and thus the expected number of independent sets of size $s$ in $G_{m^*}$ would be roughly
$${n \choose s} \left( 1 - \frac{s^2}{n^2} \right)^{m^*} \, \approx \, \bigg( \frac{n}{s} \cdot e^{-s m^* / n^2} \bigg)^s,$$ 
which tends to zero if $s > \big( \sqrt{2} + \gamma \big) \sqrt{ n \log n }$. Indeed, an easy application of our usual martingale method (see Lemma~\ref{indep:d}, below) will allow us to make this calculation rigorous for sets $S$ which contain at most $n^\delta$ elements of each neighbourhood in $G_{m^*}$. On the other hand, for those sets $S$ which intersect some neighbourhood in at least this many vertices, we shall have to do something quite different, see below.

For Proposition~\ref{prop:maxdeg} it is quite tricky even to come up with the right heuristic. One natural approach is to note that if the event $\E(m^*)$ holds, then every vertex $v$  in $G_{m^*}$ has degree roughly $2m^* / n$, and open degree roughly $n e^{-4(t^*)^2} \approx n^{1/2 + \eps}$. Since a set of size $n^{1/2 + \eps}$ is unlikely to contain an independent set of size $\gamma \sqrt{n \log n}$, and there are only $n$ choices for $v$, we should be done. However, the events involved in this calculation are not independent, and we have not succeeded in making this argument rigorous.  

Instead we shall use the fact that, again by the event $\E(m^*)$, every vertex $v$ has degree roughly $2m / n$ in $G_{m}$ for \emph{every} $m \le m^*$. It follows that, if $S$ is the neighbourhood of a vertex $v$ in $G_{n,\triangle}$, then we would expect $S$ to contain roughly $\big( {|S| \choose 2} - \frac{2m^2}{n^2} \big) e^{-4t^2}$ open edges at time~$t$, since roughly $2m^2/n^2$ edges of $S$ will have been closed by $v$ by this time. Moreover, we can approximate the probability that $S \subseteq N(v)$ by summing over the sequence of steps at which the edges are added; the probability that a particular open edge $e$ is added in step $m$ is exactly $1 / Q(m)$. Since we have about $(m^*)^{2m^*/n}$ choices for this sequence, we should obtain an upper bound on the probability that $N(v) = S$ of roughly
$$(m^*)^{2m^*/n} \bigg( \prod_{j = 1}^{2m^*/n} \frac{1}{Q(jn/2)} \bigg) \prod_{m=1}^{m^*} \left( 1 - \frac{|S|^2}{n^2} + \frac{4m^2}{n^4} \right) \, \approx \, \left( \frac{m^*}{n^{11/6}} \exp\left( \frac{2(t^*)^2}{3} - \frac{|S|^2}{2n} \right) \right)^{2m^*/n}$$
where we used Lemma~\ref{lem:Qprod} to estimate the first product. Summing over choices of $v$ and $S$, we would thus obtain an upper bound on the probability that there exists a vertex of degree at least $s = (2 + \gamma) m^* / n$ in $G_{n,\triangle}$ of
$$\left( \frac{n}{s} \cdot \frac{m^*}{n^{11/6}} \cdot e^{2(t^*)^2/3 - s^2/2n} \right)^{2m^*/n} \, \approx \, n^{-\gamma s}$$
as required, since $e^{2(t^*)^2} \approx n^{1/4 - \eps}$ and $s^2 / 2n \approx \big( \frac{1}{4} + \gamma \big) \log n$. 

Once again, the outline above can only be made rigorous if there are no other vertices which send many edges into $S$. In the next subsection we shall describe how we deal with the other cases.

\subsection{Partitioning the bad events}

As the reader will have noticed from the discussion in the previous subsection, it is not true that the number of open edges in $S$ is well-behaved for every set $S$ of size $\Theta(m^*/n)$; indeed, those sets which happen to have a large intersection with the neighbourhood(s) of some (or many) vertices will have fewer open edges than expected. This motivates the following definition.\footnote{We remark that if $e(G_{n,\triangle}) < m^*$, i.e., if the triangle-free process ends before step $m^*$, then we may define $G_{m^*} = G_{n,\triangle}$. Since in that case the event $\Q(m^*)^c$ holds, this choice does not affect the validity of either Proposition~\ref{prop:maxdeg} or~\ref{prop:indep}.} 

\begin{defn}\label{def:aJ}
Given a set $S \subseteq V(G_{n,\triangle})$ and $\delta > 0$, define $J = J(S,\delta) = \{v_1,\ldots,v_{|J|}\}$ and $\a = \a(S,\delta) = (a_1,\ldots,a_{|J|})$ to be the following random variables:
$$J(S,\delta) \, = \, \Big\{ v \in V(G_{n,\triangle}) \,:\, |N_{G_{m^*}}(v) \cap S| \ge n^\delta \Big\},$$
and $a_j = |N_{G_{m^*}}(v_j) \cap S|$, where the labels are chosen so that $a_1 \ge \dots \ge a_{|J|}$.
\end{defn}

Given a set $S \subseteq V(G_{n,\triangle})$ and $m \in \N$, let $\I(S,m)$ denote the event that $S$ is an independent set in $G_m$. It is easy to see that the following events form a cover of $\I(S,m^*)$.

\begin{defn}\label{def:ABCD}
Given a set $S \subseteq V(G_{n,\triangle})$ and $\delta > 0$, define:
\begin{itemize}
\item[$(i)$] $\A(S,\delta) \, = \, \I(S,m^*) \cap \Big\{ o(G_m[S]) \ge (1 - \eps) {|S| \choose 2} e^{-4t^2} \text{ for every } m \le m^* \Big\}$.
\item[$(ii)$] $\B(S,\delta) \, = \, \Big\{ \sum_J a_j^2 < \delta |S|^2 \Big\} \cap \Big\{ o(G_m[S]) < (1 - \eps) {|S| \choose 2} e^{-4t^2} \text{ for some } m \le m^* \Big\}$.
\item[$(iii)$] $\C(S,\delta) \, = \, \I(S,m^*) \cap \Big\{  \sum_J a_j^2 \ge \delta |S|^2  \Big\} \cap \Big\{ \sum_J a_j < n^{1/2 + 2\delta} \Big\}$.
\item[$(iv)$] $\D(S,\delta) = \, \I(S,m^*) \cap \Big\{ \sum_J a_j \ge n^{1/2 + 2\delta} \Big\}$.
\end{itemize}
\end{defn}

We shall bound from above the probability that both $\E(m^*) \cap \Y(m^*) \cap \Z(m^*) \cap \Q(m^*)$ and the event  
\begin{equation}\label{IinABCD}
\bigcup_{S \subseteq V(G_{n,\triangle}) \,:\, |S| = s} \A(S,\delta) \cup \B(S,\delta) \cup \C(S,\delta) \cup \D(S,\delta)  \, \supseteq \, \bigcup_{S \subseteq V(G_{n,\triangle}) \,:\, |S| = s} \I(S,m^*)
\end{equation}
hold, where $s = \big( \sqrt{2} + \gamma \big) \sqrt{ n \log n }$. The easiest of the probabilities to bound is that of $\A(S,\delta) \cap \Q(m^*)$ (see Lemma~\ref{indep:a}), which follows from a simple calculation, as outlined above. Bounding the probability of $ \B(S,\delta) \cap \Y(m^*) \cap \Q(m^*) \cap \D(S,\delta)^c$ is also relatively straightforward (see Lemma~\ref{indep:b}): indeed, we simply apply our usual martingale method, using the fact that $\sum_J a_j^2 < \delta s^2$ to control the maximum possible size of a single step. Dealing with the events $\C(S,\delta)$ and $\bigcup_{|S| = s} \D(S,\delta)$ is significantly harder, and we postpone a discussion of the ideas involved to later in the section.

Next, let us turn our attention to the event that $G_{n,\triangle}$ contains a vertex whose degree is significantly larger than $2m^*/n$. Given a set $S \subseteq V(G_{n,\triangle})$ and a vertex $v \in V(G_{n,\triangle})$, observe first that if $N(v) = S$ in $G_{n,\triangle}$ then
\begin{itemize}
\item[$(a)$] $S$ is independent in $G_{m^*}$.\smallskip
\item[$(b)$] $\{u,v\} \in O(G_{m^*}) \cup E(G_{m^*})$ for every $u \in S$. \smallskip
\item[$(c)$] $N(v) \subseteq S$ in $G_{m^*}$.
\end{itemize}
It follows that the event $\I(S,m^*) \cap \W(S,v)$ must hold, where
\begin{equation}\label{def:W}
\W(S,v) \, = \, \Big\{ \{u,v\} \in O(G_{m^*}) \cup E(G_{m^*}) \text{ for every } u \in S \Big\} \cap \Big\{ N_{G_{m^*}}(v) \subseteq S \Big\}.
\end{equation}
Note also that the event $\E(m^*)$ implies that 
\begin{equation}\label{eq:degreeboundsgivenbyE}
d_{G_m}(v) \, \in \, \frac{2m}{n} \pm \sqrt{n}
\end{equation}
for every $m \le m^*$. Motivated by this, we define the following events, which are slight modifications of those above.

\begin{defn}\label{def:ABCprime}
Given a set $S \subseteq V(G_{n,\triangle})$ and $\delta > 0$, define:
\begin{itemize}
\item[$(i)$] $\A'(S,\delta) = \I(S,m^*) \cap \Big\{ o\big( G_m[S] \big) \ge (1 - \eps) \big( {|S| \choose 2} - \frac{2m^2}{n^2} \big) e^{-4t^2} \text{ for every } m \le m^* \Big\}$.
\item[$(ii)$] $\B'(S,\delta) = \Big\{ \sum_{j=2}^{|J|} a_j^2 < \delta s^2 \Big\} \cap {\ds \bigcup_{m=1}^{m^*}} \Big\{ o\big( G_m[S] \big) < (1 - \eps) \big( {|S| \choose 2} - \frac{2m^2}{n^2} \big) e^{-4t^2} \Big\}$.
\item[$(iii)$] $\C'(S,\delta) \, = \, \Big\{ \sum_{j=2}^{|J|} a_j^2 \ge \delta s^2 \Big\} \cap \Big\{ \sum_J a_j < n^{1/2 + 2\delta} \Big\}$.
\end{itemize}
\end{defn}

\noindent Let $\T(S,v)$ denote the event that $N_{G_{n,\triangle}}(v) = S$. We claim that
\begin{equation}\label{eq:UinABCD}
\bigcup_{|S| \ge s} \T(S,v) \cap \E(m^*)  \, \subseteq \, \bigcup_{|S| = s} \Big( \A'(S,\delta) \cup \B'(S,\delta) \cup \C'(S,\delta) \cup \D(S,\delta) \Big) \cap \W(S,v)
\end{equation}
for every $s \ge 2m^*/n + \sqrt{n}$. Indeed, we have already observed that $\T(S,v)$ implies $\W(S,v)$, and by~\eqref{def:W} and~\eqref{eq:degreeboundsgivenbyE}, the event $\W(S,v) \cap \E(m^*)$ implies that $\W(S',v)$ holds for some $S' \subseteq S$ with $|S'| = s$. The other implications now follow exactly as before, since the neighbourhood of a vertex in $G_{n,\triangle}$ is an independent set in $G_{m^*}$.

We can now bound the probability that the various events hold for some $S$ with $|S| = s$ as before, the main differences being (as noted above) that bounding $\Pr\big( \A'(S,\delta) \cap \W(S,v) \cap \E(m^*) \cap \Q(m^*) \big)$ is slightly more technical than bounding $\Pr\big( \A(S,\delta) \cap \Q(m^*) \big)$, and that dealing with the event  $\C'(S,\delta)$ is relatively easy, since we shall be able to show that 
$$\C'(S,\delta) \cap \W(S,v) \subseteq \big( \Y(m^*) \cap \Z(m^*) \big)^c,$$
which implies that the bad event corresponding to $\C'(S,\delta)$ is in fact impossible.

\subsection{The events $\A(S,\delta)$ and $\A'(S,\delta)$}

We begin with the easiest part of the proof, which requires only some straightforward counting. 

\begin{lemma} \label{indep:a}
If $s \ge \big( \sqrt{2} + \gamma \big) \sqrt{n \log n}$, then
$$\sum_{S \,:\, |S| = s} \Pr\Big( \A(S,\delta) \cap \Q(m^*) \Big) \, \le \, n^{ - \delta s}.$$
\end{lemma}

For each $0 \le m' \le m^*$, let $\O(S,m')$ denote the event that 
\begin{equation}\label{def:event:O}
o(G_m[S]) \ge (1 - \eps){|S| \choose 2} e^{-4t^2}
\end{equation}
for every $0 \le m \le m'$, and note that $\A(S,\delta) = \I(S,m^*)\cap \O(S,m^*)$. 

\begin{proof}[Proof of Lemma~\ref{indep:a}]
The lemma is an easy consequence of the following observation: the probability that an open edge inside $S$ is chosen in step $m+1$ is exactly $o(G_m[S]) / Q(m)$. Hence 
$$\Pr\Big( \I(S,m^*) \cap \O(S,m^*) \cap \Q(m^*) \Big)  \, \le \, \max_{\O(S,m^*) \cap \Q(m^*)} \prod_{m=0}^{m^*-1} \left( 1 - \frac{o(G_m[S])}{Q(m)} \right),$$
where the maximum is over all realizations of the triangle-free process for which both $\O(S,m^*)$ and $\Q(m^*)$ hold. Since $\O(S,m^*) \cap \Q(m^*)$ implies that 
$$\frac{o(G_m[S])}{Q(m)} \, \ge \, \frac{(1 - 2\eps)|S|^2}{n^2}$$ 
for every $0 \le m \le m^*$, and $s m^* \ge \big( \sqrt{2} + \gamma \big) \big( \frac{1}{2\sqrt{2}} - \eps \big) n^2 \log n > \big( 1/2 + 2\eps) n^2 \log n$, it follows that
$$\sum_{S \,:\, |S| = s} \Pr\big( \A(S,\delta) \cap \Q(m^*) \big)  \, \le \, {n \choose s} \exp\left( - \frac{(1 - 2\eps)s^2 m^*}{n^2} \right) \, \le \, \left( \frac{en}{s} \cdot \frac{1}{n^{1/2 + \delta}} \right)^s \, \le \, n^{- \delta s},$$
as claimed.
\end{proof}

Next, let's make precise the calculation sketched earlier for $\A'(S,\delta)$. For each edge $f \in E(G_{m^*})$, let $m(f) \in [m^*]$ denote the step of the triangle-free process at which it was added, i.e., such that $f \in E(G_{m(f)}) \setminus E(G_{m(f)-1})$. We shall need the following simple lemma.

\begin{lemma}\label{lem:Qprod}
Let $v \in V(G_{n,\triangle})$, set $d := d_{G_{m^*}}(v)$, and let $f_1,\ldots,f_d$ be the edges of $G_{m^*}$ that are incident to $v$. If $\E(m^*) \cap \Q(m^*)$ holds, then
$$\prod_{j = 1}^d Q\big( m(f_j) \big) \, \ge \, n^{11d/6}.$$
\end{lemma}
 
\begin{proof}
Without loss of generality, let us assume that $m(f_1) < \cdots < m(f_d)$. Recall from~\eqref{eq:degreeboundsgivenbyE} that $d_{G_m}(v) \in \frac{2m}{n} \pm \sqrt{n}$ for every $0 \le m \le m^*$ (since $\E(m^*)$ holds). Thus, in particular,
$$d \, \le \, \frac{2m^*}{n} + \sqrt{n} \, = \, \big( 2t^* + 1 \big) \sqrt{n}\qquad \text{ and } \qquad t(f_j) \, \le \, \frac{j}{2\sqrt{n}} + 1$$
for every $j \in [d]$, where $t(f_j) = m(f_j) \cdot n^{-3/2}$. Moreover, since $\E(m^*) \cap \Q(m^*)$ holds, we have $d \gg \sqrt{n}$ and $Q\big( m(f_j) \big) \ge e^{-4t_j^2} n^2 / 4$ for each $j \in [d]$. It follows that
\begin{multline*}
n^{-2d} \prod_{j = 1}^d Q\big( m(f_j) \big) \, \ge \, 4^{-d} \exp\bigg( - 4\sum_{j=1}^d t(f_j)^2 \bigg) \, \ge \,  \exp\bigg( - \sum_{j=1}^d \bigg( \frac{j^2}{n} + \frac{4j}{\sqrt{n}} + 6 \bigg) \bigg) \\
 \, \ge \,  \exp \bigg( - \frac{d^3}{3n} - \frac{3d^2}{\sqrt{n}} \bigg) \, \ge \, \exp \bigg( - \bigg( \frac{4}{3} + \delta \bigg)  (t^*)^2 d  \bigg) \, \ge \, \exp\bigg( - \frac{d \log n}{6} \bigg),
\end{multline*}
as claimed. The final step holds since $8(t^*)^2 \le (1 - \eps) \log n$.
\end{proof}

For each $0 \le m' \le m^*$, let $\O'(S,m')$ denote the event that the inequality
\begin{equation}\label{def:Oprime}
o(G_m[S]) \, \ge \, \big( 1 - \eps \big) \left( {|S| \choose 2} - \frac{2m^2}{n^2} \right) e^{-4t^2}
\end{equation}
holds for every $0 \le m \le m'$. Recall also from~\eqref{def:W} the definition of the event $\W(S,v)$.

\begin{lemma} \label{maxdeg:a}
If $s \ge  \big( \frac{1}{\sqrt{2}} + \gamma \big) \sqrt{n \log n}$ and $v \in V(G_{n,\triangle})$, then
\begin{equation}\label{eq:maxdeg:a}
\sum_{S \,:\, |S| = s} \Pr\Big( \A'(S,\delta) \cap \W(S,v) \cap \E(m^*) \cap \Q(m^*) \Big) \, \le \, n^{-\delta s}.
\end{equation}
\end{lemma}

\begin{proof}
Recall first that $\A'(S,\delta) = \I(S,m^*) \cap \O'(S,m^*)$, and note that if $\O'(S,m^*) \cap \Q(m^*)$ holds then
\begin{equation}\label{eq:maxdeg:a:ooverQ}
\frac{o(G_m[S])}{Q(m)} \, \ge \, \big( 1 - 2\eps \big) \left(  \frac{|S|^2}{n^2} - \frac{4m^2}{n^4} \right)
\end{equation}
for every $0 \le m \le m^*$. Let $f_1,\ldots,f_d$ be the edges of $G_{m^*}$ which are incident to $v$, and recall from~\eqref{def:W} that the event $\W(S,v)$ implies that $N_{G_{m^*}}(v) \subseteq S$. Note also that
$$d \, \in \, D \, := \, \bigg( \frac{1}{\sqrt{2}} \pm 3\eps \bigg) \sqrt{n \log n},$$ 
since $\E(m^*)$ holds. Given edges $f_1,\ldots,f_d$ and steps $m(f_1),\ldots,m(f_d) \in [m^*]$ such that $d_{G_m}(v)$ satisfies~\eqref{eq:degreeboundsgivenbyE} for every $m \in [m^*]$, we shall bound the probability that the edge $f_j$ is chosen\footnote{Note that we do not lose anything by assuming that $f_j$ is open after $m(f_j) - 1$ steps, since this follows automatically if $N_{G_{m^*}}(v) \subseteq S$ and $S$ is an independent set.} in step $m(f_j)$ for each $j \in [d]$, and that at every other step, we do not choose an open edge in $S$. Having done so, it will suffice to sum over the at most $2^s (m^*)^d$ choices for the edges $f_j$ and steps $m(f_j)$. 

Note first that, by Lemma~\ref{lem:Qprod}, and since the event $\E(m^*) \cap \Q(m^*)$ holds, the probability that the edge $f_j$ is chosen in step $m(f_j)$ for each $j \in [d]$ is
 \begin{equation}\label{eq:maxdeg:a:Qprod}
\prod_{j = 1}^d \frac{1}{Q\big( m(f_j) \big)} \, \le \, n^{-11d/6}.
\end{equation}
Moreover, by~\eqref{eq:maxdeg:a:ooverQ}, the probability that at every other step we do not choose an open edge in $S$ is at most
 \begin{equation}\label{eq:maxdeg:a:openprod}
\max_{\O'(S,m^*) \cap \Q(m^*)} \prod_{m \in M} \left( 1 - \frac{o(G_m[S])}{Q(m)} \right) \, \le \, \exp\bigg( - \big( 1 - 3\eps \big) \left( \frac{s^2 m^*}{n^2} - \frac{4(m^*)^3}{3n^4} \right) \bigg),
\end{equation}
where $M = [m^*] \setminus \{ m(f_1),\ldots,m(f_d)\}$. Note that $s^2 = d^2 + (s - d)(s + d)$, and that
$$\frac{m^* d}{n^2} \ge \left( \frac{1}{4} - 3 \eps \right) \log n, \quad \frac{\big( s + d \big) m^*}{n^2} \ge \bigg( \frac{2 + \gamma}{4} \bigg) \log n \quad \text{and} \quad  \frac{(m^*)^3}{n^4} \le \frac{d \log n}{16}.$$ 
for every $d \in D$, and hence
$$\frac{s^2 m^*}{n^2} - \frac{4(m^*)^3}{3n^4} \, \ge \, \left( \frac{1}{4} - 3 \eps \right) d \log n + \bigg( \frac{2 + \gamma}{4} \bigg) (s - d) \log n - \frac{d \log n}{12}.$$
It follows that the right-hand side of~\eqref{eq:maxdeg:a:openprod} is at most
\begin{equation}\label{eq:maxdeg:a:openprod2}
\exp\bigg( - \left( \frac{1}{6} - 4 \eps \right) d \log n \,-\,  \left( \frac{3 + \gamma}{6} \right) \big( s - d \big) \log n \bigg)
\end{equation}

Hence, combining~\eqref{eq:maxdeg:a:Qprod} and~\eqref{eq:maxdeg:a:openprod2}, and summing over sets $S$, integers $d \in D$, edges $f_1,\ldots,f_d$ and steps $m(f_1),\ldots,m(f_d)$, and noting that $s - d \ge \gamma d$, we obtain an upper bound on the left-hand side of~\eqref{eq:maxdeg:a} of 
\begin{multline*}
{n \choose s} \sum_{d \in D} \frac{2^s (m^*)^d}{n^{11d/6}} \cdot \exp\bigg( - \, \left( \frac{1}{6} - 4 \eps \right) d \log n \,-\,  \left( \frac{3 + \gamma}{6} \right) \big( s - d \big) \log n \bigg)\\
\, \le \, \sum_{d \in D} \left( \frac{2en}{s} \cdot \frac{m^*}{n^{11/6}} \cdot \frac{n^{4\eps}}{n^{1/6}} \right)^d \left( \frac{2en}{s} \cdot \frac{1}{n^{1/2 + \gamma/6}} \right)^{s - d} \, \le \, \sum_{d \in D} n^{5\eps d} \cdot n^{-\gamma(s-d)/6} \, \le \, n^{-\delta s}, 
\end{multline*}
as required. 
\end{proof}

\subsection{The events $\B(S,\delta) \cap \D(S,\delta)^c$ and $\B'(S,\delta) \cap \D(S,\delta)^c$}\label{SubSecB}

We shall next apply our usual martingale method in order to show that $o\big( G_m[S] \big)$ is (with very high probability) well-behaved, as long as $\sum_J a_j^2 < \delta s^2$ and $\sum_J a_j < n^{1/2 + 2\delta}$. (The latter condition will be necessary in order to bound the number of choices for the neighbourhoods in $S$ of the vertices of $J$.) For convenience, we remind the reader of the following notation: 
\begin{itemize}
\item $\B(S,\delta) = \big\{ \sum_{j=1}^{|J|} a_j^2 < \delta |S|^2 \big\} \cap \O(S,m^*)^c$, where we recall from~\eqref{def:event:O} that 
$$\O(S,m^*) \, = \, \bigg\{ o(G_m[S]) \ge (1 - \eps){|S| \choose 2} e^{-4t^2} \textup{ for every $m \in [m^*]$} \bigg\}.$$
\item $\B'(S,\delta) = \big\{ \sum_{j=2}^{|J|} a_j^2 < \delta |S|^2 \big\} \cap \O'(S,m^*)^c$, where we recall from~\eqref{def:Oprime} that 
$$\O'(S,m^*) \, = \, \bigg\{ o(G_m[S]) \ge (1 - \eps)\bigg( {|S| \choose 2} - \frac{2m^2}{n^2} \bigg) e^{-4t^2} \textup{ for every $m \in [m^*]$} \bigg\}.$$ 
\end{itemize}
We shall prove the following two lemmas. 

\begin{lemma}\label{indep:b}
If $3 \sqrt{n} \le s \le n^{1/2 + \eps}$, then
$$\Pr\Big( \B(S,\delta) \cap \Y(m^*) \cap \Q(m^*) \cap \D(S,\delta)^c \Big) \, \le \, e^{-s n^\delta}$$ 
for every $S \subseteq V(G_{n,\triangle})$ with $|S| = s$. 
\end{lemma}

As in the previous subsection, in the maximum degree setting we shall use the event $\E(m^*)$ to control the number of $G_m$-neighbours of the vertex $v_1$ in $S$, for each $m \in [m^*]$. 

\begin{lemma}\label{maxdeg:b}
If $\big( \frac{1}{\sqrt{2}} + \gamma \big) \sqrt{n \log n} \le s \le n^{1/2+ \eps}$, then
$$\Pr\Big( \B'(S,\delta) \cap \E(m^*) \cap \Y(m^*) \cap \Q(m^*)  \cap \D(S,\delta)^c \Big) \, \le \, e^{-s n^\delta}$$
for every $S \subseteq V(G_{n,\triangle})$ with $|S| = s$. 
\end{lemma} 

We shall use the following notation in the proofs of Lemmas~\ref{indep:b} and~\ref{maxdeg:b}. Given $m \in \N$, a set $S \subseteq V(G_{n,\triangle})$ and a collection $\n = (A_1,\ldots,A_k)$ of subsets of $S$, set
$$O_\n(S,m) \, = \, O\big( G_m[S] \big) \setminus \bigcup_{j=1}^k O\big( G_m[A_j] \big),$$
and set $o_\n(S,m) = |O_\n(S,m)|$. Set $g_o(t) = n^{3\delta} g_x(t) = C e^{2t^2} n^{-1/4 + 3\delta}(\log n)^4$, and note that $g_o(t) \ll 1$ for every $0 < t \le t^*$, since $e^{4(t^*)^2} \le n^{1/2 - \eps}$ and $\delta = \delta(\eps)$ was chosen sufficiently small. Define the normalized error to be 
$$o_\n^*(S,m) \, = \, \frac{o_\n(S,m) - e^{-4t^2} o_\n(S,0)}{g_o(t) e^{-4t^2} o_\n(S,0)},$$
and write $\O_\n(S,m')$ for the event that $|o_\n^*(S,m)| \le 1$ for every $m \le m'$. Observe that if $\O_\n(S,m)$ holds and $\sum_{j=1}^k |A_j|^2 \le \delta |S|^2$, then $\O(S,m)$ also holds.

In a slight abuse of notation, given a collection $\n$ as above, we define $\tilde{\n}(S,m)$ to be the event\footnote{We shall show that this event holds for some collection $\n$, see Observations~\ref{Bobs} and~\ref{Bprimeobs}, and then apply the union bound, using the event $\D(S,\delta)^c$ in order to bound the number of choices for $\n$.} that  $N_{G_{m+1}}(v) \cap S \subseteq A_j$ for some $j \in [k]$, for every vertex $v \in V(G_{n,\triangle})$ with $|N_{G_m}(v) \cap S| \ge n^\delta$. That is, 
$$\tilde{\n}(S,m) = \bigcap_{v \in V(G_{n,\triangle})} \bigg( \bigcup_{j=1}^k \Big\{ N_{G_{m+1}}(v) \cap S \subseteq A_j \Big\} \cup \Big\{ |N_{G_m}(v) \cap S| \le n^\delta \Big\} \bigg).$$
Observe that if two edges $e_1,e_2 \in O(G_m[S])$ are closed by the addition (in step $m+1$) of edge $f$, then $e_1$, $e_2$ and $f$ share a common vertex, and $e_1,e_2 \in Y_f(m)$. Moreover, $e_1$ and $e_2$ are both contained in the neighbourhood in $G_{m+1}$ of one of the endpoints of $f$. It follows immediately that if $\tilde{\n}(S,m)$ holds, then $|\Delta o_\n(S,m)| \le n^\delta$.

We will deduce Lemmas~\ref{indep:b} and~\ref{maxdeg:b} from the following result, which follows by the martingale method of Section~\ref{MartSec}. Since the proof is similar to several of those above (cf. in particular the proof of Proposition~\ref{Xprop} in Section~\ref{SecX}), we shall omit some of the details. 

\begin{lemma}\label{lemma:b}
Let $S \subseteq V(G_{n,\triangle})$, and let $\n = (A_1,\ldots,A_k)$ be a collection of subsets of $S$. If $|S| \sqrt{n} \le o_\n(S,0) \le n^{5/4}$, then 
$$\Pr\Big( \O_\n(S,m)^c \cap \tilde{\n}(S,m) \cap \Y(m) \cap \Q(m) \Big) \, \le \, e^{- |S| n^{4\delta}}$$
for every $m \in [m^*]$. 
\end{lemma}

We need to deal separately with the case $m \le \omega \cdot n^{3/2}$; since we again use Bohman's method from~\cite{Boh} in this case, we postpone the details to the Appendix~\cite{App}. 

\begin{lemma}\label{lemma:b:landbeforetime}
Let $S \subseteq V(G_{n,\triangle})$, and let $\n = (A_1,\ldots,A_k)$ be a collection of subsets of $S$. If $|S| \sqrt{n} \le o_\n(S,0) \le n^{5/4}$, then 
$$\Pr\Big( \big\{ |o^*_\n(S,m)| > 1 \big\} \cap \tilde{\n}(S,m) \cap \Y(m) \cap \Q(m) \Big) \, \le \, n^{- |S| n^{4\delta}}$$
for every $m \le \omega \cdot n^{3/2}$. 
\end{lemma}

\begin{proof}[Proof of Lemma~\ref{lemma:b}]
We shall apply the martingale method of Section~\ref{MartSec}. 
Indeed, we have
\begin{equation}\label{eq:ExDeltaOpenEdges}
\Ex\big[ \Delta o_\n(S,m) \big] \, = \, - \frac{1}{Q(m)} \sum_{f \in O_\n(S,m)} \Big( Y_f(m) + 1 \Big),
\end{equation}
from which it follows easily (see Lemma~A.2.1 of the Appendix~\cite{App}) that, for each $\omega \cdot n^{3/2} < m \le m^*$, if $\O_\n(S,m) \cap \Y(m) \cap \Q(m)$ holds then
$$\Ex\big[ \Delta o^*_\n(S,m) \big] \, \in \, \frac{4t}{n^{3/2}} \Big( -\, o^*_\n(S,m) \pm \eps \Big).$$
Hence, setting $\K(m) = \O_\n(S,m) \cap \tilde{\n}(S,m) \cap \Y(m) \cap \Q(m) \cap \big\{ |o^*_\n(S,a)| < 1/2 \big\}$ and $I = [a,b] = [\omega \cdot n^{3/2},m^*]$, it follows that $o_\n(S,m)$ is $(g_o,h;\K)$-self-correcting on $[a,b]$, where $h(t) = t \cdot n^{-3/2}$. Moreover, we claim that if $\K(m)$ holds then
$$|\Delta o_\n(S,m) | \, \le \, n^\delta \qquad \text{and} \qquad \Ex\big[ | \Delta o_\n(S,m) | \big] \, \le \, \frac{C \cdot t}{n^{3/2}} \cdot e^{-4t^2} o_\n(S,0).$$ 
Indeed, the first inequality follows from the event $\tilde{\n}(S,m)$, as noted above, and the second from~\eqref{eq:ExDeltaOpenEdges}, combined with $\O_\n(S,m) \cap \Y(m) \cap \Q(m)$, since $o_\n(S,m)$ is decreasing. Hence, by Lemma~\ref{lem:chainstar}, we obtain
$$|\Delta o^*_\n(S,m) | \, \le \, \frac{e^{4t^2} n^\delta}{g_o(t) o_\n(S,0)} =: \alpha(t) \qquad \text{and} \qquad \Ex\big[ | \Delta o^*_\n(S,m) | \big] \, \le \, \frac{C \cdot t}{g_o(t) n^{3/2}} =: \beta(t).$$ 
Since $\alpha(t)$ and $\beta(t)$ are $C$-slow on $[a,b]$, and $\min\big\{ \alpha(t), \, \beta(t), \, h(t) \big\} \ge \frac{\eps t}{n^{3/2}}$ and $\alpha(t) \le \eps$ for every $\omega < t \le t^*$ (since $\sqrt{n} \le o_\n(S,0) \le n^{5/4}$), it follows that $(C,\eps;g_o,h;\alpha,\beta;\K)$ is a reasonable collection, and that $o^*_\n(S,m)$ satisfies the conditions of Lemma~\ref{lem:self:mart}.

Finally, observe that
$$\alpha(t) \beta(t) n^{3/2} \, = \, \frac{C \cdot t e^{4t^2} n^\delta}{g_o(t)^2 o_\n(S,0)} \, \le \, \frac{1}{n^{5\delta} |S|}$$
for every $\omega < t \le t^*$, since $o_\n(S,0) \ge |S| \sqrt{n}$ and $g_o(t) = n^{3\delta} g_x(t)$. By Lemma~\ref{lem:self:mart}, it follows that
$$\Pr\Big( \O_\n(S,m')^c \cap \K(m'-1) \text{ for some $m' \in [a,b]$} \Big) \, \le \, n^4 \exp\Big( - \delta' n^{5\delta} |S| \Big) \, \le \, n^{-n^{4\delta} |S|}.$$
Combining this bound with Lemma~\ref{lemma:b:landbeforetime}, and summing over choices of $m' \le m$, we obtain the claimed bound on the probability of the event $\O_\n(S,m)^c \cap \tilde{\n}(S,m) \cap \Y(m) \cap \Q(m)$.
\end{proof}

Lemmas~\ref{indep:b} and~\ref{maxdeg:b} follow by applying Lemma~\ref{lemma:b} to all possible collections $\n$ of sets $\big\{ N_{G_{m^*}}(v) \cap S : v \in J \big\}$. We shall need the following two simple observations. 

\begin{obs}\label{Bobs}
Let $S \subseteq V(G_{n,\triangle})$, and suppose that $\B(S,\delta) \cap \O(S,m)^c$ holds. Then there exists a collection $\n = (A_1,\ldots,A_k)$, with 
$$\sum_{j=1}^k |A_j|^2 < \delta |S|^2,$$
such that $\O_\n(S,m)^c \cap \tilde{\n}(S,m)$ holds.
\end{obs}

\begin{proof}
Recall that the event $\B(S,\delta)$ implies that $\sum_{j=1}^k a_j^2 < \delta |S|^2$. Set $A_j = N_{G_{m+1}}(v_j) \cap S$ for each $j \in [k]$, where $J(S,\delta) = \{ v_1,\ldots,v_k \}$, and note that therefore $\sum_{j=1}^k |A_j|^2 < \delta |S|^2$. As noted above, this implies that $\O_\n(S,m) \subseteq \O(S,m)$. Finally, $N_{G_{m+1}}(v) \cap S \subseteq A_j$ for every $v \in J$, and $|N_{G_m}(v) \cap S| \le n^\delta$ for every $v \not\in J$, so $\tilde{\n}(S,m)$ holds, as claimed.
\end{proof}

\begin{obs}\label{Bprimeobs}
Let $S \subseteq V(G_{n,\triangle})$ be a set with $|S| \ge \big( \frac{1}{\sqrt{2}} + \gamma \big) \sqrt{n \log n}$, and suppose that $m \in [m^*]$ is minimal such that $\B'(S,\delta) \cap \E(m^*) \cap \O'(S,m)^c$ holds. Then there exists a collection $\n = (A_1,\ldots,A_k)$, with 
$$|A_1| \le \frac{\big( 2 + \delta \big) m}{n} \quad \text{and} \quad \sum_{j=2}^k |A_j|^2 < \delta |S|^2,$$
such that $\O_\n(S,m)^c \cap \tilde{\n}(S,m)$ holds.
\end{obs}

\begin{proof}
Recall that the event $\E(m^*)$ implies that $d_{G_m}(v) \le (2 + \delta) m/ n$ for every $m \in [m^*]$, and that $\B'(S,\delta)$ implies that $\sum_{j=2}^k a_j^2 < \delta |S|^2$. Thus, setting $A_j = N_{G_{m+1}}(v_j) \cap S$ for each $j \in [k]$, where $J(S,\delta) = \{v_1,\ldots,v_k \}$, we have
$$\sum_{j=1}^k {|A_j| \choose 2} \le {(2 + \delta) m /n \choose 2} + \delta |S|^2 \le \frac{2m^2}{n^2} +  \frac{3\delta}{\gamma} \bigg( {|S| \choose 2} - \frac{2m^2}{n^2} \bigg),$$
where the final inequality follows since ${|S| \choose 2} - \frac{2m^2}{n^2} \ge \gamma \cdot \big( {|S| \choose 2} + \frac{2m^2}{n^2} \big)$. Hence
$$o_\n(S,0) \ge {|S| \choose 2} - \sum_{j=1}^k {|A_j| \choose 2} \ge \bigg( 1 - \frac{3\delta}{\gamma} \bigg) \left( {|S| \choose 2} - \frac{2m^2}{n^2} \right).$$
Now, observe that  
$$o\big( G_m[S] \big) \, < \, \big( 1 - \eps \big) \left( {|S| \choose 2} - \frac{2m^2}{n^2} \right) e^{-4t^2},$$
since $m \in [m^*]$ is minimal such that $\O'(S,m)^c$ holds. Hence, since $\delta = \delta(\eps,\gamma)$ was chosen sufficiently small, 
$$o_\n(S,m) \, \le \, o\big( G_m[S] \big) \, < \, \big( 1 - \eps \big) \left( {|S| \choose 2} - \frac{2m^2}{n^2} \right) e^{-4t^2} \, \le \, \big( 1 - \delta \big) e^{-4t^2} o_\n(S,0),$$
which implies $\O_\n(S,m)^c$. Finally, as in the previous proof, $N_{G_{m+1}}(v) \cap S \subseteq A_j$ for every $v \in J$ and $|N_{G_m}(v) \cap S| \le n^\delta$ for every $v \not\in J$, so $\tilde{\n}(S,m)$ holds, as required.
\end{proof}

Finally, note that if $\D(S,\delta)^c$ holds, then there are at most $n^{2n^{1/2+2\delta}}$ different possible collections $\n$ given by the observations above, since there are at most $n^{1/2+2\delta}$ edges between $S$ and $J$. If moreover $|S| \ge 2\sqrt{n}$, then this is at most $n^{|S| n^{2\delta}}$. 

We can now deduce Lemmas~\ref{indep:b} and~\ref{maxdeg:b}.

\begin{proof}[Proof of Lemma~\ref{indep:b}]
Suppose that $\B(S,\delta) \cap \Y(m^*) \cap \Q(m^*) \cap \D(S,\delta)^c$ holds, and let $m \in [m^*]$ be minimal such that $\O(S,m)^c$ holds. By Observation~\ref{Bobs}, there exists a collection $\n = (A_1,\ldots,A_k)$, with $\sum_{j=1}^k |A_j|^2 < \delta |S|^2$, such that $\O_\n(S,m)^c \cap \tilde{\n}(S,m)$ holds. We shall apply Lemma~\ref{lemma:b} for each such collection, and use the union bound. 

Since $3 \sqrt{n} \le s \le n^{1/2 + \eps}$ and $\sum_{j=1}^k |A_j|^2 < \delta |S|^2$, we have $|S| \sqrt{n} \le o_\n(S,0) \le n^{5/4}$. Hence, by Lemma~\ref{lemma:b}, we obtain
$$\sum_{\n,m} \Pr\Big( \O_\n(S,m)^c \cap \tilde{\n}(S,m) \cap \Y(m) \cap \Q(m) \Big) \, \le \, m^* \cdot n^{|S| n^{2\delta}} \cdot e^{-|S| n^{4\delta}} \, \le \, e^{-s n^\delta},$$
where the sum is over $m \in [m^*]$ and families $\n$ as described above, which satisfy $\D(S,\delta)^c$. It follows immediately that
$$\Pr\Big( \B(S,\delta) \cap \Y(m^*) \cap \Q(m^*) \cap \D(S,\delta)^c \Big) \, \le \, e^{-s n^\delta},$$
as required.
\end{proof}


The proof of Lemma~\ref{maxdeg:b} is similar, using Observation~\ref{Bprimeobs}. 

\begin{proof}[Proof of Lemma~\ref{maxdeg:b}]
Suppose that $\B'(S,\delta) \cap \E(m^*) \cap \Y(m^*) \cap \Q(m^*) \cap \D(S,\delta)^c$ holds, and let $m \in [m^*]$ be minimal such that $\O'(S,m)^c$ holds. By Observation~\ref{Bprimeobs}, there exists a collection $\n = (A_1,\ldots,A_k)$, with $|A_1| \le ( 2 + \delta ) m / n$ and $\sum_{j=2}^k |A_j|^2 < \delta |S|^2$, such that $\O_\n(S,m)^c \cap \tilde{\n}(S,m)$ holds. We shall again apply Lemma~\ref{lemma:b}, and use the union bound. 

Since $\big( \frac{1}{\sqrt{2}} + \gamma \big) \sqrt{n \log n} \le s \le n^{1/2+\eps}$ and 
$$o_\n(S,0) \,\ge\, {|S| \choose 2} - \sum_{j=1}^k {|A_j| \choose 2} \, \ge \, \frac{1}{2} \left( {|S| \choose 2} - \frac{2m^2}{n^2} \right) \, \ge \, \frac{\gamma |S|}{4} \cdot \sqrt{n \log n},$$
as in the proof of Observation~\ref{Bprimeobs}, it follows that $|S| \sqrt{n} \le o_\n(S,0) \le n^{5/4}$. Hence, by Lemma~\ref{lemma:b}, we obtain
$$\sum_{\n,m} \Pr\Big( \O_\n(S,m)^c \cap \tilde{\n}(S,m) \cap \Y(m) \cap \Q(m) \Big) \, \le \, m^* \cdot n^{|S| n^{2\delta}} \cdot e^{-|S| n^{4\delta}} \, \le \, e^{-s n^\delta},$$
where the sum is over $m \in [m^*]$ and families $\n$ as described above, which satisfy $\D(S,\delta)^c$. It follows immediately that
$$\Pr\Big( \B'(S,\delta) \cap \E(m^*) \cap \Y(m^*) \cap \Q(m^*) \cap \D(S,\delta)^c \Big) \, \le \, e^{-s n^\delta},$$
as required.
\end{proof}

\subsection{The events $\C(S,\delta)$ and $\C'(S,\delta)$}\label{SubSecC}

Recall that $\C(S,\delta)$ denotes the event that $\a(S,\delta) = (a_1,\ldots, a_{|J|})$ satisfies $\sum_J a_j^2 \ge \delta |S|^2$ and $\sum_J a_j < n^{1/2 + 2\delta}$, and that $S$ is an independent set in $G_{m^*}$. The main result of this subsection is the following lemma.

\begin{lemma}\label{indep:c}
If $s \ge \big( \sqrt{2} + \gamma \big) \sqrt{n \log n}$, then
$$\sum_{S \,:\, |S| = s} \Pr\Big( \C(S,\delta) \cap \E(m^*) \cap \Y(m^*) \cap \Z(m^*) \cap \Q(m^*) \Big) \, \le \, n^{-\delta s }.$$
\end{lemma}

We shall also prove the following easy lemma which shows that if $\E(m^*) \cap \Z(m^*)$ holds, then the event $\C'(S,\delta) \cap \W(S,v)$ does not hold for any pair $(S,v)$ with $|S| \ge \sqrt{n}$.  

\begin{lemma}\label{maxdeg:c}
For every $S \subseteq V(G_{n,\triangle})$ with $|S| \ge \sqrt{n}$, and every $v \in V(G_{n,\triangle})$,
$$\C'(S,\delta) \cap \W(S,v) \subseteq \big( \E(m^*) \cap \Z(m^*) \big)^c.$$
\end{lemma}

Recall that $a_1 \ge \dots \ge a_{|J|}$, and let $k = k(S,\delta)$ be minimal such that $\sum_{j > k} a_j^2 < \delta |S|^2$. We begin with an easy but key observation which will be used in both proofs.

\begin{obs}\label{obs:C}
Let $S \subseteq V(G_{n,\triangle})$ with $|S| \ge \sqrt{n}$. Then 
$$\C(S,\delta) \cup \C'(S,\delta) \, \subseteq \, \Big\{ a_k \ge n^{1/2-3\delta} \Big\} \, \cap \, \Big\{  k \le n^{5\delta} \Big\}.$$
\end{obs}

\begin{proof}
Noting that the event $\C(S,\delta) \cup \C'(S,\delta)$ implies that $\sum_J a_j < n^{1/2 + 2\delta}$, we obtain
$$n^{1 - \delta} \, \le \, \delta |S|^2 \, \le \, \sum_{j \ge k} a_j^2 \, \le \, a_k \sum_{j \ge k} a_j \, \le \, a_k n^{1/2 + 2\delta},$$
by the definition of $k$. The claimed bounds on $a_k$ and $k$ now follow immediately.
\end{proof}

\begin{proof}[Proof of Lemma~\ref{maxdeg:c}]
Fix $S \subseteq V(G_{n,\triangle})$ with $|S| \ge \sqrt{n}$ and $v \in V(G_{n,\triangle})$, and suppose that the event $\C'(S,\delta) \cap \W(S,v) \cap \E(m^*) \cap \Z(m^*)$ holds; we shall show that this is impossible. Recall first that the event $\W(S,v)$ implies that $\{u,v\} \in O(G_{m^*}) \cup E(G_{m^*})$ for every $u \in S$. We shall show, using the event $\E(m^*) \cap \Z(m^*)$, that every vertex $w \in V(G_{n,\triangle})$ other than $v$ has at most 
\begin{equation}\label{eq:maxdeg:c:nbrsofw}
n^{4\eps} + (\log n)^2
\end{equation}
$G_{m^*}$-neighbours in $S$. Indeed, fix $w \in V(G_{n,\triangle})$, and note that (as in the proof of Proposition~\ref{Zprop}), since the event $\E(m^*)$ holds, there are at most 
$$\big( 1 + o(1) \big) \cdot 4t^* e^{-4(t^*)^2} \sqrt{n} \, \le \, n^{4\eps}$$ 
vertices $u \in S$ such that $\{u,v\} \in O(G_{m^*})$ and $\{u,w\} \in E(G_{m^*})$. Similarly, since the event $\Z(m^*)$ holds, there are at most $(\log n)^2$ vertices $u \in S$ such that $\{u,v\}, \{u,w\} \in E(G_{m^*})$. 

It follows from~\eqref{eq:maxdeg:c:nbrsofw} that $a_2 \le n^{4\eps} + (\log n)^2$. But $\sum_{j=2}^{|J|} a_j^2 \ge \delta s^2$, since $\C'(S,\delta)$ holds, and so $k(S,\delta) \ge 2$. Hence, by Observation~\ref{obs:C}, we have $a_2 \ge n^{1/2-3\delta}$, which is a contradiction, as required.
\end{proof}

The proof of Lemma~\ref{indep:c} is considerably harder, and so we shall give a brief sketch before plunging into the details. We are again motivated by Observation~\ref{obs:C}, but since we no longer assume that the event $\W(S,v)$ holds (and so very many vertices can have very high degree into $S$) we shall need some extra ideas. We will ignore the `high degree vertices' in $J$ with between $n^\delta$ and $n^{1/2 - 3\delta}$ neighbours in $S$, and focus on the set $J' \subseteq J$ of `very high degree vertices', which form an unusually dense bipartite graph $H = G_{m^*}[J',S]$. Moreover, there is a trade-off in choosing the graph $H$: the more edges it has, and the earlier the edges of $H$ are chosen, the less likely it is to occur, but the easier it is to keep the set $S$ independent. We shall need to keep track of each of these competing influences.

In order to show that such a structure (a bipartite graph $H$ as described above, sitting on an independent set $S$) is unlikely to exist in $G_{m^*}$, we partition the space according to the sets $S$ and $J'$, the graph $H$, and the collection $\m = \big( m(f) : f \in E(H) \big)$ of steps of the triangle-free process at which the edges of $H$ were chosen. The probability that the edge $f$ is chosen in step $m(f)$ is $1 / Q\big( m(f) \big) \approx 2 e^{4t^2} / n^2$, since the event $\Q(m^*)$ holds; the hard part will be to bound the number of `forbidden' open edges at each step. 

The forbidden open edges come in two types: those inside $S$, and those which would close not-yet-chosen edges of $H$, i.e., in the set 
\begin{equation}\label{def:YHm:T}
Y_H(m) \, := \, \bigcup_{f \in T(m)} Y_f(m), \quad \text{where} \quad T(m) \,:=\, \big\{ f \in E(H) \,:\, m(f) > m \big\}.
\end{equation}
In order to keep these sets disjoint, we shall ignore open edges in $S$ which are closed by vertices in $J'$, i.e., we shall consider the sets $O_\n(S,m)$, where $\n = (A_1, \ldots, A_k)$ encodes the neighbourhoods $A_j = N_{G_{m^*}}(v_j) \cap S$ of the vertices of $J'$, cf. Section~\ref{SubSecB}. Observe that if $f \in O_\n(S,m) \cap Y_h(m)$ for some $h \in E(H)$ with $m(h) > m$, then $f \subseteq N_{G_{m^*}}(v) \cap S$, where $v$ is the endpoint of $f$ in $J'$, which contradicts the assumption that $f \in O_\n(S,m)$. It follows that the two sets of open edges we consider are indeed disjoint, as claimed.

Finally, note that we can bound the probability that $o_\n(S,m)$ is smaller than expected using Lemma~\ref{lemma:b}; thus, all that remains is to bound from below the size of the set $Y_H(m)$. Our key tool in doing so will be the following lemma. Given a graph $G$ and an oriented edge $f \in V(G)^2$, let $d_G^L(f)$ and $d_G^R(f)$ denote the degrees of the (left and right, respectively) endpoints of $f$ in $G$, and let
$$\Xi(G) \, := \, \bigcup_{\{u,v\} \in {V(G) \choose 2}} {N_G(u) \cap N_G(v) \choose 2}$$
denote the set of edges which are `double-covered' by $G$, that is, those which are contained in the neighbourhood of at least two vertices of $G$. Recall from Section~\ref{Ysec} that $Y_e^L(m) \subseteq Y_e(m)$ denotes the collection of $Y$-neighbours $f$ of $e$ such that the vertex $v \in e \setminus f$ has label $L$. 

\begin{lemma}\label{lem:bigcupY}
Let $m \in [m^*]$, let $H$ be a graph with $E(H) \subseteq E(G_{m^*}) \cap O(G_m)$, and give an orientation to each edge of $E(H)$. If $\Z(m^*)$ holds then
$$\bigg| \bigcup_{f \in E(H)} Y^L_f(m) \bigg| \, \ge \, \sum_{f \in E(H)} Y^L_f(m) \, - \, (\log n)^2 \bigg( \big| \Xi\big( G_{m^*}[V(H)] \big) \big| + \sum_{f \in E(H)} d^R_H(f) \bigg).$$
\end{lemma}

We remark that we shall in fact apply Lemma~\ref{lem:bigcupY} to two different subgraphs, $T(m)$ and $T'(m) \subseteq T(m)$ (defined below), of the graph $H = G_{m^*}[J',S]$.\footnote{We shall also use it in Section~\ref{SubSecD}, below.}

\begin{proof}[Proof of Lemma~\ref{lem:bigcupY}]
The lemma follows by inclusion-exclusion, together with the observation that 
\begin{equation}\label{eq:inex}
\sum_{h \in E(H) \setminus \{f\}} \big| Y^L_f(m) \cap Y^L_h(m) \setminus \Xi\big( G_{m^*}[V(H)] \big) \big| \, \le \, d^R_H(f) (\log n)^2.
\end{equation}
To see~\eqref{eq:inex}, let $f,h \in E(H)$ and suppose first that $f$ and $h$ are disjoint. Then (clearly) $|Y^L_f(m) \cap Y^L_h(m)| \in \{0,1\}$, and moreover we claim that 
$$Y^L_f(m) \cap Y^L_h(m) \subseteq \Xi\big( G_{m^*}[V(H)] \big).$$ 
Indeed, to see this simply note that $f,h \in E(H) \subseteq E(G_{m^*})$, so if $e \in Y^L_f(m) \cap Y^L_h(m)$ then it is double-covered by $G_{m^*}[V(H)]$. Note also that if $f \neq h$, but $f$ and $h$ intersect in the `left' endpoint of either $f$ or $h$, then $Y^L_f(m) \cap Y^L_h(m) = \emptyset$.  

Now, since $\Z(m)$ holds, we have $|Y^L_f(m) \cap Y^L_h(m)| \le (\log n)^2$ for every pair of edges $f,h \in O(G_m)$, by Observation~\ref{obs:twostep}. Moreover, there are at most $d^R_H(f)$ edges $h \in E(H) \setminus \{f\}$ which contain the `right' endpoint of $f$, and so~\eqref{eq:inex} follows.

To deduce the lemma from~\eqref{eq:inex}, simply note that (for any set $A \subseteq E(K_n)$)
$$\bigg| \bigcup_{f \in E(H)} Y^L_f(m) \setminus A \bigg| \, \ge \, \sum_{f \in E(H)} \big| Y^L_f(m) \setminus A \big| \, - \, \sum_{\substack{f,h \in E(H)\\ f \neq h}} \big| Y^L_f(m) \cap Y^L_h(m) \setminus A \big|,$$
set $A = \Xi\big( G_{m^*}[V(H)] \big)$, and observe that, since $\Z(m^*)$ holds and $H \subseteq G_{m^*}$, each edge\footnote{In fact, this holds for every edge, not just those in $\Xi\big( G_{m^*}[V(H)] \big)$.} of $\Xi\big( G_{m^*}[V(H)] \big)$ appears in $Y^L_f(m)$ for at most $(\log n)^2$ edges $f \in E(H)$. 
\end{proof}

We are now ready to prove Lemma~\ref{indep:c}.

\begin{proof}[Proof of Lemma~\ref{indep:c}]
Let $\big( \sqrt{2} + \gamma \big) \sqrt{n \log n} \le s \le n^{1/2 + \eps}$; we begin the proof by breaking up  the event whose probability we wish to bound into more manageable pieces. Indeed, for each $S \subseteq V(G_{n,\triangle})$ with $|S| = s$, each $k \le n^{5\delta}$, each $J' \subseteq V(G_{n,\triangle}) \setminus S$ with $|J'| = k$, each bipartite graph $H$ on $S \cup J'$ with $d_H(v) \ge n^{1/2 - 3\delta}$ for each $v \in J'$, and each collection $\m = \big( m(f) : f \in E(H) \big) \in [m^*]^{e(H)}$, let $C(H,\m)$ denote the event that the following all hold:
\begin{itemize}
\item[$(a)$] $\I(S,m^*) \cap \E(m^*) \cap \Y(m^*) \cap \Z(m^*) \cap \Q(m^*)$.\smallskip
\item[$(b)$] $\big\{ G_{m^*}[J',S] = H \big\} \cap \big\{ \sum_{J \setminus J'} a_j^2 < \delta |S|^2 \big\} \cap \big\{ \sum_J a_j < n^{1/2 + 2\delta} \big\}$. \smallskip
\item[$(c)$] For each $f \in E(H)$, the edge $f$ was added in step $m(f)$ of the triangle-free process.
\end{itemize}
Note that we suppress the dependence of $C(H,\m)$ on $S$ and $J'$, by encoding both sets in the graph $H$. By Observation~\ref{obs:C}, we have
\begin{equation}\label{eq:partitionoftheeventC}
\sum_{S \,:\, |S| = s} \Pr\Big( \C(S,\delta) \cap \E(m^*) \cap \Y(m^*) \cap \Z(m^*) \cap \Q(m^*) \Big) \, \le \, \sum_{H,\, \m} \Pr\big( C(H,\m) \big),
\end{equation}
where the second sum is over all graphs $H$ as described above, and sequences $\m \in [m^*]^{e(H)}$. 

As noted in the sketch above, our bound on the probability of $C(H,\m)$ will have two parts: a bound on the probability that the edges of $H$ are chosen at the steps corresponding to $\m$, and a bound on the number of `forbidden' open edges at each step. Recall that an open edge is forbidden if it is in the set $S$, or if it is a $Y$-neighbour of a still-open edge of $H$. We begin by controlling the number of open edges inside $S$. 

Let $J' = \{v_1,\ldots,v_k\}$ and $\n = (A_1,\ldots,A_k)$, where $A_j = N_H(v_j) \cap S$ for each $j \in [k]$. 

\medskip
\noindent \textbf{Claim 1:} With probability at least $1 - e^{-s n^{3\delta}}$, either $C(H,\m)^c$ holds, or
\begin{equation}\label{eq:indep:c:claim1}
o_\n(S,m) \, \ge \, \bigg( (1 - 2\delta){s \choose 2} - \sum_{j=1}^k {a_j \choose 2} \bigg) e^{-4t^2}
\end{equation}
for every $m \in [m^*]$.

\begin{proof}[Proof of claim]
The claim follows by applying Lemma~\ref{lemma:b} to a large family of collections $\n' \supseteq \n$, cf. Section~\ref{SubSecB}. Indeed, we claim that if $C(H,\m)$ holds, but~\eqref{eq:indep:c:claim1} fails to hold, then there exists $m \in [m^*]$ and $\n' = (A_1,\ldots,A_\ell)$, with $\sum_{j=k+1}^\ell |A_j|^2 < \delta |S|^2$, such that $\O_{\n'}(S,m)^c \cap \tilde{\n}'(S,m)$ holds. To see this, set $A_j = N_{G_{m^*}}(v_j) \cap S$ for each $k < j \le \ell$, where $J(S,\delta) = \{ v_1,\ldots,v_\ell \}$, and observe that the event $C(H,\m)$ implies that 
$$\sum_{j=k+1}^\ell |A_j|^2 \, = \, \sum_{v_j \in J \setminus J'} a_j^2 \, < \, \delta |S|^2.$$
It follows that
\begin{align*}
o_{\n'}(S,m) & \, \le \, o_\n(S,m) \, < \, \bigg( (1 - 2\delta){s \choose 2} - \sum_{j=1}^k {a_j \choose 2} \bigg) e^{-4t^2} \\
& \, \le \, \bigg( (1 - \delta){s \choose 2} - \sum_{j=1}^\ell {a_j \choose 2} \bigg) e^{-4t^2} \, < \, (1 - \delta) e^{-4t^2} o_{\n'}(S,0),
\end{align*}
assuming that~\eqref{eq:indep:c:claim1} fails to hold for $m$, and so we have $\O_{\n'}(S,m)^c$, as claimed. Moreover, note that $N_{G_{m+1}}(v_j) \cap S \subseteq A_j$ for every $j \in [\ell]$ and $|N_{G_m}(v) \cap S| \le n^\delta$ for every $v \not\in J$, and so the event $\tilde{\n}'(S,m)$ also holds, as required. 

We shall now apply Lemma~\ref{lemma:b} to each possible such collection $\n'$, and use the union bound. Observe that the claim is trivial if the bound in~\eqref{eq:indep:c:claim1} is negative, and so we may assume that
$$\sum_{j=1}^\ell {a_j \choose 2} \, \le \, \sum_{j=1}^k {a_j \choose 2} + \delta {s \choose 2} \, \le \, (1 - \delta){s \choose 2},$$
and hence, since $\sqrt{n} \ll s \le n^{1/2+\eps}$, that $|S| \sqrt{n} \le o_\n(S,0) \le n^{5/4}$. Recall also that, since $\sum_J a_j < n^{1/2 + 2\delta}$ and $|S| \ge 2\sqrt{n}$, there are at most $n^{|S| n^{2\delta}}$ possible collections $\n'$ as described above. By Lemma~\ref{lemma:b}, and summing over $m \in [m^*]$ and collections $\n'$ as described above, it follows that
$$\sum_{\n',m} \Pr\Big( \O_{\n'}(S,m)^c \cap \tilde{\n}'(S,m) \cap \Y(m) \cap \Q(m) \Big) \, \le \, m^* \cdot n^{|S| n^{2\delta}} \cdot e^{-|S| n^{4\delta}} \, \le \, e^{-s n^{3\delta}},$$
as required.
\end{proof}

Next we turn to the set $Y_H(m)$, i.e., to the forbidden open edges which are $Y$-neighbours of some not-yet-chosen edge of $H$. Recall from~\eqref{def:YHm:T} the definition of 
$$T(m) \, = \, \big\{ f \in E(H) \,:\, m(f) > m \big\},$$
and set $a(f) = a_j$ for each edge $f \in E(H)$, where $v_j$ is the endpoint of $f$ in $J'$. The following claim is a straightforward consequence of Lemma~\ref{lem:bigcupY}.  

\medskip
\noindent \textbf{Claim 2:} Suppose that $C(H,\m)$ holds. Then, for every $\omega \cdot n^{3/2} < m \le m^*$, 
$$\bigg| \bigcup_{f \in T(m)} Y_f(m) \bigg| \, \ge \, \left( \frac{1}{2} - \delta^2 \right) \Yt(m) |T(m)| \,+\, \sum_{f \in T(m)} \max\bigg\{  \left( \frac{1}{2} - \delta^2 \right) \Yt(m) \, - \, a(f) (\log n)^2 , \, 0\bigg\}.$$

\begin{proof}[Proof of claim]
Let us give an orientation to each edge $f \in T(m)$, by saying that its right foot is in $S$. The claim follows by applying Lemma~\ref{lem:bigcupY} twice, first to $T(m)$, and then to a suitably-chosen subset $T' \subseteq T$, with reversed orientations. 

Let us begin by observing that since $\I(S,m^*) \cap \Z(m^*)$ holds, we have
\begin{equation}\label{eq:boundingDGH}
\big| \Xi\big( G_{m^*}[V(H)] \big) \big| \, \le \, {|J'| \choose 2} + (\log n)^4 {|J'| \choose 2} \, \le \, n^{O(\delta)} \, \ll \, n^\eps \, \le \, \frac{\Yt(m^*)}{(\log n)^2}
\end{equation}
Indeed, since $S$ is an independent set in $G_{m^*}$, every edge which is double-covered by edges of $H$ is either inside $J'$, or is in the common $G_{m^*}$-neighbourhood of two vertices of $J'$. The first inequality now follows from the event $\Z(m^*)$, since each pair of vertices of $J'$ has at most $(\log n)^4$ edges in their common neighbourhood. The other inequalities follow since $|J'| \le n^{5\delta}$, since $\delta = \delta(\eps) \ll \eps$, and by the definitions of $\Yt$ and $m^*$.

We next claim that, since $\I(S,m^*) \cap \E(m^*) \cap \Z(m^*)$ holds, we have
\begin{equation}\label{eq:claimtwo:firstapp}
\bigg| \bigcup_{f \in T(m)} Y^L_f(m) \bigg| \, \ge \, \left( \frac{1}{2} - \delta^2 \right) \Yt(m) |T(m)|.
\end{equation}
To prove~\eqref{eq:claimtwo:firstapp}, apply Lemma~\ref{lem:bigcupY} to $T(m)$, and use~\eqref{eq:boundingDGH} to bound $\big| \Xi\big( G_{m^*}[V(H)] \big) \big|$. Since we have $Y^L_f(m) \in \big( \frac{1}{2} + o(1) \big) \Yt(m)$, by $\E(m^*)$, and $d^R_H(f) \le |J'| \le n^{5\delta} \ll \Yt(m) / (\log n)^2$ for every $f \in E(H)$, we obtain~\eqref{eq:claimtwo:firstapp}, as claimed.
 
Finally, recall that $a(f) = d_H^L(f)$ for each edge $f \in E(H)$, and define 
$$T'(m) \, = \, \bigg\{ f \in T(m) \,:\, \left( \frac{1}{2} - \delta^2 \right) \Yt(m) \ge a(f) (\log n)^2 \bigg\}.$$
We claim that
\begin{equation}\label{eq:claimtwo:secondapp}
\bigg| \bigcup_{f \in T'(m)} Y^R_f(m) \bigg| \, \ge \, \sum_{f \in T(m)} \max\bigg\{  \left( \frac{1}{2} - \delta^2 \right) \Yt(m) \, - \, a(f) (\log n)^2 , \, 0\bigg\}.
\end{equation}
This holds trivially if $|T'(m)| = 0$, so assume not and apply Lemma~\ref{lem:bigcupY} to $T'(m)$. Using the event $\E(m^*)$ and~\eqref{eq:boundingDGH}, we obtain 
$$\bigg| \bigcup_{f \in T'(m)} Y^R_f(m) \bigg| \, \ge \, \left( \frac{1}{2} - \delta^2 \right) \Yt(m) |T'(m)| \, - \, \sum_{f \in T'(m)} a(f) (\log n)^2,$$
which immediately implies~\eqref{eq:claimtwo:secondapp}. Combining~\eqref{eq:claimtwo:firstapp} and~\eqref{eq:claimtwo:secondapp}, and noting that $Y_f^L(m) \cap Y_h^R(m) = \emptyset$ for every $f,h \in E(H)$, since $S$ is independent, the claim follows.
\end{proof}

We are ready to prove our desired bound on the probability of the event $C(H,\m)$. 

\medskip
\noindent \textbf{Claim 3:} For every bipartite graph $H$ on $S \cup J'$, with $|S| = s$, $|J'| = k \le n^{5\delta}$ and $d_H(v) \ge n^{1/2 - 3\delta}$ for each $v \in J'$, and each collection $\m = \big( m(f) : f \in E(H) \big) \in [m^*]^{e(H)}$, 
$$\Pr\big( C(H,\m) \big) \, \le \, \ds\exp\bigg( \frac{m^*}{n^2} \sum_{j=1}^k a_j^2 - \left( \frac{1}{2} + \gamma^2 \right) s \log n \bigg) \prod_{j=1}^k \bigg( \frac{a_j}{n^{4-4\delta}} \bigg)^{a_j/2} +  e^{-s n^{3\delta}}.$$

\begin{proof}[Proof of claim]
If $C(H,\m)$ occurs, then at each non-$\m$ step of the triangle-free process we do not choose a forbidden open edge, and at step $m(f)$ we choose edge $f$, for each $f \in E(H)$. By Claim~1, the probability that $C(H,\m)$ occurs and~\eqref{eq:indep:c:claim1} fails to hold is at most $ e^{-s n^{3\delta}}$, so let us assume from now on that 
$$o_\n(S,m) \, \ge \, \bigg( (1 - 2\delta){s \choose 2} - \sum_{j=1}^k {a_j \choose 2} \bigg) e^{-4t^2}$$
for every $m \in [m^*]$. Moreover, by Claim~2, we have 
$$|Y_H(m)| \, \ge \, \left( \frac{1}{2} - \delta^2 \right) \Yt(m) |T(m)| \,+\, \sum_{f \in T(m)} \max\bigg\{  \left( \frac{1}{2} - \delta^2 \right) \Yt(m) \, - \, a(f) (\log n)^2 , \, 0\bigg\}$$
for every $\omega \cdot n^{3/2} < m \le m^*$, where $Y_H(m)$ was defined in~\eqref{def:YHm:T}. Moreover, as noted earlier, $O_\n(S,m)$ and $Y_H(m)$ are \emph{disjoint} sets of forbidden open edges, since if $f \in O_\n(S,m) \cap Y_H(m)$, then $f$ is contained in the $G_{m^*}$-neighbourhood (in $S$) of some vertex of $J'$ (that is, in one of the sets $A_j \in \n$), which contradicts the definition of~$O_\n(S,m)$.

It follows that the number of forbidden open edges at step $m+1 > \omega \cdot n^{3/2}$ is 
\begin{multline}\label{eq:forbiddenopenedges}
o_\n(S,m) + |Y_H(m)| \, \ge \, e^{-4t^2} \max\bigg\{ (1 - 2\delta){s \choose 2} - \sum_{j=1}^k {a_j \choose 2}, \, 0 \bigg\}  \\
\,+\, \left( \frac{1}{2} - \delta^2 \right) \Yt(m) |T(m)| \,+ \sum_{f \in T(m)} \max\bigg\{  \left( \frac{1}{2} - \delta^2 \right) \Yt(m) \, - \, a(f) (\log n)^2, \, 0\bigg\},
\end{multline}
and hence the probability of choosing a forbidden open edge\footnote{Note that, if this probability is $p$, then the probability of not choosing such an edge is at most $e^{-p}$, and the choices at each step are independent.} is at least the right-hand side of~\eqref{eq:forbiddenopenedges} divided by $Q(m)$. Let us take the three terms on the right-hand side of~\eqref{eq:forbiddenopenedges} one at a time. Indeed, since $\Q(m^*)$ holds, and setting $m_0 = \lceil \omega \cdot n^{3/2} \rceil$, 
we have
$$\sum_{m = m_0}^{m^*} \frac{e^{-4t^2}}{Q(m)} \max\bigg\{ (1 - 2\delta){s \choose 2} - \sum_{j=1}^k {a_j \choose 2}, \, 0 \bigg\}  \, \ge \, (1 - 3\delta) \frac{m^* s^2}{n^2} \,-\, \frac{m^*}{n^2} \sum_{j=1}^k a_j^2$$
and, recalling that $t(f) = m(f) \cdot n^{-3/2}$,
$$\sum_{m = 1}^{m^*} \frac{\Yt(m)|T(m)|}{Q(m)} \ge \sum_{m = 1}^{m^*} \sum_{f \in T(m)} \frac{(8-\delta)m}{n^3}  = \sum_{f \in E(H)} \sum_{m = 1}^{m(f)-1} \frac{(8-\delta)m}{n^3} \, \ge \, (4 - \delta) \sum_{f \in E(H)} t(f)^2.$$
In order to bound the final term in~\eqref{eq:forbiddenopenedges}, let us write $\hat{t}(f)$ for the time $\omega < t \le t^*$ at which $\Yt(m) = C \cdot a(f) (\log n)^2$, if such a time exists, and set $\hat{t}(f) = 0$ or $\hat{t}(f) = t^*$ otherwise,\footnote{Of course, if $\Yt(\omega \cdot n^{3/2}) < C \cdot a(f) (\log n)^2$ then set $\hat{t}(f) = 0$, and if $\Yt(m^*) > C \cdot a(f) (\log n)^2$  then set $\hat{t}(f) = t^*$. Note for future reference that we have $e^{-4\hat{t}(f)^2} \le a(f) \cdot e^{-4(t^*)^2}$, since $e^{4(t^*)^2} \le n^{1/2 - \eps}$.} in the obvious way, and set $\hat{m}(f) = \hat{t}(f) \cdot n^{3/2}$. We obtain
\begin{multline*}
\sum_{m=1}^{m^*} \frac{1}{Q(m)} \sum_{f \in T(m)} \max\bigg\{  \left( \frac{1}{2} - \delta^2 \right) \Yt(m) \, - \, a(f) (\log n)^2 , \, 0 \bigg\} \\
 \, \ge \sum_{f \in E(H)} \sum_{m=1}^{\min\{ \hat{m}(f), m(f)\} }  \left( \frac{1}{2} - 2\delta^2 \right) \frac{8m}{n^3} \, \ge \, \big( 2 - \delta \big) \sum_{f \in E(H)}  \min\big\{ \hat{t}(f), \, t(f) \big\}^2.
\end{multline*}
Combining the last several inequalities, it follows that the probability that we avoid choosing a forbidden open edge at every (non-$\m$) step of the triangle-free process is at most\footnote{Here we use the fact that $\sum_{m = 1}^{m_0} \frac{\Yt(m)|T(m)|}{Q(m)} \le  O(\omega^2) \cdot e(H)$, which is swallowed by the error term.}
\begin{equation}\label{eq:probavoidforbiddenopenedges}
\exp\Bigg[ \frac{m^*}{n^2} \sum_{j=1}^k a_j^2 - \big( 1 - 3\delta \big) \bigg( \frac{m^* s^2}{n^2} + 2 \sum_{f \in E(H)} \Big( t(f)^2 + \min\big\{ \hat{t}(f), \, t(f) \big\}^2 \Big) \bigg) \Bigg].
\end{equation}
Finally, note that, since $\Q(m^*)$ holds, the probability that we choose the edge $f$ at step $m(f)$ for each edge $f \in E(H)$ is
\begin{equation}\label{eq:Qprod:c}
\prod_{f \in E(H)} \frac{1}{Q\big( m(f) \big)} \, \le \, \bigg( \frac{4}{n^2} \bigg)^{e(H)} \exp\bigg( 4 \sum_{f \in E(H)} t(f)^2 \bigg).
\end{equation}
It follows that the probability of the event $C(H,\m)$ is at most the product of~\eqref{eq:probavoidforbiddenopenedges} and the right-hand side of~\eqref{eq:Qprod:c}. The remainder of the proof is a straightforward calculation. Indeed, note first that
\begin{equation}\label{eq:increasingts}
\varphi(H,\m) \, := \, \big( 1 + 3\delta \big)  \sum_{f \in E(H)} t(f)^2  - \big( 1 - 3\delta \big) \sum_{f \in E(H)} \min\big\{ \hat{t}(f), \, t(f) \big\}^2
\end{equation}
is increasing in $t(f)$ for every $f \in E(H)$. Moreover, recalling that $\hat{t}(f)$ depends only on $a(f) = d_H\big( v(f) \big)$, where $v(f)$ is the endpoint of $f$ in $J'$, we may define $\hat{t}(v_j) = \hat{t}(f)$ for any edge $f \in E(H)$ which is incident to $v_j$, for each $j \in [k]$. It follows that
$$\sum_{f \in E(H)} \hat{t}(f)^2 \, = \, \sum_{j = 1}^{k} \hat{t}(v_j)^2 a_j,$$
since each vertex $v_j$ is counted exactly $d_H(v_j) = a_j$ times in the sum on the left. Since $t(f) \le t^*$ and $\hat{t}(f) \le t^*$ for every $f \in E(H)$, and $e(H) = \sum_{j=1}^k a_j$, it follows that 
$$\varphi(H,\m) \, \le \; \sum_{j = 1}^{k} a_j \Big( (1 + 6\delta) (t^*)^2 - \hat{t}(v_j)^2 \Big).$$
Recalling that $e^{-4\hat{t}(v_j)^2} \le a_j \cdot e^{-4(t^*)^2}$, and noting that $e^{24\delta(t^*)^2} \le n^{3\delta}$, it follows that
\begin{equation}\label{eq:boundon:phiHm}
e^{4 \varphi(H,\m)} \, \le \, \prod_{j = 1}^{k} \big( n^{3\delta} \cdot a_j \big)^{a_j}.
\end{equation}
Now, multiplying~\eqref{eq:probavoidforbiddenopenedges} by the right-hand side of~\eqref{eq:Qprod:c}, and using~\eqref{eq:boundon:phiHm}, we obtain
$$\Pr\big( C(H,\m) \big) \, \le \, \bigg( \frac{4}{n^2} \bigg)^{e(H)} \Bigg[ \prod_{j = 1}^{k} \big( n^{3\delta} \cdot a_j \big)^{a_j/2} \Bigg] \exp\Bigg[ \frac{m^*}{n^2} \sum_{j=1}^k a_j^2 - \big( 1 - \eps \big) \frac{m^* s^2}{n^2} \Bigg]  +  e^{-s n^{3\delta}}.$$
Finally, observe that
$$\exp\bigg( - \big( 1 - \eps \big) \frac{m^* s^2}{n^2} \bigg) \, \le \, \exp\bigg( - \left( \frac{1}{2} + \gamma^2 \right) s \log n \bigg),$$
since $s \ge \big( \sqrt{2} + \gamma \big) \sqrt{n \log n}$. It follows that 
\begin{equation}\label{eq:finalboundonCHm}
\Pr\big( C(H,\m) \big) \, \le \, \prod_{j=1}^k \bigg( \frac{a_j}{n^{4-4\delta}} \bigg)^{a_j/2} \exp\Bigg[ \frac{m^*}{n^2} \sum_{j=1}^k a_j^2 - \left( \frac{1}{2} + \gamma^2 \right) s \log n \Bigg]  +  e^{-s n^{3\delta}},
\end{equation}
as claimed.
\end{proof}

We are finally ready to sum the probability of $C(H,\m)$ over $H$ and $\m$. To simplify the counting, let us first fix $k$ and $\a$. Note that we have ${n \choose s}$ choices for $S$, at most $n^k$ choices for $J'$ (given $k$), at most $\prod_{j=1}^k {s \choose a_j}$ choices for $H$ (given $S$, $J'$ and $\a$), and at most $(m^*)^{e(H)}$ choices for $\m$. Note that, since $e(H) \le n^{1/2 + 2\delta}$, we may disregard the final term in~\eqref{eq:finalboundonCHm}. 

Since $e(H) = \sum_{j=1}^k a_j$, we have
$$(m^*)^{e(H)} \prod_{j=1}^k {s \choose a_j} \bigg( \frac{a_j}{n^{4-4\delta}} \bigg)^{a_j/2} \exp\bigg( \frac{m^*}{n^2} \sum_{j=1}^k a_j^2 \bigg) \le \; \prod_{j=1}^k \bigg[ \frac{es}{a_j} \cdot \frac{m^* \sqrt{a_j}}{n^{2-2\delta}} \cdot \exp\bigg(  \frac{m^* a_j}{n^2} \bigg) \bigg]^{a_j},$$
and hence, since $s m^* \le n^2 \log n$, it follows that\footnote{Here the sum is over those $H$ and $\m$ with the given values of $k$ and $\a$.}
$$\sum_{H,\, \m} \Pr\big( C(H,\m) \big) \, \le \, n^k \cdot {n \choose s} \cdot \exp\bigg( - \left( \frac{1}{2} + \gamma^2 \right) s \log n \bigg) \cdot \prod_{j=1}^k \bigg[ \frac{n^{3\delta}}{\sqrt{a_j}} \cdot \exp\bigg(  \frac{m^* a_j}{n^2} \bigg) \bigg]^{a_j}.$$
But since $\E(m^*)$ holds, we have $a_j \le \big( \frac{1}{\sqrt{2}} + \delta \big) \sqrt{n \log n}$ for each $j \in [k]$, and thus
$$\frac{n^{3\delta}}{\sqrt{a_j}} \cdot \exp\bigg( \frac{m^* a_j}{n^2} \bigg)  \, \le \, \frac{n^{3\delta}}{n^{1/4}} \exp \bigg( \left( \frac{1}{4} - \eps^2 \right) \log n \bigg) \, \ll \, 1.$$ 
Hence, summing over $k \le n^{5\delta}$ and sequences $\a = (a_1,\ldots,a_k)$, we obtain
\begin{align*}
\sum_{H,\, \m} \Pr\big( C(H,\m) \big) & \, \le \, {n \choose s} \sum_{k = 1}^{n^{5\delta}} s^k \cdot n^k \cdot \exp\bigg( - \left( \frac{1}{2} + \gamma^2 \right) s \log n \bigg)\\
& \, \le \,  \sum_{k = 1}^{n^{5\delta}} n^{2k} \bigg( \frac{en}{s} \cdot n^{-1/2 - \gamma^2} \bigg)^s \, \le \,  \sum_{k = 1}^{n^{5\delta}} n^{2k - \gamma^2 s} \, \le \, n^{-\delta s},
 \end{align*}
since $s \gg \sqrt{n}$, as required. By~\eqref{eq:partitionoftheeventC}, the lemma follows.
\end{proof}

\subsection{The event $\D(S,\delta)$}\label{SubSecD}

Recall that $\D(S,\delta)$ denotes the event that $S$ is an independent set in $G_{m^*}$ and $\sum_J a_j \ge n^{1/2 + 2\delta}$. In this subsection we shall bound the probability that the event $\D(S,\delta)$ occurs for some $S \subseteq V(G_{n,\triangle})$ with $|S| \le \sqrt{n} \log n$. Unlike in the previous three subsections, it will not suffice to use the union bound over sets $S$; nevertheless, we shall prove the following bound.

\begin{lemma}\label{indep:d}
If $s \le n^{1/2 + \delta^3}$, then
\begin{equation}\label{eq:lemma:d}
\Pr\bigg( \bigcup_{S \,:\, |S| = s} \D(S,\delta) \cap \E(m^*) \cap \Z(m^*) \cap \Q(m^*) \bigg) \, \le \, n^{-\sqrt{n}}.
\end{equation}
\end{lemma}

The idea of the proof is as follows: if $\sum_J a_j \ge n^{1/2 + 2\delta}$, then the bipartite graph $G_{m^*}[S,J]$ contains a subgraph $H = G_{m^*}[S^*,J^*]$ whose appearance in $G_{m^*}$ is highly unlikely. This subgraph will have the following two properties: 
$$(a) \;\; e(H) \ge n^{1/2 + \delta} \qquad \text{and} \qquad (b) \; \; H \text{ is `close to regular on both sides',}$$ 
i.e., $d_H(u) / d_H(v) = \Theta(1)$ for every $u,v \in S^*$, and also for every $u,v \in J^*$. (This regularity property will be useful to us twice: in bounding the number of forbidden open edges at each step, and in the final calculation.) The calculation required to bound the probability that such a graph $H$ occurs in $G_{m^*}$ is similar to (but somewhat simpler than) that in the previous section; the main difference is that we cannot control the number of open edges inside $S$, and so the only forbidden open edges are those in $Y_H(m)$. 

We begin with a straightforward (and probably well-known) lemma. 

\begin{lemma}\label{lem:choosingHforD}
Let $G$ be a bipartite graph on vertex set $A \cup B$ with $e(G) \ge v(G)$, and let $\nu \in (0,1)$. There exist sets $A^* \subseteq A$ and $B^* \subseteq B$ such that the induced bipartite subgraph $H = G[A^*,B^*]$ has the following properties:
\begin{equation}\label{eq:lem:choosingHforD}
e(H) \ge \frac{e(G)^{1 - \nu}}{4}, \qquad  \Delta_H(A^*) \, \le \, 2^{7/\nu} \delta_H(A^*) \qquad \text{and} \qquad \Delta_H(B^*) \, \le \, 2^{7/\nu}  \delta_H(B^*),
\end{equation}
where $\Delta_H(S) = \max\{ d_H(u) : u \in S\}$ and $\delta_H(S) = \min\{ d_H(u) : u \in S \}$. 
\end{lemma}

In the proof of Lemma~\ref{lem:choosingHforD} we will use the following simple (and standard) observation.

\begin{obs}\label{obs:mindegree}
Let $G$ be a bipartite graph  on vertex set $A \cup B$. There exist sets $A^* \subseteq A$ and $B^* \subseteq B$ such that the induced bipartite subgraph $H = G[A^*,B^*]$ has the following properties:
$$e(H) \ge \frac{e(G)}{2}, \qquad  \delta_H(A^*) \, \ge \, \frac{\ol{d}_G(A)}{4} \qquad \text{and} \qquad \delta_H(B^*) \, \ge \, \frac{\ol{d}_G(B)}{4},$$
where $\ol{d}_G(A)$ (resp. $\ol{d}_G(B)$) denotes the average degree of a vertex of $A$ (resp. $B$) in $G$.
\end{obs}

\begin{proof}
Simply remove, one by one, vertices in $A$ with degree at most $\ol{d}_G(A)/4$ in the remaining graph, and vertices in $B$ with degree at most $\ol{d}_G(B)/4$. Since we clearly remove at most $e(G)/2$ edges in total, this process ends when we reach the desired subgraph $H$.
\end{proof}

We can now prove Lemma~\ref{lem:choosingHforD}.

\begin{proof}[Proof of Lemma~\ref{lem:choosingHforD}]
Note first that, by Observation~\ref{obs:mindegree}, it will suffice to find an induced subgraph $H = G[A^*,B^*]$ with 
\begin{equation}\label{eq:lem:choosingHforD2}
e(H) \ge \frac{e(G)^{1 - \nu}}{2}, \qquad  \Delta_H(A^*) \, \le \, 2^{5/\nu} \ol{d}_H(A^*) \qquad \text{and} \qquad \Delta_H(B^*) \, \le \, 2^{5/\nu} \ol{d}_H(B^*).
\end{equation}
We shall therefore aim to satisfy the properties~\eqref{eq:lem:choosingHforD2} instead of~\eqref{eq:lem:choosingHforD}. 

Let $A^* = A' \setminus A''$ and $B^* = B' \setminus B''$, where $A'$, $B'$, $A''$ and $B''$ are defined as follows. Set $\ell = 2/\nu$, and let $A' \subseteq A$ and $B' \subseteq B$ be a pair of subsets which maximizes 
$$w\big( A', B' \big) \, = \, \frac{e\big( G[A',B'] \big)^\ell}{|A'| |B'|}.$$
Now let $A'' \subseteq A'$ and $B'' \subseteq B'$ denote the high degree vertices of $H' = G[A',B']$, i.e.,
$$A'' \, = \, \big\{ u \in A' \,:\, d_{H'}(u) \ge 4^\ell \cdot \ol{d}_{H'}(A') \big\} \quad \text{and} \quad B'' \, = \, \big\{ u \in B' \,:\, d_{H'}(u) \ge 4^\ell \cdot \ol{d}_{H'}(B') \big\}.$$

In order to see that the subgraph $H = G[A^*,B^*]$ satisfies the properties~\eqref{eq:lem:choosingHforD2}, observe first that, by the maximality of the pair $(A',B')$, we have
$$w \big(A',B' \big) \ge w \big(A,B \big).$$
We claim that therefore $e(H') \ge e(G)^{1-\nu}$. Indeed, we have
$$e(H') \, = \, e\big( G[A',B'] \big) \, \ge \, \bigg( \frac{|A'| \cdot |B'|}{|A| \cdot |B|} \bigg)^{1/\ell} e(G) \, \ge \, \frac{e(G)}{v(G)^{2/\ell}} \, \ge \, e(G)^{1-\nu},$$
since $v(G) \le e(G)$ and $\ell = 2 / \nu$. Moreover, note that we have $\Delta_H(A^*) \le 4^\ell \cdot \ol{d}_{H'}(A')$ and $\Delta_H(B^*) \le 4^\ell \cdot \ol{d}_{H'}(B')$, by construction. 

We claim that $e(H) \ge e(H')/2$, which, together with the bounds proved above, implies the inequalities~\eqref{eq:lem:choosingHforD2}. Indeed, by counting edges we have
$$4^\ell \cdot \ol{d}_{H'}(A') \cdot |A''| \, \le \, e(A'',B') \, \le \, e(A',B') \, = \, \ol{d}_{H'}(A') \cdot |A'|,$$ 
and so $|A''| \le 4^{-\ell} |A'|$ and similarly $|B''| \le 4^{-\ell} |B'|$. It follows that 
$$2 \cdot w \big(A',B' \big) \, \ge \, w\big( A'',B' \big) + w\big( A',B'' \big) \, \ge \, \frac{4^\ell}{|A'| |B'|} \Big( e\big( A'',B' \big)^\ell + e\big( A',B'' \big)^\ell \Big),$$ 
by the maximality of $(A',B')$. Using Jensen's inequality, we obtain
$$e\big( A'',B' \big) + e\big( A',B'' \big) \, \le \, 2^{1 - 1/\ell} \cdot \Big( e\big( A'',B' \big)^\ell + e\big( A',B'' \big)^\ell \Big)^{1/\ell} \, \le \, \frac{e\big( A',B' \big)}{2},$$
and hence $e(H) \ge e(H') / 2$, as required.
\end{proof}

We are now ready to prove Lemma~\ref{indep:d}.

\begin{proof}[Proof of Lemma~\ref{indep:d}]
Let $S \subseteq V(G_{n,\triangle})$ be a set of size $s \le n^{1/2 + \delta^3}$, and recall that we denote by $J = J(S,\delta) = \{v_1,\dots,v_k\}$ the set of vertices with at least $n^\delta$ $G_{m^*}$-neighbours in $S$, and set $a_j = |N_{G_m^*}(v_j) \cap S|$ for each $j \in [k]$. Suppose that the event $\sum_{J} a_j \ge n^{1/2 + 2\delta}$ occurs for $S$; we make the following deterministic claim.

\medskip
\noindent \textbf{Claim~1:} There exist sets $S^* \subseteq S$ and $J^* \subseteq J$ such that the induced bipartite subgraph $H = G_{m^*}[S^*,J^*]$ has the following properties:
\begin{equation}\label{eq:d:claim1}
e(H) \ge n^{1/2 + \delta}, \quad  n^{\delta/2} \le \Delta_H(S^*) \le C \cdot \delta_H(S^*) \quad \text{and} \quad n^{\delta/2} \le \Delta_H(J^*) \le C \cdot \delta_H(J^*).
\end{equation}

\begin{proof}[Proof of claim]
Choose a minimal subset $J_0 \subset J$ such that $\sum_{j \in J_0} a_j \ge n^{1/2 + 2\delta}$, and apply Lemma~\ref{lem:choosingHforD} to the graph $G_{m^*}[S,J_0]$, with $\nu = \delta / 2$. The upper bounds on $\Delta_H(S^*)$ and $\Delta_H(J^*)$ follow immediately, since $C = C(\eps,\delta)$ was chosen to be sufficiently large as a function of $\delta$, and it also follows that $e(H) \gg n^{1/2 + 3\delta/2}$. Since $|S|,|J_0| \le n^{1/2 + \delta}$ (the latter by minimality), we also obtain the claimed lower bounds.
\end{proof}

Now, using Claim~1, let us break up the bad event in~\eqref{eq:lemma:d} into more manageable pieces, as follows. For each $S^*,J^* \subseteq V(G_{n,\triangle})$, each bipartite graph $H$ on $S^* \cup J^*$ satisfying~\eqref{eq:d:claim1}, and each collection $\m = \big( m(f) : f \in E(H) \big) \in [m^*]^{e(H)}$, let $D(H,\m)$ denote the event that the following all hold:
\begin{itemize}
\item[$(a)$] $\I(S,m^*) \cap \E(m^*) \cap \Z(m^*) \cap \Q(m^*)$.\smallskip
\item[$(b)$] $\big\{ G_{m^*}[S^*,J^*] = H \big\}$. \smallskip
\item[$(c)$] For each $f \in E(H)$, the edge $f$ was added in step $m(f)$ of the triangle-free process.
\end{itemize}
As in the previous subsection we suppress the dependence of $D(H,\m)$ on $S^*$ and $J^*$, by encoding both sets in the graph $H$. By Claim~1, we have
\begin{equation}\label{eq:partitionoftheeventD}
\Pr\bigg( \bigcup_{S \,:\, |S| = s} \D(S,\delta) \cap \E(m^*) \cap \Z(m^*) \cap \Q(m^*) \bigg) \, \le \, \sum_{H,\, \m} \Pr\big( D(H,\m) \big),
\end{equation}
where the sum is over all graphs $H$ as described above, and sequences $\m \in [m^*]^{e(H)}$. 

Our bound on the probability of $D(H,\m)$ is similar to that of $C(H,\m)$ in the previous subsection. It again consists of two parts: a bound on the probability that the edges of $H$ are chosen at the steps corresponding to $\m$, and a bound on the number of `forbidden' open edges at each step. This time, however, an open edge is forbidden only if it is in $Y_H(m)$, i.e., if it is a $Y$-neighbour of a still-open edge of $H$. 

In order to bound the number of forbidden edges, we shall need the following claim, which follows by modifying the proof\footnote{Observe that our upper bound on $|J^*|$ is too weak to control the size of $\Xi\big( G_{m^*}[V(H)] \big)$, and so it is not sufficient to simply apply Lemma~\ref{lem:bigcupY}.} of Lemma~\ref{lem:bigcupY}. We again define 
$$T(m) \, = \, \big\{ f \in E(H) \,:\, m(f) > m \big\},$$
and let $d_H(f) = d_H^L(f) + d_H^R(f)$ for each edge $f \in E(H)$. 

\medskip
\noindent \textbf{Claim~2:} Suppose that $D(H,\m)$ holds. Then, for every $\omega \cdot n^{3/2} < m \le m^*$,  
$$\bigg| \bigcup_{f \in T(m)} Y_f(m) \bigg| \, \ge \, \sum_{f \in T(m)} \max\Big\{  \big( 1 - \delta^3 \big) \Yt(m) - d_H(f) (\log n)^3, \, 0 \Big\}.$$

\begin{proof}[Proof of claim]
We consider, for each open edge $f \in T(m)$, the following subset of the $Y$-neighbours of $f$:
$$\hat{Y}^H_f(m) \, := \, \Big\{ h \in Y_f(m) \,:\, \{u,v\} \not\in E(H), \textup{ where } u = f \setminus h \textup{ and } v = h \setminus f \Big\}.$$
In words, we consider only those $Y$-neighbours of $f$ in $G_m$ such that the step in the $Y$-graph (from $f$ to $h$, say) does not use an edge of $H$. The motivation for this definition is that $\hat{Y}^H_f(m) \ge Y_f(m) - d_H(f)$, and if $f$ and $h$ are disjoint edges of $T(m)$, then 
\begin{equation}\label{eq:emptyintersectionofhats}
\hat{Y}^H_f(m) \cap \hat{Y}^H_h(m) = \emptyset.
\end{equation}
To see~\eqref{eq:emptyintersectionofhats}, let $e \in \hat{Y}^H_f(m) \cap \hat{Y}^H_h(m)$, and note that either $e$ has an endpoint in each of $S^*$ and $J^*$, which is impossible because $S$ is independent, or $e$ is contained in either $S^*$ or $J^*$, which is impossible because $H = G_{m^*}[S^*,J^*]$ is the bipartite graph induced by $S^* \cup J^*$, and so the edge linking $e$ and $f$ in the $Y$-graph must be an edge of $H$. 

It follows from~\eqref{eq:emptyintersectionofhats}, together with the proof of~\eqref{eq:inex} in the previous subsection, that
\begin{equation}\label{eq:intersection:D}
\sum_{h \in T(m) \setminus \{f\}} \big| \hat{Y}^H_f(m) \cap \hat{Y}^H_h(m) \big| \, \le \, d_H(f) (\log n)^2.
\end{equation}
Indeed, since $\Z(m)$ holds we have $|\hat{Y}^H_f(m) \cap \hat{Y}^H_h(m)| \le (\log n)^2$ for every pair of edges $f,h \in O(G_m)$, by Observation~\ref{obs:twostep}, and there are at most $d_H(f)$ edges $h \in T(m) \setminus \{f\}$ which intersect $f$, so~\eqref{eq:intersection:D} follows.

The claim now follows by inclusion-exclusion. Indeed, let us define
$$T'(m) \, = \, \bigg\{ f \in T(m) \,:\, \big( 1 - \delta^3 \big) \Yt(m) \ge d_H(f) (\log n)^3 \bigg\},$$
and observe that, by~\eqref{eq:intersection:D}, and since the event $\E(m^*)$ holds,
\begin{align*}
\bigg| \bigcup_{f \in T(m)} Y_f(m) \bigg| & \, \ge \, \bigg| \bigcup_{f \in T'(m)} \hat{Y}^H_f(m) \bigg| \, \ge \, \sum_{f \in T'(m)} \big| \hat{Y}^H_f(m) \big| \, - \, \sum_{\substack{f,h \in T(m')\\ f \neq h}} \big| \hat{Y}^H_f(m) \cap \hat{Y}^H_h(m) \big|\\
& \, \ge \, \sum_{f \in T'(m)} \Big( \big( 1 - \delta^3 \big) \Yt(m) - d_H(f) \Big) \, - \, \sum_{f \in T'(m)} d_H(f) (\log n)^2\\
& \, \ge \, \sum_{f \in T(m)} \max\Big\{  \big( 1 - \delta^3 \big) \Yt(m) - d_H(f) (\log n)^3, \, 0 \Big\}
\end{align*}
for every $\omega \cdot n^{3/2} < m \le m^*$, as required.
\end{proof}

We are ready to prove our desired bound on the probability of the event $D(H,\m)$. 

\medskip
\noindent \textbf{Claim 3:} For every bipartite graph $H$ satisfying~\eqref{eq:d:claim1}, and each collection $\m \in [m^*]^{e(H)}$, 
\begin{equation}\label{eq:d:claim3}
\Pr\big( D(H,\m) \big) \, \le \, \bigg( \frac{1}{n^{2 + \delta^2}} \bigg)^{e(H)} \prod_{f \in E(H)} d_H(f).
\end{equation}

\begin{proof}[Proof of claim]
If $D(H,\m)$ occurs, then at each non-$\m$ step of the triangle-free process we do not choose a forbidden open edge, and at step $m(f)$ we choose edge $f$, for each $f \in E(H)$. By Claim~2, there are at least
$$\sum_{f \in T(m)} \max\Big\{  \big( 1 - \delta^3 \big) \Yt(m) - d_H(f) (\log n)^3, \, 0 \Big\}$$
forbidden open edges in $G_m$, for every $\omega \cdot n^{3/2} < m \le m^*$.

Similarly (but not identically) as before, let us write $\hat{t}(f)$ for the time $t > \omega$ at which $\Yt(m) = C \cdot d_H(f) (\log n)^3$, if such a time exists, and set $\hat{t}(f) = 0$ otherwise. Note that the degree $d_H(f)$ of an edge is the same up to a factor of $C$ for all edges of $H$, by~\eqref{eq:d:claim1}. Since the event $\Q(m^*)$ holds, we have
\begin{multline*}
\sum_{m=1}^{m^*} \frac{1}{Q(m)} \sum_{f \in T(m)} \max\Big\{  \big( 1 - \delta^3 \big) \Yt(m) - d_H(f) (\log n)^3, \, 0 \Big\} \\
 \, \ge \sum_{f \in E(H)} \sum_{m=1}^{\min\{ \hat{m}(f), m(f)\} }  \Big( 8 - O\big( \delta^3 \big) \Big) \frac{m}{n^3} \, \ge \, \Big( 4 - O\big( \delta^3 \big) \Big) \sum_{f \in E(H)}  \min\big\{ \hat{t}(f), \, t(f) \big\}^2,
\end{multline*}
where $\hat{m}(f) = \hat{t}(f) \cdot n^{3/2}$, and hence the probability that we avoid choosing a forbidden open edge at every (non-$\m$) step of the triangle-free process is at most
\begin{equation}\label{eq:probavoidforbiddenopenedges:D}
\exp\bigg( - \big( 4 - \delta^2 \big) \sum_{f \in E(H)} \Big( \min\big\{ \hat{t}(f), \, t(f) \big\}^2  - \omega^3 \Big) \bigg),
\end{equation}
since $\sum_{m = 1}^{\omega \cdot n^{3/2}} \sum_{f \in T(m)} \frac{\Yt(m)}{Q(m)} \le  \omega^3 \cdot e(H)$. Next, note that since $\Q(m^*)$ holds, the probability that we choose the edge $f$ at step $m(f)$ for each edge $f \in E(H)$ is
\begin{equation}\label{eq:Qprod:d}
\prod_{f \in E(H)} \frac{1}{Q\big( m(f) \big)} \, \le \, \bigg( \frac{4}{n^2} \bigg)^{e(H)} \exp\bigg( 4 \sum_{f \in E(H)} t(f)^2 \bigg).
\end{equation}
It follows that the probability of the event $D(H,\m)$ is at most the product of~\eqref{eq:probavoidforbiddenopenedges:D} and the right-hand side of~\eqref{eq:Qprod:d}. Note that this product is increasing in $t(f)$ for each $f \in E(H)$, so we may set $t(f) = t^*$. Since $e^{4(t^*)^2} \le n^{1/2 - 2\eps}$, we obtain\footnote{Note that $e^{4\omega^3} \ll n^\delta$, which is easily swallowed by the error term.} 
\begin{equation}\label{eq:claim3:almost}
\Pr\big( D(H,\m) \big) \, \le \, \bigg( \frac{1}{n^{3/2 + \eps}} \bigg)^{e(H)} \exp\bigg( - \big( 4 - \delta^2 \big) \sum_{f \in E(H)} \min\big\{ \hat{t}(f), \, t^* \big\}^2 \bigg).
\end{equation}
Suppose first that $\hat{t}(f) \le t^*$ for every $f \in E(H)$. Then~\eqref{eq:d:claim3} follows easily from~\eqref{eq:claim3:almost}, since we have $e^{-4\hat{t}(f)^2} \le d_H(f) \cdot n^{-1/2 + \delta}$ for every $f \in E(H)$, and hence
$$\Pr\big( D(H,\m) \big) \, \le \, \bigg( \frac{1}{n^{3/2 + \eps}} \bigg)^{e(H)} \prod_{f \in E(H)} \Big( d_H(f) \cdot n^{-1/2 + \delta} \Big)^{1-\eps}  \, \le \, \bigg( \frac{1}{n^{2 + \delta}} \bigg)^{e(H)} \prod_{f \in E(H)} d_H(f),$$
as claimed. On the other hand, if $\hat{t}(f) > t^*$ for some $f \in E(H)$, then $\hat{t}(h)^2 > (t^*)^2 - C$ for every $h \in E(H)$, by~\eqref{eq:d:claim1}. It follows from~\eqref{eq:probavoidforbiddenopenedges:D} and~\eqref{eq:Qprod:d} that
$$\Pr\big( D(H,\m) \big) \, \le \, \bigg( \frac{\log n}{n^2} \bigg)^{e(H)} \exp\bigg( \delta^2 \sum_{f \in E(H)} t(f)^2 \bigg) \, \le \, \bigg( \frac{1}{n^{2 + \delta^2}} \bigg)^{e(H)} \prod_{f \in E(H)} d_H(f),$$
since $t(f) \le t^*$ and $d_H(f) \ge \delta(H) \ge n^{\delta/2}$ for every $f \in E(H)$, as required.
\end{proof}

We are finally ready to sum the probability of $D(H,\m)$ over $H$ and $\m$. Fixing $|S^*|$, $|J^*|$ and $e(H)$ (we will sum over these at the very end), observe that we have at most 
$${n \choose |S^*|} {n \choose |J^*|} {|S^*||J^*| \choose e(H)} \, \le \, \bigg( \frac{3 |S^*||J^*|}{e(H)} \bigg)^{e(H)}$$ 
choices for $H$, since $e(H) \gg \big( |S^*| + |J^*| \big) \log n$, by the lower bounds on $\delta_H(S^*)$ and $\delta_H(J^*)$ given by Claim~1. Note also that, by the AM-GM inequality, 
$$\prod_{f \in E(H)} d_H(f) \, \le \, \bigg( \frac{1}{e(H)} \sum_{f \in E(H)} d_H(f) \bigg)^{e(H)} \, \le \, \bigg( \frac{C \cdot e(H)}{|J^*|} \bigg)^{e(H)},$$
where the second inequality follows since $H$ is almost regular. Hence, by Claim~3,
$$\sum_{H,\, \m} \Pr\big( D(H,\m) \big) \le \sum_{\substack{|S^*|,|J^*|\\ e(H)}} \bigg( \frac{3 |S^*||J^*|}{e(H)} \cdot m^* \cdot  \frac{1}{n^{2 + \delta^2}} \cdot \frac{C \cdot e(H)}{|J^*|} \bigg)^{e(H)} \le \sum_{e(H)} n^{-\delta^3 e(H)} \, \le \, n^{-\sqrt{n}},$$
since $|S^*| \le s \le n^{1/2 + \delta^3}$ and $e(H) \ge n^{1/2 + \delta}$, as required. By~\eqref{eq:partitionoftheeventD}, the lemma follows.
\end{proof}

\subsection{The proof of Propositions~\ref{prop:maxdeg} and~\ref{prop:indep}}

The two main propositions of this section follow easily from the lemmas above. Recall from~\eqref{def:W} and Definitions~\ref{def:ABCD} and~\ref{def:ABCprime} the definitions of the events $\W(S,v)$, $\A'(S,\delta)$, $\B'(S,\delta)$, $\C'(S,\delta)$ and $\D(S,\delta)$. 

\begin{proof}[Proof of Proposition~\ref{prop:maxdeg}]
Set $s =  \big( \frac{1}{\sqrt{2}} + \gamma \big) \sqrt{ n \log n }$, and observe that if $\Delta\big( G_{n,\triangle} \big) \ge s$, then the event 
$$\T(S,v) \, = \, \Big\{ N(v) = S \text{ in } G_{n,\triangle} \Big\},$$ 
holds for some $v \in V(G_{n,\triangle})$, and some $S \subseteq V(G_{n,\triangle})$ with $|S| \ge s$. Now, by~\eqref{eq:UinABCD}, we have
$$\bigcup_{|S| \ge s} \T(S,v) \cap \E(m^*)  \, \subseteq \, \bigcup_{|S| = s} \Big( \A'(S,\delta) \cup \B'(S,\delta) \cup \C'(S,\delta) \cup \D(S,\delta) \Big) \cap \W(S,v),$$
and thus, setting $\F'(S,\delta) = \A'(S,\delta) \cup \B'(S,\delta) \cup \C'(S,\delta) \cup \D(S,\delta)$, it will suffice to show that
\begin{equation}\label{maxdeg:finalbound}
\sum_{v \in V(G_{n,\triangle})} \Pr\bigg( \bigcup_{S \,:\, |S| = s} \F'(S,\delta) \cap \W(S,v) \cap \E(m^*) \cap \Y(m^*) \cap \Z(m^*) \cap \Q(m^*) \Big) \, \le \, e^{-\sqrt{n}}.
\end{equation}
To prove~\eqref{maxdeg:finalbound}, recall that by Lemma~\ref{maxdeg:a}, we have
$$\sum_{S \,:\, |S| = s} \Pr\Big( \A'(S,\delta) \cap W(S,v) \cap \E(m^*) \cap \Q(m^*) \Big) \, \le \, n^{-\delta s},$$
by Lemma~\ref{maxdeg:b}, we have
$$\sum_{S \,:\, |S| = s} \Pr\Big( \B'(S,\delta) \cap \E(m^*) \cap \Y(m^*) \cap \Q(m^*) \cap \D(S,\delta)^c \Big) \, \le \, e^{-s n^\delta},$$
by Lemma~\ref{maxdeg:c}, we have
$$\C'(S,\delta) \cap \W(S,v) \subseteq \big( \E(m^*) \cap \Z(m^*) \big)^c,$$
and by Lemma~\ref{indep:d}, we have
$$\Pr\bigg( \bigcup_{S \,:\, |S| = s} \D(S,\delta) \cap \E(m^*) \cap \Z(m^*) \cap \Q(m^*) \bigg) \, \le \, n^{-\sqrt{n}}.$$
This proves~\eqref{maxdeg:finalbound}, and thus completes the proof of Proposition~\ref{prop:maxdeg}, and hence of Theorem~\ref{triangle}.
\end{proof}

The deduction of Proposition~\ref{prop:indep} is similar. Recall from Definition~\ref{def:ABCD} the definitions of the events $\I(S,m)$, $\A(S,\delta)$, $\B(S,\delta)$, $\C(S,\delta)$ and $\D(S,\delta)$.

\begin{proof}[Proof of Proposition~\ref{prop:indep}]
Set $s =  \big( \sqrt{2} + \gamma \big) \sqrt{ n \log n }$, and recall that, by~\eqref{IinABCD}, 
$$\bigcup_{S \subseteq V(G_{n,\triangle}) \,:\, |S| = s} \A(S,\delta) \cup \B(S,\delta) \cup \C(S,\delta) \cup \D(S,\delta)  \, \supseteq \, \bigcup_{S \subseteq V(G_{n,\triangle}) \,:\, |S| = s} \I(S,m^*).$$ 
Thus, noting that $\alpha\big( G_{m^*} \big) \ge \alpha\big( G_{n,\triangle} \big)$, and writing $\F(S,\delta) = \A(S,\delta) \cup \B(S,\delta) \cup \C(S,\delta) \cup \D(S,\delta)$, it will suffice to prove that
$$\Pr\bigg( \bigcup_{S \subseteq V(G_{n,\triangle}) \,:\, |S| = s} \F(S,\delta) \cap \E(m^*)  \cap \Y(m^*) \cap \Z(m^*) \cap \Q(m^*) \bigg) \, \le \, e^{-\sqrt{n}}.$$
Now, by Lemma~\ref{indep:a}, we have
$$\sum_{S \,:\, |S| = s} \Pr\Big( \A(S,\delta) \cap \Q(m^*) \Big) \, \le \, n^{ - \delta s}.$$
by Lemma~\ref{indep:b}, we have
$$\sum_{S \,:\, |S| = s} \Pr\Big( \B(S,\delta) \cap \Y(m^*) \cap \Q(m^*)  \cap \D(S,\delta)^c \Big) \, \le \, e^{-s n^\delta}.$$
by Lemma~\ref{indep:c}, we have
$$\sum_{S \,:\, |S| = s} \Pr\Big( \C(S,\delta) \cap \E(m^*) \cap \Y(m^*) \cap \Z(m^*) \cap \Q(m^*) \Big) \, \le \, n^{-\delta s }.$$
and by Lemma~\ref{indep:d}, we have
$$\Pr\bigg( \bigcup_{S \,:\, |S| = s} \D(S,\delta) \cap \E(m^*) \cap \Z(m^*) \cap \Q(m^*) \bigg) \, \le \, n^{-\sqrt{n}}.$$
It follows that
$$\Pr\Big( \Big( \alpha\big( G_{n,\triangle} \big) \ge s \Big) \cap \E(m^*) \cap \Y(m^*) \cap \Z(m^*) \cap \Q(m^*) \Big) \, \le \, e^{-\sqrt{n}},$$
which completes the proof of Proposition~\ref{prop:indep}, and hence of Theorem~\ref{R3k}.
\end{proof}

\section*{Acknowledgements}

The authors would like to express their sincere gratitude to the referee for an extremely careful and thorough reading of the paper, and for numerous insightful and constructive comments on the presentation. They would also like to thank Roberto Imbuzeiro Oliveira for suggesting the problem to them, and for many interesting conversations.

\end{document}


\begin{abstract}
This file contains some of the straightforward but technical calculations from the paper ``The triangle-free process and the Ramsey number $R(3,k)$" which were omitted from the proof in order not to distract from the main argument. We have gathered these calculations here in order to save the interested reader the trouble of reproving them herself. 
\end{abstract}

\maketitle

\pagestyle{myheadings}
\markboth{}{}

\section{Introduction}

This short file is an Appendix to the paper~\cite{FGMO}. In that paper we follow the triangle-free process to its asymptotic end, and show that it is an excellent `Ramsey graph', in the sense that it gives very strong lower bounds on the Ramsey numbers $R(3,k)$. 

Recall from~\cite{FGMO} that we denote by $G_{n,\triangle}$ the (random) maximal triangle-free graph on $\{1,\ldots,n\}$ obtained via the triangle-free process. The main results of~\cite{FGMO} were as follows:\footnote{We remark that when we restate results from~\cite{FGMO} we shall use the numbering of that paper, whereas new statements will be given the prefix `A'.}

\setcounter{subsubsection}{1}
\setcounter{subsection}{1}
\setcounter{thm}{0}
\setcounter{thma}{0}

\begin{thm}\label{triangle}
$$e\big( G_{n,\triangle} \big) \,=\, \left( \frac{1}{2\sqrt{2}} + o(1) \right) n^{3/2} \sqrt{\log n },$$
with high probability as $n \to \infty$.
\end{thm}

The Ramsey number $R(3,k)$ is the smallest integer $n$ such that every red-blue colouring of the edges of the complete graph $K_n$ contains either a red $K_k$ or a blue triangle.

\begin{thm}\label{R3k}
$$R(3,k) \, \ge \, \left( \frac{1}{4} - o(1) \right) \frac{k^2}{\log k}$$
as $k \to \infty$.
\end{thm}

The basic heuristic behind Theorems~\ref{triangle} and~\ref{R3k} is that, with high probability, the graph $G_m$ obtained after $m$ steps of the triangle-free process approximates (in a certain sense) the Erd\H{o}s-R\'enyi random graph $G_{n,m}$, except in the fact that it contains no triangles. More precisely, there exists a (large) collection $\S$ of variables all of which take (approximately) the values one would expect in $G_{n,m}$, and all of whose derivatives at time $t = m \cdot n^{-3/2}$ may be bounded by functions which depend only on the values of variables in~$\S$ at time~$t$. We showed that moreover almost all of these variables exhibit a certain `self-correction', and were thus able to control their evolution with a fairly high degree of precision. 
 

In this file we shall give the details of various straightforward but technical calculations which were omitted from~\cite{FGMO} due to considerations of space and aesthetics. More precisely:
\begin{itemize}
\item In Section~\ref{Xsec} we shall give an extended version of~\cite[Section~3.4]{FGMO}, prove a generalized version of Lemma~\ref{selfX}, and prove Lemma~\ref{lem:chainstar} and Proposition~\ref{Uprop}. \smallskip
\item In Sections~\ref{PrelimSec} and~\ref{LBTsec} we shall do the same for~\cite[Section~4]{FGMO}; in particular, we shall prove Propositions~\ref{NFa} and~\ref{lem:landbeforetime} in Section~\ref{LBTsec}. \smallskip
\item In Section~\ref{XYQsec} we shall derive the equations which govern the `whirlpool' of~\cite[Section~6]{FGMO}, and prove Lemma~\ref{XYQalpha}. \smallskip
\item Finally, in Section~\ref{AppMartSec}, we shall (for completeness) adapt the proof (from~\cite{Colin}) of Lemma~\ref{mart} to our setting. 
\end{itemize}



\section{Section~3.2: Tracking the variables $X_e$}\label{Xsec}

\setcounter{subsubsection}{3}
\setcounter{subsection}{1}
\setcounter{thm}{7}
\setcounter{thma}{0}

Recall that $C = C(\eps) > 0$ is chosen sufficiently large, and that 
$$g_y(t) = e^{2t^2} n^{-1/4} (\log n)^4 \qquad \text{and} \qquad g_x(t) = C g_y(t).$$
Set $a = \omega \cdot n^{3/2}$ and define, for each $m \in [m^*]$,  
$$\K^\X(m) = \E(m) \cap \X(a) \cap \Y(m) \cap \Q(m).$$
In this section we shall prove Lemmas~\ref{selfX} and~\ref{Xe_gamma} of~\cite{FGMO}, which were used in~\cite[Section~3]{FGMO} to prove the following proposition. 

\begin{prop}\label{Xprop}
Let $\omega \cdot n^{3/2} < m \le m^*$. With probability at least $1 - n^{-C\log n}$ either $\K^\X(m-1)^c$ holds, or
\begin{equation}\label{eq:Xprop}
X_e(m) \, \in \, \Big( 1 \pm C e^{2t^2} n^{-1/4} (\log n)^4 \Big) \cdot 2 e^{-8t^2} n \, = \, \big( 1 \pm g_x(t) \big) \Xt(m)
\end{equation}
for every open edge $e \in O(G_m)$. 
\end{prop}

Let $(W,A)$ denote the graph structure pair with $v(W) = 4$, $v_A(W) = 1$, $e(W) = 1$ and $o(W) = 2$. We shall use the following two immediate consequences of the event $\E(m)$: that 
$$X_e(m) \le 2 \cdot \Xt(m)$$
for every open edge $e \in O(G_m)$, and that
$$N_\phi(W)(m) \, \le \, \max\big\{ 4t e^{-8t^2} \sqrt{n}, (\log n)^\omega \big\}$$
for every $\omega \cdot n^{3/2} < m \le m^*$ and every faithful map $\phi$.

The first step is to show that the variables $X_e$ are self-correcting as long as the event $\K(m) = \K^X(m) \cap \X(m)$ holds. Define
$$X_e^*(m) \, = \, \frac{X_e(m) - \Xt(m)}{g_x(t) \Xt(m)},$$
the normalized error. We shall prove the following lemma.

\begin{lemma}\label{selfX}
Let $\omega \cdot n^{3/2} \le m \le m^*$. If $\E(m) \cap \X(m) \cap \Y(m) \cap \Q(m)$ holds, then 
$$\Ex \big[ \Delta X^*_e(m) \big] \, \in \, \frac{4t}{n^{3/2}} \Big( - X_e^*(m) \pm \eps \Big)$$
for every $e \in O(G_m)$.
\end{lemma}

It is easy to see that 
\begin{equation}\label{eq:Xeq}
\Ex\big[ \Delta X_e(m) \big] \, = \, - \frac{2}{Q(m)} \sum_{f \in X_e(m)} \Big( Y_f(m) + 1 \Big), 
\end{equation}
for every $e \in O(G_m)$. Indeed, for each open triangle $T$ in $G_m$ containing $e$, the probability that one of the open edges ($f$ and $h$, say) of $T$ other than $e$ is closed (or chosen) in step $m+1$ is equal to 
$$\frac{\big| Y_f(m) \cup Y_h(m) \big| + 2}{Q(m)} \, = \, \frac{Y_f(m) + Y_h(m) + 2}{Q(m)}.$$
To see this, simply note that if $Y_f(m) \cap Y_h(m) \neq \emptyset$, then the endpoints of $e$ have a common neighbour in $G_m$, which means that $e \not\in O(G_m)$, a contradiction. Since $X_e$ decreases by two for each open triangle which is destroyed,~\eqref{eq:Xeq} follows. 

In order to deduce Lemma~\ref{selfX} from~\eqref{eq:Xeq}, we shall prove the following, more general statement. We also use this more general version in~\cite[Section~7.4]{FGMO}. 

\begin{lemmaa}\label{selfA}
Let $A(m)$ be a random variable which denotes both a collection of open edges of $G_m$, and the size of that collection. Suppose that 
\begin{equation}\label{eq:Aeq}
\Ex\big[ \Delta A(m) \big] \, \in \, - \frac{\ell \pm o(1)}{Q(m)} \sum_{f \in A(m)} Y_f(m)
\end{equation}
for some $\ell  \in \N$ and every $\omega \cdot n^{3/2} < m \le m^*$, and set $\At(m) = e^{-4\ell  t^2} A(0)$ and
$$A^*(m) \, = \, \frac{A(m) - \At(m)}{g(t) \At(m)},$$
for some $g \colon (0,t^*) \to \RR^+$ 
which satisfies $g(t) \ge g_x(t)$ and $g \sim g_x$.\footnote{We write $g \sim g_x$ to indicate that $g(t) = \lambda(n) \cdot g_x(t)$ for some function $\lambda(n)$. Since $g(t) \ge g_x(t)$, we have $\lambda(n) \ge 1$ for every $n \in \N$.} 
Then 
$$\Ex \big[ \Delta A^*(m) \big] \, \in \, \frac{4t}{n^{3/2}} \Big( - A^*(m) \pm \eps \Big).$$
for every $\omega \cdot n^{3/2} \le m \le m^*$ such that $\Y(m) \cap \Q(m)$ holds and $|A^*(m)| \le 1$.
\end{lemmaa}

We emphasize that $\ell$ is an absolute fixed constant; in fact, in our applications we shall need to consider only the cases $\ell  = 1$ and $\ell  = 2$. In the proof of Lemma~\ref{selfA}, we shall use the product rule, which was stated in~\cite[Section~4]{FGMO}

\begin{chain}
 For any random variables $a(m)$ and $b(m)$,  
$$\Ex \big[ \Delta \big( a(m) b(m) \big) \big] \, = \, a(m) \Ex \big[ \Delta b(m) \big] + b(m) \Ex \big[ \Delta a(m) \big] + \Ex\big[ \big( \Delta a(m) \big) \big( \Delta b(m) \big) \big].$$
In particular, if $a(m)$ is deterministic, then
$$\Ex \big[ \Delta \big( a(m) b(m) \big) \big] \, = \, a(m) \Ex \big[ \Delta b(m) \big] + \Delta a(m) \Big( b(m) + \Ex\big[ \Delta b(m) \big] \Big).$$
\end{chain}


To simplify the calculations below, we shall write $a \approx b$ to denote that the inequalities $a/b \in 1 \pm O\big( 1/n \big)$ hold.

\begin{proof}[Proof of Lemma~\ref{selfA}]
Observe first that, differentiating with respect to $t$, we have
\begin{equation}\label{eq:deltaAt}
\Delta \At(m) \approx -\frac{8\ell t}{n^{3/2}} \cdot \At(m) \quad \text{and} \quad \Delta \big( g(t) \At(m) \big) \approx - \frac{(8\ell -4)t}{n^{3/2}} \cdot g(t) \At(m).
\end{equation}
We claim that if $\Y(m) \cap \Q(m)$ holds and $|A^*(m)| \le 1$, then
\begin{align}
\Ex\big[ \Delta A(m) \big] - \Delta \At(m) & \, \in \, - \frac{\ell \pm o(1)}{Q(m)} \sum_{f \in A(m)} Y_f(m) \,+\, \frac{8\ell t}{n^{3/2}} \cdot \At(m) \nonumber \\
& \, \in \, \bigg( \At(m) \,-\, \frac{1 \pm g_y(t)}{1 \pm g_q(t)} \cdot A(m) \bigg) \cdot \big(\ell \pm o(1) \big) \cdot \frac{\Yt(m)}{\Qt(m)} \nonumber \\
&  \, \subseteq \, \Big( 1 \,-\, \big( 1 \pm 2g_y(t) \big) \big( 1 + g(t) A^*(m) \big) \Big) \cdot \big(\ell \pm o(1) \big) \cdot \frac{\Yt(m) \cdot \At(m)}{\Qt(m)} \nonumber \\
& \, \subseteq \, \Big( - \ell  \cdot A^*(m) \pm \eps^2 \Big) \cdot \frac{g(t) \cdot \Yt(m) \cdot \At(m)}{\Qt(m)}. \nonumber
\end{align}
Indeed, this follows since $g_q(t) \ll g_y(t) \le \eps^3 g_x(t) \le \eps^3 g(t)$ and $\omega < t \le t^*$. Thus
\begin{equation}\label{eq:ExdeltaAAt}
\frac{\Ex\big[ \Delta A(m) \big] - \Delta \At(m)}{g(t) \At(m)} \, \in \, - \, \frac{8\ell t}{n^{3/2}} \cdot A^*(m) \,\pm\, \frac{\eps t}{n^{3/2}}.
\end{equation}
Now, since $A(m) - \At(m) = g(t)  \At(m) \cdot A^*(m)$, by the product rule we have
\begin{equation}\label{eq:Achain}
\frac{\Ex\big[ \Delta A(m) \big] - \Delta \At(m)}{g(t) \At(m)} \, = \, \Ex\big[ \Delta A^*(m) \big] \,+\, \frac{\Delta \big( g(t) \At(m) \big)}{g(t) \At(m)} \Big( A^*(m) \,+\,  \Ex\big[ \Delta A^*(m) \big] \Big).
\end{equation}
Combining~\eqref{eq:deltaAt},~\eqref{eq:ExdeltaAAt} and~\eqref{eq:Achain}, 
we obtain
\begin{align*}
\Ex\big[ \Delta A^*(m) \big] & \, \in \, - \, \frac{8\ell t}{n^{3/2}} \cdot A^*(m) + \frac{(8\ell -4)t}{n^{3/2}}  \cdot A^*(m) \,\pm\, \frac{2\eps t}{n^{3/2}}\\
& \, \subseteq \, \frac{4t}{n^{3/2}} \cdot \Big( - A^*(m) \pm \eps \Big),
\end{align*}
as required.
\end{proof}


Recall next the following bound on $|\Delta X_e(m)|$, which was proved in~\cite{FGMO}.

\setcounter{subsubsection}{3}
\setcounter{thm}{9}

\begin{lemma}\label{Xe_alpha}
Let $\omega \cdot n^{3/2} < m \le m^*$. If $\E(m)$ holds, then
$$|\Delta X_e(m)| \, \le \, 2 \cdot \max\big\{ 4t e^{-8t^2} \sqrt{n}, (\log n)^\omega \big\}$$
for every $e \in O(G_m)$.
\end{lemma}

Using Lemmas~\ref{selfX} and~\ref{Xe_alpha}, we can bound $|\Delta X^*_e(m)|$ and $\Ex\big[ |\Delta X_e^*(m)| \big]$.

\begin{lemma}\label{Xe_gamma}
Let $\omega \cdot n^{3/2} < m \le m^*$. If $\E(m) \cap \X(m) \cap \Y(m) \cap \Q(m)$ holds, then 
$$|\Delta X^*_e(m)| \, \le \, \frac{C}{g_x(t)} \cdot \frac{e^{8t^2}}{n} \max\big\{ t e^{-8t^2} \sqrt{n}, (\log n)^\omega \big\} \quad \text{and} \quad \Ex \big[ |\Delta X^*_e(m)| \big] \, \le \, \frac{C}{g_x(t)} \cdot \frac{\log n}{n^{3/2}}$$
for every $e \in O(G_m)$.
\end{lemma}

In order to deduce Lemma~\ref{Xe_gamma} from Lemma~\ref{Xe_alpha}, and also in Section~\ref{AwhirlSec}, below, we shall use the following result, which was stated (but not proved) in~\cite{FGMO}. 

\setcounter{subsubsection}{4}
\setcounter{thm}{22}

\begin{lemma}\label{lem:chainstar}
Let $A(m)$ be a random variable, let $\At(m)$ and $g(t)$ be functions, and set  
$$A^*(m) \, = \, \frac{A(m) - \At(m)}{g(t) \At(m)}.$$
If $|A(m)| \le \big( 1 + g(t) \big) \At(m)$,
\begin{equation}\label{eq:deltaAgA}
|\Delta \At(m) | \, \ll \, \frac{\log n}{n^{3/2}} \cdot \At(m)  \quad \text{and} \quad |\Delta \big( g(t) \At(m) \big) | \ll \frac{\log n}{n^{3/2}} \cdot g(t) \At(m),
\end{equation}
then
$$| \Delta A^*(m) | \, \le \, 2 \cdot \left( \frac{| \Delta A(m) |}{g(t) \At(m)} \,+\, \frac{1 + g(t)}{g(t)} \cdot \frac{\log n}{n^{3/2}} \right).$$
\end{lemma}

\begin{proof}
Observe first that, for arbitrary functions $a,b,c \colon \N \to \RR^+$, if $a^*(m) = \frac{a(m) - b(m)}{c(m)}$ then
$$\Delta a(m) - \Delta b(m) \, = \, \Delta\big( a^*(m) c(m) \big) \, = \, c(m+1) \Delta a^*(m) + a^*(m) \Delta c(m),$$
and hence
\begin{equation}\label{eq:chainstar}
\Delta a^*(m) \, \in \, \frac{\Delta a(m)}{c(m+1)} \,\pm\, \frac{|\Delta b(m)| \cdot c(m) + \big( a(m) + b(m) \big) |\Delta c(m)|}{c(m) \cdot c(m+1)}.
\end{equation}
Applying~\eqref{eq:chainstar} to the functions $A(m)$, $\At(m)$ and $g(t) \At(m)$, and using the assumptions that $A(m) \le \big( 1 + g(t) \big) \At(m)$ and
$$|\Delta \At(m) | \, \ll \, \frac{\log n}{n^{3/2}} \cdot \At(m)  \quad \text{and} \quad |\Delta \big( g(t) \At(m) \big) | \ll \frac{\log n}{n^{3/2}} \cdot g(t) \At(m),$$ 
we obtain
$$| \Delta A^*(m) | \, \le \, 2 \cdot \left( \frac{| \Delta A(m) |}{g(t) \At(m)} \,+\, \frac{1 + g(t)}{g(t)} \cdot \frac{\log n}{n^{3/2}} \right),$$
as claimed.
\end{proof}

Specializing to the case $X_e$, we obtain the following easy corollary.

\setcounter{subsubsection}{4}
\setcounter{thm}{21}

\begin{lemmaa}\label{lem:chainstarX}
For every $\omega \cdot n^{3/2} < m \le m^*$, if $\X(m)$ holds, then
$$| \Delta X_e^*(m) | \, \le \, \frac{3}{g_x(t)} \cdot \left( \frac{| \Delta X_e(m) |}{\Xt(m)} \,+\, \frac{\log n}{n^{3/2}} \right)$$
for every $e \in O(G_m)$.
\end{lemmaa}

\begin{proof}
We apply Lemma~\ref{lem:chainstar} with $A(m) = X_e(m)$ and $g(t) = g_x(t)$. Since $\Xt(m)$ is equal to $e^{-8t^2}$ times a function of $n$, and $g_x(t) \Xt(m)$ is equal to $e^{-6t^2}$ times a function of $n$, we have
$$\Delta \Xt(m) \in \frac{- 16t \pm o(1)}{n^{3/2}} \cdot  \Xt(m) \quad \text{and} \quad \Delta \big( g_x(t) \Xt(m) \big) \in  \frac{- 12t \pm o(1)}{n^{3/2}} \cdot g_x(t) \Xt(m),$$
and hence
$$|\Delta \Xt(m) | \, \ll \, \frac{\log n}{n^{3/2}} \cdot \Xt(m)  \quad \text{and} \quad |\Delta \big( g_x(t) \Xt(m) \big) | \ll \frac{\log n}{n^{3/2}} \cdot g_x(t) \Xt(m).$$ 
Moreover, the event $\X(m)$ implies that $X_e(m) \le \big( 1 + g_x(t) \big) \Xt(m)$ for every $e \in O(G_m)$, and so
$$| \Delta X_e^*(m) | \, \le \, \frac{3}{g_x(t)} \cdot \left( \frac{| \Delta X_e(m) |}{\Xt(m)} \,+\, \frac{\log n}{n^{3/2}} \right),$$
as claimed.
\end{proof}

We can now easily deduce Lemma~\ref{Xe_gamma}.

\begin{proof}[Proof of Lemma~\ref{Xe_gamma}]
By Lemmas~\ref{Xe_alpha} and~\ref{lem:chainstarX}, we have 
\begin{align*}
| \Delta X_e^*(m) | & \, \le \,  \frac{3}{g_x(t)} \cdot \left( \frac{| \Delta X_e(m) |}{\Xt(m)} \,+\, \frac{\log n}{n^{3/2}} \right)\\
& \, \le \, \frac{3}{g_x(t) \Xt(m)} \cdot \left( 2 \cdot \max\big\{ 4t e^{-8t^2} \sqrt{n}, (\log n)^\omega \big\} \,+\,  \frac{e^{-8t^2} \log n}{\sqrt{n}} \right)\\
& \, \le \, \frac{C}{g_x(t)} \cdot \frac{e^{8t^2}}{n} \max\big\{ t e^{-8t^2} \sqrt{n}, (\log n)^\omega \big\},
\end{align*}
as claimed. Moreover, observe that, by~\eqref{eq:Xeq}, and using the event $\X(m) \cap \Y(m) \cap \Q(m)$ and the fact that $X_e$ is decreasing, we have
$$\Ex\big[ | \Delta X_e(m) | \big] \, = \, \frac{2}{Q(m)} \sum_{f \in X_e(m)} \Big( Y_f(m) + 1 \Big) \, \le \, \frac{3 \cdot \Xt(m) \cdot \Yt(m)}{\Qt(m)} \, = \, \frac{12t}{n^{3/2}} \cdot \Xt(m).$$
Thus, by Lemma~\ref{lem:chainstarX}, we have 
\begin{align*}
\Ex\big[ | \Delta X_e^*(m) | \big] & \, \le \,  \frac{3}{g_x(t)} \cdot \left( \frac{\Ex\big[ | \Delta X_e(m) | \big]}{\Xt(m)} \,+\, \frac{\log n}{n^{3/2}} \right)\\
& \, \le \, \frac{3}{g_x(t)} \cdot \left( \frac{12t}{n^{3/2}} \,+\, \frac{\log n}{n^{3/2}} \right) \, \le \, \frac{C}{g_x(t)} \cdot \frac{\log n}{n^{3/2}},
\end{align*}
as required.
\end{proof}

\subsection{The proof of Proposition~\ref{Uprop}}

We take this opportunity to prove a similar proposition from~\cite[Section~5]{FGMO}. 

\setcounter{subsection}{1}
\setcounter{subsubsection}{5}
\setcounter{thm}{4}
\setcounter{thma}{2}

\begin{prop}\label{Uprop}
Let $\omega \cdot n^{3/2} < m \le m^*$. Then, with probability at least $1 - n^{-C\log n}$, either $\K^\Y(m-1)^c$ holds, or 
\begin{equation}\label{eq:YL}
Y^L_e(m) \, \in \, \big( 1 \pm g_x(t) \big) \cdot \big( 2t e^{-4t^2} \sqrt{n}  \big)
\end{equation}
for every $e \in O(G_m)$.
\end{prop}

The proof is almost identical to that of Proposition~\ref{Xprop}, above, so we shall skip some of the details. Recall first the following special case of~\cite[Lemma~\ref{Ueq}]{FGMO}, which is obtained from the version stated there by setting $\sigma = L$ and noting that $U_e^L(m) V_e^L(m) = \sum_{f \in Y_e^L(m)} Y_f(m)$. 

\setcounter{subsubsection}{5}
\setcounter{thm}{19}

\begin{lemma}\label{Ueq}
Let $\omega \cdot n^{3/2} < m \le m^*$. If $\E(m) \cap \U(m) \cap \X(m) \cap \Z(m)$ holds, then
$$\Ex\big[ \Delta Y^L_e(m) \big] \, \in \, - \frac{1}{Q(m)} \sum_{f \in Y_e^L(m)} Y_f(m) \,+\, \left( \frac{1}{2} \pm g_x(t) \right) \frac{\Xt(m)}{Q(m)},$$
for every $e \in O(G_m)$.
\end{lemma}

Define
$$(Y^L_e)^*(m) \, = \, \frac{2 \cdot Y^L_e(m) - \Yt(m)}{g_x(t) \Yt(m)},$$
the normalized error. Since $\Xt(m) \ll \Yt(m)^2$ for $t > \omega$, it follows from Lemmas~\ref{Ueq} and~\ref{selfA} that the variables $Y^L_e$ are self-correcting. 


\begin{lemmaa}\label{selfYL}
Let $\omega \cdot n^{3/2} \le m \le m^*$. If $\E(m) \cap \U(m) \cap \X(m) \cap \Z(m) \cap \Q(m)$ holds, then 
$$\Ex \big[ \Delta (Y^L_e)^*(m) \big] \, \in \, \frac{4t}{n^{3/2}} \Big( - (Y^L_e)^*(m) \pm \eps \Big)$$
for every $e \in O(G_m)$.
\end{lemmaa}

Next, recall the following lemma from~\cite[Section~5.3]{FGMO}.

\setcounter{thm}{16}

\begin{lemma}\label{deltaU}
Let $\omega \cdot n^{3/2} < m \le m^*$. If $\E(m) \cap \Z(m)$ holds, then 
$$\big| \Delta Y_e^L(m) \big| \, \le \, (\log n)^3$$
for every $e \in O(G_m)$.
\end{lemma}

Using Lemmas~\ref{deltaU} and~\ref{selfYL}, we can bound $|\Delta (Y^L_e)^*(m)|$ and $\Ex\big[ |\Delta (Y^L_e)^*(m)| \big]$.

\begin{lemmaa}\label{YeL_gamma}
Let $\omega \cdot n^{3/2} < m \le m^*$. If $\E(m) \cap \U(m) \cap \X(m) \cap \Z(m) \cap \Q(m)$ holds, then 
$$|\Delta (Y^L_e)^*(m)| \, \le \, \frac{C}{g_x(t)} \cdot \frac{\log n}{\sqrt{n}} \qquad \text{and} \qquad \Ex \big[ |\Delta (Y^L_e)^*(m)| \big] \, \le \, \frac{C}{g_x(t)} \cdot \frac{\log n}{n^{3/2}}$$
for every $e \in O(G_m)$.
\end{lemmaa}

In order to deduce Lemma~\ref{YeL_gamma} from Lemma~\ref{deltaU}, we shall use the following easy consequence of Lemma~\ref{lem:chainstar}.

\setcounter{subsubsection}{4}
\setcounter{thm}{21}

\begin{lemmaa}\label{lem:chainstarYL}
For every $\omega \cdot n^{3/2} < m \le m^*$, if $\U(m)$ holds, then
$$| \Delta (Y^L_e)^*(m) | \, \le \, \frac{4}{g_x(t)} \cdot \left( \frac{| \Delta Y^L_e(m) |}{\Yt(m)} \,+\, \frac{\log n}{n^{3/2}} \right)$$
for every $e \in O(G_m)$.
\end{lemmaa}

\begin{proof}
We apply Lemma~\ref{lem:chainstar} with $A(m) = Y^L_e(m)$ and $g(t) = g_x(t)$. Since $\Yt(m)$ is equal to $t \cdot e^{-4t^2}$ times a function of $n$, and $g_x(t) \Yt(m)$ is equal to $t \cdot e^{-2t^2}$ times a function of $n$, we have
$$\Delta \Yt(m) \in \frac{- 8t \pm o(1)}{n^{3/2}} \cdot \Yt(m) \quad \text{and} \quad \Delta \big( g_x(t) \Yt(m) \big) \in  \frac{- 4t \pm o(1)}{n^{3/2}} \cdot g_x(t) \Yt(m),$$
and hence
$$|\Delta \Yt(m) | \, \ll \, \frac{\log n}{n^{3/2}} \cdot \Xt(m)  \quad \text{and} \quad |\Delta \big( g_x(t) \Yt(m) \big) | \ll \frac{\log n}{n^{3/2}} \cdot g_x(t) \Xt(m).$$ 
Moreover, the event $\U(m)$ implies that $2 \cdot Y^L_e(m) \le \big( 1 + g_x(t) \big) \Yt(m)$ for every $e \in O(G_m)$, and so
$$| \Delta (Y^L_e)^*(m) | \, \le \, \frac{4}{g_x(t)} \cdot \left( \frac{| \Delta Y^L_e(m) |}{\Yt(m)} \,+\, \frac{\log n}{n^{3/2}} \right),$$
as claimed.
\end{proof}

We can now easily deduce Lemma~\ref{YeL_gamma}.

\begin{proof}[Proof of Lemma~\ref{YeL_gamma}]
By Lemmas~\ref{deltaU} and~\ref{lem:chainstarYL}, we have 
$$| \Delta (Y^L_e)^*(m) | \, \le \,  \frac{ 4 }{ g_x(t) } \cdot \left( \frac{ |\Delta Y^L_e(m)| }{ \Yt(m) } \,+\, \frac{ \log n }{ n^{3/2} } \right) \, \le \, \frac{ C }{ g_x(t) } \cdot \frac{ (\log n)^3 }{ \sqrt{n} },$$
as claimed. Moreover, since
$$\frac{\Ex\big[ | \Delta Y^L_e(m) | \big]}{\Yt(m)} \, \le \, \frac{\Yt(m)}{\Qt(m)} \, \le \, \frac{\log n}{n^{3/2}},$$
it follows from Lemma~\ref{lem:chainstarYL} that
$$\Ex\big[ | \Delta (Y^L_e)^*(m) | \big] \, \le \,  \frac{4}{g_x(t)} \cdot \left( \frac{\Ex\big[ | \Delta Y^L_e(m) | \big]}{\Yt(m)} \,+\, \frac{\log n}{n^{3/2}} \right) \, \le \, \frac{C}{g_x(t)} \cdot \frac{\log n}{n^{3/2}},$$
as required.
\end{proof}

Finally, let's deduce Proposition~\ref{Uprop}.

\begin{proof}[Proof of Proposition~\ref{Uprop}]
We begin by choosing a family of parameters as in~\cite[Definition~3.4]{FGMO}. Set $\K(m) = \K^\Y(m) \cap \U(m) \cap \big\{ |(Y^L_e)^*(a)| < 1/2 \big\}$ and $I = [a,b] = [\omega \cdot n^{3/2},m^*]$, and let
$$\alpha(t) \, = \, \frac{C}{g_x(t)} \cdot \frac{ (\log n)^3 }{ \sqrt{n} } \qquad \text{and} \qquad \beta(t) \, = \, \frac{C}{g_x(t)} \cdot \frac{\log n}{n^{3/2}}.$$
Moreover, set $\lambda = C$, $\delta = \eps$ and $h(t) = t \cdot n^{-3/2}$. 
We claim that $(\lambda,\delta;g_x,h;\alpha,\beta;\K)$ is a reasonable collection, and that $Y^L_e$ satisfies the conditions of~\cite[Lemma~3.2]{FGMO} if $e \in O(G_m)$.

To prove the first statement, we need to show that $\alpha$ and $\beta$ are $\lambda$-slow, and that 
$$\min\big\{ \alpha(t), \, \beta(t), \, h(t) \big\} \, \ge \, \ds\frac{\eps t}{n^{3/2}}$$ 
and $\alpha(t) \le \eps^2$ for every $\omega < t \le t^*$, each of which is obvious, since $g_x(t) \le 1$ for all $t \le t^*$. To prove the second, we need to show that $Y^L_e$ is $(g_x,h;\K)$-self-correcting, which follows from Lemma~\ref{selfYL}, that, for every $\omega \cdot n^{3/2} < m \le m^*$, if $\K(m)$ holds then 
$$|\Delta (Y^L_e)^*(m)| \le \alpha(t) \qquad  \text{and} \qquad \Ex\big[ |\Delta (Y^L_e)^*(m)| \big] \le \beta(t),$$
which follows from Lemma~\ref{YeL_gamma}, and that $|(Y^L_e)^*(a)| < 1/2$, which follows from $\K(m)$.


Observe that
$$\alpha(t) \beta(t) n^{3/2} \, \le \, \frac{C^2}{g_x(t)^2} \cdot \frac{ (\log n)^4 }{ \sqrt{n} } \, \le \, \frac{1}{(\log n)^3}$$
for every $\omega < t \le t^*$. By~\cite[Lemma~3.1]{FGMO}, and summing over edges $e \in E(K_n)$ the probability that $e \in O(G_m)$ and $(Y^L_e)^*(m) > 1$, it follows that
$$\Pr\Big( \U(m)^c \cap \K(m-1) \text{ for some $m \in [a,b]$} \Big) \, \le \, n^6 \exp\Big( - \delta' (\log n)^3 \Big) \, \le \, n^{-C \log n}.$$
Finally, we remark that 
$$\Pr\Big( \K^\Y(a) \cap \big\{ |(Y^L_e)^*(a)| \ge 1/2 \big\} \Big) \, \le \, n^{-C \log n},$$
by Proposition~\ref{lem:landbeforetime} (see Section~\ref{LBTsec}, below), and so the proposition follows.
\end{proof}

\setcounter{subsubsection}{2}

\section{Section~4: Everything Else}\label{PrelimSec}

In this section we shall give the details omitted from~\cite[Section~4]{FGMO}, recall from some of the lemmas from that section which we shall need below, and prove some simple variants. Recall the following important definitions from~\cite{FGMO}.

\setcounter{subsubsection}{2}
\setcounter{subsection}{1}
\setcounter{thm}{9}
\setcounter{thma}{0}

\begin{defn}\label{def:tAF}
Define
\begin{equation}\label{def:t*}
t_A^*(F) \, = \, \inf\Big\{ t > 0 \,:\, \Nt_A(F)(m) \le (2t)^{e(F)} \Big\} \, \in \, [0,\infty]
\end{equation}
and
$$t_A(F) \, = \, \min\Big\{ \min \big\{ t_A^*(H) : A \subsetneq H \subseteq F \big\}, \, t^* \Big\}.$$
We call $t_A(F)$ the \emph{tracking time} of the pair $(F,A)$. 
\end{defn} 

Moreover, for each graph structure pair $(F,A)$ with $t_A(F) > 0$, define 
\begin{equation}\label{def:cFA}
c \, = \, c(F,A) \, := \, \max\bigg\{ \max_{A \subsetneq H \subseteq F} \bigg\{ \frac{2o(H)}{2v_A(H) - e(H)} \bigg\}, \, 2 \bigg\},
\end{equation}
so in particular $e^{ct^2} = n^{1/4}$ when $t = t_A(F)$ if $t_A(F) < t^*$. 

We begin with the following crucial remark. 

\setcounter{subsubsection}{4}
\setcounter{thm}{1}

\begin{rmk}\label{Eremark}
Let $(F,A)$ be a graph structure pair with $v(F) = n^{o(1)}$, and let $m \in [m^*]$. If the event $\E(m)$ holds, then $N_\phi(F)(m)$ satisfies the conclusion (i.e., either $(a)$, $(b)$ or $(c)$) of~\cite[Theorem~4.1]{FGMO}, for every faithful injection $\phi \colon A \to V(G_m)$.
\end{rmk}

\begin{proof}[Proof of Remark~\ref{Eremark}]
Let us denote by~$(F',A',\phi')$ the graph structure obtained by removing the isolated vertices from $F$, so $A' = V(F') \cap A$ and $\phi' = \phi|_{A'}$. Note that $t_{A'}(F') = t_A(F)$, since $t^*_A(H' \cup A) = t^*_{A'}(H')$ for every $A' \subsetneq H' \subseteq F'$. Assume that the event $\E(m)$ holds and that~$\phi$ is faithful at time~$t$. We shall consider in turn the three cases corresponding to parts $(a)$, $(b)$ and $(c)$ of the theorem. 

Suppose first that $0 < t \le \omega < t_A(F)$, and note that 
$$N_{\phi'}(F')(m) \, \in \, \Nt_{A'}(F')(m) \pm f_{F',A'}(t) \Nt_{A'}(F')(n^{3/2}),$$ 
since the event $\E(m)$ holds, and $t_{A'}(F') = t_A(F)$. Now, each copy of $F'$  in $G_m$, rooted at $\phi'(A')$, extends to between $(n - v(F))^k$ and $n^k$ copies of $F$ rooted at $\phi(A)$. Set $k = v_A(F) - v_{A'}(F')$, and note that $\Nt_A(F) = n^k \cdot \Nt_{A'}(F')$ and that $\gamma(F,A) > \gamma(F',A')$ if $k \ne 0$. It follows that 
$$N_\phi(F)(m) \, \in \, \Nt_A(F)(m) \pm f_{F,A}(t) \Nt_A(F)(n^{3/2}),$$ 
as required.

The proof in case $(b)$ is almost identical, so we skip the details and proceed to case $(c)$. Here we need the observation that the minimal $A \subsetneq H \subseteq F$ such that $t < t_H(F)$ contains no isolated vertices. Indeed, this is obvious, since moving an isolated vertex from $H$ to $F$ can only increase $t_H(F)$, and $t > t_A(F)$. It follows that the minimal $A \subsetneq H \subseteq F$ such that $t < t_H(F)$, is equal to $A$ union the minimal $A' \subsetneq H' \subseteq F'$ such that $t < t_{H'}(F')$. Since $\Delta(F',H',A') < \Delta(F,H,A)$ unless $k = 0$, it follows that
$$N_\phi(F)(m) \, \le \, n^k \cdot ( \log n )^{\Delta(F',H',A')} \Nt_{H'}(F')(m^+)  \, \le \, ( \log n )^{\Delta(F,H,A)} \Nt_H(F)(m^+),$$
by the event $\E(m)$, as required.
\end{proof}

We remark that the same argument implies that the building sequences 
$$A \subseteq H_0 \subsetneq \dots \subsetneq H_\ell = F \qquad \textup{and} \qquad A' \subseteq H_0' \subsetneq \dots \subsetneq H_\ell' = F'$$ 
of $(F,A)$ and $(F',A')$ respectively are identical, in the sense that $H_j = H_j' \cup A$ for every $0 \le j \le \ell - 1$. That is, all of the isolated vertices of $F$ lie in $V(F) \setminus H_{\ell-1}$. 

We next recall three simple properties of the collection $\F^o_F$, see~\cite[Section~4.2]{FGMO}.

\setcounter{subsubsection}{4}
\setcounter{thm}{19}

\begin{obs}\label{obs:NtF-}
Let $(F,A)$ be a graph structure pair. Then $|\F^o_F| = e(F)$, and 
$$2te^{4t^2} \cdot \Nt_A(F^o)(m) \, = \, \sqrt{n} \cdot \Nt_A(F)(m)$$
for every $F^o \in \F_F^o$.
\end{obs}

\setcounter{subsubsection}{4}
\setcounter{thm}{23}

\begin{obs}\label{obs:F-}
If $(F,A)$ is a graph structure pair and $F^o \in \F_F^o$, then $t_A(F) \le t_A(F^o)$.
\end{obs}

\begin{proof}
If $ t_A(F^o) = t^*$ there is nothing to prove, so suppose that $A \subsetneq H^o \subseteq F^o$ satisfies $t^*_A(H^o) = t_A(F^o) < t^*$. Then either $H^o \subseteq F$, in which case $t_A(F) \le t^*_A(H^o) = t_A(F^o)$, as required, or $H^o \in \F_H^o$ for some $A \subsetneq H \subseteq F$. 

We claim that, in the latter case, we have $t_A(F) \le t^*_A(H) \le t^*_A(H^o) = t_A(F^o)$. Note that $e(H) = e(H^o) - 1$, 
and therefore, by Observation~\ref{obs:NtF-},
$$\Nt_A(H^o)(m) \, = \, \frac{\sqrt{n}}{2te^{4t^2}} \cdot \Nt_A(H)(m) \, \ge \, \frac{n^\eps}{2t} \cdot \Nt_A(H)(m) \, > \, (2t)^{e(H) - 1} \, = \, (2t)^{e(H^o)}$$
for every $t \le t^*_A(H) < t^*$. By~\eqref{def:t*}, it follows that $t^*_A(H) \le t^*_A(H^o)$, as required.
\end{proof}


\begin{obs}\label{obs:gfat}
Let $(F,A)$ be a graph structure pair with $v_A(F) \ge 1$, and let $F^o \in \F^o_F$. Then 
$$o(F) g_y(t) \ll g_{F,A}(t) \qquad \text{and} \qquad g_{F^o,A}(t) \ll g_{F,A}(t)$$ 
as $n \to \infty$.
\end{obs}

\begin{proof}
Observe first that $c(F,A) \ge c(F^o,A) \ge 2$ for every $F^o \in \F^o_F$, by Observation~\ref{obs:F-}. Indeed, $c = c(F,A)$ was chosen (see~\eqref{def:cFA}) so that $e^{ct^2} = n^{1/4}$ at $t = t_A(F)$, and $t_A(F) \le t_A(F^o)$. Now, noting that $v_A(F) = v_A(F^o)$, $e(F) = e(F^o) + 1$ and $o(F) = o(F^o) - 1$, we have 
$$\gamma(F,A) \, \ge \, \gamma(F^o,A) + \sqrt{\gamma(F,A)},$$ 
by the definition of $\gamma(F,A)$. 
\end{proof}

Recall next that
\begin{equation}\label{def:fFAt}
f_{F,A}(t) \, = \, e^{C(o(F)+1)(t^2 + 1)} n^{-1/4} (\log n)^{\Delta(F,A) - \sqrt{\Delta(F,A)}}
\end{equation}
for each graph structure pair $(F,A)$, and that
$$f_y(t) \, = \, e^{Ct^2} n^{-1/4} (\log n)^{5/2} \qquad \text{and} \qquad f_x(t) \, = \, e^{-4t^2} f_y(t).$$
The following variant of~\cite[Observation~4.25]{FGMO} follows easily from the definitions. 

\begin{obsa}\label{obs:gfat}
Let $(F,A)$ be a graph structure pair with $v_A(F) \ge 1$, and let $F^o \in \F^o_F$. Then 
$$(\log n)^{e(F) + o(F)} \big( f_y(t) + f_{F^o,A}(t) \big) \ll f_{F,A}(t)$$ 
as $n \to \infty$.
\end{obsa}

\begin{proof}
Since $e(F^o) = e(F) - 1$ and $o(F^o) = o(F) + 1$, we have 
$$\Delta(F,A) \, = \, \big( \delta(F^o,A) + 1 \big)^C \, \ge \, \Delta(F^o,A) + \sqrt{\Delta(F,A)}.$$
The claimed bound now follows immediately.
\end{proof}

Our next few results give useful properties of the collection of graph structures $\F_{F,A}^-$  which was defined in~\cite[Section~4.3]{FGMO}. The first four were also proved there.

\setcounter{subsubsection}{4}
\setcounter{thm}{31}

\begin{obs}\label{obs:F+inF-}
$\F_{F,A}^+ \subseteq \F_{F,A}^-$ for every graph structure pair $(F,A)$. 
\end{obs}

Note that, by this observation, the following results all also hold for each $(F',A') \in \F_{F,A}^+$.


\setcounter{subsubsection}{4}
\setcounter{thm}{33}

\begin{obs}\label{obs:NH'F'NHF}
Let $(F,A)$ be a graph structure pair, and let $(F',A') \in \F_{F,A}^-$.  For every $A' \subseteq H' \subseteq F'$, we have $\Nt_{H'}(F') \le \Nt_{H' \cap F}(F)$.
\end{obs}

\setcounter{subsubsection}{4}
\setcounter{thm}{36}

\begin{obs}\label{obs:deltaF'H'A'deltaFA}
Let $(F,A)$ be a graph structure pair, and let $(F',A') \in \F_{F,A}^-$. Then 
$$\Delta(F',H',A') \, \le \, \Delta(F'-v,A) \, \le \, \Delta(F,A) - 3\sqrt{\Delta(F,A)}$$
for every $A' \subsetneq H' \subsetneq F'$. Moreover, the same bounds holds if $H' = F'$ and $A'  \cap F \neq A$.
\end{obs}

\setcounter{thm}{38}

\begin{obs}\label{obs:deltaaddlittle}
Let $(F,A)$ be a graph structure pair, and let $(F',A') \in \F_{F,A}^-$. Then 
$$\Delta(F',A') \, \le \, (1 + \eps) \Delta(F,A).$$ 
\end{obs}

\setcounter{thm}{40}

The following lemma was stated but not proved in~\cite{FGMO}.
 
\begin{lemma}\label{obs:deltaFdeltaFv}
Let $(F,A)$ be a graph structure pair with $t_A(F) > 0$, and let $(F',A') \in \F_{F,A}^-$. If $(\log n)^{\gamma(F,A)} \le n^{v_A(F)+e(F)+1}$, then
$$\max\Big\{ \Delta(F',A') - \Delta(F,A), \sqrt{\Delta(F,A)} \Big\} \, \le \, \frac{\eps^2}{v_A(F)} \cdot \frac{\log n}{\log\log n}.$$
\end{lemma}

\begin{proof}
We have, as in the proof above, 
\begin{align*} 
\big| \Delta(F',A') - \Delta(F,A) \big| + \sqrt{\Delta(F,A)} & \, \le \, \big( \delta(F,A) + 2 \big)^C - \delta(F,A)^C + \delta(F,A)^{C/2} \\[+0.5ex]
& \, \le \, 4C \cdot \delta(F,A)^{C - 1} \, \le \, \frac{4C \cdot \Delta(F,A)}{\delta(F,A)} \, \le \, \frac{4}{C} \cdot \frac{\Delta(F,A)}{v_A(F)^2}.
\end{align*}
Note that $e(F) \le 2 v_A(F) - 1$, since $t_A(F) > 0$. Since $(\log n)^{\gamma(F,A)} \le n^{v_A(F)+e(F)+1}$, it follows that 
$$v_A(F) \, \ge \, \frac{v_A(F) + e(F)+1}{3} \, \ge \, \frac{\gamma(F,A)}{3} \cdot \frac{\log\log n}{\log n}  \, \ge \, \frac{\Delta(F,A)}{4} \cdot \frac{\log\log n}{\log n}.$$ 
It follows that
$$\big| \Delta(F',A') - \Delta(F,A) \big|  + \sqrt{\Delta(F,A)} \, \le \, \frac{16}{C \cdot v_A(F)} \cdot \frac{\log n}{\log \log n} \, \le \, \frac{\eps}{v_A(F)} \cdot \frac{\log n}{\log\log n},$$
as claimed.
\end{proof}

We next prove three new lemmas about the collection $\F_{F,A}^-$. Recall Table~4.1.

\vskip0.5cm
\begin{center}
\begin{tabular}{c|c|c|c|c|c|c}
& $(a)$ & $(b)$ & $(c)$ & $(d)$ & $(e)$ & $(f)$ \\[+0.6ex] 
\hline &&&&& \\[-2.1ex]
$v_{A^-}(F^-) - v_A(F)$ \, & \; 0 \; & \; $-1$ \; & \; 0 \; & \; $-2$ \; & \; $-1$ \; & \; $-1$ \;  \\[+0.6ex] 
\hline &&&&&& \\[-2.1ex]
$e(F^-) - e(F)$ \, & \; 1 \; & $\le 0$ & \; 1 \; & $\le 0$ & $\le 0$ & $\le 1$ \\[+0.6ex] 
\hline &&&&& \\[-2.1ex]
$o(F^-) - o(F)$ \, & $0$ 
& $\le 0$ & \; 0 \; & $\le 0$  & $\le 0$ & $\le 0$ 
\end{tabular}\\\
\end{center}
\begin{center} 
Table~4.1
\end{center}
\vskip0.2cm

\begin{lemmaa}\label{lem:fFAtdelta}
Let $(F,A)$ be a graph structure pair with $t_A(F) > 0$, and let $(F',A') \in \F_{F,A}^-$. If $(\log n)^{\gamma(F,A)} \le n^{v_A(F)+e(F)+1}$ and $t \le \omega$, then
$$f_{F',A'}(t) \, \le \, \min\Big\{ n^\eps, \, (\log n)^{\eps \Delta(F,A)} \Big\} \cdot (\log n)^{\Delta(F,A) - 3\sqrt{\Delta(F,A)}}.$$
\end{lemmaa}

\begin{proof}
This follows from the definition~\eqref{def:fFAt}, using Observation~\ref{obs:deltaaddlittle} and Lemma~\ref{obs:deltaFdeltaFv}.
\end{proof}

\pagebreak

\begin{lemmaa}\label{NA'F'NAF}
Let $(F,A)$ be a graph structure pair with $t_A(F) > 0$. If $(F',A') \in \F_{F,A}^-$, then
\begin{equation}\label{eq:nafnaf}
\Nt_{A'}(F')(n^{3/2}) \, \le \, \frac{2 \cdot e^{4o(F)}}{\sqrt{n}} \cdot \Nt_{A}(F)(n^{3/2}).
\end{equation}
\end{lemmaa}

\begin{proof}
Note that $\Nt_A(F)(n^{3/2}) = 2^{e(F)} e^{-4o(F)} n^{v_A(F) - e(F)/2}$ for every pair $(F,A)$, since $t = 1$ when $m = n^{3/2}$, and that $e(F') \le e(F) + 1$. We are thus required to prove that
\begin{equation}\label{eq:vvee}
v_{A'}(F') - v_A(F) \, \le \, \frac{e(F') - e(F) - 1}{2}.
\end{equation}
In cases $(a)$ and $(c)$, this follows immediately from Table~4.1. In cases $(b)$, $(d)$, $(e)$ and $(f)$ we need the following extra observation: $t_A(F) > 0$ implies that $t_A^*(A' \cap F) > 0$. This gives 
$$e(F) - e(F') \, \le \, e(A' \cap F) \, \le \, 2 v_A(A' \cap F) - 1 \, = \, 2 \big( v_{A}(F) - v_{A'}(F') \big) - 1,$$
as required.
\end{proof}

\begin{lemmaa}\label{lem:F3omega}
Let $(F,A,\phi)$ be a graph structure triple, and let $(F',A') \in \F_{F,A}^-$. Suppose that $t_{A'}(F') < t \le \omega < t_A(F)$, and that $\phi' \colon A' \to V(G_m)$ is faithful at time $t$. If $\E(m)$ holds and $(\log n)^{\Delta(F,A)} \le n^{2e(F) + 2}$, then
\begin{equation}\label{eq:F3omega:main}
N_{\phi'}(F')(m) \, \le \, \frac{(\log n)^{\Delta(F,A) - 2\sqrt{\Delta(F,A)}}}{\sqrt{n}} \cdot \Nt_A(F)(n^{3/2}).
\end{equation}
\end{lemmaa}

\begin{proof}
Note that since $t_{A'}(F') < t \le \omega$, it follows that $t_{A'}(F') = 0$. There are two cases: either $(F',A')$ is balanced, or it is not. In the former case, since $\M(m)$ holds we have 
$$N_{\phi'}(F')(m) \, \le \, (\log n)^{\Delta(F'-v,A')} \, \le \, (\log n)^{\Delta(F,A) - 3\sqrt{\Delta(F,A)}},$$
and so in this case we are done. In the latter case, since $\E(m)$ holds we have
\begin{equation}\label{eq:proof:F3omega}
N_{\phi'}(F')(m) \, \le \, (\log n)^{\Delta(F',H',A')} \Nt_{H'}(F')(m^+),
\end{equation}
where $A' \subsetneq H' \subsetneq F'$ is minimal such that $t < t_{H'}(F')$, and $m^+ = \max\{m,n^{3/2}\}$. Now
$$(\log n)^{\Delta(F',H',A')} \Nt_{H'}(F')(m^+) \, \le \, (\log n)^{\Delta(F,A) - 3\sqrt{\Delta(F,A)}} \Nt_H(F)(m^+),$$
where $H = H' \cap F$, by Observations~\ref{obs:NH'F'NHF} and~\ref{obs:deltaF'H'A'deltaFA}. 
Next, note that, since $t_A(F) > 0$, we have $t^*_A(H) > 0$, and thus $e(H) \le 2v_A(H) - 1$. Hence
$$\Nt_A(H)(n^{3/2}) \,\ge\, 2^{e(F)} e^{-4 o(H)} \sqrt{n}.$$
Since $0 < t \le \omega$, it follows that
$$\Nt_H(F)(m^+) \,\le \, \omega^{e(F)} e^{4 o(F)} \cdot \Nt_H(F)(n^{3/2}) \, \le \, \frac{\omega^{e(F)} e^{8 o(F)}}{\sqrt{n}} \cdot \Nt_A(F)(n^{3/2}),$$ 
and thus
$$N_{\phi'}(F')(m) \, \le \, \frac{(\log n)^{\Delta(F,A) - 2\sqrt{\Delta(F,A)}}}{\sqrt{n}} \cdot \Nt_A(F)(n^{3/2}),$$
as claimed.
\end{proof}

\section{Section~4.6: the land before time $t = \omega$}\label{LBTsec}

In this section we shall use the method of Bohman~\cite{Boh} to control the variables $N_\phi(F)$ up to time $t = \omega$, assuming that $t_A(F) > 0$. Recall once again that
$$f_{F,A}(t) \, = \, e^{C(o(F)+1)(t^2 + 1)} n^{-1/4} (\log n)^{\Delta(F,A) - \sqrt{\Delta(F,A)}}$$
for each graph structure pair $(F,A)$, that
$$f_y(t) \, = \, e^{Ct^2} n^{-1/4} (\log n)^{5/2} \qquad \text{and} \qquad f_x(t) \, = \, e^{-4t^2} f_y(t),$$
and that $\K^\E(m) = \Y(m) \cap \Z(m) \cap \Q(m)$. We shall prove the following proposition. 

\setcounter{subsubsection}{4}
\setcounter{subsection}{1}
\setcounter{thm}{54}
\setcounter{thma}{0}

\begin{prop}\label{NFa}
Let $(F,A)$ be a graph structure pair, and let $0 < t \le \omega < t_A(F)$. Then, with probability at least $1 - n^{- 3\log n}$, either $\big( \E(m - 1) \cap \M(m - 1) \cap \K^\E(m-1) \big)^c$ holds, or
\begin{equation}\label{eq:propNFa}
N_\phi(F)(m) \, \in \, \Nt_A(F)(m) \,\pm\, f_{F,A}(t) \cdot \Nt_A(F)(n^{3/2})
\end{equation}
for every $\phi \colon A \to V(G_m)$ which is faithful at time~$t$. 
\end{prop}

We remark, for future reference, that~\eqref{eq:propNFa} holds trivially if $(\log n)^{\Delta(F,A)} > n^{e(F)+2}$, since then $f_{F,A}(t) \cdot \Nt_A(F)(n^{3/2}) > n^{v_A(F)}$, and that we have $e(F) < 2v_A(F)$, since $t_A(F) > 0$.

The proof of Proposition~\ref{NFa} relies heavily on the fact that the event $\Y(m)$ gives us stronger bounds (in the range $t \le \omega$) on the variables $Y_e$ than those given by the event $\E(m)$. The bounds we need are essentially due to Bohman~\cite{Boh}, although he stated a slightly weaker version of the following proposition. For completeness, we shall sketch the proof.

Recall from~\cite[Section~5]{FGMO} the definition of the variables $Y_e^L(m)$. We shall need the following slight strengthening of~\cite[Proposition~\ref{lem:landbeforetime}]{FGMO} in~\cite[Section~5]{FGMO} in order to show that the event $\U(a)$ holds, where $a = \omega \cdot n^{3/2}$.

\begin{prop}[Bohman~\cite{Boh}]\label{lem:landbeforetime}
Let $m \le \omega \cdot n^{3/2}$. With probability at least $1 - n^{- C\log n}$, either $\big( \Z(m-1) \cap \Q(m-1) \big)^c$ holds, or 
\begin{equation}\label{eq:XandYlandbeforetime:strengthenedversion}
X_e(m) \in \Xt(m) \pm f_x(t) \Xt(n^{3/2}) \qquad \text{and} \qquad 2 \cdot Y^L_e(m) \in \Yt(m) \pm f_y(t) \Yt(n^{3/2})
\end{equation}
for every $e \in O(G_m)$.
\end{prop}

Recall the following martingale lemma from~\cite[Section~3]{FGMO}.

\setcounter{subsubsection}{3}
\setcounter{thm}{1}

\begin{lemma}\label{Bohmart}
Let $M$ be a super-martingale, defined on $[0,s]$, such that 
$$- \beta \le \Delta M(m) \le \alpha$$ 
for every $m \in [0,s - 1]$. Then, for every $0 \le x \le \min\{ \alpha, \beta \} \cdot s$,
$$\Pr\big( M(s) > M(0) + x \big) \, \le \, \exp\left( - \frac{x^2}{4 \alpha \beta s} \right).$$
\end{lemma}

\begin{proof}[Proof of Proposition~\ref{lem:landbeforetime}]
Consider the event that $m_0 \le \omega \cdot n^{3/2}$ is the minimal value of $m$ such that the event $\Z(m) \cap \Q(m)$ holds, and moreover either
$$\X(m)^c \textup{ holds,} \qquad \textup{or} \qquad 2 \cdot Y^L_e(m) \not\in \Yt(m) \pm f_y(t) \Yt(n^{3/2}).$$ 
We shall consider the two subcases: $\X(m_0)$ holds, or it does not, separately.

Suppose first that $\X(m_0)^c$ holds, and note that $\Xt(m) = \Theta \big( e^{-8t^2} \cdot \Xt(n^{3/2}) \big)$ and $\Yt(m) =  \Theta\big( t \cdot e^{-4t^2} \cdot \Yt(n^{3/2}) \big)$. By~\eqref{eq:Xeq}, and using the event $\X(m) \cap \Y(m) \cap \Q(m)$, which holds for all $m < m_0$, we have
\begin{align*}
\Ex\big[ \Delta X_e(m) \big] & \, = \, - \frac{2}{Q(m)} \sum_{f \in X_e(m)} Y_f(m) \\
&  \, \in \, - \frac{2}{\Qt(m)} \Big( 1 \pm \eps \cdot f_y(t) e^{4t^2} \Big) \Big( \Xt(m) \pm f_x(t) \Xt(n^{3/2}) \Big) \Big( \Yt(m) \pm f_y(t) \Yt(n^{3/2}) \Big).
\end{align*}
Multiplying out the brackets, we obtain
$$\Xt(m) \Yt(m) \, \pm \, O \Big( \eps t e^{-8t^2} f_y(t) + t e^{-4t^2} f_x(t) + e^{-8t^2} f_y(t)  + f_x(t) f_y(t) \Big) \Xt(n^{3/2}) \Yt(n^{3/2}),$$
and hence, recalling that $f_y(t) = e^{4t^2} f_x(t) \ll e^{-8\omega^2}$, we obtain
$$\Ex\big[ \Delta X_e(m) \big] \, \in \, -  \frac{16t }{n^{3/2}} \cdot \Xt(m) \,\pm \, \sqrt{C} \cdot \bigg( \frac{t + 1}{n^{3/2}} \bigg) \cdot f_x(t) \Xt(n^{3/2}).$$
It follows that
$$M_{X_e}^\pm(m') \, = \sum_{m=0}^{m'-1} \bigg[ \Delta X_e(m) + \frac{16t }{n^{3/2}} \cdot \Xt(m) \, \pm \, \sqrt{C}  \cdot \bigg( \frac{t + 1}{n^{3/2}} \bigg) \cdot f_x(t) \Xt(n^{3/2}) \bigg].$$
is a super-/sub-martingale pair on $0 \le m' < m_0$. Moreover, 
the number of open triangles which are destroyed in step $m+1$ of the triangle-free process is at most 
$$\max_{e \in O(G_m)} Y_e(m) \, \le \, \sqrt{n}$$
since $\Y(m)$ holds for every $m < m_0$. It follows that $-\sqrt{n} \le \Delta X_e(m) \le 0$, and hence
\begin{equation}\label{eq:MXbounds}
- 2 \sqrt{n} \, \le \, \Delta M^\pm_{X_e}(m) \, \le \, \frac{C}{\sqrt{n}},
\end{equation}
for every $e \in O(G_m)$, while $\Y(m)$ holds. Now, set
$$\alpha \, = \, \frac{C}{\sqrt{n}} + \frac{f_x(\omega) \Xt(n^{3/2})}{m_0} \qquad \textup{and} \qquad \beta \, = \, 2 \sqrt{n} + \frac{f_x(\omega) \Xt(n^{3/2})}{m_0},$$
and observe that, since $f_x(t) \Xt(n^{3/2}) \ge n^{3/4} (\log n)^{5/2}$ and we may assume\footnote{If $m_0 \le n^{1/4}$, then~\eqref{eq:MXbounds} implies that $X_e(m_0) > n - O(n^{3/4})$, and hence $\X(m_0) \cup \Y(m_0-1)^c$ holds.} that $m_0 \ge n^{1/4}$, 
$$\alpha \cdot \beta \cdot m_0 \, \le \, \frac{\big( f_x(t) \Xt(n^{3/2}) \big)^2}{(\log n)^4}.$$ 
Hence, applying Lemma~\ref{Bohmart}, we obtain
$$\Pr\bigg( \bigg( M_{X_e}^-(m_0) > \frac{1}{2} f_x(t) \Xt(n^{3/2}) \bigg) \cap \K(m-1) \bigg) \, \le \, e^{-(\log n)^3},$$
and similarly for $M_{X_e}^+$. Finally, noting that 
$$\frac{1}{n^{3/2}} \sum_{m=0}^{m'-1} 16t  \cdot \Xt(m) \in \big( 1 - e^{-8t'^2} \big) n \pm 1 \quad \textup{and} \quad \frac{1}{n^{3/2}} \sum_{m=0}^{m'-1} (t + 1) f_x(t) \le \frac{1}{C} \cdot f_x(t'),$$
it follows that, with probability at least $1 - n^{-\log n}$, we have
\begin{align*}
X_e(m') & \, \in \, n \,-\, \sum_{m=0}^{m'-1} \bigg[ \frac{16t }{n^{3/2}} \cdot \Xt(m) \pm \sqrt{C}  \cdot \bigg( \frac{t + 1}{n^{3/2}} \bigg) \cdot f_x(t) \Xt(n^{3/2}) \bigg] \,\pm \, \frac{1}{2} f_x(t') \Xt(n^{3/2})\\
& \, \in \, \Xt(m') \, \pm \, f_x(t') \Xt(n^{3/2})
\end{align*}
for every $e \in O(G_m)$, as required.

The proof in the case that $\X(m_0)$ holds is similar. The first step is to break $\Delta Y^L_e(m)$ into two pieces, $C_e^Y(m)$ and $D_e^Y(m)$, where $C_e^Y(m)$ denotes the number of edges of $Y^L_e(m+1)$ which were created in step $m+1$ of the triangle-free process, and $D_e^Y(m)$ denotes the number of edges of $Y^L_e(m)$ which were closed in that step. It follows immediately from the definition that
\begin{equation}\label{eq:N=C+D}
Y^L_e(m') \, = \, \sum_{m = 0}^{m'-1} \Big( C_e^Y(m) - D_e^Y(m) \Big).
\end{equation}
We claim first that
\begin{align*}
\Ex\big[ C^Y_e(m) \,|\, G_m \big] & \, = \, \frac{1}{2} \cdot \frac{X_e(m)}{Q(m)} \, \in \, \frac{\big( 1 \pm O(\eps) e^{4t^2} f_y(t)  \big) \big( \Xt(m) \pm f_x(t) \Xt(n^{3/2}) \big)}{2 \cdot \Qt(m)} \\
& \, \subseteq \, \frac{2e^{-4t^2}n \pm \big( e^{4t^2} f_x(t) + O(\eps) f_y(t) \big) \Xt(n^{3/2})}{n^2} \, \subseteq \, \frac{2e^{-4t^2} \pm f_y(t)}{n},
\end{align*}
where we used the event $\X(m) \cap \Q(m)$, and the facts that $f_y(t) = e^{4t^2} f_x(t)$, and that $\Xt(m) = \Theta \big( e^{-8t^2} \cdot \Xt(n^{3/2}) \big) \gg f_x(t) \Xt(n^{3/2})$ for $t \le \omega$. Similarly, we have
\begin{align*}
\Ex\big[ D^Y_e(m) \,|\, G_m \big] & \, = \, \frac{1}{Q(m)} \sum_{f \in Y^L_e(m)} Y_f(m) \, \in \, \frac{\big( 1 \pm O(\eps) e^{4t^2} f_y(t)  \big) \big( \Yt(m) \pm f_y(t) \Yt(n^{3/2}) \big)^2}{2 \cdot \Qt(m)} \\
& \, \subseteq \, \frac{\Yt(m)^2}{2 \cdot \Qt(m)} \, \pm \, O\Big( \eps \cdot t^2 e^{-4t^2} f_y(t) + t e^{-4t^2} f_y(t) + f_y(t)^2 \Big) \frac{\Yt(n^{3/2})^2}{\Qt(m)}\\
& \, \subseteq \, \frac{16t^2e^{-4t^2} \pm \big( t^2 + O(t) \big) f_y(t)}{n},
\end{align*}
using the event $\Y(m) \cap \Q(m)$, and the fact that $\Yt(m) =  \Theta\big( t \cdot e^{-4t^2} \cdot \Yt(n^{3/2}) \big)$. Hence
$$M^\pm_{C^Y_e}(m') \, = \sum_{m=0}^{m'-1} \bigg[ C^Y_e(m) - \frac{2e^{-4t^2} \pm f_y(t)}{n} \bigg].$$
and
$$M_{D^Y_e}^\pm(m') \, = \sum_{m=0}^{m'-1} \bigg[ D^Y_e(m)  - \frac{16t^2e^{-4t^2} \pm \big( t^2 + O(t) \big) f_y(t)}{n} \bigg].$$
are both super-/sub-martingale pairs while the event $\X(m) \cap \Y(m) \cap \Q(m)$ holds. 

Next, observe that 
$$0 \le C_e^Y(m) \le 1 \qquad  \text{and} \qquad 0 \le D_e^Y(m) \le (\log n)^2,$$ 
while $\Z(m)$ holds, and hence 
$$- \frac{C}{n} \, \le \, \Delta M^\pm_{C^Y_e}(m) \, \le \, 2 \qquad \text{and} \qquad - \frac{C}{n} \, \le \, \Delta M^\pm_{D^Y_e}(m) \, \le \, 2(\log n)^2.$$
Now, set
$$\alpha \, = \, 2(\log n)^2 + \frac{f_y(\omega) \Yt(n^{3/2})}{m_0} \qquad \textup{and} \qquad \beta \, = \, \frac{C}{n} + \frac{f_y(\omega) \Yt(n^{3/2})}{m_0},$$
and observe that, since $f_y(t) \Yt(n^{3/2}) \ge n^{1/4} (\log n)^{5/2}$ and we may assume that $m_0 \ge n^{1/4}$, 
$$\alpha \cdot \beta \cdot m_0 \, \le \, \frac{\big( f_y(t) \Yt(n^{3/2}) \big)^2}{\omega \cdot (\log n)^2}.$$ 
Hence, applying Lemma~\ref{Bohmart}, we obtain
$$\Pr\bigg( \bigg( M_{D^Y_e}^-(m) > \frac{1}{4} f_y(t) \Yt(n^{3/2}) \bigg) \cap \K(m-1) \bigg) \, \le \, n^{-2 C \log n},$$
and similarly for $M_{C^Y_e}^+$, $M_{C^Y_e}^-$ and $M_{D^Y_e}^+$. 


Finally, note that $\frac{\textup{d}}{\textup{d}t}\, \Yt(m) = \big( 4 - 32t^2 \big) e^{-4t^2} \sqrt{n}$ and $\frac{1}{n^{3/2}} \sum_{m=0}^{m'-1} (t^2+1) f_y(t) \le \frac{1}{C} \cdot f_y(t')$. Hence, with probability at least $1 - n^{- 2C\log n}$, we have
\begin{align*}
2 \cdot Y_e(m') &  \, = \, 2 \cdot \sum_{m = 0}^{m'-1} \Big( C_e^Y(m) - D_e^Y(m) \Big)\\
& \, \in \, \sum_{m=0}^{m'-1} \bigg[  \frac{(4 - 32t^2) e^{-4t^2}}{n} \pm \frac{O(t^2 + 1) f_y(t)}{n} \bigg] \,\pm \, \frac{1}{2} f_y(t') \Yt(n^{3/2})\\
& \, \in \, \Yt(m') \, \pm \, f_y(t') \Yt(n^{3/2}),
\end{align*}
as required. Putting the pieces together, it follows that 
$$\Pr\bigg( \bigcup_{t \le \omega} \big( \X(m) \cap \Y(m) \big)^c \cap \Z(m) \cap \Q(m) \bigg)  \, \le \, n^{- C \log n},$$ 
and moreover the same holds if we replace $\Y(m)$ by the event in~\eqref{eq:XandYlandbeforetime:strengthenedversion}.
\end{proof}

\setcounter{subsection}{0}

\subsection{Proof of Proposition~\ref{NFa}}

\setcounter{subsection}{1}

Returning to the proof of Proposition~\ref{NFa}, let us fix a graph structure triple $(F,A,\phi)$ with $t_A(F) > 0$. The first step in the proof is to break up $N_\phi(F)$ as follows: 
\begin{equation}\label{eq:N=C+D}
N_\phi(F)(m') \, = \, \sum_{m = 0}^{m'-1} \Big( C_\phi(F)(m) - D_\phi(F)(m) \Big),
\end{equation}
where $C_\phi(F)(m)$ denotes the number of copies of $F$ rooted at $\phi(A)$ which are created at step $m+1$ of the triangle-free process, and $D_\phi(F)(m)$ denotes the number of such copies which are destroyed in that step. Let's deal first with $C_\phi(F)$. It was proved in~\cite[Lemma~4.21]{FGMO} that 
\begin{equation}\label{ExdeltaC}
\Ex\big[ C_\phi(F)(m) \,|\, G_m \big] \, \in \, \frac{1}{Q(m)} \sum_{F^o \in \F^o_F} N_{\phi}(F^o)(m)
\end{equation}
for every $\phi$ which is faithful at time $t$. Using the events $\E(m)$ and $\Q(m)$, we easily obtain a super-/sub-martingale pair.

\setcounter{subsubsection}{4}
\setcounter{thm}{57}

\begin{lemma}\label{Cmart}
Let $(F,A,\phi)$ be a graph structure triple, 
and suppose that $\phi$ is faithful at time~$t$, where $0 < t \le \omega < t_A(F)$. If $\E(m) \cap \Z(m) \cap \Q(m)$ holds, then
\begin{equation}\label{Cmarteq}
\Ex\big[ C_\phi(F)(m) \,|\, G_m \big] - \frac{e(F) \cdot \Nt_A(F)(m)}{t \cdot n^{3/2}} \, \in \, \pm \, \frac{f_{F,A}(t) \cdot \Nt_A(F)(n^{3/2})}{n^{3/2}}.
\end{equation}
\end{lemma}

\begin{proof}
This follows from~\eqref{ExdeltaC} via a straightforward calculation. Note first that the lemma holds trivially if $e(F) = 0$, since then $C_\phi(F)(m) = 0$ for every $m \in \N$. So assume that $e(F) > 0$, and recall that the event $\Q(m)$ implies that $Q(m) \in \big( 1 \pm \eps \cdot e^{4t^2} f_y(t)  \big)\Qt(m)$, and that the event $\E(m)$ (and the fact that $t \le \omega < t_A(F) \le t_A(F^o)$, by Observation~\ref{obs:F-}) implies that 
$$N_{\phi}(F^o)(m) \, \in \, \Nt_A(F^o)(m) \,\pm\, f_{F^o,A}(t) \cdot \Nt_A(F^o)(n^{3/2}).$$ 
Moreover $2t e^{4t^2} \cdot \Nt_A(F^o)(m) = \sqrt{n} \cdot \Nt_A(F)(m)$ and $|\F^o_F| = e(F)$,  by Observation~\ref{obs:NtF-}, and hence, by~\eqref{ExdeltaC},
$$\Ex\big[ C_\phi(F)(m) \,|\, G_m \big] \in \frac{1 \pm e^{4t^2} f_y(t)}{\Qt(m)} \bigg( \Nt_A(F)(m) \frac{e(F) \sqrt{n}}{2t e^{4t^2}} \,\pm\, \sum_{F^o \in \F^o_F} f_{F^o,A}(t) \Nt_A(F)(n^{3/2}) \cdot \frac{\sqrt{n}}{2 e^{4}} \bigg).$$
Thus the left-hand side of~\eqref{Cmarteq} lies in the interval 
$$\pm \frac{1}{\Qt(m)} \bigg( \frac{e(F) \sqrt{n}}{2t} \cdot f_y(t) \Nt_A(F)(m) +  \sum_{F^o \in \F^o_F} f_{F^o,A}(t) \Nt_A(F)(n^{3/2}) \cdot \frac{\sqrt{n}}{e^{4}} \bigg).$$
Since $\Nt_A(F)(m) = t^{e(F)} e^{-4 o(F) (t^2 - 1)} \Nt_A(F)(n^{3/2})$, this is in turn contained in
$$\pm \frac{\Nt_A(F)(n^{3/2})}{n^{3/2}} \Big( e(F) \cdot t^{e(F)-1} e^{4 o(F) + 4t^2} f_y(t)  +  e^{4t^2} f_{F^o,A}(t) \Big) \, \subseteq \, \pm \frac{f_{F,A}(t) \cdot \Nt_A(F)(n^{3/2})}{n^{3/2}},$$
as required, where the final inequality follows since $e(F) > 0$ and 
$$(\log n)^{e(F) + o(F)} \big( f_y(t) + f_{F^o,A}(t) \big) \ll f_{F,A}(t)$$ 
by Observation~\ref{obs:gfat}. 
\end{proof}

We shall also need a corresponding lemma for the variables $D_\phi(F)$. It was proved in~\cite[Lemma~4.21]{FGMO} that, if $\Z(m)$ holds and $\phi$ is faithful at time $t$, then 
\begin{equation}\label{ExdeltaD}
\Ex\big[ D_\phi(F)(m) \,|\, G_m \big] \, \in \, \frac{1}{Q(m)} \bigg( \sum_{F^* \in N_\phi(F)} \sum_{f \in O(F^*)} Y_f(m) \, \pm \, o(F)^2 (\log n)^2 N_\phi(F) \bigg).
\end{equation}
We shall use the events $\E(m)$, $\Y(m)$ and $\Q(m)$ to obtain a super-/sub-martingale pair.

\begin{lemma}\label{Dmart}
Let $(F,A,\phi)$ be a graph structure triple, and suppose that $0 < t \le \omega < t_A(F)$, and that $\phi$ is faithful at time~$t$. If $\E(m) \cap \Y(m) \cap \Z(m) \cap \Q(m)$ holds, then
\begin{equation}\label{Dmarteq}
\Ex\big[ D_\phi(F)(m) \,|\, G_m \big] - \frac{8t \cdot o(F) \cdot \Nt_A(F)(m)}{n^{3/2}} \, \in \, \pm \, \frac{C \cdot o(F) \cdot ( t + 1 )}{n^{3/2}} \cdot  f_{F,A}(t) \Nt_A(F)(n^{3/2}).
\end{equation}
\end{lemma}

\begin{proof}
Note first that if $o(F) = 0$ then the result holds trivially, since in that case copies of $F$ cannot be destroyed, and so $D_\phi(F)(m) = 0$ for every $m \in \N$. So assume that $o(F) > 0$ and recall that, since $\E(m) \cap \Y(m) \cap \Q(m)$ holds and $t \le \omega < t_A(F)$, we have 
$$Q(m) \in \big( 1 \pm \eps \cdot e^{4t^2} f_y(t)  \big)\Qt(m), \qquad Y_f(m) \in \Yt(m) \pm f_y(t) \Yt(n^{3/2})$$ 
and
$$N_{\phi}(F)(m) \, \in \, \Nt_A(F)(m) \pm f_{F,A}(t) \cdot \Nt_A(F)(n^{3/2}).$$ 
By~\eqref{ExdeltaD}, it follows that $\Ex\big[ D_\phi(F)(m) \,|\, G_m \big]$ is contained in the interval
$$\frac{1 \pm \eps \cdot e^{4t^2} f_y(t)}{\Qt(m)} \cdot o(F) \cdot \Big(  \Nt_A(F)(m) \pm f_{F,A}(t) \Nt_A(F)(n^{3/2}) \Big) \Big( \Yt(m) \pm 2 f_y(t) \Yt(n^{3/2}) \Big),$$
where we used our assumption\footnote{Recall that Proposition~\ref{NFa} is trivial if $o(F) \ge (\log n)^{1/5}$.} that $o(F) \le n^{o(1)} \le n^{1/4} \le f_y(t) \Yt(n^{3/2})$ to absorb the final error term in~\eqref{ExdeltaD}. It follows that the left-hand side of~\eqref{Dmarteq} is contained in
$$\pm \, \frac{o(F) \cdot \Yt(n^{3/2})}{\Qt(m)} \bigg( \Big( t \cdot f_y(t) + 2 f_y(t) \Big) \Nt_A(F)(m) + \Big(  O(t) \cdot e^{-4t^2} + 2 f_y(t) \Big) f_{F,A}(t) \Nt_A(F)(n^{3/2}) \bigg),$$
since $\Yt(m) = O(t) \cdot e^{-4t^2} \cdot \Yt(n^{3/2})$. This in turn is a subset of
$$\pm \, \frac{C \cdot o(F)}{n^{3/2}} \cdot \Big( t^{e(F)} (t +1) e^{4 o(F)} f_y(t) \, + \, \big( t + e^{4t^2} f_y(t) \big) f_{F,A}(t)  \Big) \cdot \Nt_A(F)(n^{3/2}).$$
since $\Nt_A(F)(m) = \Nt_A(F)(n^{3/2}) \cdot t^{e(F)} e^{-4 o(F) (t^2 - 1)}$ and $o(F) > 0$. 

Finally, note that $e^{4t^2} f_y(t) \ll 1$ for every $t \le \omega$, and recall that $(\log n)^{e(F) + o(F)} f_y(t) \ll f_{F,A}(t)$, by Observation~\ref{obs:gfat}. It follows that
$$\Ex\big[ D_\phi(F)(m) \,|\, G_m \big] - \frac{8t \cdot o(F) \cdot \Nt_A(F)(m)}{n^{3/2}} \, \in \, \pm \, \frac{C \cdot o(F) \cdot ( t + 1 )}{n^{3/2}} \cdot  f_{F,A}(t) \Nt_A(F)(n^{3/2}),$$
as required.
\end{proof}

In order to use Lemma~\ref{Bohmart}, we shall need bounds on $C_\phi(F)(m)$ and $D_\phi(F)(m)$ which hold deterministically for all $0 < t \le \omega$. We shall prove the following bounds. 

\begin{lemma}\label{Cbound}
Let $(F,A,\phi)$ be a graph structure triple, and suppose that $0 < t \le \omega < t_A(F)$, and that $\phi$ is faithful at time~$t$. If $\E(m) \cap \M(m)$ holds, then 
\begin{equation}\label{Cboundeq}
0 \, \le \, C_\phi(F)(m) \, \le \, \min\Big\{ n^\eps, \, (\log n)^{\eps \Delta(F,A)} \Big\} \cdot \frac{(\log n)^{\Delta(F,A) - 2\sqrt{\Delta(F,A)}}}{\sqrt{n}} \cdot \Nt_A(F)(n^{3/2}).
\end{equation}
Moreover, the same bounds also hold for $D_\phi(F)(m)$.
\end{lemma}

The required bounds on $C_\phi(F)(m)$ follow easily from Lemmas~\ref{lem:fFAtdelta},~\ref{NA'F'NAF} and~\ref{lem:F3omega}.

\begin{proof}[Proof of Lemma~\ref{Cbound}]
Recall first that, by~\cite[Lemma~4.28]{FGMO}, we have
$$C_\phi(F)(m) \, \le \, \sum_{(F',A') \in \F_{F,A}^+} \max_{\phi' : \, A' \to V(G_m)} N_{\phi'}(F')(m).$$
There are two cases to consider: $t \le \omega < t_{A'}(F')$ and $t > t_{A'}(F') = 0$. Set 
$$\Upsilon(F,A) \, := \, \min\Big\{ n, \, (\log n)^{\Delta(F,A)} \Big\},$$
and recall that, by Lemmas~\ref{lem:fFAtdelta} and~\ref{NA'F'NAF}, we have 
$$f_{F',A'}(t) \le \Upsilon(F,A)^\eps \cdot (\log n)^{\Delta(F,A) - 3\sqrt{\Delta(F,A)}} \quad \textup{and} \quad \Nt_{A'}(F')(n^{3/2}) \le \frac{e^{4o(F)+1}}{\sqrt{n}} \cdot \Nt_{A}(F)(n^{3/2}).$$ 
Together with the event $\E(m)$, this implies that 
\begin{align*}
N_{\phi'}(F')(m) & \, \le \, \Nt_{A'}(F')(m) \, + \, f_{F',A'}(t) \cdot \Nt_{A'}(F')(n^{3/2}) \\
& \, \le \, \Upsilon(F,A)^\eps \cdot \frac{(\log n)^{\Delta(F,A) - 2\sqrt{\Delta(F,A)}}}{\sqrt{n}} \cdot \Nt_A(F)(n^{3/2}).
\end{align*}
In the latter case, Lemma~\ref{lem:F3omega} gives us
$$N_{\phi'}(F')(m) \, \le \, \frac{(\log n)^{\Delta(F,A) - 2\sqrt{\Delta(F,A)}}}{\sqrt{n}} \cdot \Nt_A(F)(n^{3/2}).$$
Since $|\F_{F,A}^+| \le v_A(F)^2 \le (\log n)^{2/5}$, the upper bound in~\eqref{Cboundeq} follows. The proof of the bounds on $D_\phi(F)(m)$ is identical, using~\cite[Lemma~4.30]{FGMO}.
\end{proof}

We can now apply Lemma~\ref{Bohmart} to the variables $C_\phi(F)$ and $D_\phi(F)$. 

\begin{proof}[Proof of Proposition~\ref{NFa}]
For each $m \in [m^*]$, set $\K(m) = \E(m) \cap \M(m) \cap \K^\E(m)$. We shall bound, for each $m_0 \le \omega \cdot n^{3/2}$, the probability that $m_0$ is the minimal $m \in \N$ such that $\K(m-1)$ holds, and
$$N_\phi(F)(m) \, \not\in \, \Nt_A(F)(m) \,\pm\, f_{F,A}(t) \cdot \Nt_A(F)(n^{3/2})$$
for some $\phi$ which is faithful at time~$t = m \cdot n^{-3/2}$. Note that the event in the statement of the proposition implies that this event holds for some $m \le \omega \cdot n^{3/2}$.

Fix $m_0 \le \omega \cdot n^{3/2}$, and for each $m' \le m_0$, define random variables
$$M_C^\pm(m') = \sum_{m=0}^{m'-1} \bigg[ C_\phi(F)(m) - \frac{e(F) \cdot \Nt_A(F)(m)}{t \cdot n^{3/2}} \pm \, \frac{f_{F,A}(t) \cdot  \Nt_A(F)(n^{3/2})}{n^{3/2}} \bigg]$$
and
$$M_D^\pm(m') = \sum_{m=0}^{m'-1} \bigg[ D_\phi(F)(m) \,-\, \frac{8t \cdot o(F) \cdot \Nt_A(F)(m)}{n^{3/2}} \, \pm \, \frac{C \cdot o(F) \cdot ( t + 1 )}{n^{3/2}} \,  f_{F,A}(t) \Nt_A(F)(n^{3/2}) \bigg].$$
It follows from Lemmas~\ref{Cmart} and~\ref{Dmart} that, while the event $\E(m) \cap \Y(m) \cap \Z(m) \cap \Q(m)$ holds, $M_C^\pm$ and $M_D^\pm$ are both super-/sub-martingale pairs. Now, set 
$$\alpha \, = \, \bigg( \Upsilon(F,A)^\eps \cdot \frac{(\log n)^{\Delta(F,A) - 2\sqrt{\Delta(F,A)}}}{\sqrt{n} } + \frac{f_{F,A}(\omega)}{m_0} \bigg) \cdot \Nt_A(F)(n^{3/2})$$
where $\Upsilon(F,A) = \min\big\{ n, \, (\log n)^{\Delta(F,A)} \big\}$, and
$$\beta \, = \, \bigg( \frac{(\log n)^{e(F) + o(F)}}{n^{3/2}} + \frac{f_{F,A}(\omega)}{m_0} \bigg) \cdot \Nt_A(F)(n^{3/2}).$$
By Lemma~\ref{Cbound}, we have
$$- \beta \, \le \, \Delta M^\pm_C(m) + \Delta M^\pm_D(m) \, \le \, \alpha$$
while $\E(m) \cap \M(m)$ holds. Moreover, since $f_{F,A}(t_0) \ge n^{-1/4} (\log n)^{\Delta(F,A) - \sqrt{\Delta(F,A)}}$, and we may assume that $m_0 \ge n^\eps$, we have
$$\frac{\alpha \cdot \beta \cdot m_0}{\Nt_A(F)(n^{3/2})^2} \, \le \, \frac{f_{F,A}(t_0)^2}{(\log n)^4}.$$
Hence, by Lemma~\ref{Bohmart}, we obtain 
$$\Pr\bigg( \bigg( M_C^-(m_0) > \frac{1}{4} f_{F,A}(t_0) \Nt_A(F)(n^{3/2}) \bigg) \cap \K(m_0-1) \bigg) \, \le \, e^{-(\log n)^3},$$
and similarly for $M_C^+$, $M_D^-$ and $M_D^+$.

To complete the proof, note that 
$$\sum_{m=0}^{m'-1} C( t + 1) (o(F) + 1) \cdot f_{F,A}(t) \, \le \, \frac{n^{3/2}}{C} \cdot f_{F,A}(t').$$
Therefore, if
$$\max\Big\{ \min\big\{ |M_C^+(m)|, |M_C^-(m)| \big\}, \min\big\{ |M_D^+(m)|, |M_D^-(m)| \big\} \Big\} \, \le \, \frac{1}{4} f_{F,A}(t) \Nt_A(F)(n^{3/2})$$
then 
$$\sum_{m = 0}^{m'-1} \Big( C_\phi(F)(m) - D_\phi(F)(m) \Big) \in \sum_{m=0}^{m'-1} \bigg( \frac{e(F)}{t} - 8t \cdot o(F)  \bigg) \frac{\Nt_A(F)(m)}{n^{3/2}} \pm f_{F,A}(t') \Nt_A(F)(n^{3/2}).$$
Finally, note that $\frac{\textup{d}}{\textup{d}t}\, \Nt_A(F)(m) = \big( \frac{e(F)}{t} - 8t \cdot o(F) \big) \Nt_A(F)(m)$, and that the number of choices for $\phi$ is negligible, since $|A| \le (\log n)^{1/5}$. Hence, with probability at least $1 - n^{-3\log n}$, we have
$$N_\phi(F)(m') \, = \, \sum_{m = 0}^{m'-1} \Big( C_\phi(F)(m) - D_\phi(F)(m) \Big) \, \in \, \Nt_A(F)(m') \,\pm\, f_{F,A}(t') \Nt_A(F)(n^{3/2}),$$
as required.
\end{proof}

\setcounter{subsection}{1}

\subsection{The number of open edges in a set before time $t = \omega$}

\setcounter{subsection}{1}

Let us finish this section by using Bohman's method to prove~\cite[Lemma~7.12]{FGMO}. The proof is almost identical to that of the bounds on $X_e(m)$ in Proposition~\ref{lem:landbeforetime}. 

\setcounter{subsubsection}{7}
\setcounter{thm}{11}

\begin{lemma}\label{lemma:b:landbeforetime}
Let $S \subseteq V(G_m)$, and let $\n = (A_1,\ldots,A_k)$ be a collection of subsets of $S$. If $|S| \sqrt{n} \le o_\n(S,0) \le n^{5/4}$, then 
$$\Pr\Big( \big\{ |o^*_\n(S,m)| > 1 \big\} \cap \tilde{\n}(S,m) \cap \Y(m) \cap \Q(m) \Big) \, \le \, n^{- |S| n^{4\delta}}$$
for every $m \le \omega \cdot n^{3/2}$. 
\end{lemma}

In fact we shall prove the following, slightly stronger statement. Set 
$$f_o(t) \, = \, n^{3\delta} \cdot f_y(t),$$ 
and note that we have $f_o(t) \ll g_o(t) e^{-4t^2}$ for every $m \le \omega \cdot n^{3/2}$, and so the following lemma trivially implies Lemma~\ref{lemma:b:landbeforetime}. 

\begin{lemmaa}\label{lem:openedges:landbeforetime}
Let $S \subseteq V(G_m)$, and let $\n = (A_1,\ldots,A_k)$ be a collection of subsets of $S$. Suppose that $|S| \sqrt{n} \le o_\n(S,0) \le n^{5/4}$, and let $m \le \omega \cdot n^{3/2}$. Then, with probability at least $1 - n^{- |S| n^{4\delta}}$ either $\big( \tilde{\n}(S,m) \cap \Y(m) \cap \Q(m) \big)^c$ holds, or 
\begin{equation}\label{eq:lemma:b:landbeforetime}
o_\n(S,m) \in e^{-4t^2} o_\n(S,0) \pm f_o(t) o_\n(S,0).
\end{equation}
\end{lemmaa}

\begin{proof}
Let $m_0 \le \omega \cdot n^{3/2}$ be minimal such that the event $\tilde{\n}(S,m_0) \cap \Y(m_0) \cap \Q(m_0)$ 
holds, and 
$$o_\n(S,m_0) \not\in e^{-4t_0^2} o_\n(S,0) \pm f_o(t_0) o_\n(S,0),$$
where $t_0 = m_0 \cdot n^{-3/2}$. It follows that, if $m < m_0$, then 
\begin{align*}
\Ex\big[ \Delta o_\n(S,m) \big] & \, = \, - \frac{1}{Q(m)} \sum_{f \in O_\n(S,m)} Y_f(m) \\
&  \, \in \, - \frac{e^{-4t^2} o_\n(S,0)}{\Qt(m)} \big( 1 \pm e^{4t^2} f_o(t)  \big)^2 \Big( \Yt(m) \pm f_y(t) \Yt(n^{3/2}) \Big)\\
&  \, \subset \, - \frac{8t}{n^{3/2}} \cdot e^{-4t^2} o_\n(S,0) \, \pm \, \sqrt{C} \cdot \bigg( \frac{t + 1}{n^{3/2}} \bigg) \cdot f_o(t) o_\n(S,0),
\end{align*}
since $\Yt(m) =  \Theta\big( t \cdot e^{-4t^2} \cdot \Yt(n^{3/2}) \big)$, and using the event $\Y(m) \cap \Q(m)$ and the bounds~\eqref{eq:lemma:b:landbeforetime}, which hold for all $m < m_0$. It follows that
$$M_{S,\n}^\pm(m') \, = \sum_{m=0}^{m'-1} \bigg[ \Delta o_\n(S,m) + \frac{8t}{n^{3/2}} \cdot e^{-4t^2} o_\n(S,0) \, \pm \, \sqrt{C} \cdot \bigg( \frac{t + 1}{n^{3/2}} \bigg) \cdot f_o(t) o_\n(S,0) \bigg].$$
is a super-/sub-martingale pair on $0 \le m' < m_0$. Moreover, we have
$$- \, n^\delta \, \le \, \Delta o_\n(S,m) \, \le \, 0$$
for every $m < m_0$, since $\tilde{\n}(S,m)$ holds (see~\cite[Section~7.4]{FGMO}), and so
$$- \, 2 \cdot n^\delta \, \le \, \Delta M^\pm_{S,\n}(m) \, \le \, \frac{C \cdot o_\n(S,0)}{n^{3/2}},$$
for every $m < m_0$. Since $m_0 \ge \ds\frac{f_o(t) o_\n(S,0)}{4 \cdot n^{\delta}} \gg \frac{f_o(t) o_\n(S,0)^2}{n^{3/2}}$ and $f_o(t)^2 o_\n(S,0) \ge n^{6\delta} |S|$, it follows by Lemma~\ref{Bohmart} that 
$$\Pr\bigg( M_{X_e}^-(m_0) > \frac{1}{2} f_o(t) o_\n(S,0) \bigg) \, \le \, \exp\left( - \frac{ f_o(t)^2 o_\n(S,0) }{ C^2 \cdot n^\delta } \right) \, \ll \, n^{-|S| n^{4\delta}},$$
and similarly for $M_{X_e}^+$. 
It follows that, with probability at least $1 - n^{-|S| n^{4\delta}}$, we have
\begin{align*}
o_\n(S,m') & \in o_\n(S,0) - \sum_{m=0}^{m'-1} \bigg[ \frac{8t}{n^{3/2}} \cdot o_\n(S,m) \pm \sqrt{C} \bigg( \frac{t + 1}{n^{3/2}} \bigg) f_o(t) o_\n(S,0) \bigg] \pm \frac{1}{2} f_o(t') o_\n(S,0)\\
& \, \subset \, e^{-4t'^2} o_\n(S,0) \pm f_o(t') o_\n(S,0)
\end{align*}
as required.
\end{proof}

\setcounter{subsubsection}{5}

\section{Section~6: Whirlpools}\label{XYQsec}\label{AwhirlSec}

In this section, we shall prove that the variables $\Xb$, $\Yb$ and $Q$ follow (in expected value) a three-dimensional dynamical system which looks like a whirlpool. Set 
$$\Xs(m) \, = \,  \frac{\Xb(m) - \tilde{X}(m)}{g_q(t) \tilde{X}(m)}, \quad \Ys(m) \, = \,  \frac{\Yb(m) - \Yt(m)}{g_q(t) \Yt(m)} \quad \text{and} \quad \Qs(m) \, = \,  \frac{Q(m) - \Qt(m)}{g_q(t) \Qt(m)}$$
for each $m \in \N$. We shall prove the following lemma.

\setcounter{subsubsection}{6}
\setcounter{subsection}{1}
\setcounter{thm}{1}
\setcounter{thma}{0}

\begin{lemma}\label{whirlpool}
Let $\omega \cdot n^{3/2} < m \le m^*$, and suppose that $\X(m) \cap \Y(m) \cap \Q(m)$ holds. Then
\begin{itemize}
\item[$(a)$] $\Ex \big[ \Delta \Qs(m) \big] \, \in \, \ds\frac{4t}{n^{3/2}} \Big( - 2 \Ys(m) + \Qs(m) \pm o(1) \Big)$.
\item[$(b)$] $\Ex \big[ \Delta \Ys(m) \big] \, \in \, \ds\frac{4t}{n^{3/2}} \Big( - 3 \Ys(m) + 2\Qs(m) \pm o(1) \Big)$.
\item[$(c)$] $\Ex \big[ \Delta \Xs(m) \big] \, \in \, \ds\frac{4t}{n^{3/2}} \bigg( - \Xs(m) - 4\Ys(m) + 4\Qs(m) \pm o(1) \Big)$.
\end{itemize}
\end{lemma}

We begin by calculating the expected step-change in the variables $\Xb(m)$, $\Yb(m)$ and $Q(m)$. 

\begin{lemma}\label{Qeq}
For every $m \in \N$, 
$$\Ex \big[ \Delta Q(m) \big] \, = \, - \, \Yb(m) - 1.$$
\end{lemma}

\begin{proof}
This is trivial, since if edge $e$ is chosen in step $m+1$, then $\Delta Q(m) = - Y_e(m) - 1$. 
\end{proof}

\begin{lemma}\label{Ybeq}
Let $\omega \cdot n^{3/2} < m \le m^*$. If $\X(m) \cap \Y(m) \cap \Q(m)$ holds, then
$$\Ex \big[ \Delta \Yb(m) \big] \, \in \,  \frac{1}{Q(m)} \Big(  - \Yb(m)^2 + \Xb(m) - 2 \cdot \Var\big(Y_e(m)\big) \pm  O\big( \Yt(m) \big) \Big).$$
\end{lemma}

\begin{lemma}\label{Xbeq}
Let $\omega \cdot n^{3/2} < m \le m^*$. If $\X(m) \cap \Y(m) \cap \Q(m)$ holds, then
$$\Ex \big[ \Delta \Xb(m) \big] \, \in \, \frac{1}{Q(m)} \bigg( - 2 \cdot \Xb(m) \Yb(m) - 3 \cdot \Cov(X,Y) \pm O\Big( \Xt(m) + \Yt(m)^2 \Big) \bigg).$$
\end{lemma}

Recall first that the $Y$-graph has vertex set $O(G_m)$, and an edge between each pair $\{f,f'\}$ such that $f' \in Y_f(m)$. We begin with a simple observation about this graph.

\begin{obsa}\label{obs:Y:trianglefree}
The $Y$-graph is triangle-free, i.e., for every $m \in \N$, there do not exist three distinct open edges $e$, $f$ and $g$ of $G_m$ with $e \in Y_f(m)$, $f \in Y_g(m)$ and $g \in Y_e(m)$.
\end{obsa}

\begin{proof}
This follows from the fact that $G_m$ is triangle-free. Indeed, let $\{e, f, g \} \subset O(G_m)$ be open edges which form a triangle in the $Y$-graph, and suppose first that they have a common endpoint. Then the other endpoints of $e$, $f$ and $g$ form a triangle in $G_m$, which is a contradiction. On the other hand, if $e$, $f$ and $g$ do not share a common endpoint, then they must form a triangle (since they are pairwise intersecting), which contradicts (e.g.) the assumption that $e \in Y_f(m)$. It follows that no such triple exists, as claimed.
\end{proof}

In order to prove Lemmas~\ref{Ybeq} and~\ref{Xbeq}, we shall use the variables
$$\YY(m) = \sum_{e \in Q(m)} Y_e(m) \qquad \text{and} \qquad \XX(m) = \sum_{e \in Q(m)} X_e(m),$$
which are exactly twice the number of edges in the $Y$-graph, and six times the number of open triangles in $G_m$, respectively. 

We first bound the expected change in $\YY(m)$. 

\begin{lemmaa}\label{YY1}
$$\Ex\big[ \Delta \YY(m) \big] \, = \, \frac{1}{Q(m)} \bigg( \XX(m) \, - \, 2 \sum_{e \in Q(m)} Y_e(m)^2 \bigg).$$
\end{lemmaa}

\begin{proof}
Suppose we add edge $e \in O(G_m)$ in step $m+1$. 
We claim that
\begin{equation}\label{eq:deltaYY}
\Delta \YY(m) \, = \, X_e(m) - 2 \sum_{f \in Y_e(m)} Y_f(m).
\end{equation}
To see this, observe that in step $m+1$ we remove from the $Y$-graph the vertices corresponding to each $f \in Y_e(m)$, and add a matching between the vertices corresponding to $X_e(m)$. Since the $Y$-graph is triangle-free, by Observation~\ref{obs:Y:trianglefree}, and each $f \in Y_e(m)$ is (by definition) a $Y$-neighbour of $e$, it follows that the number of edges of the $Y$-graph which are removed is exactly $\sum_{f \in Y_e(m)} Y_f(m)$. Since $\YY(m)$ is equal to twice the number of edges in the $Y$-graph,~\eqref{eq:deltaYY} follows.

Summing over edges $e \in Q(m)$, it follows that
$$\Ex\big[ \YY(m+1) - \YY(m) \,\big|\, G_m \big] \, = \, \frac{1}{Q(m)} \sum_{e \in Q(m)} \Big[X_e(m) - 2\sum_{f \in Y_e(m)} Y_f(m) \Big ].$$
Now recall that $f \in Y_e(m)$ if and only if $e \in Y_f(m)$, and thus 
$$\sum_{e \in Q(m)} \sum_{f \in Y_e(m)} Y_f(m) \,=\, \sum_{f \in Q(m)} Y_f(m)^2.$$
(Indeed, this is simply the number of walks of length two in the $Y$-graph.) Hence
$$\Ex\big[ \Delta \YY(m) \big] \, = \, \frac{\XX(m)}{Q(m)} \, - \, \frac{2}{Q(m)}\sum_{e \in Q(m)} Y_e(m)^2,$$
as required.
\end{proof}

Recall the following simple lemma, which was stated (but not proved) in~\cite[Section~5]{FGMO}.

\setcounter{subsubsection}{5}
\setcounter{thm}{27}

\begin{lemma}\label{A/B}
$$\Ex \bigg[ \Delta \left( \frac{A(m)}{B(m)} \right) \bigg] \, = \, \frac{\Ex \big[ \Delta A(m) \big]}{B(m)} \, - \, \frac{1}{B(m)}\Ex\Bigg[\frac{A(m+1) \Delta B(m)}{B(m+1)} \,\Big| \, G_m \Bigg]$$
\end{lemma}

\begin{proof}
We have
\begin{multline*}
\Ex\Bigg[ \frac{A(m+1)}{B(m+1)} - \frac{A(m)}{B(m)} \,\Big| \, G_m \Bigg] = \Ex\Bigg[\frac{B(m)A(m+1)-B(m+1)A(m)}{B(m)(B(m+1)} \,\Big| \, G_m \Bigg]\\
 = \frac{1}{B(m)} \Ex\Bigg[\frac{B(m+1)\big(A(m+1)-A(m)\big) + \big(B(m)-B(m+1)\big)A(m+1)}{B(m+1)} \,\Big| \, G_m \Bigg]\\
 = \frac{\Ex\big[ \Delta A(m)\big]}{B(m)} \, + \, \frac{1}{B(m)}\Ex\Bigg[\frac{\big(B(m)-B(m+1)\big ) A(m+1)}{B(m+1)} \,\Big| \, G_m \Bigg],
\end{multline*}
as claimed.
\end{proof}

Recall that $\Var\big( Y_e(m) \big) = \Ex\big[ Y_e(m)^2 \big] - \Yb(m)^2$, where the expectation is over the choice of the edge $e \in O(G_m)$, and that the event $\X(m) \cap \Y(m)$ implies that
\begin{equation}\label{eq:recallXYbounds}
X_e(m) \in \big( 2 \pm o(1) \big) e^{-8t^2} n \quad \text{and} \quad Y_e(m) \in \big( 4 \pm o(1) \big) t e^{-4t^2} \sqrt{n}
\end{equation}
for every $e \in O(G_m)$. 
We can now prove Lemma~\ref{Ybeq}.

\begin{proof}[Proof of Lemma~\ref{Ybeq}]
By Lemma~\ref{A/B}, we have
\begin{equation}\label{eq:deltaYb}
\Ex \big[ \Delta \Yb(m) \big] \, = \, \frac{\Ex \big[ \Delta \YY(m) \big]}{Q(m)} \, - \, \frac{1}{Q(m)}\Ex\Bigg[\frac{\YY(m+1) \Delta Q(m)}{Q(m+1)} \Bigg].
\end{equation}
Now, we claim that
$$\frac{\YY(m+1)}{Q(m+1)} \, \in \, \left( 1 \pm \frac{1}{\Yt(m)^2} \right) \frac{\YY(m)}{Q(m)} \, \subset \, \Yb(m) \,\pm\, 1.$$
To see this, note that by~\eqref{eq:deltaYY} 
and~\eqref{eq:recallXYbounds}, since $\X(m) \cap \Y(m) \cap \Q(m)$ holds we have 
$$\frac{|\Delta \YY(m)|}{\YY(m)} \, \le \, 3 \cdot \frac{\Xt(m) + \Yt(m)^2}{\Yt(m) \cdot \Qt(m)} \, \ll \, \frac{1}{\Yt(m)^2} \quad \text{and} \quad \frac{|\Delta Q(m)|}{Q(m)} \, \le \, 2 \cdot \frac{\Yt(m)}{\Qt(m)} \, \ll \, \frac{1}{\Yt(m)^2}.$$ 
Since $\Ex\big[ \Delta Q(m) \big] = - \Yb(m) - 1$,  and using~\eqref{eq:deltaYb}, it follows that
$$\Ex \big[ \Delta \Yb(m) \big] \, \in \, \frac{\Ex \big[ \Delta \YY(m) \big]}{Q(m)} \, + \, \frac{\Yb(m)^2 \pm O\big( \Yt(m) \big)}{Q(m)}.$$
By Lemma~\ref{YY1}, and since $\Ex\big[ Y_e(m)^2 \big] = \Var\big( Y_e(m) \big) + \Yb(m)^2$, this implies that
\begin{align*}
\Ex \big[ \Delta \Yb(m) \big] & \, \in \, \frac{1}{Q(m)} \left(  \frac{\XX(m)}{Q(m)} \, - \, 2 \Big( \Yb(m)^2 + \Var\big(Y_e(m)\big) \Big) \right) \, + \, \frac{\Yb(m)^2 \pm  O\big( \Yt(m) \big)}{Q(m)}\\
& \, = \, \frac{1}{Q(m)} \Big( \Xb(m) - \Yb(m)^2 - 2 \cdot \Var\big(Y_e(m)\big) \pm  O\big( \Yt(m) \big) \Big),
\end{align*}
as required.
\end{proof}

The proof for $\Xb(m)$ is similar. Recall that $\Cov(X,Y) = \Ex\big[ X_e \cdot Y_e \big] - \Xb \cdot \Yb$, where the expectation is over the (uniformly random) choice of the edge $e \in O(G_m)$. 

\begin{lemmaa}\label{XX1}
If $\X(m) \cap \Y(m)$ holds, then
$$\Ex\big[ \Delta \XX(m) \big] \, \in \, - \frac{3}{Q(m)} \sum_{f \in Q(m)} X_f(m) \cdot Y_f(m) \, \pm \, O\Big( \Xt(m) + \Yt(m)^2 \Big).$$
\end{lemmaa}

\begin{proof}
Observe first that $\XX(m)$ is simply the number of labelled open triangles in $G_m$ (i.e., six times the number of unlabelled open triangles). We claim that if edge $e$ is chosen in step $m + 1$ of the triangle-free process, then
\begin{equation}\label{eq:deltaXX}
- 3 \cdot X_e(m)  \, \le \, \Delta \XX(m) + 3 \sum_{f \in Y_e(m)} X_f(m) \, \le \, 6 \cdot Y_e(m)^2
\end{equation}
To see the lower bound, note that no new open triangles are created, and each edge $f \in Y_e(m) \cup \{e\}$ which is closed in step $m+1$ destroys at most $X_f(m)/2$ unlabelled open triangles. On the other hand, an open triangle is destroyed in two different ways  only if it has two edges in $Y_e(m) \cup \{e\}$, and hence we have double-counted at most $Y_e(m)^2$ unlabelled open triangles, which gives the upper bound in~\eqref{eq:deltaXX}.

Now, summing over edges $e \in Q(m)$, it follows that
$$\Ex\big[ \Delta \XX(m) \big] \, \in \, - \frac{3}{Q(m)} \sum_{e \in Q(m)} \sum_{f \in Y_e(m)} X_f(m) \, \pm \, O\Big( \Xt(m) + \Yt(m)^2 \Big).$$
Since $f \in Y_e(m)$ if and only if $e \in Y_f(m)$, it follows that
$$\sum_{e \in Q(m)} \sum_{f \in Y_e(m)} X_f(m) \,=\, \sum_{f \in Q(m)} X_f(m) \cdot Y_f(m),$$ 
and hence
$$\Ex\big[ \Delta \XX(m) \big] \, \in \, - \frac{3}{Q(m)} \sum_{f \in Q(m)} X_f(m) \cdot Y_f(m) \, \pm \, O\Big( \Xt(m) + \Yt(m)^2 \Big),$$
as required.
\end{proof}

We can now prove Lemma~\ref{Xbeq}.

\begin{proof}[Proof of Lemma~\ref{Xbeq}]
By Lemma~\ref{A/B} we have
$$\Ex \big[ \Delta \Xb(m) \big] \, = \, \frac{\Ex \big[ \Delta \XX(m) \big]}{Q(m)} \, - \, \frac{1}{Q(m)}\Ex\Bigg[\frac{\XX(m+1) \Delta Q(m)}{Q(m+1)} \Bigg],$$
and we have
$$\frac{\XX(m+1)}{Q(m+1)} \, \in \, \left( 1 \pm \frac{1}{\Yt(m)^2} \right) \frac{\XX(m)}{Q(m)} \, \subseteq \, \Xb(m) \,\pm\, 1.$$
To see this, note that by~\eqref{eq:recallXYbounds} and~\eqref{eq:deltaXX}, 
since $\X(m) \cap \Y(m) \cap \Q(m)$ holds we have 
$$\frac{|\Delta \XX(m)|}{\XX(m)} \, \le \, 4 \cdot \frac{\Xt(m) \cdot \Yt(m)}{\Xt(m) \cdot \Qt(m)} \, \ll \, \frac{1}{\Yt(m)^2} \quad \text{and} \quad \frac{|\Delta Q(m)|}{Q(m)} \, \le \, 2 \cdot \frac{\Yt(m)}{\Qt(m)} \, \ll \, \frac{1}{\Yt(m)^2}.$$ 
Since $\Ex\big[ \Delta Q(m) \big] = - \Yb(m) - 1$, it follows that
$$\Ex \big[ \Delta \Xb(m) \big] \, \in \, \frac{\Ex \big[ \Delta \XX(m) \big]}{Q(m)} \, + \, \frac{\Xb(m)\Yb(m) \pm O\big( \Xb(m) + \Yb(m) \big)}{Q(m)}.$$
By Lemma~\ref{XX1}, and since $\Ex\big[ X_e \cdot Y_e \big] = \Xb \cdot \Yb + \Cov(X,Y)$, this implies that
$$\Ex \big[ \Delta \Xb(m) \big] \, \in \, \frac{1}{Q(m)} \bigg( - 2 \cdot \Xb(m) \Yb(m) - 3 \cdot \Cov(X,Y) \pm O\Big( \Xt(m) + \Yt(m)^2 \Big) \bigg),$$
as required.
\end{proof}

We can now deduce Lemma~\ref{whirlpool} via a rather tedious calculation. 

\begin{proof}[Proof of Lemma~\ref{whirlpool}] 
Let $\omega \cdot n^{3/2} < m \le m^*$, and suppose that $\X(m) \cap \Y(m) \cap \Q(m)$ holds. By Lemma~\ref{A/B} we have
$$\Ex\big[ \Delta \Qs(m) \big] \, = \, \frac{\Ex \big[ \Delta Q(m) - \Delta \Qt(m) \big]}{g_q(t) \Qt(m)} \, - \, \frac{\Delta \big( g_q(t) \Qt(m) \big)}{g_q(t) \Qt(m)} \cdot \Ex\Big[\Qs(m+1) \,\big| \, G_m \Big].$$
Recall that $\Ex\big[ \Delta Q(m) \big] = - \Yb(m) - 1$ and $\Delta \Qt(m) \in - \Yt(m) \pm 1$, and note that
$$\Delta \big( g_q(t) \Qt(m) \big) \, \in \, - \frac{4t}{n^{3/2}} \cdot g_q(t) \Qt(m) \pm 1,$$
since $g_q(t) \Qt(m)$ is equal to $e^{-2t^2}$ times some function of $n$.  It follows that
$$\bigg( 1 - \frac{\big( 4 + o(1)\big) t}{n^{3/2}} \bigg) \Ex\big[ \Delta \Qs(m) \big] \, \in \, \frac{\Yt(m) - \Yb(m)}{g_q(t) \Qt(m)} + \frac{4t}{n^{3/2}} \cdot \Qs(m)  \pm \frac{3}{g_q(t) \Qt(m)},$$
since the event $\Q(m)$ implies that $|\Ys(m)| + |\Qs(m)| \le 2$, and hence
$$\Ex\big[ \Delta \Qs(m) \big] \, \in \, - \frac{\Ys(m) \cdot \Yt(m)}{\Qt(m)} + \frac{4t}{n^{3/2}} \cdot \Qs(m)  \pm \frac{1}{n^{3/2}} \, \subset \, \frac{4t}{n^{3/2}} \Big( - 2 \Ys(m) + \Qs(m) \pm o(1) \Big),$$
as required.

We turn next to $\Ys(m)$. Observe that, by Lemma~\ref{A/B}, we have
\begin{equation}\label{eq:ExdeltaYs}
\Ex\big[ \Delta \Ys(m) \big] \, = \, \frac{\Ex \big[ \Delta \Yb(m) - \Delta \Yt(m) \big]}{g_q(t) \Yt(m)} \, - \, \frac{\Delta \big( g_q(t) \Yt(m) \big)}{g_q(t) \Yt(m)} \cdot \Ex\Big[\Ys(m+1) \,\big| \, G_m \Big],
\end{equation}
and recall that, by Lemma~\ref{Ybeq}, if $\X(m) \cap \Y(m) \cap \Q(m)$ holds then
\begin{equation}\label{eq:ExdeltaYb}
\Ex \big[ \Delta \Yb(m) \big] \, \in \, \frac{1}{Q(m)} \Big( - \Yb(m)^2 + \Xb(m) - 2 \cdot \Var\big(Y_e(m)\big) \pm  O\big( \Yt(m) \big) \Big).
\end{equation}
Now, since $\Yt(m) = 4t e^{-4t^2} \sqrt{n}$, a simple calculation gives 
\begin{equation}\label{eq:deltaYt}
\Delta \Yt(m) \, \in \, - \left( \frac{8t^2 - 1}{t \cdot n^{3/2}} \right) \Yt(m) \pm \frac{1}{n^2} \, \subseteq \, \frac{-\Yt(m)^2 + \Xt(m)}{\Qt(m)} \pm \frac{1}{n^2},
\end{equation}
and similarly, since $g_q(t) \Yt(m) = 4t e^{-2t^2} n^{1/4} (\log n)^3$ and $t \ge \omega$, we obtain
\begin{equation}\label{eq:delta:gqYt}
\Delta \big( g_q(t) \Yt(m) \big) \in - \left( \frac{4t^2 - 1}{t \cdot n^{3/2}} \right) g_q(t) \Yt(m) \pm \frac{1}{n^2}  \, \subseteq \, - \big( 1 \pm o(1) \big) \cdot \frac{4t}{n^{3/2}} \cdot g_q(t) \Yt(m).
\end{equation}
Moreover, since $\Y(m)$ holds, we have
\begin{equation}\label{eq:varYsmall}
\Var\big( Y_e(m) \big) \, \le \, g_y(t)^2 \Yt(m)^2 \, \ll \, g_q(t) \Yt(m)^2.
\end{equation}
Recalling that $\Yb(m) = \big( 1 + g_q(t) \Ys(m) \big) \Yt(m)$ and $Q(m) = \big( 1 + g_q(t) \Qs(m) \big) \Qt(m)$, it follows from~\eqref{eq:ExdeltaYb},~\eqref{eq:deltaYt} and~\eqref{eq:varYsmall} that
\begin{multline*}
\Ex\big[ \Delta \Yb(m) \big] - \Delta \Yt(m)  \, \in \, - \, \frac{\Yt(m)^2}{Q(m)} \Big( \big( 1 + g_q(t) \Ys(m) \big)^2 - \big( 1 + g_q(t) \Qs(m) \big) \Big) \\
 \, + \, \frac{\Xt(m)}{Q(m)} \Big( \big( 1 + g_q(t) \Xs(m) \big)  - \big( 1 + g_q(t) \Qs(m) \big) \Big) \, \pm \, \frac{o(1) \cdot g_q(t) \Yt(m)^2}{\Qt(m)}. 
\end{multline*}
Observing that $\Xt(m) \ll \Yt(m)^2$, which holds since $t \ge \omega$, we deduce that 
\begin{equation}\label{eq:ExdeltaYb:Yt}
\Ex\big[ \Delta \Yb(m) \big] - \Delta \Yt(m) \, \in \, \frac{g_q(t) \Yt(m)^2}{Q(m)} \Big( - 2 \Ys(m) + \Qs(m) \pm o(1) \Big)
\end{equation}
if $|\Xs(m)| + |\Ys(m)| + |\Qs(m)| = O(1)$, which follows from the event $\Q(m)$.

Now, combining~\eqref{eq:ExdeltaYs},~\eqref{eq:delta:gqYt} and~\eqref{eq:ExdeltaYb:Yt}, we obtain
$$\Ex\big[ \Delta \Ys(m) \big] \in \frac{\Yt(m)}{Q(m)} \Big( - 2 \Ys(m) + \Qs(m) \pm o(1) \Big)  + \big( 1 \pm o(1) \big) \frac{4t}{n^{3/2}} \Big( \Ys(m) + \Ex\big[ \Delta \Ys(m) \big] \Big),$$
which easily implies that 
$$\Ex\big[ \Delta \Ys(m) \big] \, \in \, \frac{4t}{n^{3/2}} \Big( - 3 \Ys(m) + 2 \Qs(m) \pm o(1) \Big),$$
as required.

Finally, we turn to $\Xs(m)$. By Lemma~\ref{A/B}, we have
\begin{equation}\label{eq:ExdeltaXs}
\Ex\big[ \Delta \Xs(m) \big] \, = \, \frac{\Ex \big[ \Delta \Xb(m) - \Delta \Xt(m) \big]}{g_q(t) \Xt(m)} \, - \, \frac{\Delta \big( g_q(t) \Xt(m) \big)}{g_q(t) \Xt(m)} \cdot \Ex\Big[ \Xs(m+1) \,\big| \, G_m \Big],
\end{equation}
and by Lemma~\ref{Xbeq}, if $\X(m) \cap \Y(m) \cap \Q(m)$ holds then
\begin{equation}\label{eq:ExdeltaXb}
\Ex \big[ \Delta \Xb(m) \big] \, \in \, \frac{1}{Q(m)} \bigg( - 2 \cdot \Xb(m) \Yb(m) - 3 \cdot \Cov(X,Y) \pm O\Big( \Xt(m) + \Yt(m)^2 \Big) \bigg).
\end{equation}
Now, since $\Xt(m) = 2 e^{-8t^2} n$, a simple calculation gives 
\begin{equation}\label{eq:deltaXt}
\Delta \Xt(m) \, \in \, - \frac{16t}{n^{3/2}} \cdot \Xt(m) \pm \frac{1}{n^2} \, \subseteq \, - \frac{2 \cdot \Xt(m) \cdot \Yt(m)}{\Qt(m)} \pm \frac{o(1)}{n^{3/2}},
\end{equation}
and similarly, since $g_q(t) \Xt(m) = 2 e^{-6t^2} n^{3/4} (\log n)^3$ and $t \ge \omega$, we obtain
\begin{equation}\label{eq:delta:gqXt}
\Delta \big( g_q(t) \Xt(m) \big) \in - \frac{12t}{n^{3/2}} \cdot g_q(t) \Xt(m) \pm \frac{1}{n^2}  \, \subseteq \, - \big( 1 \pm o(1) \big) \cdot \frac{12t}{n^{3/2}} \cdot g_q(t) \Xt(m).
\end{equation}
Moreover, since $\X(m) \cap \Y(m)$ holds, we have
\begin{equation}\label{eq:covXYsmall}
\Cov\big( X_e(m), Y_e(m) \big) \, \le \, g_x(t) \Xt(m) \cdot g_y(t) \Yt(m) \, \ll \, g_q(t) \Xt(m) \Yt(m),
\end{equation}
and note also that $\Xt(m) + \Yt(m)^2 \ll  g_q(t) \Xt(m) \Yt(m)$ for every $m \le m^*$. 

Combining~\eqref{eq:ExdeltaXb},~\eqref{eq:deltaXt} and~\eqref{eq:covXYsmall}, it follows that $\Ex\big[ \Delta \Xb(m) \big] - \Delta \Xt(m)$ is contained in
$$\frac{2 \cdot \Xt(m) \Yt(m)}{Q(m)} \Big( - \big( 1 + g_q(t) \Xs(m) \big) \big( 1 + g_q(t) \Ys(m) \big) + \big( 1 + g_q(t) \Qs(m) \big)  \pm o(1) \cdot g_q(t) \Big),$$
and hence
\begin{equation}\label{eq:ExdeltaXb:Xt}
\Ex\big[ \Delta \Xb(m) \big] - \Delta \Xt(m) \, \in \, \frac{2 \cdot g_q(t) \Xt(m) \Yt(m)}{Q(m)} \Big( - \Xs(m) -  \Ys(m) 
+ \Qs(m) \pm o(1) \Big).
\end{equation}

Finally, it follows from~\eqref{eq:ExdeltaXs},~\eqref{eq:delta:gqXt} and~\eqref{eq:ExdeltaXb:Xt} that
\begin{align*}
& \Ex\big[ \Delta \Xs(m) \big] \in \frac{2 \cdot \Yt(m)}{Q(m)} \Big( - \Xs(m) -  \Ys(m) 
+ \Qs(m) \pm o(1) \Big) \\
& \hspace{5cm} + \, \big( 1 \pm o(1) \big) \cdot \frac{12t}{n^{3/2}} \cdot \Big( \Xs(m) + \Ex\big[ \Delta \Xs(m) \big] \Big),
\end{align*}
which, since $|\Xs(m)| + |\Ys(m)| + |\Qs(m)| = O(1)$, by $\Q(m)$, easily implies that 
$$\Ex\big[ \Delta \Xs(m) \big] \, \in \,  \frac{4t}{n^{3/2}} \Big( - \Xs(m) - 4 \Ys(m) + 4 \Qs(m) \pm o(1) \Big),$$
as required.
\end{proof}

Finally, let's prove Lemma~\ref{XYQalpha}. Recall that 
$$\left( \begin{array}{c}
\Ys \\
\Qs
\end{array} \right) \, = \, 
\eps \left( \begin{array}{cc}
4 & 5 \\
4 & 3
\end{array} \right)
\left( \begin{array}{c}
\lambda \\
\mu
\end{array} \right)$$
and $\Lambda(m) \, = \, \lambda(m)^2 + \mu(m)^2$.

\setcounter{subsubsection}{6}
\setcounter{thm}{6}

\begin{lemma}\label{XYQalpha}
Let $\omega \cdot n^{3/2} < m \le m^*$, and suppose that $\X(m) \cap \Y(m) \cap \Q(m)$ holds. Then
$$| \Delta \Xs(m)| + |\Delta \Ys(m)| + |\Delta \Qs(m)| \, \le \, \frac{(\log n)^3}{g_q(t) \cdot n^{3/2}},$$
and hence
$$ |\Delta \Lambda(m) | \, \le \, \frac{(\log n)^4}{g_q(t) \cdot n^{3/2}} \qquad \text{and} \qquad  \Ex\big[ | \Delta \Lambda(m) | \big] \, \le \, \frac{(\log n)^4}{g_q(t) \cdot n^{3/2}}.$$
\end{lemma}

We shall use the following bound, which follows easily from Lemma~\ref{lem:chainstar}. 

\begin{lemmaa}\label{lem:chainstarXYQ}
Let $A \in \{\Xb,\Yb,Q\}$ and let $\At$ be the corresponding member of $\{\Xt,\Yt,\Qt\}$. For every $\omega \cdot n^{3/2} < m \le m^*$, if $\Q(m)$ holds, then
$$| \Delta A^*(m) | \, \le \, \frac{3}{g_q(t)} \cdot  \left( \frac{| \Delta A(m) |}{\At(m)} \,+\, \frac{\log n}{n^{3/2}} \right).$$
\end{lemmaa}

\begin{proof}
Since $\At(m)$ is equal to either $te^{-kt^2}$ or $e^{-k t^2}$ times some function of $n$, where $k \in \{4,8\}$, $g_q(t) \At(m)$ is equal to either $t e^{-(k-2)t^2}$ or $e^{-(k-2)t^2}$  times some function of $n$, and $t > \omega$, we have
$$\Delta \At(m) \in \frac{- 2k t \pm o(1)}{n^{3/2}} \cdot  \At(m) \quad \text{and} \quad \Delta \big( g_q(t) \At(m) \big) \in  \frac{- (2k - 4)t \pm o(1)}{n^{3/2}} \cdot g_q(t) \At(m),$$
and hence
$$|\Delta \At(m) | \, \ll \, \frac{\log n}{n^{3/2}} \cdot \At(m)  \quad \text{and} \quad |\Delta \big( g_q(t) \At(m) \big) | \ll \frac{\log n}{n^{3/2}} \cdot g_q(t) \At(m).$$ 
Moreover, the event $\Q(m)$ implies that $A(m) \le \big( 1 + g_q(t) \big) \At(m)$, and $g_q(t) \ll 1$. Thus, applying Lemma~\ref{lem:chainstar}, we obtain
$$| \Delta A^*(m) | \, \le \, \frac{3}{g_q(t)} \left( \frac{| \Delta A(m) |}{\At(m)} \,+\, \frac{\log n}{n^{3/2}} \right),$$
as claimed.
\end{proof}



\begin{proof}[Proof of Lemma~\ref{XYQalpha}]
By Lemma~\ref{lem:chainstarXYQ}, we have 
$$| \Delta \Xs(m)| \,+\, |\Delta \Ys(m)| \,+\, |\Delta \Qs(m)| \, \le \, \frac{6}{g_q(t)}  \left( \frac{ | \Delta \Xb(m)| }{ \Xt(m) } + \frac{ | \Delta \Yb(m)| }{ \Yt(m) } + \frac{ | \Delta Q(m)| }{ \Qt(m) } + \frac{\log n}{n^{3/2}} \right),$$
so it will suffice to prove that 
$$\max\bigg\{ \frac{ | \Delta \Xb(m)| }{ \Xt(m) }, \frac{ | \Delta \Yb(m)| }{ \Yt(m) }, \frac{ | \Delta Q(m)| }{ \Qt(m) } \bigg\} \, \le \,  \frac{\log n}{n^{3/2}}.$$
For $Q(m)$, the bound is trivial, since if $\Y(m)$ holds then $\Delta Q(m) \in - ( 1 \pm \eps ) \Yt(m)$. 
To prove the bound for $\Yb(m)$, recall first (see~\eqref{eq:deltaYY}) that
$$\Delta \YY(m) \, = \, X_e(m) - 2 \sum_{f \in Y_e(m)} Y_f(m),$$ 
and observe that therefore, since $\X(m) \cap \Y(m)$ holds, 
$$\Delta \YY(m) \, \in \, (1 \pm \eps) \Big( \Xt(m) - 2 \cdot \Yt(m)^2 \Big).$$ 
Since $\Xt(m) \ll \Yt(m)^2$ for $t \ge \omega$, and $\Delta Q(m) \in - ( 1 \pm \eps ) \Yt(m)$, it follows that
$$\Yb(m+1) \, = \, \frac{\YY(m+1)}{Q(m+1)} \, \in \, \bigg( 1 \pm \frac{\Yt(m)}{\Qt(m)} \bigg) \bigg( \frac{\YY(m) + \Xt(m) - 2 \cdot \Yt(m)^2}{Q(m) - \Yt(m)} \bigg),$$
and hence
$$|\Delta \Yb(m)| \, \le \, \bigg( 1 + \frac{\Yt(m)}{\Qt(m)} \bigg) \bigg( \frac{\YY(m) + \Xt(m) - 2 \cdot \Yt(m)^2}{Q(m) - \Yt(m)} \bigg) - \Yb(m) \, \le \, \frac{3 \cdot \Yt(m)^2}{\Qt(m)},$$
as required.

Finally, to prove the bound for $\Xb(m)$, recall from~\eqref{eq:deltaXX} that
$$- 3 \cdot X_e(m)  \, \le \, \Delta \XX(m) + 3 \sum_{f \in Y_e(m)} X_f(m) \, \le \, 6 \cdot Y_e(m)^2$$
and observe that therefore, since $\X(m) \cap \Y(m)$ holds, 
$$\Delta \XX(m) \, \in \, - (3 \pm \eps) \cdot \Xt(m) \Yt(m).$$ 
As before, it follows that
$$\Xb(m+1) \, = \, \frac{\XX(m+1)}{Q(m+1)} \, \in \, \bigg( 1 \pm \frac{\Yt(m)}{\Qt(m)} \bigg) \bigg( \frac{\XX(m) - 3 \cdot \Xt(m) \Yt(m)}{Q(m) - \Yt(m)} \bigg),$$
and hence
$$|\Delta \Xb(m)| \, \le \, \bigg( 1 + \frac{\Yt(m)}{\Qt(m)} \bigg) \bigg( \frac{\XX(m) - 3 \cdot \Xt(m) \Yt(m)}{Q(m) - \Yt(m)} \bigg) - \Xb(m) \, \le \, \frac{4 \cdot \Xt(m) \Yt(m)}{\Qt(m)},$$
as required. As noted above, it follows from Lemma~\ref{lem:chainstarXYQ} and our bounds on $|\Delta Q(m)|$, $|\Delta \Yb(m)|$ and $|\Delta \Xb(m)|$ that
$$| \Delta \Xs(m)| + |\Delta \Ys(m)| + |\Delta \Qs(m)| \, \le \, \frac{(\log n)^3}{g_q(t) \cdot n^{3/2}}.$$

Finally, in order to deduce the claimed bounds on $|\Delta \Lambda(m) |$ and $\Ex\big[ | \Delta \Lambda(m) | \big]$, simply note that
$$|\Delta \Lambda(m) | \, \le \, 2 \Big( \big| \lambda(m) \cdot \Delta \lambda(m) \big| + \big| \mu(m) \cdot \Delta \mu(m) \big| \Big) + |\Delta \lambda(m)|^2 + |\Delta \mu(m)|^2,$$
and that 
$$|\Delta \lambda(m)| + |\Delta \mu(m)| = O\big( |\Delta \Ys(m)| + |\Delta \Qs(m)| \big) \quad \text{and} \quad |\lambda(m)| + |\mu(m)| = O(1),$$
since the event $\Q(m)$ holds. It follows immediately that
$$ |\Delta \Lambda(m) | \, \le \, \frac{(\log n)^4}{g_q(t) \cdot n^{3/2}},$$
and therefore that the same bound holds for $\Ex\big[ | \Delta \Lambda(m) | \big]$, as required.
\end{proof}

\section{The martingale inequalities}\label{AppMartSec}

In this section, for completeness, we shall give the proof our main martingale lemma~\cite[Lemma~3.1]{FGMO}. The proof below is taken 
from the survey of McDiarmid~\cite{Colin}. Recall that we assume throughout that for each martingale $M$ we consider, $M(m)$ depends only on the graph $G_{m+r}$ for each $m \in [0,s]$, for some (fixed) $r \in \N$ depending on~$M$. 

\setcounter{subsubsection}{3}
\setcounter{subsection}{1}
\setcounter{thm}{0}
\setcounter{thma}{0}

\begin{lemma}[Theorem~3.15 of~\cite{Colin}]\label{mart}
Let $M$ be a super-martingale, defined on $[0,s]$, such that 
\begin{equation}\label{eq:martlemma}
|\Delta M(m)| \le \alpha \qquad  \text{and} \qquad \Ex\big[ |\Delta M(m)| \big] \le \beta
\end{equation}
for every $m \in [0,s-1]$. Then, for every $0 \le a \le \beta s$,
$$\Pr\big( M(s) > M(0) + a \big) \, \le \, \exp\left( - \frac{a^2}{4 \alpha \beta s} \right).$$
\end{lemma}

We begin with a straightforward observation.

\begin{lemmaa}[Lemma~2.8 of~\cite{Colin}]\label{Colin:28}
Let $X$ be a random variable with $\Ex[X] \le 0$ and $X \le b$, and for each $x \in \RR \setminus \{0\}$, set
$$g(x) \, = \, \frac{e^x - x - 1}{x^2}.$$
The function $g$ is increasing on $\RR \setminus \{0\}$, and 
\begin{equation}\label{eq:Exofetothex}
\Ex\big[ e^X \big] \, \le \, \exp\Big( g(b) \Ex\big[ X^2 \big] \Big).
\end{equation}
\end{lemmaa}

\begin{proof}
The fact that $g$ is increasing follows from simple calculus. Indeed, for each $x \in \RR \setminus \{0\}$,
$$g'(x) \, = \, \frac{(x-2)e^x + x + 2}{x^3} \, \ge \, 0$$
since $h(x) = (x-2)e^x + x + 2$ satisfies $h(0) = 0$ and $h'(x) = (x - 1)e^x + 1 \ge 0$ for every $x \in \RR$. To see the latter bound, simply note that $h'(0) = 0$ and $h''(x) = x e^x$. 

To deduce~\eqref{eq:Exofetothex}, simply note that
$$e^x \, = \, 1 + x + x^2 g(x) \, \le \, 1 + x + x^2 g(b)$$
for every $x \le b$ (setting $g(0) = 0$). Since $\Ex[X] \le 0$ and $X \le b$, it follows immediately that 
$$\Ex\big[ e^X \big] \, \le \, 1 + g(b) \Ex[X^2] \, \le \, \exp\Big( g(b) \Ex[X^2] \Big),$$
as required.
\end{proof}

We next prove another relatively straightforward preliminary lemma. 

\begin{lemmaa}[Lemma~3.16 of~\cite{Colin}]\label{Colin:316}
Let $M$ be a super-martingale defined on $[0,s]$, and let $r \in \N$. Suppose that $M(m)$ depends only on $G_{m+r}$ for each $m \in [0,s]$. Then, for any $h \in \RR$,
\begin{equation}\label{eq:secondcolinlemma}
\Ex\Big[ e^{h( M(s) - M(0) )} \,\Big|\, G_r \Big] \, \le \, \sup \bigg( \prod_{m=0}^{s-1} \Ex\Big[ e^{h \Delta M(m)} \,\big|\, G_{m+r} \Big] \,\Big|\, G_r \bigg).
\end{equation}
\end{lemmaa}

\begin{proof}
The proof is by induction on $s$. When $s = 0$ the result is trivial, since~\eqref{eq:secondcolinlemma} reduces to $1 \le 1$. So let $s \ge 1$, and suppose the result holds for $s -1$. Set $A = e^{h (M(s) - M(1))}$ and 
$$B \, = \, \prod_{m=1}^{s-1} \Ex\Big[ e^{h \Delta M(m)} \,\big|\, G_{m+r} \Big].$$
By the induction hypothesis, we have $\Ex\big[ A \,|\, G_{r+1} \big] \le \Ex\big[ B \,|\, G_{r+1} \big]$, and note that trivially $\sup\big( B \,|\, G_r \big) \le \sup\big( B \,|\, G_{r+1} \big)$. It follows that
\begin{multline*}
\Ex\Big[ e^{h( M(s) - M(0) )} \,\Big|\, G_r \Big] \, = \, \Ex\Big[ e^{h \Delta M(0)} \Ex\big[ A \,\big|\, G_{r+1} \big] \,\Big|\, G_r \Big] \\
\, \le \, \Ex\Big[ e^{h \Delta M(0)} \sup\big( B \,\big|\, G_{r+1} \big) \,\Big|\, G_r \Big] \, \le \, \Ex\Big[ e^{h \Delta M(0)} \sup\big( B \,\big|\, G_r \big) \,\Big|\, G_r \Big] \\
 \, = \, \sup\big( B \,\big|\, G_r \big) \Ex\big[ e^{h \Delta M(0)}  \,\big|\, G_r \big] \, = \, \sup \bigg( \prod_{m=0}^{s-1} \Ex\Big[ e^{h \Delta M(m)} \,\big|\, G_{m+r} \Big] \,\Big|\, G_r \bigg),
\end{multline*}
as required.
\end{proof}

We shall use one more trivial observation, which is an immediate consequence of~\cite[Lemma~2.4]{Colin}. (Here $\log$ denotes the natural logarithm.)

\begin{obsa}\label{Colin:24}
For every $0 \le x \le 1$, 
$$(1 + x) \log(1 + x) - x \, \ge \, \frac{x^2}{4}.$$
\end{obsa}

We are now ready to prove Lemma~\ref{mart}.

\begin{proof}[Proof of Lemma~\ref{mart}]
Note that $\Ex\big[ \Delta M(m) \big] \le 0$, since $M$ is a super-martingale, and recall that $|\Delta M(m)| \le \alpha$. Applying Lemma~\ref{Colin:28} to the random variable $X = h \cdot \Delta M(m)$, it follows that
\begin{equation}\label{eq:martproof1}
\Ex\Big[ e^{h \Delta M(m)} \,\big|\, G_{m+r} \Big] \, \le \, \exp\Big( h^2 g(h \alpha ) \Ex\big[ \big( \Delta M(m) \big)^2 \big] \Big).
\end{equation}
Next, note that
\begin{equation}\label{eq:martproof2}
\Ex\big[ \big( \Delta M(m) \big)^2 \big] \, \le \, \alpha \beta,
\end{equation}
since $|\Delta M(m)| \le \alpha$ and $\Ex\big[ |\Delta M(m)| \big] \le \beta$. 

Now, combining Lemma~\ref{Colin:316} with~\eqref{eq:martproof1} and~\eqref{eq:martproof2}, we obtain
\begin{align*}
\Ex\Big[ e^{h( M(s) - M(0) )} \,\Big|\, G_r \Big] & \, \le \, \sup \bigg( \prod_{m=0}^{s-1} \Ex\Big[ e^{h \Delta M(m)} \,\big|\, G_{m+r} \Big] \,\Big|\, G_r \bigg)\\
& \, \le \, \sup \bigg( \prod_{m=0}^{s-1} \Ex\Big[ \exp\Big( h^2 g(h \alpha ) \Ex\big[ \big( \Delta M(m) \big)^2 \big] \Big) \,\Big|\, G_{m+r} \Big] \,\Big|\, G_r \bigg) \\
& \, \le \, \exp\Big( h^2 g(h \alpha ) \alpha \beta s \Big).
\end{align*}
Thus, for any $h > 0$, by Markov's inequality, 
\begin{align*}
& \Pr\big( M(s) > M(0) + a \big) \, = \, \Pr \Big( e^{h( M(s) - M(0) )} \ge e^{ha} \Big)\\
& \hspace{2cm} \, \le \, e^{-ha} \Ex\Big[ e^{h( M(s) - M(0) )} \,\Big|\, G_r \Big] \, \le \, \exp\Big( -ha + h^2 g(h \alpha ) \alpha \beta s \Big).
\end{align*}
Setting $h = \frac{1}{\alpha} \log \big( \frac{a + \beta s}{\beta s} \big)$, and noting that
\begin{align*}
-ha + h^2 g(h \alpha ) \alpha \beta s & \, = \, - \frac{a}{\alpha} \log \bigg( \frac{a + \beta s}{\beta s} \bigg) + \frac{\beta s}{\alpha} \bigg( \frac{a + \beta s}{\beta s} - \log \bigg( \frac{a + \beta s}{\beta s} \bigg) - 1 \bigg)\\
& 
\, = \, \frac{\beta s}{\alpha} \bigg( -\, \bigg( 1 + \frac{a}{\beta s} \bigg) \log\bigg( 1 + \frac{a}{\beta s} \bigg) + \frac{a}{\beta s} \bigg).
\end{align*}
Finally, applying Observation~\ref{Colin:24} with $x = a / \beta s$, it follows that
$$\Pr\big( M(s) > M(0) + a \big) \, \le \, \exp\bigg( - \frac{\beta s}{\alpha} \cdot \frac{1}{4} \bigg( \frac{a}{\beta s} \bigg)^2 \bigg) \, = \, \exp\bigg( - \frac{a^2}{4 \alpha \beta s} \bigg),$$
since $a \le \beta s$, as required.
\end{proof}